\documentclass[a4paper,10pt,reqno]{amsart}

\usepackage{amsmath, amssymb, amsthm,amsbsy}
\usepackage{enumitem}
\usepackage{mathrsfs, eucal}
\usepackage{bbm}
\usepackage[all]{xy}
\usepackage{stackrel}

\numberwithin{equation}{subsection}
\numberwithin{figure}{subsection}
\numberwithin{table}{section}
\allowdisplaybreaks

\theoremstyle{definition}
\newtheorem{thm}{Theorem}[subsection]
\newtheorem{prp}[thm]{Proposition}
\newtheorem{lem}[thm]{Lemma}
\newtheorem{cor}[thm]{Corollary}
\newtheorem{dfn}[thm]{Definition}
\newtheorem{asp}[thm]{Assumption}
\newtheorem{ntn}[thm]{Notation}
\newtheorem{fct}[thm]{Fact}
\newtheorem{rmk}[thm]{Remark}
\newtheorem*{thm*}{Theorem}
\newtheorem*{prp*}{Proposition}
\newtheorem*{lem*}{Lemma}
\newtheorem*{cor*}{Corollary}
\newtheorem*{dfn*}{Definition}
\newtheorem*{ntn*}{Notation}
\newtheorem*{fct*}{Fact}
\newtheorem*{rmk*}{Remark}
\newtheorem*{eg*}{Example}

\newcommand{\ol}{\overline}

\newcommand{\wh}{\widehat}
\newcommand{\wt}{\widetilde}
\newcommand{\ep}{\epsilon}
\newcommand{\ve}{\varepsilon}

\newcommand{\inj}{\hookrightarrow}
\newcommand{\surj}{\twoheadrightarrow}
\newcommand{\longto}{\longrightarrow}
\newcommand{\longinj}{\lhook\joinrel\longrightarrow}

\newcommand{\rlto}{\rightleftarrows}
\newcommand{\xr}[1]{\xrightarrow{#1}}
\newcommand{\xrr}[1]{\xrightarrow{\, #1 \, }}
\newcommand{\simto}{\xr{\sim}}
\newcommand{\longsimto}{\xrr{\sim}}

\newcommand{\xlrarrows}[3]{
\stackrel[#3]{#2}{\mathrel{\substack{\xrightarrow{#1} \\[-.9ex] \xleftarrow{#1}}}}}
\newcommand{\adjunc}{\xlrarrows{\rule{1em}{0em}}{}{}}

\newcommand{\tbcup}{{\textstyle \bigcup}}
\newcommand{\tbscup}{{\textstyle \bigsqcup}}
\newcommand{\iclim}{\varinjlim}
\newcommand{\ilim}{\varprojlim}

\newcommand{\one}{\mathbf{1}}
\newcommand{\kdg}{k_{\dg}}
\newcommand{\ase}{a}
\newcommand{\yon}{j}
\newcommand{\Qlb}{\ol{\bbQ_\ell}}

\newcommand{\dL}{\mathrm{L}}
\newcommand{\dR}{\mathrm{R}}
\newcommand{\Hot}{\operatorname{Ho}}

\newcommand{\shHom}{\operatorname{\mathscr{H} \kern-.3em \mathit{om}}}
\newcommand{\shExt}{\operatorname{\mathscr{E} \kern-.2em \mathit{xt}}}
\newcommand{\dRHom}{\operatorname{\mathscr{R} \kern-.1em \mathit{hom}}}
\newcommand{\bsHom}{\operatorname{\pmb{\mathscr{H} \kern-.3em \mathit{om}}}}

\newcommand{\C}{\operatorname{C}}
\newcommand{\E}{\operatorname{E}}
\newcommand{\sk}{\operatorname{sk}}
\newcommand{\dD}{\operatorname{D}}

\newcommand{\Ho}{\operatorname{h}}
\newcommand{\Ob}{\operatorname{Ob}}

\newcommand{\GL}{\operatorname{GL}}

\newcommand{\mHo}{\operatorname{Ho}}
\newcommand{\mSh}{\operatorname{Sh}}
\newcommand{\Aut}{\operatorname{Aut}}

\newcommand{\Cov}{\operatorname{Cov}}

\newcommand{\Dex}{\operatorname{Dex}}

\newcommand{\Ext}{\operatorname{Ext}}
\newcommand{\Fun}{\operatorname{Fun}}
\newcommand{\Hom}{\operatorname{Hom}}
\newcommand{\Img}{\operatorname{Im}}
\newcommand{\Ind}{\operatorname{Ind}}
\newcommand{\Iso}{\operatorname{Iso}}
\newcommand{\Ker}{\operatorname{Ker}}
\newcommand{\Map}{\operatorname{Map}}
\newcommand{\Mod}{\operatorname{Mod}}
\newcommand{\Ner}{\operatorname{\ic{N}}}
\newcommand{\Ndg}{\operatorname{\ic{N}_{\dg}}}
\newcommand{\Nsp}{\operatorname{\ic{N}_{\mathrm{spl}}}}
\newcommand{\Nrm}{\operatorname{Nrm}}
\newcommand{\Rep}{\operatorname{Rep}}

\newcommand{\Sub}{\operatorname{Sub}}

\newcommand{\Trc}{\operatorname{Trc}}
\newcommand{\Sie}{\operatorname{Sieve}}
\newcommand{\cosk}{\operatorname{cosk}}
\newcommand{\join}{\mathbin{\star}}

\newcommand{\funD}{\operatorname{D}}

\newcommand{\Hall}{\operatorname{Hall}}

\newcommand{\QCoh}{\operatorname{QCoh}}
\newcommand{\Spec}{\operatorname{Spec}}

\newcommand{\dSpec}{\operatorname{dSpec}}
\newcommand{\fSing}{\operatorname{Sing}}

\newcommand{\id}{\operatorname{id}}
\newcommand{\pr}{\operatorname{pr}}
\newcommand{\pH}{{}^{\pv}H}
\newcommand{\pvp}{\operatorname{{}^{\pv}\pi}}
\newcommand{\Rst}{\operatorname{Rst}}
\newcommand{\dgMap}{\operatorname{\mathit{Map}}}
\newcommand{\frkC}{\operatorname{\mathfrak{C}}}

\newcommand{\cl}{\mathrm{cl}}
\newcommand{\cn}{\mathrm{cn}}
\newcommand{\dg}{\mathrm{dg}}
\newcommand{\fp}{\mathrm{fp}}
\newcommand{\lf}{\mathrm{lf}}
\newcommand{\op}{\mathrm{op}}
\newcommand{\pe}{\mathrm{pe}}
\newcommand{\pv}{\mathrm{p}}
\newcommand{\aff}{\mathrm{aff}}
\newcommand{\cpt}{\mathrm{cpt}}
\newcommand{\cof}{\mathrm{cof}}
\newcommand{\fib}{\mathrm{fib}}
\newcommand{\fin}{\mathrm{fin}}

\newcommand{\ord}{\mathrm{ord}}
\newcommand{\str}{\mathrm{str}}
\newcommand{\cart}{\mathrm{cart}}
\newcommand{\cstr}{\mathrm{c}}

\newcommand{\geom}{\mathrm{geom}}
 
\newcommand{\proj}{\mathrm{proj}} 
\newcommand{\stab}{\mathrm{stab}}
\newcommand{\trun}{\mathrm{trun}}
\newcommand{\txinj}{\mathrm{inj}}
\newcommand{\txred}{\mathrm{red}}
\newcommand{\et}{\mathrm{et}}
\newcommand{\sm}{\mathrm{sm}}
\newcommand{\lis}{\mathrm{lis}}
\newcommand{\txET}{\mathrm{ET}}
\newcommand{\txle}{\mathrm{lis}\textrm{-}\mathrm{et}}

\newcommand{\bbE}{\mathbb{E}}
\newcommand{\bbF}{\mathbb{F}}
\newcommand{\bbK}{\mathbb{K}}
\newcommand{\bbL}{\mathbb{L}}
\newcommand{\bbN}{\mathbb{N}}

\newcommand{\bbQ}{\mathbb{Q}}
\newcommand{\bbU}{\mathbb{U}}
\newcommand{\bbV}{\mathbb{V}}
\newcommand{\bbZ}{\mathbb{Z}}
\newcommand{\bbo}{\mathbbm{1}}

\newcommand{\bfP}{\mathbf{P}}
\newcommand{\bfQ}{\mathbf{Q}}

\newcommand{\shA}{\mathcal{A}}
\newcommand{\shB}{\mathcal{B}}
\newcommand{\shF}{\mathcal{F}}
\newcommand{\shG}{\mathcal{G}}
\newcommand{\shI}{\mathcal{I}}
\newcommand{\shJ}{\mathcal{J}}
\newcommand{\shK}{\mathcal{K}}
\newcommand{\shL}{\mathcal{L}}
\newcommand{\shM}{\mathcal{M}}
\newcommand{\shN}{\mathcal{N}}
\newcommand{\shO}{\mathcal{O}}
\newcommand{\shP}{\mathcal{P}}
\newcommand{\shR}{\mathcal{R}}
\newcommand{\calF}{\mathcal{F}}
\newcommand{\calG}{\mathcal{G}}
\newcommand{\calP}{\mathcal{P}}
\newcommand{\frkm}{\mathfrak{m}}

\newcommand{\topC}{\mathcal{C}}
\newcommand{\topCG}{\mathcal{CG}}
\newcommand{\topCW}{\mathcal{CW}}

\newcommand{\topH}{\mathcal{H}}
\newcommand{\Kan}{\mathcal{K}\mathrm{an}}

\newcommand{\spA}{\mathfrak{A}}
\newcommand{\spC}{\mathfrak{C}}
\newcommand{\spCom}{\mathfrak{C}\mathrm{om}}
\newcommand{\siCat}{\mathfrak{C}\mathrm{at}_{\infty}}

\newcommand{\cat}{\mathrm}
\newcommand{\catA}{\cat{A}}
\newcommand{\catB}{\cat{B}}
\newcommand{\catC}{\cat{C}}
\newcommand{\catD}{\cat{D}}
\newcommand{\catDD}{\mathbf{D}}
\newcommand{\catF}{\cat{F}}
\newcommand{\catM}{\cat{M}}
\newcommand{\Ab}{\cat{Ab}}
\newcommand{\St}{\cat{St}}
\newcommand{\CS}{\boldsymbol{\Delta}}
\newcommand{\dgA}{\cat{dgAlg}}
\newcommand{\dgB}{\operatorname{\cat{B}}}
\newcommand{\Aff}{\cat{Aff}}
\newcommand{\PSh}{\cat{PSh}}
\newcommand{\Com}{\cat{Com}}
\newcommand{\Set}{\cat{Set}}
\newcommand{\Sch}{\cat{Sch}}
\newcommand{\dAff}{\cat{dAff}}
\newcommand{\sCom}{\cat{sCom}}
\newcommand{\sMod}{\cat{sMod}}
\newcommand{\sCat}{\cat{Cat}_{\Delta}}
\newcommand{\sSet}{\cat{Set}_{\Delta}}
\newcommand{\Grpd}{\cat{Grpd}}
\newcommand{\dgCat}{\cat{dgCat}}
\newcommand{\AlgSp}{\cat{AS}} 
\newcommand{\AlgSt}{\cat{AlgSt}}
\newcommand{\LinOrd}{\cat{LinOrd}}
\newcommand{\Perf}{\operatorname{P}}
\newcommand{\dgMod}{\operatorname{Mod_{\dg}}}

\newcommand{\st}{\mathcal}
\newcommand{\stF}{\st{F}}
\newcommand{\stM}{\st{M}}
\newcommand{\stP}{\st{P}}

\newcommand{\stT}{\st{T}}
\newcommand{\stU}{\st{U}}
\newcommand{\stV}{\st{V}}
\newcommand{\stW}{\st{W}}
\newcommand{\stX}{\st{X}}
\newcommand{\stY}{\st{Y}}
\newcommand{\stVect}{\st{V}ect}
\newcommand{\stQCoh}{\st{QC}oh}
\newcommand{\stPf}{\operatorname{\st{P}f}}
\newcommand{\stEx}{\operatorname{\st{E}}}

\newcommand{\gsS}{\mathrm{S}}
\newcommand{\gsEt}{\operatorname{Et}}
\newcommand{\gsET}{\operatorname{ET}}
\newcommand{\gsLE}{\operatorname{LE}}
\newcommand{\gsSm}{\operatorname{Sm}}

\newcommand{\ic}{\mathsf}
\newcommand{\iS}{\mathcal{S}}
\newcommand{\iEM}{\operatorname{\mathcal{EM}}}
\newcommand{\iLTop}{\ic{LTop}}
\newcommand{\iRTop}{\ic{RTop}}
\newcommand{\iCat}{\ic{Cat}_{\infty}}

\newcommand{\iB}{\ic{B}}
\newcommand{\iC}{\ic{C}}
\newcommand{\iE}{\ic{E}}
\newcommand{\iI}{\ic{I}}
\newcommand{\iM}{\ic{M}}

\newcommand{\iQ}{\ic{Q}}
\newcommand{\iX}{\ic{X}}

\newcommand{\iAb}{\ic{Ab}}
\newcommand{\iSt}{\ic{St}}
\newcommand{\iAff}{\ic{Aff}}
\newcommand{\iCom}{\ic{Com}}
\newcommand{\idAS}{\ic{dAS}}
\newcommand{\idSt}{\ic{dSt}}
\newcommand{\iSch}{\ic{Sch}}
\newcommand{\isCom}{\ic{sCom}}
\newcommand{\isMod}{\ic{sMod}}
\newcommand{\idAff}{\ic{dAff}}
\newcommand{\idSch}{\ic{dSch}}
\newcommand{\iAlgSp}{\ic{AS}}

\newcommand{\iArtin}{\ic{Artin}}

\newcommand{\tS}{\ic{S}}
\newcommand{\tT}{\ic{T}}
\newcommand{\tU}{\ic{U}}

\newcommand{\iFE}{\operatorname{\ic{FinExt}}}
\newcommand{\iAlg}{\operatorname{\ic{Alg}}}
\newcommand{\iFun}{\operatorname{\ic{Fun}}}
\newcommand{\iDes}{\operatorname{\ic{Des}}}
\newcommand{\iMod}{\operatorname{\ic{Mod}}}
\newcommand{\iCAlg}{\operatorname{\ic{CAlg}}}
\newcommand{\iPSh}{\operatorname{\ic{PSh}}} 
\newcommand{\iSh}{\operatorname{\ic{Sh}}} 
\newcommand{\iSp}{\operatorname{\ic{Sp}}}
\newcommand{\iShv}{\operatorname{\ic{Shv}}} 
\newcommand{\iPerv}{\operatorname{\ic{Perv}}}
\newcommand{\iQCoh}{\operatorname{\ic{QCoh}}}
\newcommand{\iVect}{\operatorname{\ic{Vect}}}
\newcommand{\iD}{\operatorname{\ic{D}}}
\newcommand{\iDa}{\operatorname{\ic{D}_{\infty}}}
\newcommand{\iDm}{\operatorname{\ic{D}}^{-}_{\infty}}
\newcommand{\iDp}{\operatorname{\ic{D}^{+}_{\infty}}}
\newcommand{\iDc}{\operatorname{\ic{D}_{\infty,\cstr}}}
\newcommand{\iDcm}{\operatorname{\ic{D}}_{\infty,\cstr}^{-}}
\newcommand{\iDcp}{\operatorname{\ic{D}_{\infty,\cstr}^{+}}}
\newcommand{\iDcbm}{\operatorname{\ic{D}}_{\infty,\cstr}^{(-)}}
\newcommand{\iDcbp}{\operatorname{\ic{D}_{\infty,\cstr}^{(+)}}}
\newcommand{\iDD}{\operatorname{\mathbf{D}}}
\newcommand{\iDDc}{\operatorname{\mathbf{D}_{\infty,\cstr}}}

\newcommand{\ipDm}{\operatorname{{}^{\pv}\mathsf{D}^{\le 0}}}
\newcommand{\ipDp}{\operatorname{{}^{\pv}\mathsf{D}^{\ge 0}}}
\newcommand{\iDs}[1]{\operatorname{\ic{D}}_{\infty,#1}}
\newcommand{\iDu}[1]{\operatorname{\ic{D}}_{\infty}^{#1}}
\newcommand{\iDcu}[1]{\operatorname{\ic{D}}_{\infty,\cstr}^{#1}}
\newcommand{\iDDcu}[1]{\operatorname{\mathbf{D}}_{\infty,\cstr}^{#1}}

\newcommand{\abs}[1]{\left| #1 \right|}
\newcommand{\tsh}[1]{\langle #1 \rangle}
\newcommand{\HomR}[1]{\operatorname{\mathrm{Hom}^\mathrm{R}_{#1}}}

\newcommand{\oc}[2]{#1_{\!/#2}}

\newcommand{\rst}[2]{\left. #1 \right|_{#2}}

\begin{document}

\title{Geometric derived Hall algebra}
\author{Shintarou Yanagida}
\address{Graduate School of Mathematics, Nagoya University. 
Furocho, Chikusaku, Nagoya, Japan, 464-8602.}
\email{yanagida@math.nagoya-u.ac.jp}
\date{December 10, 2019}

\begin{abstract}
We give a geometric formulation of To\"en's derived Hall algebra
by constructing Grothendieck's six operations for 
the derived category of lisse-\'etale constructible sheaves
on the derived stacks of complexes.
Our formulation is based on an variant of Laszlo and Olsson's theory of 
derived categories and six operations for algebraic stacks.
We also give an $\infty$-theoretic explanation of 
the theory of derived stacks, which was originally constructed by To\"en and Vezzosi
in terms of model theoretical language.
\end{abstract}

\maketitle
\tableofcontents

\setcounter{section}{-1}
\section{Introduction}
\label{s:intro}

\subsection{}

The \emph{derived Hall algebra} introduced by To\"{e}n \cite{T06} 
is a version of Ringel-Hall algebra, and roughly speaking 
it is a ``Hall algebra for complexes''.
In the case of the ordinary Ringel-Hall algebra 
for the abelian category of representations of a quiver,
Lusztig \cite{Lus} established a geometric formulation 
using the theory of perverse sheaves in the 
equivariant derived category of $\ell$-adic constructible sheaves 
on the moduli spaces of the representations,
and his formulation gives rise to the theory of canonical bases for quantum groups.
As mentioned in \cite[\S 1, Related and future works]{T06}, it is natural to 
expect a similar geometric formulation for derived Hall algebras
using the moduli space of \emph{complexes} of representations.

At the moment when \cite{T06} appears,
the theory of \emph{derived stacks} was just being under construction,
which would realize the moduli space of complexes.
Soon after, To\"en and Vezzosi \cite{TVe1,TVe2} 
completed the works on derived stacks, and based on them
To\"en and Vaqui\'e \cite{TVa} constructed the moduli space of dg-modules.
In \cite[\S 0.6]{TVa}, it was announced among several stuffs that 
a geometric formulation of derived Hall algebra was being studied.
As far as we understand, it has not appeared yet.

The purpose of this article is to give such a geometric 
formulation of derived Hall algebras.
More precisely speaking, we want to give the following materials.
\begin{itemize}[nosep]
\item 
The theory of derived categories of \emph{lisse-\'etale constructible sheaves}
on derived stacks and the construction of 
\emph{Grothendieck's six operations} on them.

\item
The theory of \emph{perverse $\ell$-adic constructible sheaves} 
on derived stacks.

\item
The geometric formulation of derived Hall algebra.
\end{itemize}

Any such attempt will require \emph{derived algebraic geometry}.
Although the works \cite{TVe1,TVe2,TVa} are based on the language of 
\emph{model categories} so that it might be natural to work over such a language,
we decided to work using the language of \emph{$\infty$-categories}
following Lurie's derived algebraic geometry \cite{Lur5} -- \cite{Lur14}.
The reason is that today most papers on derived algebraic geometry
uses the language of $\infty$-categories rather than that of model categories.
A ``byproduct" of our project is 
\begin{itemize}
\item 
A ``translation" of model-categorical \cite{TVe1,TVe2,TVa} into 
the language of $\infty$-categories.
\end{itemize}

Such a translation is in fact a trivial one, and it seems that 
experts on derived algebraic geometry use both languages freely.
We decided to give somewhat a textbook-like explanation on this point,
and we will start the main text with some recollection on 
the theory of $\infty$-categories and $\infty$-topoi (\S \ref{s:ic}).
The theory of derived stacks in the sense of To\"en and Vezzosi
will be explained in the language of $\infty$-categories (\S \ref{s:dSt}).

We will construct the derived category of sheaves on derived stacks
using the theory of \emph{stable $\infty$-category}
established by Lurie \cite{Lur2}.
For this purpose, we will give in \S \ref{s:is} an exposition of 
the general theory of the $\infty$-category of sheaves on $\infty$-topoi
and the associated stable $\infty$-category.
There are many overlaps between our exposition and a (very small) part of 
Lurie's derived algebraic geometry \cite{Lur5}--\cite{Lur14}.
We decided to give a rather self-contained presentation 
for the completeness of this article.

The main object of this article is the \emph{lisse-\'etale sheaves} 
on derived stacks, which will be introduced in \S \ref{s:LE}.
Our definition is a simple analogue of the lisse-\'etale sheaves
on the algebraic stacks established by Laumon and Moret-Bailly \cite{LM}.
In the case of algebraic stacks, the lisse-\'etale topos is 
constructed using the \'etale topos on algebraic spaces.
Thus we need a derived analogue of algebraic spaces,
which we call \emph{derived algebraic spaces} and introduce in \S \ref{s:dAS}.
 
In \cite{LO1, LO2}, Laszlo and Orson constructed 
the Grothendieck six operations on the derived category of 
constructible sheaves on algebraic stacks,
based on the correction \cite{O} of a technical error 
on the lisse-\'etale topos in \cite{LM}.
In the sequel \cite{LO3} they also give the theory of 
perverse $\ell$-adic sheaves and weights on them over algebraic stacks.
Our construction of six operations (\S \ref{s:6op}, \S \ref{s:adic}) 
and definition of perverse $\ell$-adic sheaves (\S \ref{s:pv})
are just a copy of their argument.

Fundamental materials will be established up to \S \ref{s:pv},
and we will then turn to the derived Hall algebra.
In \S \ref{s:mod} we review To\"en and Vaqui\'e's construction \cite{TVa} of 
the moduli space of dg-modules over a dg-category via derived stack.
In \S \ref{s:H} we explain the definition of derived Hall albebras,
and give its geometric formulation, the main purpose of this article.

Let us sketch an outline of the construction here.
Let $\catD$ be a locally finite dg-category over $\bbF_q$
(see \S \ref{s:mod} for an account on dg-categories).

\begin{fct*}[{\cite{TVa}}]
We have the \emph{moduli stack} $\stPf(\catD)$ \emph{of perfect dg-modules} 
over $\catD^{\op}$.
It is a derived stack, locally geometric and locally of finite presentation. 
\end{fct*}

We can also construct the moduli stack of cofibrations $X \to Y$ 
of perfect-modules over $\catD^{\op}$, denoted by $\stEx(\catD)$.
Then there exist morphisms
\[
 s,c,t: \stEx(\catD) \longto \stPf(\catD)
\]
of derived stacks which send $u: X \to Y$ to 
\[
 s(u) = X, \quad c(u) = Y, \quad t(u) = Y \coprod^X 0.
\]
where $s,t$ are smooth and  $c$ is proper.
Thus we have a square 
\[
 \xymatrix{ \stEx(\catD) \ar[r]^{c} \ar[d]_{p:=s \times t} & \stPf(\catD) \\ 
            \stPf(\catD) \times \stPf(\catD) }
\]
of derived stacks with smooth $p$ and proper $c$.

Next let $\Lambda := \ol{\bbQ}_\ell$ be the field of $\ell$-adic numbers
where $\ell$ and $q$ are assumed to be coprime.
Then we have the \emph{derived category} $\dD^b_{\cstr}(\stX,\Lambda)$
\emph{of constructible lisse-\'etale $\Lambda$-sheaves}
over a locally geometric derived stack $\stX$.
We also have derived functors (\S 4).

Applying the general theory to the present situation, we have 
\[
 \xymatrix{
 \dD^b_{\cstr}(\stEx(\catD),\Lambda) \ar[r]^{c_!}  & 
 \dD^b_{\cstr}(\stPf(\catD),\Lambda) 
 \\
 \dD^b_{\cstr}(\stPf(\catD) \times \stPf(\catD),\Lambda) \ar[u]^{p^*}}
\]
Now we set
\[
 \mu: \dD^b_{\cstr}(\stPf(\catD) \times \stPf(\catD),\Lambda) 
      \longto \dD^b_{\cstr}(\stPf(\catD),\Lambda), \quad
 M \longmapsto c_! p^*(M) [\dim p]
\]

\begin{thm*}[{Theorem \ref{thm:H:asc}}]
$\mu$ is associative.
\end{thm*}

We will also give a definition of \emph{derived Hall category},
which will be the span of the \emph{Lusztig sheaves}
in the sense of \cite{S2}.
By restricting to the ``abelian category part",
we can reconstruct the Hall category in \cite{S2},
and have a geometric formulation of the ordinary Ringel-Hall algebra
via the derived category of lisse-\'etale $\ell$-adic constructible sheaves
on the moduli stack of objects in the abelian category.
In \cite[p.2]{S2}, such a stacky construction is regarded as 
``(probably risky) project" and avoided.
Thus our exposition will also give a new insight to 
the ordinal Ringel-Hall algebra.

We postpone the study of \emph{the function-sheaf dictionary} for 
the lisse-\'etale constructible sheaves on derived stacks to a future article.

\subsection{Conventions and notations}
\label{ss:intro:ntn}

We will use some basic knowledge on 
\begin{itemize}[nosep]
\item dg-categories \cite{T07, T12}, 
\item model categories \cite{H}, 
\item simplicial homotopy theory \cite{GJ},
\item $\infty$-categories and $\infty$-topoi \cite{Lur1},
\item algebraic stacks in the ordinary sense \cite{LM, O:book}
\item derived algebraic geometry in the sense of To\"{e}n-Vezzosi \cite{TVe2, T14}.
\end{itemize}
We give a brief summary of the theory of $\infty$-categories and $\infty$-topoi
in \S \ref{s:ic}, and explanations on some topics 
in Appendices \ref{a:ic}--\ref{a:stb}.
Some recollections on algebraic spaces and algebraic stacks 
is given in Appendix \ref{s:Cl}.

Here is a short list of our global conventions and notations.
\begin{itemize}[nosep]
\item
$\bbN$ denotes the set of non-negative integers.

\item
We fix two universes $\bbU \in \bbV$ and work on them.
All the mathematical objects considered, such as sets, groups, rings and so on,
are elements of $\bbU$ unless otherwise stated.
For example, $\Set$ denotes the category of sets in $\bbU$.
We will sometimes say an object is \emph{small} if it belongs to $\bbU$.

\item
The word `ring' means a unital and associative one unless otherwise stated,
and the word `$\infty$-category' means the one in \cite{Lur1}.

\item
A \emph{poset} means a partially ordered set.
A poset $(I,\le)$ is \emph{filtered} if $I$ is non-empty
and for any $i,j \in I$ 
there exists $k \in I$ such that $i \le k$ and $j \le k$.

\item
A category will be identified with its nerve \cite[p.9]{Lur1} 
which is an $\infty$-category.
Basically we denote an ordinary category in a serif font
and denote an $\infty$-category in a sans-serif font.
For example, in \S \ref{sss:dSt:dring} 
we will denote by $\sCom$ the category of simplicial commutative rings, 
and by $\isCom$ the $\infty$-category of simplicial commutative rings.

\item
For categories $\catC$ and $\catD$,
the symbol $F: \catC \rlto \catD :G$ denotes an adjunction
in the sense that we have an isomorphism 
$\Hom_{\catD}(F(-),-) \simeq \Hom_{\catC}(-,G(-))$ of functors.
Thus it means $F \dashv G$.
We use the same symbol $\rlto$ in the $\infty$-categorical context.
See \S \ref{ss:ic:adjunc} for the detail.

\item
The homotopy category of a model category $\catC$ will be denoted by $\mHo \catC$,
and the homotopy category of an $\infty$-category $\iC$ will be denoted by $\Ho \iC$.
The functor category from a category $\catC$ to another $\catD$ will be denoted by
$\Fun(\catC,\catD)$.

\item
Finally, for an $\infty$-category $\iC$, 
the symbol $X \in \iC$ means that $X$ is an object of $\iC$.
\end{itemize}

\subsection*{Acknowledgements} 

The author is supported by the Grant-in-aid for 
Scientific Research (No.\ 16K17570, 19K03399), JSPS.
This work is also supported by the JSPS Bilateral Program
``Elliptic algebras, vertex operators and link invariant".

\section{Notations on $\infty$-categories and $\infty$-topoi}
\label{s:ic}

In this subsection we explain some basic notions on 
$\infty$-categories and $\infty$-topoi 
which will be used throughout the main text.
The purpose here is just to give brief accounts and introduce symbols of them.
Somewhat detailed accounts of selected topics will also be given in 
Appendix \ref{a:ic} and Appendix \ref{a:it}.
All the contents are given in \cite{Lur1, Lur2},
and our terminology basically follow loc.\ cit.,
but we use slightly different symbols.

\subsection{Simplicial sets}
\label{ss:ic:ss}

Let us start the explanation with our notations on simplicial sets.
We follow \cite[\S A.2.7]{Lur1} for symbols of simplicial sets.

\begin{dfn}\label{dfn:ic:CS}
We denote by $\CS$ the \emph{category of combinatorial simplices},
which is described as
\begin{itemize}[nosep]
\item 
An object is the linearly ordered set $[n] := \{0 \le 1 \le \cdots \le n\}$
for $n \in \bbN$.
\item
A morphisms is a non-strictly order-preserving function $[m] \to [n]$.
\end{itemize}
\end{dfn}

\begin{dfn*}
For a category $\catC$, a functor ${\CS}^{\op} \to \catC$ is called 
a \emph{simplicial object in $\catC$}.
In particular, a simplicial object in the category $\Set$ of sets 
is called a \emph{simplicial set}.
\end{dfn*}

We often use the symbol $S_n := S([n])$ for $n \in \bbN$.
An element in $S_0$ is called a \emph{vertex} of $S$,
and an element in $S_1$ is called an \emph{edge} of $S$.
We write $v \in S$ to mean $v$ is a vertex of $S$.
A simplicial set $S$ has the \emph{face map} $d_j: S_n \to S_{n-1}$
and the \emph{degeneracy map} $s_j: S_n \to S_{n+1}$ for each $j=0,1,\ldots,n$.
See \cite[\S A.2.7]{Lur1} for the precise definitions of $d_j$ and $s_j$.

As noted in \cite[Remark A.2.7.1]{Lur1},
the category $\CS$ is equivalent to the category $\LinOrd^{\fin}$ of
all finite nonempty linearly ordered sets,
and we sometimes identify them to regard simplicial sets 
(and more general simplicial objects) as functors 
which are defined on $\LinOrd^{\fin}$.

Using this convention, we introduce 

\begin{dfn*}
For a simplicial set $S$,
we define $S^{\op}$ to be the simplicial set given by
\[
 S^{\op}(J) := S(J^{\op}) \quad (J \in \LinOrd^{\fin}).
\]
Here $J^{\op} \in \LinOrd^{\fin}$ has the same underlying set as $J$ 
but with the opposite ordering of $J$.
We call $S^{\op}$ the \emph{opposite} of $S$.
\end{dfn*}

\begin{ntn*}
The \emph{category $\sSet$ of simplicial sets} 
is defined to be the functor category 
\[
 \sSet := \Fun({\CS}^{\op}, \Set).
\]
A morphism in $\sSet$ is called a \emph{simplicial map} 
or \emph{a map of simplicial sets}.
\end{ntn*}

The category $\sSet$ has a model structure called 
the \emph{Kan model structure}. 
See in Appendix \ref{ss:ic:Kan} for an account.

For a linearly ordered set $J$, 
we denote by $\Delta^J$ the simplicial set $[n] \mapsto \Hom([n],J)$,
where the morphisms are taken in the category of linearly ordered sets.
For $n \in \bbN$, we simply write $\Delta^n := \Delta^{[n]}$, 
and for $0 \le j \le n$, we denote by $\Lambda_j^n \subset \Delta^n$ 
the $j$-th horn \cite[Example A.2.7.3]{Lur1}.

Finally we introduce 

\begin{dfn*}
A \emph{Kan complex} is a simplicial set $K$ 
such that for any $n \in \bbN$ and any $0 \le i \le n$, 
any simplicial map $f_0: \Lambda_i^n \to K$ 
admits an extension $f: \Delta^n \to K$.
\end{dfn*}

We will use the following fact repeatedly in the main text.

\begin{fct}[{\cite[Corollary 1.3.2.12]{Lur2}}]\label{fct:ic:sag=Kan}
Any simplicial abelian group is a Kan complex.
\end{fct}

\subsection{$\infty$-categories}

Next we give some symbols for $\infty$-categories.
We use the word ``$\infty$-category" in the sense of \cite{Lur1}.

\begin{dfn*}[{\cite[Definition 1.1.2.4]{Lur1}}]
An \emph{$\infty$-category} is a simplicial set $\iC$ such that 
for any $n \in \bbN$ and any $0<i<n$, any map $f_0: \Lambda_i^n \to \iC$ 
of simplicial sets admits an extension $f: \Delta^n \to \iC	$.
\end{dfn*}

Since an $\infty$-category is a simplicial set,
all the notions on simplicial sets can be transfered to 
those on an $\infty$-category.
For example, we have 

\begin{dfn*}
The \emph{opposite} of an $\infty$-category $\iC$ is defined to be the opposite
$\iC^{\op}$ of $\iC$ as a simplicial set in the sense of \S \ref{ss:ic:ss}.
\end{dfn*}

For an $\infty$-category $\iC$,
the opposite $\iC^{\op}$ is an $\infty$-category.
Thus we call $\iC^{\op}$ the \emph{opposite $\infty$-category} of $\iC$.

Next we explain the relation between ordinary categories and $\infty$-categories.

\begin{dfn*}
For an $\infty$-category $\iC$,
the \emph{objects} are the vertices of $\iC$ as a simplicial set,
and the \emph{morphisms} are the edges of $\iC$ as a simplicial set.
\end{dfn*}

Thus an object of an $\infty$-category $\iC$ is 
a simplicial maps $\Delta^0 \to \iC$.
and a morphism of $\iC$ is a simplicial maps $\Delta^1 \to \iC$.

These notions are compatible with those in the ordinary category theory,
which we now briefly explain.

\begin{dfn}[{\cite[p.9]{Lur1}}]\label{dfn:ic:nerve}
The \emph{nerve} of a category $\catC$, denoted by $\Ner(\catC)$,
is a simplicial set with $\Ner(\catC)_n = \Fun([n], \catC)$,
where the linearly ordered set $[n]=\{0,1,\ldots,n\}$ is regarded 
as a category in the obvious way.
The face maps and degeneracy maps are defined via compositions and 
inserting the identity morphisms.
\end{dfn}

By \cite[Proposition 1.1.2.2, Example 1.1.2.6]{Lur1},
the nerve of any category $\catC$ is an $\infty$-category.
Then the objects and the morphisms of $\catC$ 
coincide with those of $\Ner(\catC)$.
Also we have $\Ner(\catC^{\op}) = \Ner(\catC)^{\op}$ as simplicial sets.
%

Next we introduce functors of $\infty$-categories.
For that, let us recall 

\begin{dfn}[{\cite[Chap.\ 1 \S 5]{GJ}}]\label{dfn:ic:fcpx}
For $X,Y \in \sSet$, 
we define the simplicial set $\Map_{\sSet}(X,Y)$ by the following description.
\begin{itemize}[nosep]
\item
$\Map_{\sSet}(X,Y)_n := \Hom_{\sSet}(X \times \Delta^n,Y)$ for each $n \in \bbN$,
where $\times$ denotes the product of simplicial sets.
\item
The degeneracy maps and the face maps are induced by those on $\Delta^n$.
\end{itemize}
\end{dfn}

The simplicial set $\Map_{\sSet}(X,Y)$ is called the function complex 
in the literature of the simplicial homotopy theory, 
but we will not use this word.

\begin{dfn}[{\cite[Notation 1.2.7.2]{Lur1}}]\label{dfn:ic:fun}
Let $K$ be a simplicial set and $\iC$ be an $\infty$-category.
We set
\[
 \iFun(K,\iC) := \Map_{\sSet}(K,\iC). 
\]
If $K=\iB$ is an $\infty$-category,
then $\iFun(\iB,\iC)$ is called 
the \emph{$\infty$-category of functors} from $\iB$ to $\iC$,
and its object, i.e. a simplicial map $\iB \to \iC$, is called a \emph{functor}.
\end{dfn}

Thus a functor of $\infty$-category is nothing but a simplicial map.
The simplicial set $\iFun(K,\iC)$ is an $\infty$-category 
by \cite[Proposition 1.2.7.3 (1)]{Lur1}.

\subsection{The homotopy category of an $\infty$-category}
\label{ss:ic:ho}

The definition of the homotopy category of an $\infty$-category
is rather complicated,
so we postpone a detailed explanation to Appendix \S \ref{ss:ic:Ho},
and here we only give relevant important notions.

For the definition of the homotopy category of an $\infty$-category,
we need to recall the notion of a topological category.
We will use the terminology on enriched categories
(see \cite[\S A.1.4]{Lur1} for example).

\begin{dfn}[{\cite[Definition 1.1.1.6]{Lur1}}]\label{dfn:ic:tc}
\begin{enumerate}[nosep]
\item
We denote by $\topCG$ the category of 
compactly generated weakly Hausdorff topological spaces.
\item 
A \emph{topological category} is defined to be 
a category enriched over $\topCG$.
\item
For a topological category $\topC$ and its objects $X,Y \in \topC$, 
we denote by $\Map_{\topC}(X,Y) \in \topCG$ 
the topological space of morphisms and call it the \emph{mapping space}.
\end{enumerate}
\end{dfn}

Next we introduce the category $\topH$ called 
the \emph{homotopy category of spaces}.
Let $\topCW$ be the topological category whose objects are CW complexes
and $\Map_{\topCW}(X,Y)$ is the set of continuous maps 
equipped with the compact-open topology.

\begin{dfn}[{\cite[Example 1.1.3.3]{Lur1}}]
\label{dfn:ic:topH}
The \emph{homotopy category $\topH$ of spaces}.
is the category defined as follows.
\begin{itemize}[nosep]
\item 
The objects of $\topH$ are defined to be the objects of $\topCW$.

\item 
For $X,Y \in \topC$, we set $\Hom_{\topH}(X,Y) := \pi_0(\Map_{\topCW}(X,Y))$.

\item
Composition of morphisms in $\topH$ is given by the application of $\pi_0$ to 
composition of morphisms in $\topC$.
\end{itemize}
\end{dfn}

The category $\topH$ is actually the homotopy category of $\topCW$ 
in the sense of \cite[Definition 1.1.3.2]{Lur1}.

Next we recall

\begin{dfn}\label{dfn:ic:sCat}
A category enriched in $\sSet$ is called a \emph{simplicial category}.
We denote by $\sCat$ the category of simplicial categories.
A $\sSet$-enriched functor between simplicial categories 
is called a \emph{simplicial functor}.
\end{dfn}

The category $\sSet$ itself is a simplicial category 
by the simplicial set $\Map_{\sSet}(\cdot,\cdot)$ in Definition \ref{dfn:ic:fcpx}.
Then there is a functor 
\begin{equation}\label{eq:ic:frkC}
 \frkC[\cdot ]: \sSet \longto \sCat
\end{equation}
by \cite[\S1.1.5]{Lur1} (see also Appendix \ref{ss:ic:Ho}).
Then
the \emph{homotopy category} $\Ho S$ of a simplicial set $S$ is defined to be 
\[
 \Ho S := \Ho \frkC[S],
\]
where the right hand side denotes 
the homotopy category of the simplicial category $\frkC[S]$
(Definition \ref{dfn:ic:Ho-sCat}).

\begin{dfn*}
The \emph{homotopy category} of an $\infty$-category $\iC$ 
is defined to be the homotopy category $\Ho \iC$ 
of $\iC$ as a simplicial set.
\end{dfn*}

For a simplicial set $S$,
the homotopy category $\Ho S$ is enriched over $\topH$.
Thus, for vertices $x,y \in S$ we may denote 
\[
 \Map_{S}(x,y) := \Hom_{\Ho S}(x,y) \in \topH.
\] 
It is called the \emph{mapping space} from $x$ to $y$ in $S$.

We see that a simplicial map $f: S \to S'$ induces a functor 
$\Ho f: \Ho S \to \Ho S'$ between the homotopy categories.
Thus the following definitions make sense.

\begin{dfn}[{\cite[Definition 1.2.10.1]{Lur1}}]
\label{dfn:ic:sfe}
Let $f: S \to S'$ be a simplicial map.
\begin{enumerate}[nosep]
\item 
$f$ is called \emph{essentially surjective}
if the induced functor $\Ho f$ is essentially surjective.
\item
$f$ is called \emph{fully faithful} if $\Ho f$ is 
a fully faithful functor of $\topH$-enriched categories.
\item
$f$ is called an \emph{equivalence} 
if it is essentially surjective and fully faithful.
\end{enumerate} 
For a morphism $g$ of $\infty$-categories,
the same terminology will be used with $g$ regarded as a simplicial map.
\end{dfn}

These definitions are compatible with the notions of the ordinary category theory.
In fact, for a category $\catC$ and its objects $X,Y \in \catC$, 
we have a bijection
$\Hom_{\catC}(X,Y) = \pi_0 \Map_{\Ner(\catC)}(X,Y)$,
where in the right hand side we regard $X$ and $Y$ as objects 
of the $\infty$-category $\Ner(\catC)$.

For later use, we also introduce

\begin{dfn}\label{dfn:pr:csv}
A functor $f: \iB \to \iC$ of $\infty$-categories is \emph{conservative}
if the following condition is satisfied:
if $\beta$ is a morphism in $\iB$ such that $f(\beta)$ is an equivalence in $\iC$,
then $\beta$ is an equivalence in $\iB$.
\end{dfn}

\begin{dfn}[{\cite[\S 1.2.11]{Lur1}}]\label{dfn:sub-inf-cat}
Let $\iC$ be an $\infty$-category.
\begin{enumerate}[nosep]
\item 
For a subcategory $\catB \subset \Ho \iC$, define $\iB$ to be the simplicial set 
appearing in the following pullback diagram of simplicial sets 
(see \cite[Chap.\ 1]{GJ} for the pullback in $\sSet$).
\[
 \xymatrix{\iB \ar[r] \ar[d] & \iC \ar[d] \\ \Ner(\catB) \ar[r] & \Ner(\Ho \iC)}
\]
Then $\iB$ is an $\infty$-category and called the 
\emph{sub-$\infty$-category of $\iC$ spanned by $\catB$}.

\item
A simplicial subset $\iC' \subset \iC$ is 
a \emph{sub-$\infty$-category of $\iC$} if it arises by this construction.
\item
If in the item (1) the subcategory $\catB \subset \Ho \iC$ 
is a full subcategory, then the subcategory
$\iB$ is called the \emph{full sub-$\infty$-category of $\iC$ spanned by $\catB$}.
Similarly, a simplicial subset $\iC' \subset \iC$ arising in this way 
is called a \emph{full sub-$\infty$-category}. 
\end{enumerate}
\end{dfn}

We close this part with

\begin{dfn}[{\cite[Definition 1.2.12.1]{Lur1}}]\label{dfn:ic:ini-fin}
Let $\iC$ be an $\infty$-category. 
\begin{enumerate}[nosep]
\item 
An object $\one_{\iC} \in \iC$ is called a \emph{final object} 
if for any $X \in \iC$ the mapping space 
$\Map_{\iC}(X,\one_{\iC})$ 
is a final object in $\topH$, i.e., a weakly contractible space. 

\item 
An object $\emptyset \in \iC$ is called an \emph{initial object}
if it is a final object of $\iC^{\op}$. 
\end{enumerate}
\end{dfn}

The following is a characterization of final objects in an $\infty$-category.

\begin{fct}[{\cite[Definition 1.2.12.3, Proposition 1.2.12.4, 
 Corollary 1.2.12.5]{Lur1}}]\label{fct:ic:f=sf}
For an object $X$ of an $\infty$-category $\iC$,
the followings are equivalent.
\begin{itemize}[nosep]
\item 
$X$ is a final object.

\item
The canonical functor $\oc{\iC}{X} \to \iC$ 
(Corollary \ref{cor:ic:ovc}) is a trivial fibration of simplicial sets
with respect to the Kan model structure (Fact \ref{fct:ic:Km}).

\item
The Kan complex $\Hom_{\iC}^R(Y,X)$ (Definition \ref{dfn:ic:HomR})
representing the mapping space $\Map_{\iC}(Y,X)$ 
is contractible for any $Y \in X$.
\end{itemize} 
\end{fct}


For later use, let us cite 

\begin{fct}[{\cite[Proposition 1.2.12.9]{Lur1}}]\label{fct:ic:fin} 
The sub-$\infty$-category spanned by final objects 
in an $\infty$-category either is empty or is a contractible Kan complex.
\end{fct}

\subsection{Simplicial nerves and the $\infty$-category of spaces}

Let us introduce the $\infty$-category of spaces in the sense of \cite{Lur1}.
We begin with

\begin{dfn}[{\cite[Definition 1.1.5.6]{Lur1}}]
\label{dfn:ic:Nsp}
For a simplicial category $\spC$, 
there is a simplicial set $\Nsp(\spC)$ characterized by the property
\[
 \Hom_{\sSet}(\Delta^n,\Nsp(\spC)) = \Hom_{\sCat}(\frkC[\Delta^n],\spC).
\]
It is called the \emph{simplicial nerve} of $\spC$.
\end{dfn}

We have the following statements for simplicial nerves.

\begin{fct}[{\cite[Proposition 1.1.5.10, Remark  1.2.16.2]{Lur1}}]
\label{fct:ic:Nsp}
\begin{enumerate}[nosep]
\item
For a simplicial category $\spC$,
the simplicial nerve $\Nsp(\spC)$ is an $\infty$-category such that 
the simplicial set $\Map_{\spC}(X,Y)$ is a Kan complex for every $X,Y \in \spC$.
\item
For any $X,Y \in \Kan$, the simplicial set $\Map_{\Kan}(X,Y)$ is a Kan complex.
\end{enumerate}
\end{fct}

By this fact, the following definition makes sense.

\begin{dfn}[{\cite[Definition 1.2.16.1]{Lur1}}]\label{dfn:ic:iS}
Let $\Kan$ be the full subcategory of $\sSet$ 
spanned by the collection of small Kan complexes,
considered as a simplicial category via 
$\Map_{\Kan}(-,-) \subset \Map_{\sSet}(-,-)$.
We set 
\[
 \iS := \Nsp(\Kan)
\]
and call it $\iS$ \emph{the $\infty$-category of spaces}.
\end{dfn}


\begin{rmk*}[{\cite[Example 1.1.5.8]{Lur1}}]
An ordinary category $\catC$ can be regarded as a simplicial category
by setting each simplicial set $\Hom_{\catC}(X,Y)$ to be constant.
Then the simplicial nerve $\Nsp(\catC)$ of this simplicial category $\catC$
agrees with the nerve $\Ner(\catC)$ of $\catC$ as an ordinary category
(Definition \ref{dfn:ic:nerve}).
\end{rmk*}

Let us recall a classic result on $\iS$ due to Quillen,
which justifies the name ``homotopy category of spaces" of $\topH$.

\begin{fct*}[Quillen]
As for the homotopy category we have 
\[
 \Ho \iS \simeq \topH.
\]
\end{fct*}

This equivalence is induced by the adjunction 
$\abs{-}: \sSet \adjunc \topCG :\fSing$ 
in \eqref{eq:ic:sSet-CG}.
See \cite[Chap.\ 1, Theorem 11.4]{GJ} for a proof.

\subsection{Presheaves and $\infty$-categorical Yoneda embedding}
\label{ss:pr:Yoneda}

Here we cite from \cite[\S 5.1]{Lur1} 
an analogue of the Yoneda embedding in the theory of $\infty$-categories. 
We begin with 

\begin{dfn}[{\cite[Definition 5.1.0.1]{Lur1}}]
\label{dfn:ic:iPSh}
For a simplicial set $K$, we define
\[
 \iPSh(K) := \iFun(K^{\op},\iS)
\]
and call it the $\infty$-category of \emph{presheaves of spaces} on $K$.
Its object is called a \emph{presheaf of spaces} on $K$.
\end{dfn}

The following is a fundamental property of $\iPSh(K)$.
For the notion of (co)limits in $\infty$-categories,
see Appendix \ref{ss:ic:lim/colim}.

\begin{fct*}
For any simplicial set $K$,
the $\infty$-category $\iPSh(K)$ is a presentable (Definition \ref{dfn:ic:pres}).
In particular, it admits small limits and small colimits (Fact \ref{fct:ic:pres}).
\end{fct*}

For a simplicial set $K$, let $\frkC[K]$ be the associated simplicial category
given at \eqref{eq:ic:frkC} in \S \ref{ss:ic:ho}.
Recall also 
the simplicial category $\Kan$ of Kan complexes (Definition \ref{dfn:ic:iS}).
Then we can consider 
the following simplicial functor (Definition \ref{dfn:ic:sCat}):
\[
 \spC[K]^{\op} \times \spC[K] \longto \Kan, \quad 
 (X,Y) \longmapsto \fSing \abs{\Hom_{\spC[K]}(X,Y)}.
\]
Here we used the adjunction \eqref{eq:ic:sSet-CG}.
We also have a natural simplicial functor 
$\frkC[K^{\op} \times K] \to \spC[K]^{\op} \times \spC[K]$.
Composing these functors, we have a simplicial functor
\[
 \frkC[K^{\op} \times K] \longto \Kan.
\]
Then recalling Definition \ref{dfn:ic:Nsp} of the simplicial nerve
and Definition \ref{dfn:ic:iS} of $\iS$,
by passing to the adjoint we have a simplicial map
$K^{\op} \times K \to \iS$ of simplicial sets.
It gives rise to a simplicial map
\[
 \yon: K \longto \iPSh(K) = \iFun(K^{\op},\iS).
\]
Then by \cite[Proposition 5.1.3.1]{Lur1}
$j$ is fully faithful (Definition \ref{dfn:ic:sfe}).

\begin{dfn}\label{dfn:pr:Yoneda}
For a simplicial set $K$, the fully faithful simplicial map
$\yon: K \to \iPSh(K)$ is called the \emph{Yoneda embedding} of $K$. 
\end{dfn}

See \cite[\S 2.4]{TVe1} for an account of the Yoneda embedding 
in the model categorical setting.

For the later use, we cite

\begin{fct}[{\cite[Corollary 5.1.5.8]{Lur1}}]\label{fct:ic:iPSh:gen}
For an $\infty$-category $\iC$,
the $\infty$-category $\iPSh(\iC)$ is freely generated 
under small colimits by the image of the Yoneda embedding $\iC$.
\end{fct}

\subsection{The $\infty$-category of small $\infty$-categories}

For later use, let us now introduce

\begin{dfn}[{\cite[Definition 3.0.0.1]{Lur1}}]\label{dfn:ic:iCat}
\begin{enumerate}[nosep]
\item 
Define the simplicial category $\siCat$ as follows.
\begin{itemize}[nosep]
\item 
The objects are small $\infty$-categories.
\item
For $\infty$-categories $\iB$ and $\iC$,
define $\Map_{\siCat}(\iB,\iC)$ to be the largest Kan complex 
contained in the $\infty$-category $\iFun(\iB,\iC)$.
\end{itemize} 
\item
We set 
\[
 \iCat := \Nsp(\siCat)
\]
and call it the \emph{$\infty$-category of small $\infty$-categories}.
\end{enumerate}
\end{dfn}

Note that $\iCat$ is indeed an $\infty$-category by Fact \ref{fct:ic:Nsp} (1).
Furthermore, by \cite[\S 3.3.3, \S 3.3.4, Corollary 4.2.4.8]{Lur1}, 
$\iCat$ admits small limits and small colimits.

Applying various constructions for simplicial sets to $\iCat$,
we obtain the corresponding constructions for $\infty$-categories.
For example, we have

\begin{dfn}\label{dfn:pr:fp}
For functors $f: \iB \to \iD$ and $g:\iC \to \iD$ of small $\infty$-categories,
we have the \emph{fiber product} $\iB \times_{f,\iD,g} \iC$
by applying Definition \ref{dfn:ic:pb-po} of fiber product to 
the $\infty$-category $\iCat$.
\end{dfn}

\subsection{$\infty$-sites}
\label{ss:pr:inf-site}

In this subsection 
we recall the notion of Grothendieck topology on an $\infty$-category
and a construction of $\infty$-topos following \cite[\S6.2.2]{Lur1}.
See also \cite{TVe1} for a presentation in the theory of model category.

Using the notion of over-$\infty$-category (Appendix \S \ref{ss:ic:o/u}),
we introduce

\begin{dfn*}[{\cite[Definition 6.2.2.1]{Lur1}}]
Let $\iC$ be an $\infty$-category.
\begin{enumerate}[nosep]
\item
A \emph{sieve} on $\iC$ is a full sub-$\infty$-category 
${\iC}^{(0)} \subset \iC$ such that if $f: X \to Y$ is a morphism in $\iC$
and $Y \in {\iC}^{(0)}$, then $X$ also belongs to ${\iC}^{(0)}$.
\item
For $X \in \iC$, a \emph{sieve on $X$} is a sieve 
on the over-$\infty$-category $\oc{\iC}{X}$.
\end{enumerate}
\end{dfn*}

Next we want to introduce the pullback of sieves.
For that, we prepare

\begin{lem}\label{lem:pr:pb}
Given a functor $F: \iB \to \iC$ of $\infty$-categories 
and a sieve $\iC^{(0)} \subset \iC$, 
\[
 F^{-1} \iC^{(0)} := \iC^{(0)}\times_{\iC} \iB \subset \iB
\]
is a sieve on $\iB$.
Here the right hand side denotes the fiber product of $\infty$-categories
(Definition \ref{dfn:pr:fp}).
\end{lem} 

The proof is obvious.

Recall also that by Corollary \ref{cor:ic:ovc}, 
given a morphism $f: X \to Y$ in $\iC$,
we have a morphism $f_*: \oc{\iC}{X} \to \oc{\iC}{Y}$ of over-$\infty$-categories. 
Applying Lemma \ref{lem:pr:pb} to $F=f_*$, we have 

\begin{dfn*}
Let $\iC$ and $f: X \to Y$ as above.
Given a sieve $\oc{\iC^{(0)}}{Y}$ on $Y$,
we define the \emph{pullback} $f^* \oc{{\iC}^{(0)}}{Y}$ to be
$f^* \oc{\iC^{(0)}}{Y} := (f^*)^{-1} \oc{\iC^{(0)}}{Y}$.
\end{dfn*}

\begin{dfn}[{\cite[Definition 6.2.2.1]{Lur1}}]\label{dfn:ic:inf-site}
Let $\iC$ be an $\infty$-category.
\begin{enumerate}[nosep]
\item
A \emph{Grothendieck topology} $\tau$ on $\iC$ consists of 
a collection $\Cov(X)$ of sieves on each $X \in \iC$, 
called \emph{covering sieves} on $X$, satisfying the following conditions.
 \begin{enumerate}[nosep, label=(\alph*)]
 \item
 For each $X \in \iC$, the over-$\infty$-category $\oc{\iC}{X}$ as a sieve on $X$ 
 belongs to $\Cov(X)$.
 \item 
 For each morphism $f:X \to Y$ in $\iC$ and each $\oc{\iC^{(0)}}{Y} \in \Cov(Y)$, 
 the pullback $f^* \oc{\iC^{(0)}}{Y}$ belongs to $\Cov(X)$.
 \item
 Let $Y \in \iC$ and $\oc{\iC^{(0)}}{Y} \in \Cov(Y)$. 
 If $\oc{\iC^{(1)}}{Y} $ be a sieve on $Y$ such that
 $f^* \oc{\iC^{(1)}}{Y} \in \Cov(X)$ for any $f: X \to Y$ in $\oc{\iC^{(0)}}{Y}$,
 then $\oc{\iC^{(1)}}{Y} \in \Cov(Y)$.
 \end{enumerate}
We denote $\Cov_{\tau}(X) := \Cov(X)$ to emphasize that it is associated to $\tau$.

\item
A pair $(\iC,\tau)$ of an $\infty$-category $\iC$ and 
a Grothendieck topology $\tau$ on $\iC$ is called an \emph{$\infty$-site}.
\end{enumerate}
\end{dfn}

\begin{rmk*}
As mentioned in \cite[Remark 6.2.2.3]{Lur1},
if $\iC$ is the nerve of an ordinary category, then the above notion reduces to 
the ordinary definitions of a sieve and a Grothendieck topology 
\cite[II \S 1]{SGA4}. 
Indeed, as for the definition of a Grothendieck topology, we have 
\begin{itemize}[nosep]
\item
The condition (a) in Definition \ref{dfn:ic:inf-site} (1)
implies that the one-member family $\{\id: X \to X\}$ is a covering sieve on $X$
(the condition T3 in loc.\ cit.).
\item
The condition (b) implies that collections of covering sieves 
are stable under fiber product (the condition T1 in loc.\ cit.).
\item
The condition (c) is nothing but the local character 
(the condition T2 in loc.\ cit.).
\end{itemize}
\end{rmk*}

\begin{rmk}\label{rmk:ic:inf-site}
In \cite[Remark 6.2.2.3]{Lur1} it is explained that 
giving a Grothendieck topology on an $\infty$-category $\iC$ is equivalent to 
giving a Grothendieck topology in the ordinary sense 
on the homotopy category $\Ho \iC$.
The latter one is equivalent to 
the definition of a Grothendieck topology on $\iC$ in \cite{TVe1,TVe2,T14}.

In the main text we construct a Grothendieck topology 
on an $\infty$-category using this equivalence.
Thus we will only specify the data of covering sieves on the homotopy category
of the given $\infty$-category.
\end{rmk}

Having introduced the notion of $\infty$-sites, 
we now define \emph{sheaves on an $\infty$-site}.

Let $(\iC,\tau)$ be an $\infty$-site.
One can construct a set $S_\tau$ of monomorphisms $U \to \yon(X)$ 
corresponding to covering sieves of $X \in \iC$ of $\tau$ 
(Definition \ref{dfn:it:sh}).
Then a presheaf $\shF \in \iPSh(\iC)$ is called a \emph{$\tau$-sheaf} 
if it is $S_\tau$-local.
We denote by 
\[
 \iSh(\iC,\tau) \subset \iPSh(\iC)
\]
the full sub-$\infty$-category spanned by $\tau$-sheaves.


The $\infty$-category $\iSh(\iC,\tau)$ has the following properties.

\begin{fct}[{\cite[Proposition 6.2.2.7]{Lur1}}]\label{fct:pr:sh}
Let $(\iC,\tau)$ be an $\infty$-site.
\begin{enumerate}[nosep]
\item 
The $\infty$-category $\iSh(\iC,\tau)$ is a \emph{topological localization}
\cite[\S 6.2.1]{Lur1} of $\iPSh(\iC)$.
We call the localization functor $\iPSh(\iC) \to \iSh(\iC,\tau)$
the \emph{sheafification functor}.

\item
The $\infty$-category $\iSh(\iC,\tau)$ is an $\infty$-topos.
We call it the \emph{associated $\infty$-topos} of $(\iC,\tau)$.
\end{enumerate}
\end{fct}

The definition of an $\infty$-topos will be given 
in the next subsection.

\subsection{$\infty$-topoi}
\label{sss:pr:it}

Recall the notion of an accessible functor (Definition \ref{dfn:ic:acc-func}),
a left exact functor (Definition \ref{dfn:ic:exact})
and a localization functor (Definition \ref{dfn:ic:loc-func}).

\begin{dfn}[{\cite[Definition 6.1.0.4]{Lur1}}]
\label{dfn:pr:it}
An $\infty$-category $\tT$ is called an \emph{$\infty$-topos} if 
there exist a small $\infty$-category $\iB$ and 
an accessible left exact localization functor $\iPSh(\iB) \to \tT$. 
\end{dfn}

This definition makes sense since $\iPSh(\iB)$ is accessible for any 
small $\infty$-category $\iB$ by \cite[Example 5.4.2.7, Proposition 5.3.5.12]{Lur1}
so that we can ask if a functor from $\iPSh(\iB)$ is accessible or not.

Fact \ref{fct:pr:sh} can now be shown by 
taking $\tT = \iSh(\iC,\tau)$ and $\iB = \iC$ 
since we have $\iSh(\iC,\tau) = \iPSh(\iC)[S_\tau^{-1}]$.


The next statement is obvious from the definitions.

\begin{lem}\label{lem:pr:cl}
Let $\tT := \iSh(\iC,\tau)$ be the $\infty$-topos 
obtained from an $\infty$-site $(\iC,\tau)$.
Taking $\pi_0$ in the value of objects in $\tT$, one gets an topos 
in the ordinary sense \cite{SGA4} on the underlying category of $\iC$.
We denote the obtained topos by $\tT^{\cl}$.
\end{lem}

As for a general $\infty$-topos, we have the following Giraud-type theorem.

\begin{fct}[{\cite[Theorem 6.1.0.6]{Lur1}}]
\label{fct:pr:Giraud}
For an $\infty$-category $\tT$, the following conditions are equivalent.
\begin{enumerate}[nosep,label=(\roman*)]
\item 
$\tT$ is an $\infty$-topos.
\item
$\tT$ satisfies the following four conditions:
(a) $\tT$ is presentable.
(b) Colimits in $\tT$ are universal \cite[Definition 6.1.1.2]{Lur1}.
(c) Coproducts in $\tT$ are disjoint \cite[\S 6.1.1, p.532]{Lur1}.
(d) Every groupoid object of $\tT$ is effective \cite[\S 6.1.2]{Lur1}.
\end{enumerate}
\end{fct}

We will repeatedly use the following consequences.

\begin{cor}\label{cor:pr:lcl}
An $\infty$-topos admits arbitrary small limits and small colimits.
\end{cor}

\begin{proof}
Since an $\infty$-topos is presentable by Fact \ref{fct:pr:Giraud}, 
it admits small colimits by Definition \ref{dfn:ic:pres},
and admits small limits by Fact \ref{fct:ic:pres} (1).
\end{proof}

\begin{cor}\label{cor:pr:final}
An $\infty$-topos $\tT$ has an initial object,
which will be typically denoted by $\emptyset_{\tT}$.
Also $\tT$ has a final object, typically denoted by $\one_{\tT}$.
\end{cor}

\begin{proof}
An initial object is a colimit of the empty diagram,
so that $\emptyset_{\tT}$ exists by Corollary \ref{cor:pr:lcl}.
Similarly a final object exists since it is a limit of the empty diagram.
\end{proof}

Let us also cite 

\begin{fct}[{\cite[Proposition 6.3.5.1 (1)]{Lur1}}]
\label{fct:pr:ovc}
For an $\infty$-topos $\tT$ and an object $U \in \tT$,
the over-$\infty$-category $\oc{\tT}{U}$ is an $\infty$-topos.
We call it the \emph{localized $\infty$-topos} of $\tT$ on $U$,
and sometimes denote it by $\rst{\tT}{U}$.
\end{fct}

See Fact \ref{fct:is:biadj} for a complementary explanation 
of the localized $\infty$-topos.

Let us now introduce another notion of sheaves.

\begin{dfn}[{\cite[Notation 6.3.5.16]{Lur1}}]
\label{dfn:pr:iShv}
For $\infty$-topoi $\tT$ and $\iC$ admitting small limits,
we denote the full sub-$\infty$-category of $\iFun(\tT^{\op},\iC)$
spanned by functors preserving small limits by 
\[
 \iShv_{\iC}(\tT) \subset \iFun(\tT^{\op},\iC),
\]
and call it the \emph{$\infty$-category of $\iC$-valued sheaves on $\tT$}.
We usually assume $\tT$ to be an $\infty$-topos.
\end{dfn}

This notion of sheaves on an $\infty$-topos 
will be fundamental for our study of sheaves on derived stacks 
(\S \ref{s:dAS}, \S \ref{s:LE}).
At first it looks strange since the definition does not include
the standard sheaf conditions.
However, as the following fact implies, it behaves nicely on an $\infty$-topos.

\begin{fct}[{\cite[Remark 6.3.5.17]{Lur1}}]
\label{fct:pr:Shv(X)=X} 
Let $\tT$ be an $\infty$-topos.
Then the Yoneda embedding $\tT \to \iShv_{\iS}(\tT)$ is an equivalence.
\end{fct}

See Appendix \S \ref{ss:it:yon-topos} for account on 
the $\infty$-categorical Yoneda embedding.

Another remark on $\iShv$ is that we have the following compatibility 
with the symbol $\iSh$ (Definition \ref{dfn:it:sh}).

\begin{fct*}[{\cite[Proposition 1.1.12]{Lur5}}]
Let $(\iC,\tau)$ be an $\infty$-site with $\iC$ admitting small limits,
and $\iB$ be an $\infty$-category admitting small limits.
Then we have the following equivalence of $\infty$-categories.
\[
 \iShv_{\iB}(\iSh(\iC,\tau)) \longsimto \iShv_{\iB}(\iC).
\]
\end{fct*}

For later use, we introduce 

\begin{dfn}\label{dfn:pr:cst-sh}
Let $\tT$ be an $\infty$-topos and $S \in \iS$.
The \emph{constant sheaf} valued in $S$ on $\tT$ 
is an object of $\iShv_{\iS}(\tT)$ determined by the correspondence 
$U \mapsto S$ for any $U \in \tT$.
It will be denoted by the same symbol $S$.
\end{dfn}

We close this subsection by introducing basic notions on $\infty$-topoi.

\begin{dfn}[{\cite[Corollary 6.2.3.5]{Lur1}}]
\label{dfn:it:eff-epi}
A morphism $f: U \to V$ in an $\infty$-topos $\tT$ is called 
an \emph{effective epimorphism} if as an object $f \in \oc{\tT}{V}$
the truncation $\tau_{\le -1}(f)$ is a final object 
(Definition \ref{dfn:ic:ini-fin}) of $\oc{\tT}{V}$.
\end{dfn}

Here we used the truncation functor 
$\tau_{\le -1}: \oc{\tT}{V} \to \tau_{\le -1} \oc{\tT}{V}$ 
in Definition \ref{dfn:ic:tau_k}.
This definition makes sense since $\oc{\tT}{V}$ is an $\infty$-topos by
Fact \ref{fct:pr:ovc} so that it is presentable by Fact \ref{fct:pr:Giraud}.

Using the notion of effective epimorphism,
we can introduce analogues of ordinary notions of topoi.
We also use \emph{coproducts} in an $\infty$-category (Definition \ref{dfn:ic:coprod}).

\begin{dfn}
A \emph{covering} of an $\infty$-topos $\tT$ is an effective epimorphism
$\coprod_{i \in I} U_i \to \one_{\tT}$,
where $\tT$ denotes a final object of $\tT$ (Corollary \ref{cor:pr:final}).
We denote it by $\{U_i\}_{i \in I}$, suppressing the morphism to $\tT$, 
if no confusion will occur.
\end{dfn}
 
We have an obvious notion of \emph{subcovering} of a covering of an $\infty$-topos.

\begin{dfn}[{\cite[Definition 3.1]{Lur7}}]
\label{dfn:pr:qcpt}
Let $\tT$ be an $\infty$-topos.
\begin{enumerate}[nosep]
\item 
An $\infty$-topos $\tT$ is \emph{quasi-compact}
if any covering of $\tT$ has a finite subcovering.

\item
An object $T \in \tT$ is \emph{quasi-compact} 
if the $\infty$-topos $\oc{\tT}{T}$ is quasi-compact in the sense of (1).
\end{enumerate}
\end{dfn}

\section{Recollections on derived stacks}
\label{s:dSt}

In this section we explain 
the theory of derived stacks which will be used throughout the main text.
The main references are \cite{TVe2, T14}.

Our presentation is based on the $\infty$-categorical language,
although the fundamentals of the theory of derived stacks 
is developed in \cite{TVe1,TVe2} via the model theoretic language.
We refer \cite{Lur1} for the $\infty$-categorical language used here.

\subsection{Higher Artin stacks}
\label{ss:dSt:St}

We cite from \cite[\S 2.1]{TVe2} the notion of higher Artin stacks,
which enables one to develop an extension of the ordinary theory of algebraic stacks.
Although higher Artin stacks will not play essential roles in our study,
we give a summary of the theory by the following two reasons.
\begin{itemize}[nosep]
\item
The theory of higher Artin stacks and the theory of derived stacks
are developed in a parallel way in \cite[\S 2.1, \S2,2]{TVe2},
and in the case of higher Artin stacks some parts of the theory are simple.
We give an explanation on higher Artin stacks 
as a warming-up for the theory of derived stacks.
The latter one will be explained in the next \S \ref{ss:dSt:dSt}.
\item
Our discussion on derived Hall algebra has a non-derived counterpart
which is developed in the region of ordinary Artin stacks.
The theory of ordinary Artin stack is naturally embedded 
in the theory of higher Artin stacks, 
so we need some terminology on higher Artin stacks.
\end{itemize}

\subsubsection{Definition}

Fix a commutative ring $k$.

Let us consider the (ordinary) category $\Com_k$ of commutative $k$-algebras.
We denote the $\infty$-category of the nerve of $\Com_k$ by
\[
 \iCom_k := \Ner(\Com_k).
\]

\begin{rmk*}
Let us explain other ways to define $\iCom_k$, 
all of which give equivalent $\infty$-categories 
in the sense of Definition \ref{dfn:ic:sfe}.
\begin{itemize}[nosep]
\item
One way is to use the $\infty$-localization (Definition \ref{dfn:ic:loc}) to 
the pair $(\Com_k,W)$, where $W$ is the set of ring isomorphisms.

\item
Another way is to use the simplicial nerve (Definition \ref{dfn:ic:Nsp}).
Let $\spCom_k$ be the simplicial category (Definition \ref{dfn:ic:sCat})
of commutative $k$-algebras with the simplicial set $\Map_{\spCom_k}(X,Y)$ 
set to be $\Hom_{\Com_k}(X,Y)$ regarded as a constant simplicial set.
Then we have an $\infty$-category $\Nsp(\spCom_k)$.
\end{itemize}
\end{rmk*}

\begin{dfn*}
\begin{enumerate}[nosep]
\item
$\iCom_k$ is called the \emph{$\infty$-category of commutative $k$-algebras}.
\item
The \emph{$\infty$-category $\iAff_k$ of affine schemes} is defined to be
the opposite $\infty$-category of $\iCom_k$:
\[
 \iAff_k := (\iCom_k)^{\op}.
\]
The object of $\iAff_k$ corresponding to $A \in \iCom_k$ 
will be denoted by $\Spec A$
and called the \emph{affine scheme} of $A$.
\item
In the case $k=\bbZ$, we sometimes suppress the subscript and denote 
$\iCom := \iCom_{\bbZ}$ and $\iAff := \iAff_{\bbZ}$.
\end{enumerate}
\end{dfn*}

This definition is just an analogue of the ordinary scheme theory: 
the category $\Aff$ of affine schemes is equivalent to 
the opposite category $\Com^{\op}$ of the category of commutative rings.

The $\infty$-category $\iCom_k$ can be seen as the category $\Com_k$ of 
commutative $k$-algebras equipped with ``higher structure" on the set of morphisms.
See \S \ref{a:ic} for a summary of the theory of $\infty$-categories.
In particular, a morphism in the category $\Aff_k$ of affine schemes 
can be regarded as a morphism of the $\infty$-category $\iAff_k$.
Thus ordinary notions on morphisms of affine schemes can be transfered to 
those in $\iAff_k$.

Let us now consider the $\infty$-category 
\[
 \iPSh(\iAff_k) := \iFun\left((\iAff_k)^{\op}, \iS\right)
\]
of presheaves of spaces over $\iAff_k$ (Definition \ref{dfn:ic:iPSh}).
An object of $\iPSh(\iAff_k)$ is a functor $\iAff_k \to \iS$ 
of $\infty$-categories, where $\iS$ denotes the $\infty$-category of spaces 
(Definition \ref{dfn:ic:iS}).

\begin{rmk*}
A few ethical remarks are in order.
\begin{enumerate}[nosep]
\item 
In Lurie's theory of derived algebraic geometry \cite{Lur5}--\cite{Lur14}, 
the $\infty$-category $\iS$ of spaces is considered to be a correct 
$\infty$-theoretical replacement of the ordinary category $\Set$ of sets.
\item
In \cite{TVe1, TVe2}, 
a functor to the model category $\sSet$ of simplicial sets is called a prestack.
Since all the prestacks appearing in our discussion are valued 
in Kan complexes, we replace $\sSet$ by $\iS$.
\end{enumerate}
\end{rmk*}

Next we consider a Grothendieck topology on the $\infty$-category $\iAff_k$.
See Appendix \ref{a:it} for the relevant notions.

By \cite[Lemma 2.1.1.1]{TVe2}, \'etale coverings of affine schemes give 
a (non-$\infty$-theoretical) Grothendieck topology 
on the homotopy category $\Ho \iAff_k$.
Then by Remark \ref{rmk:ic:inf-site} we have a Grothendieck topology 
on the $\infty$-category $\iAff_k$.

\begin{dfn*}
The obtained Grothendieck topology on $\iAff_k$ is denoted by $\et$ 
and called the \emph{\'etale topology} on $\iAff_k$.
\end{dfn*}

Now we apply the construction of an $\infty$-topos 
in \S \ref{sss:pr:it} to the $\infty$-site $(\iAff_k, \et)$.
We consider the $\infty$-category 
\[
 \iSh(\iAff_k,\et)
\]
of $\et$-sheaves on $\iAff_k$ (see \S \ref{ss:pr:inf-site} for an account, 
and Definition \ref{dfn:it:sh} for a strict definition).
By Fact \ref{fct:pr:sh}, the $\infty$-category $\iSh(\iAff_k,\et)$ 
is an $\infty$-topos (Definition \ref{dfn:pr:it}),
which is a localization of the $\infty$-category $\iPSh(\iAff_k)$.
Not that by Corollary \ref{cor:ic:iSh:hc}, this $\infty$-topos 
$\iSh(\iAff_k,\et)$ is hypercomplete (Definition \ref{dfn:ic:hcpl}).

\begin{dfn}[{\cite[Definition 1.3.2.2, 1.3.6.3]{TVe2}}]\label{dfn:dSt:iSt}
The \emph{$\infty$-category $\iSt_k$ of stacks over $k$} is defined to be
\[
 \iSt_k := \iSh(\iAff_k,\et).
\]
An object in $\iSt_k$ is called a \emph{stack} (\emph{over $k$}), 
and a morphism in $\iSt_k$ is called a \emph{morphism of stacks} (\emph{over $k$}).
\end{dfn}

\begin{rmk}\label{rmk:dSt:stack}
Let us remark that in \cite{TVe2} 
the word ``stack" is used in a slightly different way.
Denoting by $\Ho \iC$ the homotopy category of an $\infty$-category $\iC$ 
(Definition \ref{dfn:ic:h:iC}),
we have an adjunction
\[
 a: \Ho \iPSh(\iAff_k) \adjunc \Ho \iSh(\iAff_k,\et) :j
\]
between the homotopy categories of $\iSh(\iAff,\et)$ and $\iPSh(\iAff)$
with $j$ fully faithful.
In \cite{TVe2} a stack means an object in $\iPSh(\iAff_k)$ whose class in 
the homotopy category $\Ho \iPSh(\iAff_k)$ is 
in the essential image of the functor $j$.
This terminology and Definition \ref{dfn:dSt:iSt} are different,
but we can identify them up to a choice of equivalence in $\iSt_k$.
Similarly, a morphism of stacks is defined in \cite{TVe2} to be a morphism 
in $\Ho \iPSh(\iAff_k)$, 
which differs from our definition.
Since there is no essential difference in our study, 
we will use Definition \ref{dfn:dSt:iSt} only.
\end{rmk}

Recall the definition of the set $\pi_0(S)$ of path components 
of a simplicial set $S$ (see \cite[\S 1.7]{GJ} for example).
Since a stack is a sheaf valued in $\iS$ 
and so is a sheaf of simplicial sets, 
the following notation makes sense.

\begin{dfn*}
\begin{enumerate}[nosep]
\item 
For a stack $\stX \in \iSt_k$, 
we denote by $\pi_0(\stX) \in \Fun((\iAff_k)^{\op}, \Set)$
the sheaf of sets obtained by taking $\pi_0$.
\item
For a morphism $f: \stX \to \stY$ of stacks, we denote by $\pi_0(f)$ 
the induced morphism $\pi_0(\stX) \to \pi_0(\stY)$ of sheaves of sets.
\end{enumerate}
\end{dfn*}

A morphism between stacks means a morphism in the $\infty$-category $\iSt_k$.
We then have the notion of monomorphisms (Definition \ref{dfn:ic:mono}).
We also have the notion of effective epimorphism (Definition \ref{dfn:it:eff-epi})
since $\iSt_k$ is an $\infty$-topos.

\begin{dfn}\label{dfn:St:epi-mono}
A \emph{monomorphism} of stacks is defined to be a monomorphism in $\iSt_k$.
An \emph{epimorphism} of stacks is defined to be 
an effective epimorphism in $\iSt_k$.
\end{dfn}

\begin{rmk}\label{rmk:dSt:epi-mono}
We can describe monomorphisms and epimorphisms in $\iSt_k$ more explicitly.
For a stack $\stX \in \iSt_k$, 
we denote by $\pi_0(\stX) \in \Fun((\iAff_k)^{\op}, \Set)$
the sheaf of sets obtained by taking $\pi_0$.
Here $(\iAff_k)^{\op}$ is regarded as an ordinary category.
For a morphism $f: \stX \to \stY$ of stacks, we denote by $\pi_0(f)$ 
the induced morphism $\pi_0(\stX) \to \pi_0(\stY)$ of sheaves of sets.
Then
\begin{itemize}[nosep]
\item
By Fact \ref{fct:ic:mono}, 
$f$ is a monomorphism if and only if the induced morphism $\pi_0(\Delta_f)$ of 
$\Delta_f: \stX \to \stX \times_{f, \stY, f} \stX$ is an isomorphism 
in $\Fun((\iAff_k)^{\op},\Set)$.
\item
$f$ is an effective epimorphism 
if and only if the induced morphism $\pi_0(f): \pi_0(\stX) \to \pi_0(\stY)$ 
is an epimorphism in the category $\Fun((\iAff_k)^{\op}, \Set)$.
\end{itemize}
Note that the fiber product $\stX \times_{f, \stY, f} \stX$ 
is the one in the $\infty$-categorical sense (Definition \ref{dfn:ic:pb-po}).
This description of monomorphisms and epimorphisms 
is consistent with \cite[Definition 1.3.1.2]{TVe2}.
\end{rmk}

As in the ordinary Yoneda embedding $X \mapsto \Hom(-,X)$, 
we have the $\infty$-theoretic Yoneda embedding (Definition \ref{dfn:pr:Yoneda:iSh})
\[
 \yon: \iAff_k \longinj \iSt_k.
\]

Following the terminology in \cite{TVe2}, we introduce 

\begin{dfn*}
A stack in the essential image of the Yoneda embedding $\yon$ 
is called a \emph{representable stack}.
\end{dfn*}

We will often consider $\iAff_k \subset \iSt_k$ by the Yoneda embedding $\yon$
and identify an affine scheme with the corresponding representable stack.
We can transfer the ordinary notions on affine schemes to 
those on representable stacks.
For example, we have 

\begin{dfn}\label{dfn:St:sm}
A \emph{smooth morphism of representable stacks} 
is defined to be a smooth morphism of affine schemes 
in the sense of \cite[4\`{e}me partie, D\'{e}finition (17.3.1)]{EGA4}.
\end{dfn}

Next we recall the notion of geometric stacks.

\begin{dfn}[{\cite[Definition 1.3.3.1]{TVe2}}]\label{dfn:St:geom}
For $n \in \bbZ_{\ge -1}$, we define an \emph{$n$-geometric stack},
an object in $\iSt_k$, inductively on $n$.
At the same time we also define an $n$-atlas of a stack,
a \emph{$n$-representable morphism} and a \emph{$n$-smooth morphism} of stacks.
\begin{itemize}[nosep]
\item
Let $n=-1$.
\begin{enumerate}[nosep, label=(\roman*)]
\item
A \emph{$(-1)$-geometric stack} is defined to be a representable stack.

\item
A morphism $f: \stX \to \stY$ of stacks is called \emph{$(-1)$-representable} 
if for any representable stack $U$ and any morphism $U \to \stY$ of stacks, 
the pullback $\stX \times_{\stY} U$ is a representable stack. 

\item
A morphism $f: \stX \to \stY$ of stacks is called \emph{$(-1)$-smooth} 
if it is $(-1)$-representable, and 
if for any representable stack $U$ and any morphism $U \to \stY$ of stacks, 
the induced morphism $\stX \times_{\stY} U \to U$ is 
a smooth morphism of representable stacks (Definition \ref{dfn:St:sm}).
\end{enumerate}

\item
Let $n \in \bbN$.
\begin{enumerate}[nosep, label=(\roman*)]
\item
Let $\stX$ be a stack.
An \emph{$n$-atlas of $\stX$} is a small family $\{U_i \to \stX\}_{i \in I}$
of morphisms of stacks satisfying the following three conditions.
\begin{itemize}[nosep]
\item 
Each $U_i$ is a representable stack.
\item
Each morphism $U_i \to \stX$ is $(n-1)$-smooth.
\item
The morphism $\coprod_{i \in I} U_i \to \stX$ is an epimorphism of stacks.
\end{itemize}
We will sometimes denote an $n$-atlas $\{U_i \to \stX\}_{i \in I}$ 
simply by $\{U_i\}_{i \in I}$.

\item
A stack $\stX$ is called \emph{$n$-geometric} 
if the following two conditions are satisfied.
\begin{enumerate}[nosep, label=(\alph*)]
\item 
The diagonal morphism $\stX \to \stX \times \stX$ is $(n-1)$-representable.
\item
There exists an $n$-atlas of $\stX$.
\end{enumerate}

\item 
A morphism $f: \stX \to \stY$ of stacks is called \emph{$n$-representable}
if for any representable stack $U$ and for any morphism $U \to \stY$ of stacks, 
the derived stack $\stX \times_{\stY} U$ is $n$-geometric.
\item
A morphism $f: \stX \to \stY$ of stacks is called \emph{$n$-smooth} 
if for any representable stack $U$ and any morphism $U \to \stY$ of stacks,
there exists an $n$-atlas $\{U_i\}_{i \in I}$ of $\stX \times_{\stY} U$ such that 
for each $i \in I$ the composition $U_i \to \stX \times_{\stY} U \to U$ 
is a smooth morphism of representable stacks (Definition \ref{dfn:St:sm}). 
\end{enumerate}
\end{itemize}

\end{dfn}

\subsubsection{Relation to algebraic stacks}
\label{sss:St:hiArt}

As an illustration of geometric stacks,
let us explain the relation to algebraic spaces and algebraic stacks 
in the ordinary sense following \cite[\S 2.1.2]{TVe2}.
We begin with

\begin{dfn}[{\cite[\S 2.1.1]{TVe2}}]\label{dfn:St:trunc}
\begin{enumerate}[nosep]
\item
Let $m \in \bbN$.
A stack $\stX \in \iSt_k$ is called \emph{$m$-truncated} 
if $\pi_j(\stX(U),p)$ is trivial 
for any $U \in \iAff_k$, any $p \in \pi_0(\stX(U))$ and any $j>m$.
\item
An \emph{Artin $m$-stack} is an $m$-truncated $n$-geometric stack 
for some $n \in \bbN$.
\end{enumerate}
\end{dfn}

Here we used the homotopy groups $\pi_i(S)$ of a simplicial set $S$
(see \cite[\S 1.7]{GJ} for example).
Let us cite a relation between truncation and geometricity.

\begin{fct*}[{\cite[Lemma 2.1.1.2]{TVe2}}]
An $n$-geometric stack is $(n+1)$-truncated.
\end{fct*}

These higher Artin stacks will be the $\infty$-theoretic counterpart
of algebraic spaces and algebraic stacks.
See the appendix \S \ref{ss:Cl:AlgSt} for the relevant definitions.
Let us now cite from \cite[\S 2.1.2]{TVe2} a general construction of 
an embedding of the theory of fibered categories in groupoids 
into the theory of simplicial presheaves.

Let $\gsS$ be a site.
The category $\Grpd/\gsS$ of fibered categories in groupoids over $\gsS$ 
has a model structure such that 
the fibrant objects are ordinary stacks over $\gsS$
and the weak equivalences are the functors of fibered categories 
inducing equivalences on the associated stacks.
The homotopy category $\mHo(\Grpd/\gsS)$ of this model category
can be described as follows.
\begin{itemize}[nosep]
\item 
The objects are ordinary stacks over $\gsS$.
\item
The morphisms are $1$-morphisms of ordinary stacks up to $2$-isomorphisms. 
\end{itemize}
By \cite[\S 2.1.2]{TVe2} we have a Quillen adjunction 
\[
 \Grpd/\gsS \adjunc \mSh^{\le 1}(\gsS),
\]
where $\mSh^{\le 1}(\gsS)$ denotes the category of sheaves of 
simplicial sets over $\gsS$ with the $1$-truncated local 
projective model structure \cite[Theorem 3.7.3]{TVe1}.

Let us take $\gsS$ in the above argument to be the site $(\Aff_k,\et)$ of 
affine schemes with the \'etale topology over a commutative ring $k$
(Definition \ref{dfn:Cl:ET}).
Then the above Quillen adjunction yields an adjunction of the underlying 
$\infty$-categories (Definition \ref{dfn:ic:um}) by Fact \ref{fct:ic:Qi}.

\begin{ntn}\label{ntn:dSt:ase}
Let us denote the obtained adjunction by 
\[
 \ase: \Nsp(\Grpd/(\Aff_k,\et)) \adjunc \iSt_k : t.
\]
\end{ntn}

Algebraic stacks over $k$ in the sense of Definition \ref{dfn:Cl:AlgSt}
belong to $\Grpd/(\Aff_k,\et)$.
In particular, algebraic spaces and schemes over $k$ belong there.

Now we have the following result.

\begin{fct}[{\cite[Proposition 2.1.2.1]{TVe2}}]\label{fct:dSt:ase}
\begin{enumerate}[nosep]
\item
If $X$ is a scheme or an algebraic space over $k$, 
then $\ase(X)$ is an Artin $0$-stack which is $1$-geometric.
\item
If $X$ is an algebraic stack over $k$, 
then $\ase(X)$ is an Artin $1$-stack which is $1$-geometric.
\end{enumerate}
\end{fct}


Hereafter we consider algebraic stacks, algebraic spaces and schemes 
as objects of $\iSt_k$, i.e., as stacks.

\begin{rmk}\label{rmk:dSt:LM-O}
The paper \cite{O} treats a slightly generalized notion of algebraic stacks.
The difference is that in \cite{O} the diagonal morphism $\Delta$ is only 
assumed to be quasi-separated \cite[1re Partie, D\'{e}finition (1.2.1)]{EGA4}.
The statement in Fact \ref{fct:dSt:ase} (2) still holds 
for an algebraic stack in this sense.
\end{rmk}

The following Fact \ref{fct:dSt:mAS=AS} shows that 
the notion of Artin $n$-stacks is a natural extension of 
the ordinary notion of schemes, algebraic spaces and algebraic stacks
(see also Fact \ref{fct:Cl:St-AS}).

\begin{fct}[{\cite[Remark 2.1.1.5]{TVe2}}]\label{fct:dSt:mAS=AS}
Let $\stX$ be an Artin $m$-stack with $m \in \bbN$. 
\begin{enumerate}[nosep]
\item
$\stX$ is an algebraic space 
if and only if the following two conditions hold.
\begin{itemize}[nosep]
\item 
$\stX$ is a Deligne-Mumford $m$-stack, i.e.,
there exists an $m$-atlas $\{U_i\}_{i \in I}$ of $\stX$ 
such that each morphism $U_i \to \stX$ of stacks is an \'etale morphism.

\item
The diagonal map $\stX \to \stX \times \stX$ is a monomorphism of stacks.
\end{itemize}
If $\stX$ is an algebraic space, 
then $\stX$ is $1$-geometric.

\item
$\stX$ is a scheme if and only if there exists an $m$-atlas $\{U_i\}_{i \in I}$ 
of $\stX$ such that each morphism $U_i \to \stX$ of stacks is a monomorphism.
If so then $\stX$ is $1$-geometric.

\item
Assume $\stX$ is an algebraic space or a scheme.
Then $\stX$ has an affine diagonal, i.e., the diagonal map is $(-1)$-representable, 
if and only if it is $0$-geometric.
\end{enumerate}
\end{fct}

\begin{rmk}\label{rmk:dSt:mAS=AS}
By Fact \ref{fct:dSt:ase}, 
if $\stX$ in Fact \ref{fct:dSt:mAS=AS} is an algebraic space or a scheme, 
then we have $m=0$ and $\stX$ is $1$-geometric.
Thus we can depict the relation between sub-$\infty$-categories of $\iSt$ 
and subcategories of the category $\St_{\ord}$ of ordinary stacks as
\[
 \xymatrix{
    \Ner(\Sch)          \ar@{^{(}->}[r]   \ar@{=}[d]
  & \Ner(\AlgSp)        \ar@{^{(}->}[r]   \ar@{=}[d] 
  & \Ner(\AlgSt)        \ar@{^{(}->}[rrr] \ar@{^{(}->}[d]  & & 
  & \Ner(\St_{\ord})    \ar@{^{(}->}[d]^{a} \\
    \iSch               \ar@{^{(}->}[r] 
  & \iAlgSp_{n=1}^{m=0} \ar@{^{(}->}[r] 
  & \iArtin_{n=1}^{m=1} \ar@{^{(}->}[r] 
  & \iArtin_{n=1}       \ar@{^{(}->}[r] 
  & \iArtin = \Ho \iSt_{\geom}^{\trun} \ar@{^{(}->}[r]
  & \iSt
 }
\]
Here $\AlgSp$ and $\AlgSt$ denote the category of algebraic spaced 
and of algebraic stacks respectively.
The subscript $n$ in the second row denotes the geometricity of a stack, 
and the superscript $m$ denotes the truncation degree.
\end{rmk}

We close this part with the recollection of a quotient stack,
a standard example of algebraic stacks.

\begin{fct*}[{\cite[Example 8.1.12]{O:book}}]
Let $S$ be a scheme, $X$ be an algebraic space over $S$, 
and $G$ be a smooth group scheme over $S$. 
Assume $G$ acts on $X$.
Define $[X/G]$ to be the ordinary stack on $S$ (Definition \ref{dfn:Cl:ord-st})
whose objects are triples $(T,P,\pi)$ consisting of the following data.
\begin{itemize}[nosep]
\item 
$T \in \Sch_S$.
\item
$P$ is a $G \times_S T$-torsor on the big \'etale site $\gsET(T)$ 
(Definition \ref{dfn:Cl:ET}).
\item
$\pi: P \to X \times_S T$ is a $G \times_S T$-equivariant morphism 
of sheaves on $\Sch_T$.
\end{itemize}
Then $[X/G]$ is an algebraic stack.
A smooth covering of $[X/G]$ is given by $q: X \to [X/G]$,
which is defined by the triple $(S,G_X,\rho)$
with $G_X := G \times X$ the trivial $G$-torsor on $X$ and 
$\rho: G_X \to X$ the action map of $G$ on $X$.

The algebraic stack $[X/G]$ is called the \emph{quotient stack}.
\end{fct*}

\begin{dfn}\label{dfn:St:BG}
For a smooth group scheme $G$ over a scheme $S$,
the \emph{classifying stack} $B G$ of $G$ is defined to be 
the quotient stack $[S/G]$ where $G$ acts trivially on $S$.
\end{dfn}

In the case $S=\Spec k$,
the classifying stack $B G$ is a $1$-geometric Artin $1$-stack
by Fact \ref{fct:dSt:ase}.

\subsection{Derived stacks}
\label{ss:dSt:dSt}

In this subsection we present a summary of derived stacks.
As in the previous \S \ref{ss:dSt:St}, 
our exposition will be given in the $\infty$-categorical language.

\subsubsection{The $\infty$-category of simplicial modules}
\label{sss:dSt:isMod}

This part is a preliminary for arguments in a simplicial setting.

Let $k$ be a commutative ring. 
We denote by $\Mod_k$ the category of $k$-modules.
We also denote by $\sMod_k$ the category of simplicial $k$-modules, 
i.e., of functors $\CS^{\op} \to \Mod_k$.
The category $\sMod_k$ is equipped with the model structure 
given by the Kan model structure on the underlying simplicial sets.
We denote by $\sMod_k^{\circ} \subset \sMod_k$ the full subcategory
of fibrant-cofibrant objects.
Now recall the underlying $\infty$-category of a simplicial model category 
(Definition \ref{dfn:ic:um}).
Using this notion, we introduce

\begin{dfn*}
The underlying $\infty$-category
\[
 \isMod_k := \Nsp(\sMod_{k}^{\circ}) 
\]
is called the \emph{$\infty$-category of simplicial $k$-modules}.
\end{dfn*}

Let us give another description of the $\infty$-category $\isMod_k$.
We will use basic notions on dg-categories 
(see \S \ref{ss:mod:dg} for a short account).
We denote by $\C^-(k)$ 
the dg-category of non-positively graded complexes of $k$-modules.
Taking the dg-nerve (Definition \ref{dfn:stb:Ndg}),
we have an $\infty$-category $\Ndg(\C^-(k))$.
We then have an equivalence of $\infty$-categories
\[
 \Ndg(\C^-(k)) \simeq \isMod_k.
\]
This equivalence is given by the Dold-Kan correspondence 
\cite[Chap.\ III, \S 2]{GJ}, \cite[Theorem 1.2.3.7]{Lur2}.

The standard tensor product on $\C^{-}(k)$ induces 
an $\infty$-operad structure on $\isMod_k$ in the sense of \cite{Lur2}.
We will often denote the obtained $\infty$-operad $\isMod_k^{\otimes}$ 
by the simplified symbol $\isMod_k$. 

\subsubsection{Derived rings}
\label{sss:dSt:dring}

Following \cite[\S2.2]{T14} 
we introduce the $\infty$-category of affine derived schemes.

Let us fix a commutative ring $k$.
As in the previous \S \ref{ss:dSt:St}, 
we denote by $\Com_k$ the category of commutative $k$-algebras.

\begin{dfn}\label{dfn:dSt:dring}
A simplicial object in $\Com_k$, 
in other words, a functor $\CS^{\op} \to \Com_k$, 
is called a \emph{simplicial commutative $k$-algebra} 
or a \emph{derived $k$-algebra}.
In the case $k=\bbZ$, we call it a \emph{derived ring}.
We denote by $\sCom_k$ the category of derived $k$-algebras.
\end{dfn}

We can associate to a derived $k$-algebra $A$ a commutative graded $k$-algebra
\[
 \pi_*(A) = \oplus_{n \in \bbN} \pi_n(A), 
\]
where the base point is taken to be $0 \in A$.
See \cite[\S 1.7]{GJ} for the definition of homotopy groups of simplicial sets.
In particular, we have a commutative $k$-algebra $\pi_0(A)$.

The category $\sCom_k$ has a model structure induced by the Kan model structure 
(Fact \ref{fct:ic:Km}) of the underlying simplicial sets.
Then, regarding $\sCom_k$ as a simplicial model category 
(Definition \ref{dfn:ic:spl-model}),
we have the underlying $\infty$-category $\Nsp(\sCom_k^{\circ})$
(Definition \ref{dfn:ic:um}).
Here $\sCom_k^{\circ} \subset \sCom_k$ is the full subcategory 
spanned by fibrant-cofibrant objects with respect to the Kan model structure,
and $\Nsp(-)$ is the simplicial nerve construction (Definition \ref{dfn:ic:Nsp}).

\begin{dfn}\label{dfn:dSt:isCom}
The obtained $\infty$-category 
\[
 \isCom_k := \Nsp(\sCom_k^{\circ})
\]
is called the \emph{$\infty$-category of derived $k$-algebras}.
\end{dfn}

\begin{rmk*}
Here is another description of $\isCom_k$.
Consider the set $W$ of weak equivalences 
in the model category $\sCom_k$ in the above sense.
We apply the $\infty$-localization (Definition \ref{dfn:ic:loc}) to 
the pair $(\Ner(\sCom_k),W)$, and have an equivalence of $\infty$-categories
\[
 \isCom_k \simeq \Ner(\sCom_k)[W^{-1}].
\]
\end{rmk*}

\begin{rmk}\label{rmk:dSt:dag/sag:1}
We have a functor 
\[
 \isCom_k \longto \iCAlg_k^{\cn}
\]
of $\infty$-categories, which is an equivalence if $k$ is a $\bbQ$-algebra.
The target $\iCAlg_k^{\cn}$ is the $\infty$-category of connected 
\emph{commutative ring spectra},
which is a foundation of Lurie's spectral algebraic geometry \cite{Lur5}--\cite{Lur14}.
See \ref{ss:sp:ring} for the detail.
\end{rmk}

\begin{dfn}\label{dfn:dSt:idAff}
\begin{enumerate}
\item 
The $\infty$-category $\idAff_k$ of \emph{affine derived schemes over $k$}
is defined to be
\[
 \idAff_k := (\isCom_k)^{\op}, 
\]
the opposite $\infty$-category of $\isCom_k$.
The object in $\idAff_k$ corresponding to $A \in \isCom_k$ 
will be denoted by $\dSpec A$ and called the \emph{affine derived scheme} of $A$.
A morphism in $\idAff_k$ will be called a \emph{morphism of affine derived schemes}.

\item
For an affine derived scheme $U = \dSpec A$ over $k$,
the morphism $U \to \dSpec k$ of affine derived schemes corresponding to $k \to A$
is called the \emph{structure morphism} of $U$.
\end{enumerate}
\end{dfn}

In terms of the $\infty$-operad $\isMod_k^{\otimes}$ (\S \ref{sss:dSt:isMod}), 
a derived $k$-algebra is nothing but 
a commutative ring object in $\isMod_k^{\otimes}$.

One may infer a relation between the $\infty$-categories $\idAff_k$ and 
$\iAff_k$ in the previous \S \ref{ss:dSt:St}.
We postpone the explanation to \S \ref{sss:dSt:trunc}.

Let us now introduce some classes of morphisms of affine derived schemes.

\begin{dfn}[{\cite[\S 2.2.2]{TVe2}}]
\label{dfn:dSt:dAff:mor}
\begin{enumerate}[nosep]
\item
A morphism $A \to B$  in $\isCom_k$ is \emph{\'etale} (resp.\ \emph{smooth}, 
resp.\ \emph{flat}) if the following two conditions are satisfied.
\begin{itemize}[nosep]
\item 
the induced morphism $\pi_0(A) \to \pi_0(B)$ is 
an \'etale (resp.\ smooth, resp.\ flat) morphism of commutative $k$-algebras,
\item
the induced morphism $\pi_i(A) \otimes_{\pi_0(A)} \pi_0(B) \to \pi_i(B)$
is an isomorphism for any $i$.
\end{itemize}

\item
A morphism $X \to Y$ in the $\infty$-category $\idAff_k$ of affine derived schemes 
is called \emph{\'etale} (resp.\ \emph{smooth}, resp.\ \emph{flat}) 
if the corresponding one in $\isCom_k$ is so.
\end{enumerate}
\end{dfn}

For later use we also introduce the notion of a finitely presented morphism.
We need some terminology.
\begin{itemize}[nosep]
\item 
By a \emph{filtered system} $\{C_i\}_{i \in I}$ of objects 
in an $\infty$-category $\iC$, we mean a diagram in $\iC$ indexed by a 
$\kappa$-filtered poset $I$ with some regular cardinal $\kappa$.

\item
$\iC_{C/}$ denotes the under-$\infty$-category 
of an $\infty$-category $\iC$ under an object $C \in \iC$ 
(Definition \ref{dfn:ic:o/u}).

\item
$\Map_{\iC}(-,-)$ denotes the mapping space for an $\infty$-category $\iC$
(Definition \ref{dfn:ic:Map}),
which is an object of the homotopy category $\topH$ of spaces 
(Definition \ref{dfn:ic:topH}).
\end{itemize}

\begin{dfn}[{\cite[Definition 1.2.3.1]{TVe2}, \cite[\S 2.3]{TVa}}]
\label{dfn:dSt:fin-pr}
\begin{enumerate}[nosep]
\item 
A morphism $f: A \to B$ in $\isCom_k$ is called \emph{finitely presented}
if for any filtered system $\{C_i\}_{i \in I}$ of objects in $(\isCom_k)_{A/}$
the natural morphism
\[
 \iclim_{i \in I} \Map_{(\isCom_k)_{A/}}(B,C_i) \longto 
 \Map_{(\isCom_k)_{A/}}(B, \iclim_{i \in I} C_i)
\]
is an isomorphism in $\topH$.

\item
A derived $k$-algebra $A \in \isCom_k$ is called \emph{finitely presented}
or \emph{of finite presentation}
if the morphism $k \to A$ is finitely presented in the sense of (1).
\end{enumerate}
\end{dfn}

%

\subsubsection{Derived stacks}

In \cite[Chap.\ 2.2]{TVe2} derived stacks are defined 
in terms of homotopical algebraic geometry context.
Here we give a redefinition in terms of the theory of $\infty$-topos
following \cite[\S 3.2]{T14}.

We begin with giving $\idAff_k$ an $\infty$-categorical Grothendieck topology 
(\S \ref{ss:pr:inf-site}).

\begin{dfn}[{\cite[\S 2.2.2]{TVe2}}]\label{dfn:dAff:et-cov}
A family $\{\dSpec A_i \to \dSpec A\}_{i \in I}$ of morphisms in $\idAff_k$ 
is called an \emph{\'etale covering} if the following conditions are satisfied.
\begin{itemize}[nosep]
\item 
For each $i \in I$, the morphism 
$\pi_*(A) \otimes_{\pi_0(A)}\pi_0(A_i) \to \pi_*(A_i)$
is an isomorphism of graded rings.
\item
There exists a finite subset $J \subset I$ such that 
the induced morphism $\coprod_{j \in J}\Spec \pi_0(A_j) \to \Spec \pi_0(A)$
of affine schemes is a surjection.
\end{itemize}
\end{dfn}

Let us note that this definition is equivalent to the one 
in \cite[Definition 2.2.2.12]{TVe2} by the argument given there.

By the argument in \cite[\S 2.2.2]{TVe2},
\'etale coverings give a Grothendieck topology on $\idAff_k$
in the sense of Definition \ref{dfn:ic:inf-site}.
In particular, 
\'etale coverings are stable under pullbacks.

\begin{dfn*} 
The obtained $\infty$-site is called the \emph{\'etale $\infty$-site} 
and denoted by $(\idAff_k, \et)$  
\end{dfn*}

Then, as in Definition \ref{dfn:dSt:iSt}, we can introduce

\begin{dfn}\label{dfn:dSt:idSt}
The \emph{$\infty$-category of derived stacks over $k$} is defined to be
\[
 \idSt_k := \iSh(\idAff_k,\et).
\]
Its object is called a \emph{derived stack}.
\end{dfn}

\begin{rmk*}
In \cite{TVe2} the model category corresponding to $\idSt_k$ is denoted
by $D^{-}\mathrm{St}(k)$ and its object is called a $D^-$-stack. 
See also Remark \ref{rmk:dSt:stack} on the difference 
of our terminology on derived stacks 
and that on $D^-$-stacks given in \cite{TVe2}.
\end{rmk*}

The $\infty$-category $\idSt_k$ has similar properties to $\iSt_k$.
For example, by Fact \ref{fct:pr:sh} and Corollary \ref{cor:ic:iSh:hc}, we have

\begin{fct}\label{fct:idSt:hyper}
$\idSt_k$ is a hypercomplete $\infty$-topos (Definition \ref{dfn:ic:hcpl}).
\end{fct}

We also have

\begin{fct}[{\cite[Lemma 2.2.2.13]{TVe2}}]\label{fct:dSt:et:qcpt}
The $\infty$-topos $\idSt_k$ is quasi-compact (Definition \ref{dfn:pr:qcpt}).
\end{fct}

As in the case of stacks, we have the $\infty$-theoretic Yoneda embedding
(Definition \ref{dfn:pr:Yoneda:iSh})
\[
 \yon: \idAff_k \longinj \idSt_k. 
\]

\begin{dfn*}
A derived stack in the essential image of $j$ is called 
a \emph{representable derived stack}.
\end{dfn*}

Hereafter, 
regarding $\idAff_k$ as a sub-$\infty$-category of $\idSt_k$ 
by the Yoneda embedding $\yon$,
we consider an affine derived scheme as a derived stack.
Thus every notion on affine derived schemes 
such as in Definition \ref{dfn:dSt:dAff:mor}
can be transfered to that on representable derived stacks.

For each $\stX \in \idSt_k$, we have a morphism $\stX \to \dSpec k$ in $\idSt_k$,
which will be called the \emph{structure morphism} of $\stX$.

Let us introduce some notions on morphisms of derived stacks.
A \emph{morphism} between derived stacks 
means a morphism in the $\infty$-category $\iSt_k$.
We then have the notion of \emph{monomorphisms} (Definition \ref{dfn:ic:mono}).
We also have the notion of \emph{effective epimorphism} 
(Definition \ref{dfn:it:eff-epi}) since $\iSt_k$ is an $\infty$-topos.

\begin{dfn}\label{dfn:dSt:epi-mono}
A \emph{monomorphism} or an \emph{injection} of derived stacks is defined to be 
a monomorphism in $\idSt_k$.
A \emph{epimorphism} or a \emph{surjection} of derived stacks is defined to be 
an effective epimorphism in $\idSt_k$.
\end{dfn}

\begin{rmk*}
As in the case of stacks (Remark \ref{rmk:dSt:epi-mono}), 
we can restate this definition more explicitly.
For a derived stack $\stX$, we denote by 
$\pi_0(\stX) \in \Fun((\idAff_k)^{\op},\Set)$ 
the sheaf of sets obtained by taking $\pi_0$.
For a morphism $f: \stX \to \stY$ of derived stacks, we denote by $\pi_0(f)$ 
the induced morphism $\pi_0(\stX) \to \pi_0(\stY)$ of sheaves of sets.
Then
\begin{itemize}[nosep]
\item
$f$ is a monomorphism  if and only if the induced morphism $\pi_0(\Delta_f)$ of 
$\Delta_f: \stX \to \stX \times_{f, \stY, f} \stX$ is an isomorphism
in $\Fun((\idAff_k)^{\op},\Set)$.

\item
$f$ is an epimorphism if and only if the induced morphism $\pi_0(f)$
is an epimorphism in the category $\Fun((\idAff_k)^{\op},\Set)$.

\end{itemize}
Again, our convention on monomorphisms and epimorphisms 
is consistent with \cite[Definition 1.3.1.2]{TVe2}.
\end{rmk*}

Let us also introduce

\begin{dfn}[{\cite[Definition 1.3.6.4 (2)]{TVe2}}]\label{dfn:dSt:qcpt}
A morphism $f: \stX \to \stY$ in $\idSt_k$ is \emph{quasi-compact}
if for any morphism $p: U \to \stY$ from an affine derived stack $U$,
there exists a finite family $\{U_i\}_{i \in I}$ of affine derived stacks
and a surjection $\coprod_{i \in I} U_i \to \stX \times_{f,\stY,p} U$. 
\end{dfn}

In the most of the following sections we will work on a fixed commutative ring,
but in \S \ref{s:pv} we need a base change.
For that, let us consider a field $k$ and an extension  $L$ of $k$.
Then, for a derived stack $\stX \in \idSt_k$,
the fiber product $\stX \times_{\dSpec k} \dSpec_L$ 
in $\idSt_k$ defines a derived stack over $L$.

\begin{ntn}\label{ntn:dSt:XL}
We denote the fiber product $\stX \times_{\dSpec k} \dSpec_L \in \idSt_L$
by $\stX_L$ or $\stX \otimes_{k} L$.
\end{ntn}

\subsubsection{Geometric derived stacks}
\label{ss:dSt:gdSt}

Following \cite[Chap.\ 2.2]{TVe2} and \cite[\S 2.3]{TVa} 
we recall the notion of geometric derived stack.

\begin{dfn}[{\cite[Definition 1.3.3.1]{TVe2}}]\label{dfn:dSt:geom}
For $n \in \bbZ_{\ge -1}$, an \emph{$n$-geometric stack} is an object in $\idSt_k$
defined in the same way as in Definition \ref{dfn:St:geom} 
with the replacement of ``representable stack" by ``representable derived stack".
An \emph{$n$-atlas}, an \emph{$n$-representable morphism} and 
an \emph{$n$-smooth morphism} of derived stacks are inductively defined in the same way.
\end{dfn}

\begin{rmk}\label{rmk:dSt:geom}
\begin{enumerate}[nosep]
\item
In the recursive definition of $n$-smoothness of morphisms,
we use Definition \ref{dfn:dSt:dAff:mor} for smoothness of morphisms in $\idAff$.

\item 
In \cite[\S1.3.4]{TVe2} it is explained that the condition (b) 
in the definition of $n$-geometric derived stacks follows from the condition (a).
Thus it is enough to assume the existence of $n$-atlas only.

\item
In \cite{T14} a geometric derived stack is called a derived Artin stack.
We will not use this terminology.
\end{enumerate}
\end{rmk}

By definition we can check the following statement
(see also Fact \ref{fct:dSt:mAS=AS} (3)).

\begin{lem}\label{lem:dSt:0-geom}
A derived stack $\stX$ is $0$-geometric if and only if it has an affine diagonal,
i.e. the diagonal morphism $\stX \to \stX \times \stX$ is $(-1)$-representable.
\end{lem}

We will repeatedly use the following properties.

\begin{fct}[{\cite[Proposition 1.3.3.3, Corollary 1.3.3.5]{TVe2}}]
\label{fct:dSt:geom}
Let $n \in \bbZ_{\ge -1}$.
\begin{enumerate}[nosep]
\item
An ($n-1$)-representable (resp.\ ($n-1$)-smooth) morphism 
is $n$-representable (resp. $n$-smooth).
In particular, an $(n-1)$-geometric stack is $n$-geometric.
\item
The class of $n$-representable (resp.\ $n$-smooth) morphisms are stable 
under isomorphisms, pullbacks and compositions.
\item 
For $n \in \bbN$, the $\infty$-category of $n$-geometric derived stacks
is stable under pullbacks and taking small coproducts.
\end{enumerate}
\end{fct}

%
%


We next recall the following extended class of geometric derived stacks.

\begin{dfn}[{\cite[Definition 2.17]{TVa}}]
\label{dfn:dSt:lcgm}
A derived stack $\stX$ is called \emph{locally geometric} 
if $\stX$ is equivalent to a filtered colimit 
\[
 \stX \simeq \iclim_{i \in I} \stX_i
\]
of derived stacks $\{\stX_i \}_{i \in I}$ satisfying the following conditions.
\begin{itemize}[nosep]
\item
Each derived stack $\stX_i$ is $n_i$-geometric for some $n_i \in \bbZ_{\ge-1}$.
\item
Each morphism $\stX_i \to \stX_i \times_{\stX} \stX_i$ of derived stacks induced by
$\stX_i \to \stX$ is an equivalence in $\idSt_k$.
\end{itemize}
\end{dfn}

\begin{dfn}[{\cite[\S 2.3]{TVa}}]\label{dfn:dSt:lfp}
\begin{enumerate}[nosep]
\item
An $n$-geometric derived stack $\stX$ is called \emph{locally of finite presentation}
if it has an $n$-atlas $\{U_i\}_{i \in I}$ such that 
for each representable derived stack $U_i \simeq \dSpec A_i$ 
the derived $k$-algebra $A_i$ is finitely presented
(Definition \ref{dfn:dSt:fin-pr}).

\item
A locally geometric derived stack $\stX$ is \emph{locally of finite presentation} 
if writing $\stX \simeq \iclim_i \stX_i$ each geometric derived stack $\stX_i$
can be chosen to be locally of finite presentation in the sense of (1).
\end{enumerate}
\end{dfn}

The class of locally geometric derived stacks locally of finite presentation
contains the moduli stack of perfect dg-modules
by the following theorem of To\"{e}n and Vaqui\'{e} \cite{TVa}.
This fact is very important for us and 
will be explained in detail in \S \ref{ss:mod:g}.

\begin{fct}[{\cite[Theorem 3.6]{TVa}}]\label{fct:dSt:TVa}
For a dg-category $\catD$ over $k$ of finite type,
we have the moduli space $\stM(\catD)$ of perfect $\catD^{\op}$-modules,
which is a locally geometric derived stack locally of finite presentation.
\end{fct}

For later use, we record 

\begin{dfn}[{\cite[Definition 1.3.6.2, Lemma 2.2.3.4]{TVe2}}]
\label{dfn:dSt:mor-Q}
Let $\bfQ$ be either of the following properties of morphisms in $\idAff_k$:
\[
 \bfQ = \text{flat, smooth, \'etale, locally of finitely presented}.
\]
A morphism $f: \stX \to \stY$ of derived stacks \emph{has property $\bfQ$}
if it is $n$-representable for some $n$ and if for any affine derived scheme $U$ 
and any morphism $U \to \stY$ there exists an $n$-atlas $\{U_i\}_{i \in I}$ 
of $\stX \times_{\stY} U$
such that each morphism $U_i \to \stX \times_{\stY} U \to U$ 
in $\idAff_k$ has property $\bfQ$.
\end{dfn}

In this definition the choice of $n$ is irrelevant 
by \cite[Proposition 1.3.3.6]{TVe2}.
We also note that for $\bfQ=\text{smooth}$, Definition \ref{dfn:dSt:mor-Q} of 
an $n$-representable morphism being a smooth morphism and 
Definition \ref{dfn:dSt:geom} of an $n$-smooth morphism are compatible.

\begin{fct}[{\cite[Proposition 1.3.6.3]{TVe2}}]\label{fct:morph:Q}
Let $\bfQ$ be one of the properties of morphisms in Definition \ref{dfn:dSt:mor-Q}.
\begin{enumerate}[nosep]
\item 
Morphisms of derived stacks having property $\bfQ$ are stable under equivalences,
compositions and pullbacks.
\item
Let $f: \stX \to \stY$ be a morphism of $n$-geometric derived stacks.
If there is an $n$-atlas $\{U_i \to \stY\}_{i \in I}$ such that 
each projection $\stX \times_{\stY} U_i \to U_i$ has property $\bfQ$,
then $f$ has property $\bfQ$.
\end{enumerate}
\end{fct}

Next we introduce the relative dimension of a smooth morphism.
One can check the following is well-defined using Fact \ref{fct:morph:Q}
and Definition \ref{dfn:dSt:dAff:mor}.

\begin{dfn}\label{dfn:dSt:reldim}
A morphism $f: \stX \to \stY$ of derived stacks is 
\emph{smooth of relative dimension $d$} if it satisfies the following condition.
\begin{itemize}[nosep]
\item
For any $U \in \idAff_k$ and any morphism $U \to \stY$ of derived stacks,
take an $n$-atlas $\{U_i\}_{i \in I}$ of $\stX \times_{\stY} U$ 
as Definition \ref{dfn:dSt:mor-Q}.
Denote by $g: U_i \to \stX \times_{\stY} U \to U$ the smooth morphism in $\idAff_k$.
Then $\pi_0(g): \pi_0(U_i) \to \pi_0(U)$ is smooth of relative dimension $d$
in the scheme-theoretic meaning \cite[Chap.\ IV, \S 17.10]{EGA4}.
\end{itemize}
\end{dfn}

Let us also introduce immersions of derived stacks.

\begin{dfn}[{\cite[Definition 2.2.3.5]{TVe2}}]\label{dfn:dSt:open}
\begin{enumerate}[nosep]
\item 
A morphism of derived stacks is an \emph{open immersion}
if it is locally of finite presentation (Definition \ref{dfn:dSt:mor-Q}),
flat (Definition \ref{dfn:dSt:mor-Q})
and a monomorphism (Definition \ref{dfn:dSt:epi-mono}).
\item
A morphism $\stF \to \stX$ of derived stacks is a \emph{closed immersion}
if it is $(-1)$-representable and if for any representable stack 
$U \simeq \dSpec A$ and any morphism $U \to \stX$ 
the morphism $\stF \times_{\stX} U \simeq \dSpec B \to U \simeq \dSpec A$
induces an epimorphism $\pi_0(A) \to \pi_0(B)$ of rings.
\end{enumerate}
\end{dfn}

Note that an open immersion in the above sense is called
a Zariski open immersion in \cite{TVe2}.
One can check by definition that a closed immersion $i: \stF \to \stX$ 
of derived stacks determines an open immersion $\stU \to \stX$ with the property
$[\stU(T)] = [\stX(T)] \setminus [\stF(T)]$ for any affine derived scheme $T$, 
where $[-]$ denotes the homotopy type of a simplicial set 
(Definition \ref{dfn:ic:Ho-type:sSet}).
Moreover $\stU$ is unique up to contractible ambiguity.

\begin{ntn}\label{ntn:dSt:c-o}
We call the open immersion $\stU \to \stX$ 
the \emph{complement} of $i: \stF \to \stX$.
\end{ntn}

We finally have 

\begin{dfn}\label{dfn:dSt:fp}
A morphism of derived stacks is \emph{of finite presentation}
if it is locally of finite presentation (Definition \ref{dfn:dSt:mor-Q})  
and quasi-compact (Definition \ref{dfn:dSt:qcpt}).
\end{dfn}

\subsubsection{Truncation}
\label{sss:dSt:trunc}

Let us now explain the relationship of higher Artin stacks 
in \S \ref{sss:St:hiArt} and geometric derived stacks.
We follow the argument in \cite[\S\S 2.1.1, 2.1.2, 2.2.4]{TVe2}.

Recall the $\infty$-category $\iCom_k$ of commutative $k$-algebras
(see \S \ref{ss:dSt:St}).
Regarding commutative $k$-algebras as constant derived $k$-algebras, 
we have an embedding $\iCom_k \to \isCom_k$ of $\infty$-category.
Taking the opposites, we have the corresponding embedding 
$i: \iAff_k \longinj \idAff_k$.
It gives rise to an adjunction
\[
 \pi_0: \idAff_k \adjunc \iAff_k :i
\]
of $\infty$-categories (Definition \ref{dfn:ic:adj}),
where $\pi_0$ denotes the functor taking $\pi_0$ of derived $k$-algebras.
This adjunction naturally extends to a new adjunction
$\iPSh(\iAff_k) \adjunc \iPSh(\idAff_k)$
of the $\infty$-categories of presheaves (Definition \ref{dfn:ic:iPSh}).
In \cite[\S4.8]{TVe1} the corresponding Quillen adjunction is denoted by 
$i_!: \PSh(\Aff_k) \rlto \PSh(\dAff_k) :i^*$.

By the argument of \cite[\S2.2.4]{TVe2}, 
the last adjunction $(i_!,i^*)$ yields a Quillen adjunction 
between simplicial categories,
whose underlying adjunction on the homotopy categories can be expressed  as 
$\Ho \iSh(\iAff_k, \et) \adjunc \Ho \iSh(\idAff_k,\et)$.
Then by Fact \ref{fct:ic:Qi}, 
we can construct from the Quillen adjunction an adjunction of $\infty$-categories.
Let us denote it by 
\[
 \Dex : \iSt_k = \iSh(\iAff_k, \et) \adjunc \idSt_k = \iSh(\idAff_k,\et) : \Trc.
\]

\begin{fct*}[{\cite[Lemma 2.2.4.1]{TVe2}}]
The functor $\Dex: \iSt_k \to \idSt_k$ is fully faithful.
\end{fct*}

\begin{dfn}[{\cite[Definition 2.2.4.3]{TVe2}}]\label{dfn:dSt:trunc}
\begin{enumerate}[nosep]
\item
$\Trc : \idSt_k \to \iSt_k$ is called the \emph{truncation} functor and 
$\Dex : \iSt_k \to \idSt_k$ is called the \emph{extension} functor.

\item
A derived stack $\stX \in \idSt_k$ is called \emph{truncated} if
the adjunction $\Dex(\Trc (\stX)) \to \stX$ is an isomorphism in $\Ho \idSt_k$.
\end{enumerate}
\end{dfn}

As noted in \cite{TVe2}, $\Trc$ commutes with limits and colimits,
and $\Dex$ is fully faithful and commutes with colimits,
but $\Dex$ does not commute with limits.
However we have the following results.

\begin{fct}[{\cite[Proposition 2.2.4.4]{TVe2}}]\label{fct:dSt:trc}
\begin{enumerate}
\item 
The truncation functor $\Trc$ sends an $n$-geometric derived stack to 
an $n$-geometric stack and a flat (resp.\ smooth, resp.\ \'etale) morphism
between $n$-geometric derived stacks to 
a morphism of the same type between $n$-geometric stacks.
The functor $\Trc$ also sends epimorphisms of derived stacks to 
epimorphisms of stacks.

\item
$\Dex$ sends $n$-geometric stacks to $n$-geometric derived stacks,
pullbacks of $n$-geometric stacks to those of $n$-geometric derived stacks,
and flat (resp.\ smooth, resp.\ \'etale) morphisms of $n$-geometric stacks to
those of $n$-geometric derived stacks.
\end{enumerate}
\end{fct}

Recall that an algebraic stack over $k$ (Definition \ref{dfn:Cl:AlgSt}) 
belongs to the category $\Ho(\Grpd/(\Aff_k,\et))$
and that we have the fully faithful functor $\ase$ (Notation \ref{ntn:dSt:ase}).
Now we introduce

\begin{dfn}\label{dfn:dSt:St-dSt}
Consider the composition 
\[
 \iota := \Dex \circ \ase: 
 \Nsp(\Grpd/(\Aff_k,\et)) \longto \iSt_k \longto \idSt_k.
\]
For an algebraic stack $X$, 
we call the image $\iota(X)$ the \emph{derived stack associated to $X$}.
\end{dfn}

By Fact \ref{fct:dSt:ase}, Remark \ref{rmk:dSt:LM-O} and Fact \ref{fct:dSt:trc},
we have

\begin{lem}\label{lem:dSt:ase}
Let $X$ be an algebraic stack.
Then the derived stack $\iota(X)$ is truncated and $1$-geometric 
(Definition \ref{dfn:dSt:geom}).
\end{lem}

\begin{rmk}\label{rmk:dSt:diag}
As in Remark \ref{rmk:dSt:mAS=AS},
we can depict the relation between sub-$\infty$-categories of $\iSt$ 
and sub-$\infty$-categories of $\idSt$  as
\[
 \xymatrix{
    \iSch                               \ar@{^{(}->}[r] 
  & \iAlgSp_{n=1}^{m=0} \ar@{^{(}->}[d] \ar@{^{(}->}[r] 
  & \iArtin_{n=1}^{m=1} \ar[d]^{\sim}   \ar@{^{(}->}[r] 
  & \iArtin_{n=1}       \ar[d]^{\sim}   \ar@{^{(}->}[r] 
  & \iArtin = \iSt_{\geom}^{\trun}      \ar@{^{(}->}[r] \ar[d]^{\sim}
  & \iSt \ar@{^{(}->}[d]^{\Dex} \\
  & \idSt_{n=1}^{m=0}     \ar@{^{(}->}[r] 
  & \idSt_{n=1}^{m=1}     \ar@{^{(}->}[r] 
  & \idSt_{n=1}^{\trun}   \ar@{^{(}->}[r] 
  & \idSt_{\geom}^{\trun} \ar@{^{(}->}[r]
  & \idSt 
 }
\]
One also can check that the notions ordinary algebraic stacks (\S \ref{ss:Cl:AlgSt})
are compatible with those on derived stacks under the functor $\Dex$.
As for the properties of morphisms, 
we collect the corresponding definitions in the following table.
We will give an additional table after introducing derived algebraic spaces
(Remark \ref{rmk:dAS:table})

\begin{table}[htbp]
\begin{tabular}{c|c|c}
property of morphisms & derived stacks & algebraic stacks \\
\hline
surjective 
& Definition \ref{dfn:dSt:epi-mono} & Definition \ref{dfn:Cl:St-mor1} \\
locally of finite presentation
& Definition \ref{dfn:dSt:mor-Q}    & Definition \ref{dfn:Cl:St-mor1} \\
\'etale
& Definition \ref{dfn:dSt:mor-Q}    & Definition \ref{dfn:Cl:St-mor2} \\
smooth of relative dimension $d$ 
& Definition \ref{dfn:dSt:reldim}   & Definition \ref{dfn:Cl:St-mor2} \\
quasi-compact
& Definition \ref{dfn:dSt:qcpt}     & Definition \ref{dfn:Cl:St-mor2} \\
closed immersion
& Definition \ref{dfn:dSt:open}     & Definition \ref{dfn:Cl:St-mor2} 
\end{tabular}
\caption{Morphisms between derived and algebraic stacks}
\end{table}
\end{rmk}

\subsubsection{Derived stacks of quasi-coherent modules and vector bundles}

We close this subsection by giving some examples of geometric derived stacks
following \cite[\S 1.3.7, \S 2.2.6.1]{TVe2}.
Let $k$ be a commutative ring as before.

For a derived $k$-algebra $A$, an \emph{$A$-module} means an $A$-module object 
in the $\infty$-operad $\isMod_k^{\otimes}$ (\S \ref{sss:dSt:isMod}).
$A$-modules form an $\infty$-category denoted by $\isMod_A$.
It has an induced $\infty$-operad structure from $\isMod_k^{\otimes}$, 
and the obtained $\infty$-operad is denoted by $\isMod_A^{\otimes}$.
\emph{Derived $A$-algebras} mean commutative ring objects in $\isMod_A^{\otimes}$,
which form an $\infty$-category denoted by $\isCom_A$.

For a derived $k$-algebra $A$, the category $\QCoh(A)$ 
of quasi-coherent modules on $A$ is defined as
\begin{itemize}[nosep]
\item 
An object is a data $(M,\alpha)$ consisting of 
a family $M=\{M_B\}_B$ of $B$-modules for $B \in \isCom_A$ and 
a family $\alpha=\{\alpha_u: M_B \otimes_B B' \to M_B'\}_{u}$
of isomorphisms for morphisms $u: B \to B'$ in $\isCom_A$
such that $\alpha_v \circ (\alpha_u \otimes_{B'} B'') = \alpha_{v \circ u}$
for any composable morphisms $B \xrightarrow{u} B' \xrightarrow{v}B''$.

\item
A morphism $f: (M,\alpha) \to (M',\alpha')$ is a family 
$f=\{f_B: M_B \to M_{B'}\}_{B}$ of morphisms of $B$-modules for $B \in \isCom_A$
such that for any $u:B \to B'$ in $\isCom_A$ we have 
$f_{B'} \circ \alpha_u = \alpha'_{u} \circ (f_B \otimes_B B'):  
 M_B \otimes_B  B' \to M'_{B'}$.
\end{itemize}

The category $\QCoh(A)$ inherits a model structure 
induced from Kan model structure of simplicial sets.
Let us now consider the subcategory $\QCoh(A)_W^c \subset \QCoh(A)$
consisting of equivalences between cofibrant objects.
Taking the nerve, we have an $\infty$-category $\iQCoh(A) := \Ner(\QCoh(A)_W^c)$.
It enjoys the property
$\pi_1(\iQCoh(A),M) \simeq \Aut_{\Ho \isMod_A}(M)$.
Now we cite

\begin{fct*}[{\cite[Theorem 1.3.7.2]{TVe2}}]
The correspondence $A \mapsto \iQCoh(A)$ determines a derived stack
\[
 \stQCoh: \isCom = (\idAff)^{\op} \longto \iS.
\]
\end{fct*}

Next we turn to the definition of vector bundles.

\begin{dfn}\label{dfn:dSt:Vect}
Let $A$ be a derived $k$-algebra, and let $r \in \bbN$.
An $A$-module $B \in \isMod_A$ is called a \emph{rank $r$ vector bundle}
if there exists an \'etale covering $A \to A'$ 
(Definition \ref{dfn:dAff:et-cov})
such that $M \otimes_A A' \simeq (A')^r$ in $\Ho \isMod_{A'}$.

We have the corresponding full sub-$\infty$-category $\iVect_r \subset \iQCoh(A)$
and the \emph{derived stack $\stVect_r$ of rank $r$ vector bundles}.
\end{dfn}

Recall the classifying stack $B G$ 
for a smooth group scheme $G$ over a scheme $S$ (Definition \ref{dfn:St:BG}).
It is an algebraic stack, 
so that we have the corresponding truncated derived stack $\iota(B G)$.

In the case $G=\GL_r$, regarded as a group scheme over $k$, 
the classifying stack $B \GL_r$ is 
nothing but the moduli space of rank $r$ vector bundles.
Now we have 

\begin{fct}[{\cite[\S 2.2.6.1]{TVe2}}]\label{fct:dSt:Vect}
We have 
\[
 \stVect_r \simeq \iota(B \GL_r).
\]
In particular, $\stVect_r$ is a truncated (Definition \ref{dfn:dSt:trunc} (3)) 
$1$-geometric derived stack
with $(-1)$-representable diagonal.
It is also locally of finite presentation (Definition \ref{dfn:dSt:lfp}).
\end{fct}

Note that in \cite[\S 2.2.6.1]{TVe2} 
the word ``affine diagonal" is used for ``$(-1)$-representable diagonal".

\section{Sheaves on ringed $\infty$-topoi}
\label{s:is}

This section is a preliminary for our study of sheaves on derived stacks.
In a general setting of ringed $\infty$-topoi,
we introduce the notions of sheaves of modules and functors between them.
Some of them are already given in Lurie's work \cite{Lur2,Lur5,Lur7}.
These will give us an $\infty$-theoretical counterpart of 
the theory of (unbounded) derived categories 
(\cite{Sp}, \cite[Chap.\ 18]{KS}).

We manly focus on sheaves of \emph{stable} modules,
so that the resulting $\infty$-category will be stable 
in the sense of \cite[Chap.\ 1]{Lur2}.
A stable $\infty$-category is an $\infty$-theoretical counterpart 
of a triangulated category.
Thus our $\infty$-category of sheaves of stable modules 
will be an $\infty$-categorical counterpart of 
the derived category of sheaves of modules.

\subsection{Geometric morphisms of $\infty$-topoi}
\label{ss:pr:gm}

In this subsection we give some complementary explanation on $\infty$-topoi.

Since an $\infty$-topos is an $\infty$-category, we automatically have

\begin{dfn*}
A morphism or a functor of $\infty$-topoi is defined to be 
a functor of $\infty$-categories (Definition \ref{dfn:ic:fun}).
\end{dfn*}

However, this definition is poor as 
the correct notion of morphisms between (ordinary) topoi implies.
Recall that a morphism $f: T \to T'$ of topoi is defined to be an adjunction
$f^*: T' \adjunc T :f_*$ \cite[IV.7]{SGA4}.
In this subsection we recall the notion of \emph{geometric morphisms},
which is the genuine notion of morphisms between $\infty$-topoi.

\begin{dfn}\label{dfn:is:gm}
A morphism $f_*: \tT \to \tT'$ of $\infty$-topoi is \emph{geometric} 
if it admits a left adjoint $f^*: \tT' \to \tT$ which is left exact 
(Definition \ref{dfn:ic:exact}).
In this case, the resulting adjunction will be denoted by 
\[
 f^*: \tT' \adjunc \tT :f_*.
\]
\end{dfn}

\begin{rmk}\label{rmk:is:gm}
Since either of $f_*$ and $f^*$ determines the other 
up to contractible ambiguity, 
we sometimes refer to $f^*$ as a geometric morphism.
Following \cite[Remark 6.3.1.7]{Lur1},
we always denote a left adjoint by upper asterisk such as $f^*$,
and denote a right adjoint by lower asterisk such as $f_*$.
\end{rmk}

By \cite[Remark 6.3.1.2]{Lur1} any equivalence of $\infty$-topoi, 
which is defined to be an equivalence as $\infty$-categories, 
is a geometric morphism.
Also, by \cite[Remark 6.3.1.3]{Lur1},
the class of geometric morphisms is stable under composition.

We now introduce the $\infty$-category $\iRTop$ 
of $\infty$-categories and geometric morphisms.

\begin{dfn}\label{dfn:is:RTop}
The $\infty$-category $\iRTop$ is given by the following description.
\begin{itemize}[nosep]
\item 
The objects are (not necessarily small) $\infty$-topoi.
\item
The morphisms are functors $f_*: \tT \to \tT'$ of $\infty$-topoi
which has a left exact left adjoint. 
\end{itemize}
\end{dfn}

We also have the $\infty$-category $\iLTop$ whose objects are the same as $\iRTop$
but morphisms are functors $f^*: \tT' \to \tT$ 
preserving small colimits and finite limits.
By \cite[Corollary 6.3.1.8]{Lur1}, we have an equivalence $\iLTop \simeq \iRTop^{\op}$.
Here is a list of some properties of $\iRTop$.

\begin{fct}\label{fct:is:RTop}
\begin{enumerate}[nosep]
\item
$\iRTop$ admits small colimits \cite[Proposition 6.3.2.1]{Lur1}.

\item
$\iRTop$ admits small limits \cite[Corollary 6.3.4.7]{Lur1}.

\item 
The $\infty$-category $\iS$ of spaces is a final object of $\iRTop$
\cite[Proposition 6.3.3.1]{Lur1}.
\end{enumerate}
\end{fct}

Recall that for an $\infty$-topos $\tT$ and $U \in \tT$,
the over-$\infty$-category $\oc{\tT}{U}$ is an $\infty$-topos (Fact \ref{fct:pr:ovc}).
It is equipped with the following geometric morphisms.

\begin{fct}[{\cite[Proposition 6.3.5.1 (2)]{Lur1}}]\label{fct:is:biadj}
For $U \in \tT$, the canonical functor $j_!: \oc{\tT}{U} \to \tT$ 
of the over-$\infty$-category $\oc{\tT}{U}$ (Corollary \ref{cor:ic:ovc})
has a right adjoint $j^*$ which commutes with colimits. 
Thus $j^*$ has a right adjoint $j_*$,
and we have two geometric morphisms of $\infty$-topoi:
\[
 j_!: \oc{\tT}{U} \adjunc \tT         :j^*, \quad 
 j^*: \tT         \adjunc \oc{\tT}{U} :j_*.
\]
\end{fct}

Following the standard terminology in \cite{SGA4},
we call the triple $(j_!,j^*,j_*)$ \emph{biadjunction}.
We close this part by 

\begin{ntn}\label{ntn:is:rst}
Given an $\infty$-topos $\tT$, $U \in \tT$ and $\shF \in \iShv_{\iC}(\tT)$
for some $\infty$-category $\iC$ admitting small limits, we denote 
$\rst{\shF}{U} := j^*(\shF)$ and call it the \emph{restriction of $\shF$ to $U$}.
\end{ntn}

\subsection{Ringed $\infty$-topoi}

\subsubsection{Definition}
\label{sss:is:rt}

We will use the language of $\infty$-operad \cite{Lur2} in the following.
Let $\iC$ be a symmetric monoidal $\infty$-category \cite[\S 2.1]{Lur2}.
Then we have the notion of \emph{commutative ring objects} in $\iC$,
and they form an $\infty$-category $\iCAlg(\iC)$.

\begin{rmk}\label{rmk:is:iCAlg}
Let us give a more strict explanation.
Let $\bbE_{\infty}^{\otimes} = \mathrm{Comm}^{\otimes}$ 
be the commutative $\infty$-operad \cite[\S 2.1.1, \S 5.1.1]{Lur2}.
Then we have the $\infty$-category $\iAlg_{\bbE_{\infty}}(\iC)$
of $\bbE_{\infty}$-algebra objects in $\iC$ \cite[\S 3.1]{Lur2}.
In the above we denoted $\iCAlg(\iC) := \iAlg_{\bbE_{\infty}}(\iC)$.
Let $R \in \iAlg_{\bbE_{\infty}}(\iC)$.
\end{rmk}

Let $\tT$ be an $\infty$-topos.
Consider the $\infty$-category $\iShv_{\iCAlg(\iC)}(\tT)$ of 
sheaves valued in $\iCAlg(\iC)$ on $\tT$ (Definition \ref{dfn:pr:iShv}).
Its object $\shR \in \iShv_{\iCAlg(\iC)}(\tT)$
is called a \emph{sheaf of commutative ring objects} on $\tT$.

%

\begin{dfn}\label{dfn:is:ringed}
Let $\iC$ be a symmetric monoidal $\infty$-category.
\begin{enumerate}[nosep]
\item
A \emph{ringed $\infty$-topos} is a pair $(\tT,\shR)$ 
of an $\infty$-topos $\tT$ and 
a sheaf of commutative algebras $\shR \in \iShv_{\iCAlg(\iC)}(\tT)$.

\item
A \emph{functor} $(\tT,\shR) \to (\tT',\shR')$ of ringed $\infty$-topoi
is a pair $(f,f^{\sharp})$ of 
\begin{itemize}
\item
A geometric morphism $f: \tT \to \tT'$
of $\infty$-topoi corresponding to an adjunction $f^*: \tT' \adjunc \tT :f_*$
(Definition \ref{dfn:is:gm}).

\item
A morphism $f^{\sharp}: \shR' \to f_* \shA$ in $\iShv_{\iCAlg(\iC)}(\tT')$,
where $f_* \shR$ is defined by $(f_* \shR)(U') := \shR(f^* U')$ 
for each $U' \in \tT'$.
\end{itemize}
\end{enumerate}
\end{dfn}

Our definition is an $\infty$-theoretical analogue of the ordinary ringed topos
\cite[IV.11, 13]{SGA4}. 
Note that in (2) we used geometric morphisms (Definition \ref{dfn:is:gm}),
which is the correct notion of functors of $\infty$-topoi.

In the later sections we mainly discuss constructible sheaves
by setting $k=\bbF_l$ or $\Qlb$ and $\iC:=\iMod_k=\Ner(\Mod_k)$, 
the nerve of the category of $k$-modules. 
The monoidal structure on $\iC$ is given by the standard $\otimes_k$.

An ethical remark is in order.

\begin{rmk*}
As mentioned in Remark \ref{rmk:dSt:dag/sag:1}, 
Lurie explores the theory of spectral algebraic geometry
based on the sheaves of spectral rings in \cite{Lur5}--\cite{Lur14}.
Under the notation above, 
his theory is developed in $\iShv_{\iCAlg_\bbK}(\tT)$,
where $\bbK$ denotes an $\bbE_{\infty}$-ring (see Appendix \ref{ss:sp:sp})
and $\iCAlg_\bbK=\iCAlg_{\bbK}(\iMod(\iSp))$ 
denotes the $\infty$-category of $\bbK$-algebra objects in $\iMod(\iSp)$,
which is the $\infty$-category of module objects in 
the $\infty$-category $\iSp$ of spectra
equipped with the smash product monoidal structure.
A pair $(\tT,\shA)$ with $ \shA \in \iShv_{\iCAlg_{\bbK}}(\tT)$ is called 
a \emph{spectrally ringed $\infty$-topos} \cite[Definition 1.27]{Lur7}.

As noted in \cite[Remark 2.1]{Lur7},
a commutative rings is a discrete $\bbE_{\infty}$-ring,
so that the theory for $\iC=\iMod_k$ with $k$ a commutative ring
can be embedded in that for $\iC=\iMod_{\bbK}(\iSp)$
with $\bbK$ an arbitrary $\bbE_{\infty}$-ring.
Thus our presentation is basically included in Lurie's exposition.
\end{rmk*}

\subsubsection{Sheaves of $\shR$-modules on $\infty$-topoi}
\label{sss:is:Rmod}

Let us first recall the notion of modules over commutative rings 
in the $\infty$-categorical setting.
Let $\iC$ be a symmetric monoidal $\infty$-category
as in the previous part.
Take a commutative algebra object $R \in \iCAlg(\iC)$.
Then we have the notion of \emph{$R$-module objects} in $\iC$,
and they form an $\infty$-category $\iMod_R(\iC)$.
It is equipped with a symmetric monoidal structure,
and the associated tensor product is denoted by $\otimes_R$.

\begin{rmk*}
Let us continue Remark \ref{rmk:is:iCAlg}.
Since the commutative operad $\bbE_{\infty}^{\otimes}$ is coherent, 
we have the $\infty$-operad $\iMod_{R}^{\bbE_{\infty}}(\iC)^{\otimes}$
of $R$-module objects in $\iC$ \cite[\S 3.3]{Lur2}.
It is equipped with a fibration 
$\iMod_{R}^{\bbE_{\infty}}(\iC)^{\otimes} \to \bbE_{\infty}^{\otimes}$
of $\infty$-operads.
We denoted $\iMod_R(\iC) := \iMod_{R}^{\bbE_{\infty}}(\iC)$ above.
\end{rmk*}

Now we turn to the discussion of sheaves on $\infty$-topoi.
Let $\tT$ be an $\infty$-topos and 
assume that the symmetric monoidal $\infty$-category $\iC$ is presentable.
Then by \cite[Lemma 1.13]{Lur7} 
the symmetric monoidal structure on $\iC$ induces 
the \emph{pointwise symmetric monoidal structure} on $\iShv_{\iC}(\tT)$.
In particular, we have the $\infty$-category 
$\iCAlg(\iShv_{\iC}(\tT))$ of commutative ring objects in 
the symmetric monoidal $\infty$-category $\iShv_{\iC}(\tT)$.

Let us continue to assume that $\iC$ is presentable.
As in the argument of \cite[Remark 1.18]{Lur7}, the forgetful functor 
$\iCAlg(\iC) \to \iC$ is conservative (Definition \ref{dfn:pr:csv})
and preserves small limits. 
Thus we have an equivalence
\[
 \iShv_{\iCAlg(\iC)}(\tT) \simeq \iCAlg(\iShv_{\iC}(\tT)),
\]
where in the right hand side 
we are considering the pointwise symmetric monoidal structure.

Now let us take a sheaf $\shR \in \iShv_{\iCAlg(\iC)}(\tT)$ 
of commutative ring objects.
Considering it as an object of $\iCAlg(\iShv_{\iC}(\tT))$, we can consider 
the $\infty$-category $\iMod_{\shR}(\iShv_{\iC}(\tT))$ of $\shR$-module objects
in the symmetric monoidal $\infty$-category $\iShv_{\iC}(\tT)$.
It is equipped with a symmetric monoidal structure.
Moreover, by \cite[Theorem 3.4.4.2]{Lur2}, 
the $\infty$-category $\iMod_{\shR}(\iShv_{\iC}(\tT))$ is presentable,
and the tensor product associated to the symmetric monoidal structure
preserves small colimits in each variable.
In total, we have 

\begin{prp}\label{prp:is:ModR}
Let $\iC$ be a symmetric monoidal presentable $\infty$-category
and $\tT$ be an $\infty$-topos.
For an $\shR \in \iShv_{\iCAlg(\iC)}(\tT)$,
we have a symmetric monoidal presentable $\infty$-category
\[
 \iMod_{\shR}(\iShv_{\iC}(\tT))
\] 
of $\shR$-module objects.
We call its object a \emph{sheaf of $\shR$-modules on $\tT$}.
The associated tensor product 
\[
 \otimes_{\shR}: 
 \iMod_{\shR}(\iShv_{\iC}(\tT)) \times \iMod_{\shR}(\iShv_{\iC}(\tT))
 \longto \iMod_{\shR}(\iShv_{\iC}(\tT))
\]
preserves small colimits in each variable.
\end{prp}



\subsection{Sheaves of spectra}
\label{ss:is:sp}

In this subsection we discuss sheaves of spectra following \cite{Lur7}.
We will use the theory of spectra in the $\infty$-categorical setting,
which is developed extensively in \cite{Lur2}.
See also Appendix \ref{a:sp} where we give a brief summary.

\subsubsection{Sheaves of spectra in general setting}
\label{sss:is:sp(C)}

Let $\iC$ be an $\infty$-category admitting finite limits.
We denote by
\[
 \iSp(\iC)
\]
the $\infty$-category of spectra in $\iC$ (Definition \ref{dfn:sp:sp}).
Hereafter $\iC$ is assumed to be presentable (Definition \ref{dfn:ic:pres}).
Then the $\infty$-category $\iSp(\iC)$ is stable (Definition \ref{dfn:stb:stb})
since a presentable $\infty$-category admits finite limits
so that Fact \ref{fct:sp:sp-stab} works.
Moreover $\iSp(\iC)$ has a $t$-structure (Definition \ref{dfn:stb:t-str})
by Fact \ref{fct:sp:sp-t}.
Hereafter we call it \emph{the} $t$-structure of $\iSp(\iC)$.

Let $\tT$ be an $\infty$-topos.
We now consider the $\infty$-category
\[
 \iShv_{\iSp(\iC)}(\tT)
\]
of sheaves on $\tT$ valued in spectra in $\iC$.
This $\infty$-category inherits properties of $\iSp(\iC)$ mentioned above.

\begin{lem}
\label{lem:is:stab}
Let $\iC$ be a presentable $\infty$-category and $\tT$ be an $\infty$-topos.
Then the $\infty$-category $\iShv_{\iSp(\iC)}(\tT)$ is stable.
\end{lem}

\begin{proof}
We follow the argument in \cite[Remark 1.3]{Lur7}.
Since $\iSp(\iC)$ is stable, the $\infty$-category $\iFun(\tT^{\op},\iSp(\iC))$ 
is also stable by \cite[Proposition 1.1.3.1]{Lur2}.
Clearly $\iShv_{\iSp(\iC)}(\tT) \subset \iFun(\tT^{\op},\iSp(\iC))$ is closed
under limits and translation. 
Then by \cite[Lemma 1.1.3.3]{Lur2} we have the conclusion.
\end{proof}

\begin{lem}\label{lem:is:sh-t}
Assume $\iC$ is a presentable $\infty$-category
equipped with a functor $\ve: \iC \to \iS_*$ which preserves small limits.
Here $\iS_*$ denotes the $\infty$-category of pointed spaces 
(Definition \ref{dfn:sp:ptd}).
Then the stable $\infty$-category 
$\iShv_{\iSp(\iC)}(\tT)$ has a $t$-structure determined by 
\[
 \left(\iShv_{\iSp(\iC)}(\tT)_{\le 0},
       \iShv_{\iSp(\iC)}(\tT)_{\ge 0}\right),
\]
which will be given in Definition \ref{dfn:LE:cncc}.
\end{lem}

Our construction of $t$-structure is a slight modification
of \cite[Proposition 1.7]{Lur7}.
For the explanation, we need some preliminaries.

Let $\Omega^{\infty}: \iSp(\iC) \to \iC$ be the functor 
in Definition \ref{dfn:sp:Omega^inf}.
Composition with $\ve: \iC \to \iS_*$ and 
the forgetful functor $\iS_* \to \iS$ yields $\iSp(\iC) \to \iS_*$,
which further induces 
\[
 \wt{\ve}: \iShv_{\iSp(\iC)}(\tT) \longto \iShv_{\iS}(\tT)
\]
since $\ve$ preserves limits.
Now recall the equivalence $\iShv_{\iS}(\tT) \simeq \tT$
(Fact \ref{fct:pr:Shv(X)=X}).
We also have the truncation functor $\tau_{\le 0}: \tT \to \tau_{\le 0} \tT$
(Definition \ref{dfn:ic:tau_k}).
Combining these functors, we have

\begin{dfn}\label{dfn:is:pi_n}
\begin{enumerate}[nosep]
\item 
We define $\pi_0: \iShv_{\iSp(\iC)}(\tT) \to \tau_{\le 0} \tT$ by 
\[
 \pi_0: \iShv_{\iSp(\iC)}(\tT) \xrr{\wt{\ve}} 
        \iShv_{\iS}(\tT) \longsimto \tT \xrr{\tau_{\le 0}} \tau_{\le 0} \tT.
\]

\item
For $n \in \bbZ$,
we define $\pi_n: \iShv_{\iSp(\iC)}(\tT) \to \tau_{\le 0} \tT$ by
\[
 \pi_n: \iShv_{\iSp(\iC)}(\tT) \xrr{\Omega^n}
        \iShv_{\iSp(\iC)}(\tT) \xrr{\pi_0} 
        \tau_{\le 0} \tT.
\]
Here $\Omega^n$ is the functor induced by the iterated loop functor
$\Omega^n: \iSp(\iC) \to \iSp(\iC)$ (Definition \ref{dfn:stb:[n]}).
We call the image $\pi_n \shM$ of $\shM \in \iShv_{\iSp(\iC)}(\tT)$
the \emph{$n$-th homotopy group} of $\shM$.
\end{enumerate}
\end{dfn}

\begin{rmk*}
As explained in \cite[Remark 1.5]{Lur7}, 
for $n \ge 2$, $\pi_n$ can be regarded as a functor from 
$\Ho \iShv_{\iSp(\iC)}(\tT)$ to the category of 
abelian group objects in $\tau_{\le 0} \tT$.
In fact, $\pi_n$ can be rewritten as the composition
\[
 \iShv_{\iSp(\iC)}(\tT) \xrightarrow{\Omega^{n-2}} 
 \iShv_{\iSp(\iC)}(\tT) \longto 
 \iShv_{\iS_*}(\tT) \longsimto \tT_*
 \xrightarrow{\pi_2} \tau_{\le 0}\tT,
\]
where we used
\begin{itemize}
\item
For an $\infty$-category $\iB$, the symbol $\iB_*$ denotes 
the full sub-$\infty$-category of $\iFun(\Delta^1,\iB)$
spanned by pointed objects (Definition \ref{dfn:sp:ptd}).

\item 
The functor $\iShv_{\iSp(\iC)}(\tT) \to \iShv_{\iS_*}(\tT)$ 
is given by the composition 
$\iShv_{\iSp(\iC)}(\tT) \to \iShv_{\iSp(\iS_*)}(\tT) \to \iShv_{\iS_*}(\tT)$.
The first functor is induced by the limit-preserving $\ve: \iC \to \iS_*$.
The second one comes from the description of $\iSp(\iS_*)$ as 
the limit of the tower
\[
 \cdots \xrr{\Omega} \iS_* \xrr{\Omega} \iS_* \xrr{\Omega} \iS_*.
\]

\item
The equivalence $\iShv_{\iS_*}(\tT) \simeq \tT_*$ is 
shown similarly to $\iShv_{\iS}(\tT) \simeq \tT$
(Fact \ref{fct:pr:Shv(X)=X}).
\end{itemize}
Then the image of $\pi_n$ is included in the full sub-$\infty$-category
$\iEM_n(\tau_{\le 0} \tT) \subset (\tau_{\le 0} \tT)_*$ 
of Eilenberg-MacLane objects \cite[Definition 7.2.2.1]{Lur1}.
In the case $\iC=\iS_*$ and $\ve=\id$,
then the image of $\pi_n$ is equal to $\iEM_n(\tau_{\le 0} \tT)$.
For $n \ge 2$, the $\infty$-category $\iEM_n(\tau_{\le 0} \tT)$ 
is equivalent to the $\infty$-category of abelian group objects 
by \cite[Proposition 7.2.2.12]{Lur1}.
\end{rmk*}

\subsubsection{Sheaves of spectra}

We now take $\iC = \iS_*$ and $\ve=\id$ in the argument so far,
and consider the $\infty$-category $\iSp = \iSp(\iS_*)$ of spectra.
See Appendix \ref{ss:sp:sp} for a summary.
Let us mention here that the \emph{smash product} gives 
a symmetric monoidal structure on $\iS$ \cite[\S 6.3.2]{Lur2}.

By the argument in the previous \S \ref{sss:is:sp(C)},
we have the $\infty$-category
\[
 \iShv_{\iSp}(\tT)
\]
of sheaves of spectra on an $\infty$-topos $\tT$.
We list its formal properties below.

\begin{lem}\label{lem:is:ShvSp}
\begin{enumerate}[nosep]
\item 
$\iShv_{\iSp}(\tT)$ is stable and equipped with a $t$-structure 
$\left(\iShv_{\iSp}(\tT)_{\le 0}, \iShv_{\iSp}(\tT)_{\ge 0}\right)$
given in Lemma \ref{lem:is:sh-t}.
Hereafter we call it \emph{the} $t$-structure of $\iShv_{\iSp}(\tT)$.

\item
$\iShv_{\iSp}(\tT)$ has a symmetric monoidal structure 
induced by the smash product on $\iSp$.
Hereafter we call it the \emph{smash product symmetric monoidal structure}.
Moreover it is compatible with the $t$-structure 
in the sense of Definition \ref{dfn:stb:t-prp}.

\item
$\iShv_{\iSp}(\tT)$ is presentable.
\end{enumerate}
\end{lem}

\begin{proof}
(1) is discussed in the previous \S \ref{sss:is:sp(C)}.
(3) is stated in \cite[Remark 1.1.5]{Lur5}.
(2) is discussed in \cite[\S 1]{Lur7} with much length, so let us give a summary.
The symmetric monoidal structure on $\iSp$ given by the smash product 
induces a symmetric monoidal structure on the $\infty$-category $\iFun(K,\iSp)$
of morphisms for any simplicial set $K$ by \cite[Remark 2.1.3.4]{Lur2}.
Then by \cite[Proposition 1.15]{Lur8},
it further induces a symmetric monoidal structure on $\iShv_{\iSp}(\tT)$.
Also it is compatible with the $t$-structure by \cite[Proposition 1.16]{Lur7}. 
\end{proof}

For later use, let us introduce 

\begin{dfn}[{\cite[Definition 1.6]{Lur7}}]\label{dfn:LE:cncc}
Let $\tT$ an $\infty$-topos and $\shM \in \iShv_{\iSp}(\tT)$.
\begin{enumerate}[nosep]
\item 
$\shM$ is \emph{connective} 
if the homotopy groups $\pi_n \shM$ vanish for $n \in \bbZ_{<0}$.
We denote by $\iShv_{\iSp}(\tT)_{\ge 0}$ 
the full sub-$\infty$-category spanned by connective objects.
\item
$\shM$ is \emph{coconnective} 
if $\Omega^{\infty}M$ is a discrete object,
i.e., a $0$-truncated object (Definition \ref{dfn:ic:n-trc}).
We denote by $\iShv_{\iSp}(\tT)_{\le 0}$ the full sub-$\infty$-category
spanned by coconnective objects.
\end{enumerate}
\end{dfn}

\subsection{Sheaves of commutative ring spectra and modules}
\label{ss:is:sm}

In this subsection we discuss sheaves of stable $\shA$-modules
with $\shA$ a sheaf of commutative ring spectra following \cite[\S 2.1]{Lur8}.
As mentioned in the beginning of this section, these objects can be regarded 
as $\infty$-theoretic counterparts of complexes of sheaves, and 
the corresponding $\infty$-categories are counterparts of derived categories.

\subsubsection{Sheaves of commutative ring spectra} 

Let $\iCAlg(\iSp)$ be the $\infty$-category of commutative ring objects 
in $\iSp$ regarded as a symmetric monoidal category with respect to 
smash product monoidal structure.
An object of $\iCAlg$ is nothing but an $\bbE_{\infty}$-ring 
(Appendix \ref{ss:sp:ring}),
so we call $\iCAlg(\iSp)$ the \emph{$\infty$-category of $\bbE_{\infty}$-rings}.

Let $\tT$ be an $\infty$-topos.
Since $\iShv_{\iSp}(\tT)$ is a symmetric monoidal presentable $\infty$-category
by Lemma \ref{lem:is:ShvSp}, 
we can apply the arguments in \S \ref{sss:is:rt} to $\iC = \iShv_{\iSp}(\tT)$.
Thus we have the $\infty$-category
\[
 \iCAlg(\iShv_{\iSp}(\tT))
\]
 of commutative ring objects.
As mentioned in \cite[Remark 1.18]{Lur7},
since $\iCAlg(\iSp) \to \iSp$ is conservative (Definition \ref{dfn:pr:csv}) 
and preserves small limits, we have an equivalence
\[
 \iCAlg(\iShv_{\iSp}(\tT)) \simeq \iShv_{\iCAlg(\iSp)}(\tT).
\]
Thus we may call an object of $\iCAlg(\iShv_{\iSp}(\tT))$
a \emph{sheaf of commutative ring spectra on $\tT$}.

\subsubsection{Sheaves of stable modules over commutative ring spectra} 
\label{sss:is:smr}

Let us continue to use the symbols in the previous part.
Next we apply the argument in \S \ref{sss:is:Rmod} to 
the symmetric monoidal presentable $\infty$-category $\iC=\iShv_{\iSp}(\tT)$.

\begin{dfn*}
Let $\tT$ be an $\infty$-topos, and $\shR \in \iShv_{\iCAlg(\iSp)}(\tT)$
be a sheaf of commutative ring spectra.
We denote the $\infty$-category of $\shR$-module objects 
in $\iShv_{\iSp}(\tT)$ by
\[
 \iMod^{\stab}_{\shR}(\tT) := \iMod_{\shR}(\iShv_{\iSp}(\tT))
\]
and call it the $\infty$-category of \emph{stable $\shR$-modules} on $\tT$.
\end{dfn*}

The $\infty$-category $\iMod^{\stab}_{\shR}(\tT)$ 
has a symmetric monoidal structure whose tensor product is denoted by 
\[
 - \otimes_{\shR} -: 
 \iMod^{\stab}_{\shR}(\tT) \times \iMod^{\stab}_{\shR}(\tT) \longto 
 \iMod^{\stab}_{\shR}(\tT).
\]

We cite some formal properties of $\iMod^{\stab}_{\shR}(\tT)$.

\begin{fct}[{\cite[Proposition 2.1.3]{Lur8}}]\label{fct:is:sm1}
Let $\tT$ be an $\infty$-topos and $\shR \in \iShv_{\iCAlg(\iSp)}(\tT)$.
\begin{enumerate}[nosep]
\item 
The $\infty$-category $\iMod^{\stab}_{\shR}(\tT)$ is stable.

\item
The $\infty$-category $\iMod^{\stab}_{\shR}(\tT)$ is presentable and 
the tensor product $\otimes_{\shR}$ preserves small colimits in each variable.

\item
The forgetful functor $\theta: \iMod^{\stab}_{\shR}(\tT) \to \iShv_{\iSp}(\tT)$
is conservative (Definition \ref{dfn:pr:csv}) 
and preserves small limits and colimits.
\end{enumerate}
\end{fct}

The item (1) is the origin of our naming of \emph{stable} $\shR$-modules.

%

%
%

Let us next recall the natural $t$-structure on $\iMod^{\stab}_{\shR}(\tT)$
which is induced by the $t$-structure on $\iShv_{\iSp}(\tT)$ 
in Lemma \ref{lem:is:sh-t}.
Recall the $\infty$-categories $\iShv_{\iSp}(\tT)_{\ge 0}$ 
and $\iShv_{\iSp}(\tT)_{\le 0}$ in Definition \ref{dfn:LE:cncc}.

\begin{fct}[{\cite[Proposition 2.1.3]{Lur8}}]\label{fct:is:sm2}
Let $\tT$ be an $\infty$-topos and $\shR \in \iShv_{\iCAlg(\iSp)}(\tT)$.
Assume that $\shR$ is connective (Definition \ref{dfn:LE:cncc}).
\begin{enumerate}[nosep]
\item
The stable $\infty$-category $\iMod^{\stab}_{\shR}(\tT)$ 
has a $t$-structure determined by
\[
 \iMod^{\stab}_{\shR}(\tT)_{\ge0} := \theta^{-1} \iShv_{\iSp}(\tT)_{\ge 0},
 \quad
 \iMod^{\stab}_{\shR}(\tT)_{\le0} := \theta^{-1} \iShv_{\iSp}(\tT)_{\le 0}.
\]
Here $\theta: \iMod^{\stab}_{\shR}(\tT) \to \iShv_{\iSp}(\tT)$ 
denotes the forgetful functor.
We call it \emph{the} $t$-structure of $\iMod^{\stab}_{\shR}(\tT)$ from now.

\item
The $t$-structure on $\iMod_{\shR}(\tT)$ is accessible, that is, 
the $\infty$-category $\iMod_{\shR}(\tT)_{\ge0}$ is presentable.

\item
The $t$-structure on $\iMod^{\stab}_{\shR}(\tT)$ is compatible with 
the symmetric monoidal structure. 
In other words, the sub-$\infty$-category $\iMod^{\stab}_{\shR}(\tT)_{\ge0}$ 
contains the unit object of $\iMod^{\stab}_{\shR}(\tT)$
is stable under tensor product.

\item
The $t$-structure on $\iMod^{\stab}_{\shR}(\tT)$ is compatible with filtered colimits.
In other words, the sub-$\infty$-category $\iMod^{\stab}_{\shR}(\tT)_{\le0}$ 
is stable under filtered colimits.

\item
The $t$-structure on $\iMod_{\shR}(\tT)$ is right complete
(Definition \ref{dfn:stb:cmpl}).
\end{enumerate}
\end{fct}

For $\shR \in \iShv_{\iCAlg}(\tT)$,
the $0$-th homotopy $\pi_0 \shR$ (Definition \ref{dfn:is:pi_n})
is a commutative ring object of $\iShv_{\iS}(\tT) \simeq \tT$,
where the equivalence is given by Fact \ref{fct:pr:Shv(X)=X}.
Let us denote by $\iMod_{\pi_0 \shR}(\iShv_{\iS}(\tT))$
the $\infty$-category of $\pi_0 \shR$-module objects.
Then the equivalence $\iShv_{\iS}(\tT) \simeq \tT$ induces 
$\iMod_{\pi_0 \shR}(\iShv_{\iS}(\tT)) \simeq \iMod_{\pi_0 \shR}(\tT)$.
If we further assume $\shR$ is connected, then the following statement holds.
The proof is by definition and omitted.

\begin{lem}[{\cite[Remark 2.1.5]{Lur8}}]\label{lem:is:heart}
For a connective sheaf $\shR$ of $\bbE_{\infty}$-rings 
on an $\infty$-topos $\tT$, we have
\[
 \iMod^{\stab}_{\shR}(\tT)^{\heartsuit} \simeq 
 \iMod_{\pi_0 \shR}(\tT),
\] 
where the left hand side is the heart (Definition \ref{dfn:stb:heart}) 
of the $t$-structure of Fact \ref{fct:is:sm2} (1).
\end{lem}

For later use, let us introduce 


\begin{ntn*}[{\cite[Notation 1.1.2.17]{Lur2}}]
Let $\iC$ be a stale $\infty$-category, and let $X,Y \in \iC$.
We define the abelian group $\Ext^n_{\iC}(X,Y)$ by
\[
 \Ext^n_{\iC}(X,Y) := \Hom_{\Ho \iC}(X[-n],Y).
\]
where $[-n]$ denotes the shift (Definition \ref{dfn:stb:[n]})
in the stable $\infty$-category $\iC$.
For $n=0$, we denote 
\[
 \Hom_{\iC}(X,Y) := \Ext^0_{\iC}(X,Y) = \Hom_{\Ho \iC}(X,Y).
\]
\end{ntn*}

For $n \in \bbZ_{\ge 0}$, we can identify
$\Ext^n_{\iC}(X,Y) \simeq \pi_{-n}\Map_{\iC}(X,Y)$.

\begin{ntn*}
Let $\tT$ be an $\infty$-topos and $\shR \in \iShv_{\iCAlg(\iSp)}(\tT)$.
For $\shM,\shN \in \iMod^{\stab}_{\shR}(\tT)$ and $n \in \bbZ$,
we denote
\[
 \Ext^n_{\shR}(\shM,\shN) := \Ext^n_{\iMod^{\stab}_{\shR}(\tT)}(\shM,\shN),
 \quad
 \Hom_{\shR}(\shM,\shN) := \Ext^0_{\shR}(\shM,\shN).
\]
\end{ntn*}

\subsubsection{Functors on stable $\shR$-modules}
\label{sss:is:fun}

Fix an $\infty$-topos $\tT$ and take $\shR \in \iShv_{\iCAlg(\iSp)}(\tT)$.
In this part we introduce some functors on the stable $\infty$-category 
of stable $\shR$-modules.

Let us begin with internal Hom functor, 
which will give an $\infty$-theoretical counterpart of the tensor-hom adjunction.
For $\shM \in \iMod^{\stab}_{\shR}(\tT)$, the functor 
\[
 - \otimes_{\shR} \shM : 
 \iMod^{\stab}_{\shR}(\tT) \longto \iMod^{\stab}_{\shR}(\tT)
\]
is right exact (Definition \ref{dfn:ic:exact}) 
since it preserves small colimits by Fact \ref{fct:is:sm1} (2)
and we can apply the criterion of right exactness (Fact \ref{fct:ic:exact} (2)).
Then by Fact \ref{fct:ic:adj} there is a right adjoint of $- \otimes_{\shR} \shM$.
We denote it by 
\[
 \shHom_{\shR}(\shM,-): 
 \iMod^{\stab}_{\shR}(\tT) \longto \iMod^{\stab}_{\shR}(\tT).
\]
By a similar argument on the first variable, 
we obtain a bifunctor
\[
 \shHom_{\shR}(-,-): 
 \iMod^{\stab}_{\shR}(\tT)^{\op} \times \iMod^{\stab}_{\shR}(\tT) 
 \longto \iMod^{\stab}_{\shR}(\tT).
\]
We call it the \emph{internal Hom functor}.

The internal Hom functor agrees with the \emph{morphism object}
\cite[Definition 4.2.1.28]{Lur2}.
Then by \cite[Proposition 4.2.1.33]{Lur2} we have

\begin{lem*}
For any $\shL,\shM,\shN \in \iMod^{\stab}_{\shR}(\tT)$,
there is a functorial equivalence
\[
 \shHom_{\shR}(\shL \otimes_{\shR} \shM,\shN) \longsimto 
 \shHom_{\shR}(\shL, \shHom_{\shR}(\shM,\shN))
\]
which is unique up to contractible ambiguity.
\end{lem*}

Next we introduce the direct image and inverse image functors.
Our presentation basically follows \cite[\S2.5]{Lur8} but with a slight modification.
Let 
\[
 (f,f^\sharp): (\tT,\shR) \longto (\tT',\shR')
\]
be a morphism 
of ringed $\infty$-topoi (Definition \ref{dfn:is:ringed}).
Thus $f: \tT \to \tT'$ is a geometric morphism of $\infty$-topoi 
with the associated adjunction $f^*: \tT' \adjunc \tT :f_*$, and 
$f^\sharp: \shR' \to f_* \shR$ is a morphism in $\iShv_{\iCAlg(\iSp)}(\tT')$. 

Composition with $f_*: \tT \to \tT'$ gives a symmetric monoidal functor
$f^{-1}: \iShv_{\iSp}(\tT') \longto \iShv_{\iSp}(\tT)$.
Then we can regard $f^{-1}\shR'$ as an object of $\iShv_{\iCAlg(\iSp)}(\tT)$. 
Moreover $f^{-1}$ induces a functor 
$\iMod^{\stab}_{\shR'}(\tT') \to \iMod^{\stab}_{f^{-1} \shR'}(\tT)$.
Abusing the symbol, we denote it by 
\[
 f^{-1}: \iMod^{\stab}_{\shR'}(\tT') \longto \iMod^{\stab}_{f^{-1} \shR'}(\tT).
\]
Note also that $f^\sharp$ and $f$ yield a morphism $f^{-1}\shR' \to \shR$
in  $\iShv_{\iCAlg_{\bbK}}(\tT)$
so that we can regard $\shR$ as an object in $\iMod^{\stab}_{f^{-1}\shR'}(\tT)$.
Thus we have the functor
\[
 f^*: \iMod^{\stab}_{\shR'}(\tT') \longto \iMod^{\stab}_{\shR}(\tT), \quad
 \shM' \longmapsto f^* \shM' := f^{-1} \shM' \otimes_{f^{-1}\shR'} \shR.
\]

Since the inverse image functor 
$f^*: \iMod^{\stab}_{\shR'}(\tT') \to \iMod^{\stab}_{\shR}(\tT)$ 
preserves colimits, and since the $\infty$-categories of stable $\shR$-modules 
are presentable (Fact \ref{fct:is:sm1} (2)),
there is a right adjoint 
\[
 f_*: \iMod^{\stab}_{\shR}(\tT) \longto \iMod^{\stab}_{\shR'}(\tT').
\]
It is also described as follows.
The composition with $f^*: \tT' \to \tT$ gives a functor
$\iShv_{\iSp}(\tT) \longto \iShv_{\iSp}(\tT')$,
which induces $\iMod^{\stab}_{\shR}(\tT) \to \iMod^{\stab}_{f_* \shR}(\tT')$.
Then the morphism $f^\sharp$ induces 
$f_*: \iMod^{\stab}_{\shR}(\tT) \to \iMod^{\stab}_{\shR'}(\tT')$.

\begin{ntn}\label{ntn:is:iidi}
We call $f^*$ the \emph{inverse image functor},
and $f_*$ the \emph{direct image functor}.
\end{ntn}

\subsection{Sheaves of commutative rings and modules on $\infty$-topoi}
\label{ss:is:crm}

This subsection is a complement of the previous \S \ref{ss:is:sm}.
We will give notations for sheaves of commutative rings,
sheaves of modules over them, derived categories and functors between them
on $\infty$-topoi.
These will be essentially the same with those on ordinary topoi,
and we write them down just for the completeness of our presentation.

\subsubsection{Sheaves of modules over commutative rings}
\label{sss:is:smrf}

Let $\Ab$ denote the category of additive groups, 
and $\iAb:=\Ner(\Ab)$ be its nerve.
The $\infty$-category $\iAb$ has the symmetric monoidal structure.
Also we see that $\iAb$ is a presentable $\infty$-category.
One can show it directly by definition,
and also by using the equivalence \cite[Proposition A.3.7.6]{Lur1}
between a presentable $\infty$-category and 
the nerve of the subcategory of fibrant-cofibrant objects 
in a combinatorial simplicial model category.

Let $\tT$ be an $\infty$-topos.
We can now apply the argument in \S \ref{sss:is:Rmod} to 
the symmetric monoidal presentable $\infty$-category $\iC=\iAb$.
Noting that $\iCAlg(\iAb) \simeq \iCom$ 
which is the $\infty$-category of commutative rings (\S \ref{ss:dSt:St}),
We call the $\infty$-category 
\[
 \iCAlg(\iShv_{\iAb}(\tT)) \simeq \iShv_{\iCAlg(\iAb)}(\tT) 
 \simeq  \iShv_{\iCom}(\tT)
\]
the \emph{$\infty$-category of sheaves of commutative rings on $\tT$}.

Taking a sheaf $\shA \in \iShv_{\iCom}(\tT)$ of commutative rings,
we also have the $\infty$-category of \emph{sheaves of $\shA$-modules on $\tT$}
(Proposition \ref{prp:is:ModR}).
We denote it by 
\[
 \iMod_{\shA}(\tT) := \iMod_{\shA}(\iShv_{\iAb}(\tT)).
\]
Recall that it is a symmetric monoidal presentable $\infty$-category.
The associated tensor product is denoted by $\otimes_{\shA}$.
We also have the following standard claim.

\begin{prp}\label{prp:is:G}
The homotopy category $\catA := \Ho \iMod_{\shA}(\tT)$
is a Grothendieck abelian category (Definition \ref{dfn:stb:G}).
\end{prp}

\begin{proof}
We omit the proof of the claim that $\catA$ is an abelian category 
since the argument for ordinary topos works.
The claim that $\catA$ is a presentable category is a consequence of 
the presentability of the $\infty$-category $\iMod_{\shA}(\tT)$.

Let us show that the class of monomorphisms is 
preserved by small filtered colimits
Let $\{f_i: A_i \to B_i\}$ be a filtered diagram of monomorphisms in $\catA$.
We have the corresponding filtered diagram of fiber sequences 
$A_i \xr{f_i} B_i \to B_i/A_i$ in $\iMod_{\shA}(\tT)$.
We can take the filtered colimit since $\iMod_{\shA}(\tT)$ is presentable
so that it has colimit.
The resulting colimit $A \xr{f} B \to B/A$ is a fiber sequence,
so that $f$ is a monomorphism in $\catA$.
\end{proof}

\begin{rmk*}
The presentability of $\catA$ means that 
there is a small family of generators.
We can explicitly give such generators as follows. 
Recall from Fact \ref{fct:is:biadj} that we have the canonical functor 
$j_!: \oc{\tT}{U} \to \tT$ for each $U \in \tT$, and we have the biadjunction
$j_!: \oc{\tT}{U} \adjunc \tT :j^*$ and  $j^*: \tT \adjunc \oc{\tT}{U} :j_*$.
Then, for $\shA \in \iShv_{\iCAlg(\iAb)}(\tT)$, we have 
$\rst{\shA}{U} := j^* \shA \in \iShv_{\iCAlg(\iAb)}(\oc{\tT}{U})$ 
(Definition \ref{dfn:is:ringed}, Notation \ref{ntn:is:rst}).
Then we also have an adjunction
\[
 j_!: \iMod_{\rst{\shA}{U}}(\oc{\tT}{U}) \adjunc 
             \iMod_{\shA}(\tT) :j^*
\] 
of $\infty$-categories.
Now $\{j_! j^* \shA \}_{U \in \tT}$
gives the desired family of generators
(note that we tacitly assume that $\tT$ is small).
In fact, for each $\shM \in \iMod_{\shA}(\tT)$ the adjunction gives 
\[
 \Map_{\iMod_{\shA}(\tT)}\left( j_! j^* \shA,\shM \right) 
 \simeq \Map_{\iMod_{\rst{\shA}{U}}(\oc{\tT}{U})}
        \left( j^* \shA, j^* \shM \right) 
 \simeq (j^* \shM)(U) = \shM(U).
\]
\end{rmk*}

Now the same argument as \S \ref{sss:is:fun} gives 
functors on sheaves of modules over sheaves of commutative rings.
We list them in the following proposition.


\begin{prp*}
Let $\tT,\tT'$ be $\infty$-topoi and 
$\shA,\shA'$ be sheaves of commutative rings on $\tT,\tT'$ respectively.
Let also $(f,f^\sharp): (\tT,\shA) \longto (\tT',\shA')$
be a functor of ringed $\infty$-topoi.
\begin{enumerate}[nosep]
\item 
For $\shM \in \iMod_{\shA}(\tT)$, the right exact functor 
\[
 - \otimes_{\shA} \shM: \iMod_{\shA}(\tT) \longto \iMod_{\shA}(\tT)
\]
has a right adjoint denoted by 
\[
 \shHom_{\shA}(\shM,-):  \iMod_{\shA}(\tT) \longto \iMod_{\shA}(\tT).
\]
It gives rise to a bifunctor
\[
 \shHom_{\shA}(-,-): 
 \iMod_{\shA}(\tT)^{\op} \times \iMod_{\shA}(\tT) 
 \longto \iMod_{\shA}(\tT)
\]
called  the \emph{internal Hom functor}.

\item
We have the \emph{inverse image functor}
\[
 f^*: \iMod_{\shA'}(\tT') \longto \iMod_{\shA}(\tT), \quad
 \shM' \longmapsto f^* \shM' := f^{-1} \shM' \otimes_{f^{-1}\shA'} \shA.
\]
It is right exact, and have a a right adjoint 
\[
 f_*: \iMod_{\shA}(\tT) \longto \iMod_{\shA'}(\tT').
\]
called the \emph{direct image functor}.
\end{enumerate}
\end{prp*}

\subsubsection{Derived $\infty$-categories}
\label{sss:is:dic}

%
%
%

For a Grothendieck abelian category $\catA$
one can construct the \emph{derived-$\infty$-category} $\iDa(\catA)$,
which is an $\infty$-categorical counterpart 
of the unbounded derived category of $\catA$.
It is a stable $\infty$-category equipped with a $t$-structure,
which is also an $\infty$-categorical counterpart 
of triangulated category with a $t$-structure.
See Appendix \ref{ss:stb:G} for an account on $\iDa(\catA)$.

Let $\tT$ be an $\infty$-topos and 
$\shA$ be a sheaf of commutative rings on $\tT$.
Denote by $\iMod_{\shA}(\tT)$ the $\infty$-category of $\shA$-modules on $\tT$.
By Proposition \ref{prp:is:G},
\[
 \Mod_{\shA}(\tT) := \Ho \iMod_{\shA}(\tT) 
\]
is a Grothendieck abelian category, 
so we can apply to it the construction in Appendix \ref{ss:stb:G}.
Thus we have the \emph{derived $\infty$-category of 
sheaves of $\shA$-modules on $\tT$} 
\[
 \iDa(\Mod_{\shA}(\tT)) = \Ndg(\C(\Mod_{\shA}(\tT))^\circ).
\]
We collect its properties in

\begin{lem}\label{lem:is:srm}
The derived $\infty$-category $\iD := \iDa(\Mod_{\shA}(\tT))$ 
enjoys the following properties.
\begin{enumerate}[nosep]
\item 
$\iD$ is a stable $\infty$-category (Definition \ref{dfn:stb:stb}).

\item
$\iD$ is equipped with a $t$-structure (Definition \ref{dfn:stb:t-str})
determined by $(\iD_{\le 0},\iD_{\ge 0})$.
Here $\iD_{\le 0}$ (resp.\ $\iD_{\ge 0}$) is the full sub-$\infty$-category 
of $\iD$ spanned by those objects $\shM$ such that $H_n(\shM)=0$ 
for any $n > 0$ (resp.\ $n<0$).
Hereafter we call it \emph{the} $t$-structure on $\iD$.

\item
The heart of the $t$-structure (Definition \ref{dfn:stb:heart}) is 
$\iD^{\heartsuit} \simeq \iMod_{\shA}(\tT)$.

\item
$\iD$ is right complete with respect to the $t$-structure
(Definition \ref{dfn:stb:cmpl}).

\item
$\iD$ is accessible with respect to the $t$-structure 
(Definition \ref{dfn:stb:t-prp}).
\end{enumerate}
\end{lem}

\begin{proof}
All the claims are explained in Appendix \ref{ss:stb:G} and \ref{ss:stb:D-t}.
\end{proof}

In particular, the homotopy category $\Ho \iDa(\iMod_{\shA}(\tT))$
is a triangulated category with a $t$-structure.
It can be regarded as the unbounded derived category 
of the abelian category $\Mod_{\shA}(\tT)$ as mentioned before.


Since a Grothendieck abelian category has enough injective objects,
we also have the left bounded derived $\infty$-category 
(Definition \ref{dfn:stb:iDm/iDp})
\[
 \iDp(\Mod_{\shA}(\tT)) = \Ndg(\C^+(\Mod_{\shA}(\tT)_{\txinj})).
\]
It is stable and equipped with a $t$-structure.
Moreover, by Fact \ref{fct:stb:ext-fun},
we have a fully faithful $t$-exact functor (Definition \ref{dfn:stb:text})
\[
 \iDp(\Mod_{\shA}(\tT)) \longto \iDa(\Mod_{\shA}(\tT))
\]
whose essential image is $\cup_{n \in \bbZ} \iDa(\Mod_{\shA}(\tT))_{\le n}$.

Recall that we have the stable $\infty$-category 
$\iMod^{\stab}_{\shA}(\tT)$ of stable $\shA$-modules
by taking $\shR := \shA$ in the argument of \S \ref{sss:is:smr}.
It is equipped with a right complete $t$-structure.
Thus we also have a fully faithful $t$-exact functor
\[
 \iDp(\Mod_{\shA}(\tT)) \longto \iMod^{\stab}_{\shA}(\tT)
\]
whose essential image is $\cup_{n \in \bbZ} \iMod^{\stab}_{\shA}(\tT)_{\le n}$.
Since $\iDa(\Mod_{\shA}(\tT))$ and $\Mod^{\stab}_{\shA}(\tT)$
are both right complete, we have 

\begin{prp}\label{prp:is:Da=Stab}
The fully faithful $t$-exact functor
$\iDa(\Mod_{\shA}(\tT)) \inj \iMod^{\stab}_{\shA}(\tT)$
yields an equivalence of $\infty$-categories 
\[
 \iDa(\Mod_{\shA}(\tT)) \simeq  \iMod^{\stab}_{\shA}(\tT).
\]
\end{prp}

In other words, the $\infty$-category of $\shA$-module spectra 
can be identified with the derived $\infty$-category
of the abelian category of (discrete) $\shA$-modules.

\begin{rmk*}
As noted in \cite[Remark 7.1.1.16]{Lur2},
this statement holds for a discrete ring spectra $\shR$.
\end{rmk*}

\subsubsection{Derived functors}

We continue to use the symbols in the previous part.
We now discuss functors on $\iDa(\Mod_{\Lambda}(\tT))$.

Let $\tT,\tT'$ be $\infty$-topoi, 
let $\shA,\shA'$ be sheaves of commutative rings on $\tT,\tT'$ respectively,
and $(f,f^\sharp): (\tT,\shA) \to (\tT',\shA')$
be a functor of ringed $\infty$-topoi.
Then we have seen in \S \ref{sss:is:smrf} that 
there are the tensor functor $\otimes_{\Lambda}$ and 
the internal Hom functor $\shHom_{\Lambda}$ on each of the $\infty$-categories 
$\iMod_{\shA}(\tT)$ and $\iMod_{\shA'}(\tT')$.
We also have the adjunction
$f^*: \iMod_{\shA}(\tT) \adjunc \iMod_{\shA'}(\tT') :f_*$
of the inverse image functor $f^*$ and the direct image functor $f_*$.
They are the ordinary functors on the category of sheaves of $\Lambda$-modules.

On the other hand, by \S \ref{sss:is:fun}, we have functors
$\otimes_{\shA}$, $\shHom_{\shA}$, $f_*$ and $f^*$ 
on the $\infty$-categories $\iMod^{\stab}_\Lambda(\tT)$ and 
$\iMod^{\stab}_\Lambda(\tT')$ of stable $\shA$-modules.
With the view of the fully faithful $t$-exact functor
\[
 \iMod_{\shA}(\tT) \simeq \iDa(\Mod_{\shA}(\tT))^{\heartsuit}
 \longinj \iDa(\Mod_{\shA}(\tT)) \simeq \iMod^{\stab}_{\shA}(\tT)
\]
(Proposition \ref{prp:is:Da=Stab}), the functors on stable $\shA$-modules are
the extensions of those on $\shA$-modules,
and correspond to the derived functors in the ordinary categorical setting.

Hereafter we change the symbols of the functors on the stable modules
by the standard derived functors,
as collected in the following proposition.

\begin{prp}\label{prp:is:df}
Let $\tT,\tT'$ be $\infty$-topoi, 
let $\shA,\shA'$ be sheaves of commutative rings on $\tT,\tT'$ respectively, and 
$(f,f^\sharp): (\tT,\shA) \to (\tT',\shA')$ be a functor of ringed $\infty$-topoi.
\begin{enumerate}[nosep]
\item
The tensor product functor 
$\otimes_{\shA}: \Mod_{\shA}(\tT) \times \iMod_{\shA}(\tT) \to \Mod_{\shA}(\tT)$
is a right exact functor in each variable, and has a left $t$-exact extension
\[
 \otimes^{\dL}_{\shA}: 
 \iDa(\Mod_{\shA}(\tT)) \times \iDa(\Mod_{\shA}(\tT))
 \longto \iDa(\Mod_{\shA}(\tT)).
\]

\item
Let $\shM \in \Mod_{\shA}(\tT)$.
Then the internal Hom functor 
$\shHom_{\shA}(\shM,-): \Mod_{\shA}(\tT) \to \Mod_{\shA}(\tT)$ 
is left exact, and has a right $t$-exact extension 
$\shHom_{\shA}(\shM,-): \iDa(\Mod_{\shA}(\tT)) \to \iDa(\Mod_{\shA}(\tT))$.
It yields a bifunctor
\[
 \shHom_{\shA}(-,-): 
 \iDa(\Mod_{\shA}(\tT))^{\op} \times \iDa(\Mod_{\shA}(\tT))
 \longto \iDa(\Mod_{\shA}(\tT)).
\]
We denote it by $\shHom_{\iDa(\Mod_{\shA}(\tT))}$
if we want to distinguish it from $\shHom_{\shA}$ on $\iMod_{\shA}(\tT)$.

\item 
The direct image functor $f_*: \Mod_{\shA}(\tT) \to \Mod_{\shA}(\tT')$
is left exact and has a right $t$-exact extension
\[
 f_*: \iDa(\Mod_{\shA}(\tT)) \longto \iDa(\Mod_{\shA}(\tT')).
\]

\item
The inverse image functor $f^*: \Mod_{\shA}(\tT') \to \Mod_{\shA}(\tT)$
is right exact and has a left $t$-exact extension
\[
 f^*: \iDa(\Mod_{\shA}(\tT')) \longto \iDa(\Mod_{\shA}(\tT)).
\]
\end{enumerate}
\end{prp}

\begin{rmk*}
Taking the homotopy categories,
we recover the derived functors on unbounded derived categories
in \cite[\S 14.4]{KS}.
\end{rmk*}

We also introduce 

\begin{ntn}\label{ntn:is:shExt}
Let $\tT$ and $\shA$ be the same as Proposition \ref{prp:is:df}.
For $\shM,\shN \in \iDa(\iMod_{\shA}(\tT))$ and $n \in \bbZ$, we set
\[
 \shExt^n_{\shA}(\shM,\shN) := 
 \pi_0 \shHom_{\iDa(\iMod_{\shA}(\tT))}(\shM[-n],\shN) 
 \in \iMod_{\shA}(\tT).
\] 
\end{ntn}


\subsection{Open and closed geometric immersions}
\label{ss:is:oci}

In this subsection we introduce the notions of 
open and closed geometric immersions in an $\infty$-topos.
As for closed geometric immersion, 
our presentation is just a citation from \cite[\S 7.3]{Lur1}.

Hereafter we fix an $\infty$-topos $\tT$.

\subsubsection{Open geometric immersion}

Let $\one_{\tT} \in \tT$ be a final object (Corollary \ref{cor:pr:final}).

\begin{dfn}\label{dfn:it:sub}
Denote by $\Sub(\one_{\tT})$ 
the set of equivalence classes of monomorphisms $U \to \one_{\tT}$ in $\tT$.
We regard $\Sub(\one_{\tT})$ as a poset under inclusion.
\end{dfn}

Note that $\Sub(\one_{\tT})$ can be identified with 
the set of equivalence classes of $(-1)$-truncated objects in $\tT$.
Also, the poset $\Sub(\one_{\tT})$ is independent of the choice of $\one_{\tT}$
up to canonical isomorphism by Fact \ref{fct:ic:fin}.

Recall that any over-$\infty$-category $\oc{\tT}{C}$ of $C \in \tT$ 
is an $\infty$-topos (Fact \ref{fct:pr:ovc}).
Thus for $U \in \Sub(\one_{\tT})$ we can define an $\infty$-topos $\oc{\tT}{U}$ by
$\oc{\tT}{U} := \oc{\tT}{U'}$, where $U' \in \tT$ is a representative of $U$.
A different choice of $U'$ will cause an equivalent $\infty$-topos.

Let us also recall the biadjunction $(j_!,j^*,j_*)$ in Fact \ref{fct:is:biadj}
associated to an object $U \in \tT$.
Here $j_!: \oc{\tT}{U} \to \tT$ is the canonical functor 
of the over-$\infty$-category $\oc{\tT}{U}$ (Corollary \ref{cor:ic:ovc}),
and we have a pair of geometric morphisms of $\infty$-topoi
\[
 j_!: \oc{\tT}{U} \adjunc \tT      :j^*, \quad 
 j^*: \tT      \adjunc \oc{\tT}{U} :j_*.
\]

Now we introduce the notion of an open geometric immersion.

\begin{dfn}\label{dfn:is:oi}
A geometric morphism $f: \tU \to \tT$ corresponding to 
the adjunction $f^*: \tT \rlto \tT':f_*$ of $\infty$-topoi is 
an \emph{open geometric immersion} if there exists $U \in \Sub(\one_{\tT})$ 
such that the composition $\tU \xr{f_*} \tT \xr{j^*} \oc{\tT}{U}$ 
is an equivalence of $\infty$-categories.
We denote an open geometric immersion typically by $f: \tU \inj \tT$.
\end{dfn}

Following \cite{SGA4} we name the functors appearing in the above argument as

\begin{dfn}\label{dfn:is:oij}
For an open geometric immersion $j: \oc{\tT}{U} \inj \tT$
with $U \in \Sub(\one_{\tT})$, the functor $j_!: \oc{\tT}{U} \to \tT$
is called the \emph{extension by empty},
and $j^*: \tT \to \oc{\tT}{U}$ is called the \emph{restriction functor}.
\end{dfn}

Let us give a few formal properties of an open geometric immersion
which is an analogue of \cite[IV Proposition 9.2.4]{SGA4}.

\begin{lem}\label{lem:is:j_!}
Let  $j: \tU \inj \tT$ be an open geometric immersion.
\begin{enumerate}[nosep]
\item
The functors $j_!$ and $j_*$ are fully faithful (Definition \ref{dfn:ic:sfe}).

\item
The counit transformation $j_! j^* \to \id_{\tT}$ (Definition \ref{dfn:ic:uc})
is a monomorphism of functors.

\item
We have a monomorphism $j_! \inj j_*$ of functors.
\end{enumerate}
\end{lem}

\begin{proof}
\begin{enumerate}[nosep]
\item 
Replace $\tU$ by $\oc{\tT}{U}$ with $U \in \Sub(\one_{\tT})$,
and denote by $i: U \to \one_{\tT}$ the corresponding monomorphism.
Then by Corollary \ref{cor:ic:ovc} on over-$\infty$-categories
and by Fact \ref{fct:ic:mono-dfn} on monomorphisms,
we have a commutative diagram
\[
 \xymatrix{
  \oc{\tT}{i} \ar[r]^{p}   \ar[d]  & \oc{\tT}{\one_{\tT}} \ar[d]^{t} \\
  \oc{\tT}{U} \ar[r]_{j_!} \ar[ru] & \tT
 }
\]
with $p$ fully faithful.
By Fact \ref{fct:ic:f=sf} on final objects,
we see that $t$ is a trivial fibration of simplicial sets
with respect to Kan model structure (Fact \ref{fct:ic:Km}).
Then we can see that $j_!$ is fully faithful.

As for $j_*$, the result follows from biadjunction of $(j_!,j^*,j_*)$.
\item
For any $X \in \tT$, the object $j_! j^*(X)$ sits in the pullback square
\[
 \xymatrix{
  j_! j^*(X) \ar[r]^(0.6){c} \ar[d] & X \ar[d] \\ U \ar@{^{(}->}[r]_i & \one_{\tT}}
\]
Then since $i$ is $(-1)$-truncated 
we deduce that $c$ is $(-1)$ by \cite[Remark 5.5.6.12]{Lur1}.

\item
Since $j_*$ is fully faithful if and only if the unit transformation
$\id_{\tU} \to j^* j_!$ is an equivalence,
the statement follows from (1) and (2).
\end{enumerate}
\end{proof}

Next we give an analogue of \cite[XVII, Lemma 5.1.2]{SGA4} 
in the ordinary topos theory.

Let $f: \tT' \to \tT$ be a geometric morphism corresponding to the adjunction 
$f^*: \tT \rlto \tT' :f_*$, and $j: \tU = \oc{\tT}{U} \inj \tT$ be 
an open geometric immersion with $U \in \Sub_{\one_{\tT}}$.
Note that $f^*(\one_{\tT})=\one_{\tT'}$ 
since $f^*$ is left exact so that it commutes with limits.
Thus $U' := f^*(U)$ belongs to $\Sub(\one_{\tT'})$,
Then, setting $\tU' := \oc{\tT}{U'}$ and denoting by $j': \tU' \inj \tT'$ 
the natural open geometric immersion, we have a square
\begin{align}\label{diag:is:oi}
 \xymatrix{
  \tU' \ar[r]^{g} \ar@{^{(}->}[d]_{j'} & \tU \ar@{^{(}->}[d]^{j} \\ 
  \tT' \ar[r]_f & \tT}
\end{align}
in $\iRTop$.

Let us give the geometric morphism $g=(g^*: \tU \rlto \tU' :g_*)$ explicitly.
Using the associated adjunction $j_!: \tU \rlto \tT :j^*$ to 
the geometric morphism $j$ and 
${j'}^*: \tT' \rlto \tU' :j'_*$ to $j'$, we set 
$g^* := (\tU \xr{j_!} \tT \xr{f^*} \tT' \xr{{j'}^*} \tU')$.
Since $g^*$ is a composition of left exact functors, it is also left exact.
Similarly we set 
$g_* := (\tU' \xr{j'_*} \tT' \xr{f_*} \tT \xr{j^*} \tU)$, 
which is right exact.
The commutativity of the square holds by definition, 
and the adjunction property of $(g^*,g_*)$ is obvious.

We call $\tU$ the \emph{inverse image} of $\tU'$ in $\tT$.

\begin{lem}\label{lem:is:j_!-bc}
Let $f: \tT' \to \tT$ be a geometric morphism of $\infty$-topoi,
and $\tU \inj \tT$ be an open geometric immersion.
Under the notation in the square \eqref{diag:is:oi}
there exists an equivalence $f^* j_! \simto j'_! g^*$
making the following square commutative up to homotopy.
\[
 \xymatrix{
  f^* j_! \ar[r]^{\sim} \ar@{^{(}->}[d] & j'_! g^* \ar@{^{(}->}[d] \\ 
  f^* j_* \ar[r]_{\alpha} & j'_* g^*
 }
\]
Here $\alpha$ is the \emph{base change morphism}
(Definition \ref{dfn:is:bc})
associated to the square \eqref{diag:is:oi}.
\end{lem}

\begin{proof}
By definition of $g^*$ and Lemma \ref{lem:is:j_!}
there is an equivalence $j'_! g^* \simto f^* j_!$.
The commutativity can be checked by restricting to $\tU$,
but it is trivial.
\end{proof}

\subsubsection{Closed geometric immersion}

Next we turn to closed geometric immersions.
We begin with 

\begin{dfn*}
Let $U \in \tT$.
\begin{enumerate}[nosep]
\item
An object $X \in \tT$ is \emph{trivial on $U$}
if for any morphism $U' \to U$ in $\tT$
the mapping space $\Map_{\tT}(U',X) \in \topH$ is contractible.

\item 
We denote by $\tT/U \subset \tT$ 
the full sub-$\infty$-category spanned by trivial objects on $U$.
\end{enumerate}
\end{dfn*}

Let $U \in \Sub(\one_{\tT})$.
We define a full sub-$\infty$-category $\tT/U \subset \tT$ by $\tT/U := \tT/U'$,
where $U' \in \tT$ is a representative of $U$.
See \cite[Lemma 7.3.2.5]{Lur1} for the independence of the choice of $U'$
in this definition.
Now we have

\begin{fct*}[{\cite[Proposition 7.3.2.3, Lemma 7.3.2.4]{Lur1}}]
For $U \in \Sub(\one_{\tT})$, the $\infty$-category $\tT/U$ 
is a localization of $\tT$ and is an $\infty$-topos.
\end{fct*}

Thus there is a functor $\tT \to \tT/U$ and 
the following definition makes sense.

\begin{dfn}[{\cite[Definition 7.3.2.6]{Lur1}}]\label{dfn:is:ci}
%
A geometric morphism $f=(f^*: \tT \rlto \tT' :f_*)$ of $\infty$-topoi 
is a \emph{closed geometric immersion} if there exists $U \in \Sub(\one_{\tT})$
such that $f_*: \tT' \to \tT$ induces an equivalence 
$\tT' \to \tT/U$ of $\infty$-categories.
\end{dfn}

Let us remark that in \cite{Lur1} it is just called a closed immersion.

\subsection{Simplicial $\infty$-topoi and descent theorem}
\label{sss:is:spr}

In this subsection we will give an $\infty$-analogue of 
the theory of cohomological descent
in \cite[$\mathrm{V}^{\text{bis}}$]{SGA4}, \cite[\S 2]{O}.
Our discussion will utilize the $\infty$-category $\iRTop$ 
of $\infty$-topoi and geometric morphisms (Definition \ref{dfn:is:RTop}).
Let us remark that the statements in this subsection are 
much more general statements than what we need in the later sections. 

\subsubsection{$K$-injective resolutions}
\label{sss:is:Kinj}

We begin with a restatement of \cite[\S 3]{Sp}
in the context of ringed $\infty$-topoi.


Let $(\tT,\shA)$ be a ringed $\infty$-topos 
with $\shA$ a sheaf of commutative rings as in \S \ref{ss:is:crm}.
Let us denote by $\Mod_{\shA}(\tT) := \Ho \iMod_{\shA}(\tT)$
the homotopy category of the $\infty$-category of sheaves of $\shA$-modules.
Recall that $\Mod_{\shA}(\tT)$ is a Grothendieck abelian category.
In this part we denote the unbounded derived $\infty$-category 
of sheaves of $\shA$-modules by 
\[
 \iDa(\tT,\shA) := \iDa(\Mod_{\shA}(\tT)).
\] 
Recall that $\iDa(\tT,\shA)$ is stable and equipped with a $t$-structure 
whose heart is equivalent to $\iMod_{\shA}(\tT)$.
We also denote by 
\[
 \iDm(\tT,\shA) 
 := \cup_{n \in \bbZ} \iDa(\tT,\shA)_{\ge -n} \subset \iDa(\tT,\shA)
\]
the sub-$\infty$-category of right bounded objects.
Then by \cite[3.6]{Sp} we have 

\begin{lem}\label{lem:is:spr}
For any $\shM \in \iDa(\tT,\shA)$, 
there exists a morphism $f: \shM \to \shI$ in $\iDa(\tT,\shA)$
satisfying the following conditions.
\begin{itemize}
\item 
$\shI = \varinjlim \shI_n$ with $\shI_n \in \iDm(\tT,\shA)$
and $\pi_j \shI_n$ injective for all $n \in \bbN$, $j \in \bbZ$.

\item
$f$ is induced by a compatible collection of equivalences
$f_n: \tau_{\ge-n}\shM \to \shI_n$.

\item
For each $n$, the morphism $\shI_n \to \shI_{n-1}$ is an epimorphism
whose kernel $K_n$ belongs to $\iDm(\tT,\shA)$ and 
$\pi_j K_n$ are injective for any $j \in \bbZ$.
\end{itemize}
\end{lem}

We now restate \cite[3.13 Proposition]{Sp} and \cite[2.1.4 Proposition]{LO1}.
For stating that, 
let us recall that for an $\infty$-site $(\iX,\tau)$ 
we denote by $\iSh(\iX,\tau)$ the $\infty$-category of $\tau$-sheaves on $\iX$,
which is an $\infty$-topos (Definition \ref{dfn:it:sh}, Fact \ref{fct:pr:sh}).

\begin{lem}\label{lem:is:S01}
Let $\iB \subset \iMod_{\shA}(\tT)$ a full sub-$\infty$-category.
Assume the following conditions for $(\tT,\shA)$ and $\iB$. 
\begin{enumerate}[nosep, label=(S\arabic*)]
\setcounter{enumi}{-1}
\item
$\tT$ is equivalent to an $\infty$-topos 
of the form $\iSh(\iX,\tau)$ with some $\infty$-site $(\iX,\tau)$.
\item
For each $U \in \iX$, there exists a covering $\{U_j \to U\}_{j \in J}$
and an integer $n_0$ such that we have $\pi_{-n} \rst{\shM}{U_j} = 0$ 
for any $\shM \in \iB$, $n \ge n_0$ and $j \in J$.
Here $\rst{\shM}{U_j}$ denotes the restriction (Notation \ref{ntn:is:rst}).
\end{enumerate}
Let $\shM \in \iDa(\tT,\shA)$ satisfy $\pi_j \shM \in \iB$ for all $j$.
Then the morphism $f:\shM \to \shI$ in Lemma \ref{lem:is:spr} 
is an equivalence in $\iDa(\tT,\shA)$.
\end{lem}

\begin{rmk*}
We can weaken the condition (S0) by the one that $\tT$ is hypercomplete,
but will not pursue this point.
\end{rmk*}

\subsubsection{Simplicial $\infty$-topoi}

We give some notations for simplicial $\infty$-topoi. 
The symbols are borrowed from \cite[2.1]{LO1}.
Let us begin with the following citation.

\begin{dfn*}[{\cite[Definition 6.3.1.6]{Lur1}}]
A map $p : X \to S$ of simplicial sets is a \emph{topos fibration}
if the following three conditions are satisfied.
\begin{enumerate}
\item 
The map $p$ is both a cartesian fibration and a cocartesian fibration
(Definition \ref{dfn:ic:cc-fib}).

\item
For every vertex $s$ of $S$, 
the corresponding fiber $X_s = X \times_S \{s\}$ is an $\infty$-topos.

\item
For every edge $e : s \to s_0$ in $S$, 
the associated functor $X_s \to X_{s_0}$ is left exact.
\end{enumerate}
\end{dfn*}

Recall the $\infty$-category $\iRTop$ 
of $\infty$-topoi and geometric morphisms (Definition \ref{dfn:is:RTop}).
By \cite[Theorem 6.3.3.1]{Lur1}, $\iRTop$ admits small filtered limits.
Let us explain some relevant notions in the proof of loc.\ cit.
for later use.

Let $\iI$ be a small filtered $\infty$-category 
and $q: \iI^{\op} \to \iRTop$ be any functor of $\infty$-categories.
Then by \cite[Proposition 6.3.1.7]{Lur1},
there is a topos fibration $p: X \to \iI^{\op}$ classified by $q$
(see \cite[Definition 3.3.2.2]{Lur1} for the definition of 
a classifying functor of cartesian fibration).

By \cite[Proposition 5.3.1.18]{Lur1}, 
for any filtered $\infty$-category $\iI$, we have a filtered poset $A$ 
and a cofinal map $\Ner(A) \to \iI$ of simplicial sets.
Here $A$ is regarded as a category in which the set of morphisms is given by 
\[
 \Hom_{A}(a,b) := 
  \begin{cases}
   \{*\} & a \le b, \\ \emptyset & a \not\le b.
  \end{cases}
\]
See also \cite[Definition 4.1.1.1]{Lur1} for the definition of 
a cofinal map of simplicial sets.
Thus, as for the functor $q: \iI^{\op} \to \iRTop$,
we can replace $\iI$ by $\Ner(A)$.

Now let us given a functor $q: \Ner(A)^{\op} \to \iRTop $
and $p: X \to \Ner(A)^{\op}$ be a topos fibration classified by $q$.
Then by \cite[Proposition 6.3.3.3]{Lur1} the simplicial subset 
$\stX \subset \Map(\Ner(A)^{\op},X)$ 
of cartesian sections of $p: X \to \Ner(A)^{\op}$ is an $\infty$-topos.
Moreover by \cite[Proposition 6.3.3.5]{Lur1} we find that 
for each $a \in A$ the evaluation $\stX \to X_a$ 
gives a geometric morphism of $\infty$-topoi.
Thus $\stX$ gives a filtered limit of $q: \Ner(A)^{\op} \to \iRTop$.

\begin{dfn*}
Let $\iI$ be a filtered $\infty$-category.
An \emph{$\iI$-simplicial $\infty$-topos} is a topos fibration 
$p: \tT_{\bullet} \to \iI^{\op}$ which is classified 
by a functor $q: \iI^{\op} \to \iRTop$ of $\infty$-categories.
We will often denote $\tT_{\bullet}$ to indicate 
an $\iI$-simplicial $\infty$-topos.
\end{dfn*}

Let $\tT_\bullet$ be an $\iI$-simplicial $\infty$-topos.
For each $\delta : i \to j$ in $\iI$,
we denote the corresponding geometric morphism of $\infty$-topoi 
by the same symbol $\delta: \tT_j \to \tT_i$,
and denote the associated adjunction by 
$\delta^{-1}: \tT_i \rlto \tT_j :\delta_*$.
Note that our convention is the opposite of \cite[2.1]{LO1}.

By the above argument $\tT_{\bullet}$ is equipped with a geometric morphism
$e_i: \tT_i \to \tT_{\bullet}$ of $\infty$-topoi for each $i \in \iI$
corresponding to an adjunction $e_i^{-1}: \tT_{\bullet} \rlto \tT_i :e_{i,*}$.

Let us explain examples of $\iI$-simplicial $\infty$-topoi
which will be used in the later text.
Recall first that $\CS$ denotes the $\infty$-category of 
combinatorial simplices (Definition \ref{dfn:ic:CS}),
and that a contravariant functor of ordinary categories from $\CS$ 
is called a simplicial object (\S \ref{ss:ic:ss}).

\begin{dfn}\label{dfn:is:CSinj} 
\begin{enumerate}[nosep]
\item
We denote by $\CS^{\str} \subset \CS$ the subcategory with the same objects
but morphisms only the injective maps.
A contravariant functor from $\CS^{\str}$ is called 
a \emph{strictly simplicial object}.

\item 
A \emph{simplicial $\infty$-topos} is defined to be 
a $\CS$-simplicial $\infty$-topos.
A \emph{strictly simplicial $\infty$-topos} is defined to be 
a $\CS^{\str}$-simplicial $\infty$-topos.

\item
For a simplicial $\infty$-topos $\tT_\bullet$, 
we denote by $\tT^{\str}_\bullet$ the strictly simplicial $\infty$-topos 
obtained by restriction to $\CS^{\str} \subset \CS$.
\end{enumerate}
\end{dfn}

\begin{rmk*}
In the papers \cite{O,LO1,LO2}, 
the symbol $\Delta^+$ is used for our $\CS^{\str}$.
In \cite[Notation 6.5.3.6]{Lur1}, the symbol $\CS_s$ is used
and a contravariant functor from $\CS_s$ is called a semisimplicial object.
We will not use these notations.
\end{rmk*}



Another example of an $\iI$-simplicial topos is

\begin{dfn}\label{dfn:is:prj}
Let $\bbN=\{0,1,2.\ldots\}$ be the set of non-negative integers,
regarded as a filtered poset with the standard order.
For an $\infty$-topos $\tT$,
we denote by $\tT^{\bbN}$ the $\Ner(\bbN)$-simplicial $\infty$-topos
associated to the constant functor $q: \Ner(\bbN) \to \iRTop$, 
$q(n) := \tT$.
We call it the \emph{$\infty$-topos of projective systems in $\tT$}.
\end{dfn}

\subsubsection{Simplicial ringed $\infty$-topoi}
\label{sss:is:spl}

We will introduce notations for simplicial ringed $\infty$-topoi.
For that, let us first consider a sheaf on a simplicial $\infty$-topoi.
Let $\iI$ be a filtered $\infty$-category and
$\tT_\bullet$ be an $\iI$-simplicial $\infty$-topos.
Then a sheaf $\shF_{\bullet}$ 
with values in an $\infty$-category $\iC$ on $\tT_\bullet$
consists of sheaves $\shF_i \in \iShv_{\iC}(\tT_\bullet)$ for $i \in \iI$
and morphisms $\delta^{-1} \shF_j \to \shF_i$ in $\iShv_{\iC}(\tT_i)$
for $\delta: i \to j$ in $\iI$.

As in the ordinary topos theory, we introduce

\begin{dfn}\label{dfn:is:flat}
A functor $f: (\tT,\shA) \to (\tT',\shA')$ of ringed $\infty$-topoi
is \emph{flat} if the inverse image functor
$f^*: \iMod_{\shA'}(\tT') \to \iMod_{\shA}(\tT)$
is left and right exact. 
\end{dfn}

Now we have

\begin{dfn*}
Let $\iI$ be a filtered $\infty$-category.
An \emph{$\iI$-simplicial ringed $\infty$-topos} is a pair 
$(\tT_{\bullet},\shA_{\bullet})$ 
of an $\iI$-simplicial $\infty$-topos $\tT_{\bullet}$ 
and a sheaf $\shA_{\bullet} \in \iShv_{\iCAlg(\iC)}(\tT_{\bullet})$
such that the functor $\delta: (\tT_j,\shA_j) \to (\tT_i,\shA_i)$
of ringed $\infty$-topoi is flat for each $\delta: i \to j$ in $\iI$.
\end{dfn*}

As in the non-ringed case, we have a geometric morphism 
$e_i: (\tT_i,\shA_i) \to (\tT_{\bullet},\shA_{\bullet})$ 
of $\infty$-topoi for each $i \in \iI$,
which corresponds to an adjunction
$e_i^{-1}: (\tT_{\bullet},\shA_{\bullet}) \rlto (\tT_i,\shA_i) :e_{i,*}$.

Under these preliminaries
let us explain the results in \cite[2.1]{LO1}.
Let $\iI$ be a filtered $\infty$-category and 
$(\tT_{\bullet},\shA_{\bullet})$ be an $\iI$-simplicial ringed $\infty$-topos.
We assume that $\shA_{\bullet}$ is a sheaf of commutative rings.
Then the homotopy category 
$\Mod_{\shA_{\bullet}}(\tT_{\bullet}) := 
 \Ho \iMod_{\shA_{\bullet}}(\tT_{\bullet})$ 
is a Grothendieck abelian category.
Thus the following definition makes sense.

\begin{ntn}\label{ntn:is:sDa}
We denote the associated derived $\infty$-category by
\[
 \iDa(\tT_{\bullet},\shA_{\bullet}) := 
 \iDa(\Mod_{\shA_{\bullet}}(\tT_{\bullet})).
\]
\end{ntn}

Let also $\iB_{\bullet}$ be a full sub-$\infty$-category of 
$\iMod_{\shA_{\bullet}}(\tT_{\bullet})$,
and for each $i \in \iI$ let $\iB_i$ be the essential image
of $\iB_{\bullet}$ under the geometric morphism
$e_i: (\tT_i,\shA_i) \to (\tT_{\bullet},\shA_{\bullet})$.

Now Lemma \ref{lem:is:S01} and the argument \cite[2.1.9 Proposition]{LO1} yield

\begin{lem}\label{lem:is:S2}
Assume the following condition on 
$(\tT_{\bullet},\shA_{\bullet})$ and $\iB_{\bullet}$.
\begin{enumerate}[nosep]
\item[(S2)]
For each $i \in \iI$, the ringed $\infty$-topos $(\tT_i,\shA_i)$ 
and the $\infty$-category $\iB_i$ satisfy the assumptions 
(S0) and (S1) in Lemma \ref{lem:is:S01}
\end{enumerate}
Then for each $\shM \in \iDa(\tT_{\bullet},\shA_{\bullet})$ 
there exists a morphism $\shM \to \shI$ 
in $\iDa(\tT_{\bullet},\shA_{\bullet})$ satisfying the conditions
in Lemma \ref{lem:is:spr}.
If moreover $\pi_j \shM \in \iB_{\bullet}$ for any $j \in \bbZ$,
then $f$ is an equivalence.
\end{lem}

\subsubsection{Descent lemma}
\label{sss:is:des}

In this part we review the result in \cite[2.2]{LO1}.

Let $(\tS,\shB)$ be a ringed $\infty$-topos 
with $\shB$ a sheaf of commutative rings,
i.e., $\shB \in \iShv_{\iCom(\iAb)}(\tS)$. 
As in \S \ref{sss:is:Kinj}, 
we denote the derived $\infty$-category of 
sheaves of stable $\shB$-modules on $\tS$ 
by $\iDa(\tS,\shB) := \iDa(\Mod_{\shB}(\tS))$
with $\Mod_{\shB}(\tS) := \Ho \iMod_{\shB}(\tS)$.
We always consider the $t$-structure on $\iDa(\tS,\shB)$ 
explained in Lemma \ref{lem:is:srm}.
In particular we have $\iDa(\tS,\shB)^{\heartsuit} \simeq \iMod_{\shB}(\tS)$. 

Let $\iB' \subset \iMod_{\shB}(\tS)$ be a sub-$\infty$-category such that 
the homotopy category $\Ho \iB'$ is a Serre subcategory of
the abelian category $\Mod_{\shB}(\tS)$.
In other words, $\Ho \iB'$ is closed under kernels, cokernels and extensions.
We denote by
\[
 \iDs{\iB'}(\tS,\shB) \subset \iDa(\tS,\shB)
\]
the full sub-$\infty$-category spanned by those $\shM$
such that all the homotopy groups $\pi_j \shM$ belongs to $\iB'$.

Next let $(\tT_{\bullet},\shA_{\bullet})$ be 
a simplicial and strictly simplicial ringed $\infty$-topoi 
respectively (Definition \ref{dfn:is:CSinj})
with $\shA_i$ a sheaf of commutative rings for each $i \in \iI$
($\iI=\Ner(\CS)$ or $\iI=\Ner(\CS^{\str})$).
Let $\ep: (\tT_{\bullet},\shA_{\bullet}) \to (\tS,\shB)$
be a functor of ringed $\infty$-topoi
such that $\ep_i: (\tT_{i},\shA_{i}) \to (\tS,\shB)$ is flat 
(Definition \ref{dfn:is:flat}) for each $i \in \iI$.

Let $\iB'_{\bullet}$ be the image of $\iB'$ under the functor 
$\ep^*: \iDa(\tS,\shB)^{\heartsuit} \to 
 \iDa(\tT_\bullet,\shA_\bullet)^{\heartsuit}$.

Now the argument in \cite[2.2.2 Lemma, 2.2.3 Theorem]{LO1} gives 

\begin{prp}\label{prp:is:223}
Assume the following conditions.
\begin{enumerate}[nosep]
\item[(S3)]
$(\tT_{\bullet},\shA_{\bullet})$ and $\iB'_{\bullet}$ satisfy 
the condition (S2) in Lemma \ref{lem:is:S2}.
\item[(S4)]
The (restricted) morphism $\ep^*: \iB' \to \iB'_{\bullet}$
is an equivalence.
\end{enumerate}
Then the homotopy category $\Ho \iB'_{\bullet}$ is a Serre subcategory of 
$\Ho \iDa(\tT_{\bullet},\shA_{\bullet})^{\heartsuit}$,
and $\ep^*$ induces an equivalence
\[
 \ep^*: \iDs{\shB}(\tS) \longto 
        \iDs{\shB_{\bullet}}(\tT_{\bullet})
\]
of stable $\infty$-categories.
The quasi-inverse is induced by the right adjoint $\ve_*$.
\end{prp}

\subsubsection{Gluing lemma}

In this part we recall the result in \cite[2.3]{LO1}.

Let $\CS^{\str}_+$ be the category of possibly empty finite ordered sets
with injective order preserving maps.
Thus the objects of $\CS^{\str}_+$ are $[n] =\{0,1,\dots,n\}$ ($n \in \bbN$)
and $\emptyset$.
Hereafter we denote $[-1]:=\emptyset$.
We can regard $\CS^{\str}$ (Definition \ref{dfn:is:CSinj})
as a full subcategory of $\CS^{\str}_+$.

Let $\tT$ be an $\infty$-topos and $U^{\str}_{\bullet} \to \emptyset_{\tT}$ 
be a strictly simplicial hypercovering of 
the initial object $\emptyset_{\tT}$ of $\tT$ (Corollary \ref{cor:pr:final}).
We denote the localized $\infty$-topos $\oc{\tT}{U_n}$ on $U_n$
(Fact \ref{fct:pr:ovc}) by the same symbol $U_n$ for $n \in \bbN$.
We also denote $U_{-1} := \emptyset_\tT$.
Thus we have a strictly simplicial $\infty$-topos $U^{\str}_{\bullet}$
with an augmentation $\pi: U^{\str}_{\bullet} \to \tT$.

Let $\shA$ be a sheaf of commutative rings on $\tT$.
We denote the induced sheaf on $U^{\str}_{\bullet}$ by the same symbol $\shA$.
Then $\pi$ induces a functor $(U^{\str}_{\bullet},\shA) \to (\tT,\shA)$
of ringed $\infty$-topoi, which will be denoted by the same symbol $\pi$.

We now consider the cartesian and cocartesian fibration over $\CS^{\str}_+$
whose fiber is given by $\iMod_{\shA}(U_n)$.
Let $\iC_{\bullet}$ be a full sub-$\infty$-category of this fibration
such that $\Ho \iC_n$ is a Serre subcategory
of the abelian category $\Ho \iMod_{\shA}(U_n)$ for each $n$.

We impose on $\tT$, $U^{\str}_{\bullet}$ and $\shA$ the following conditions.
\begin{enumerate}[nosep, label=(G\arabic*)]
\item
For any $[n] \in \CS^{\str}_+$,
the $\infty$-topos $U_n$ is equivalent to $\iSh(\iX_n,\tau)$
with some $\infty$-site $(\iX_n,\tau)$
where for each $V \in \iX_n$ there exists $n_0 \in \bbZ$ and 
a covering sieve $\{V_j \to V\}_{j \in J}$ 
such that for any $\shM \in \iC_n$ we have $H^q(V_j,\shM)=0$
for any $q \ge n_0$.

\item
The functor 
$\iC_{-1} \to 
 \{\text{cartesian sections of $\iC_{\bullet}^{\str} \to \CS^{\str}$}\}$
is an equivalence,
where $\iC_{\bullet}^{\str} := \rst{\iC_{\bullet}}{\CS^{\str}}$ 
(Definition \ref{dfn:is:CSinj}).

\item
$\iDa(\tT,\shA)$ is compactly generated.
\end{enumerate}

\begin{fct}[{\cite[2.3.3 Theorem]{LO1}}]\label{fct:is:glue}
Let $\tT$, $U^{\str}_\bullet$ and $\shA$ be as above 
and assume the conditions (G1)--(G3).
Consider the cartesian and cocartesian fibration $\iD_{\bullet} \to \CS^{\str}$ 
whose fiber over $[n] \in \CS^{\str}$ is $\iDs{\iC_n}(U_n,\shA)$.
Let $[n] \mapsto \shK_n \in \iDs{\iC_n}(U_n,\shA)$ be a cartesian section of
this fibration $\iD_{\bullet} \to \CS^{\str}$ such that 
$\shExt^i_{\shA}(\shK_0,\shK_0)=0$ for any $i \in \bbZ_{<0}$.
Then $(\shK_n)_{n \in \bbN}$ determines an object
$\shK \in \iDs{\iC_{-1}}(\tT,\shA)$ such that the natural functor
$\iDs{\iC_{-1}}(\tT,\shA) \to 
 \{\text{cartesian sections of the fibration 
         $\iD_{\bullet} \to \Delta^{\str}$}\}$
recovers $(\shK_n)_{n \in \bbN}$.
Moreover $\shK$ is determined up to contractible ambiguity.
\end{fct}

We will not repeat the proof, but record one useful lemma.
It is used in \cite{LO1} to show the uniqueness of the object $\shK$
in the above Fact \ref{fct:is:glue}.

\begin{lem}[{\cite[2.3.4. Lemma]{LO1}}]\label{lem:is:234}
Let $\tT$ and $\shA$ be as above.
Also let $\shM$ and $\shN$ be objects in $\iDa(\tT,\shA)$ 
satisfying $\shExt_{\shA}^i(\shM,\shN) = 0$ for every $i < 0$.
For $U \in \tT$, we denote the $\infty$-topos $\oc{\tT}{U}$ 
by the same symbol $U$.
Then the correspondence
\[
 U \longmapsto \Map_{\iDa(U,\shA)}(\shM_U, \shL_U) 
\] 
for each $U \in \tT$ determines a unique object of $\iMod_{\shA}(\tT)$.
\end{lem}

\section{Derived algebraic spaces and \'etale sheaves}
\label{s:dAS}

In this section we introduce \emph{derived algebraic spaces},
which are derived analogue of algebraic spaces, 
and \'etale sheaves on them.
These  will be used to define the lisse-\'etale sheaves 
on derived stacks in \S \ref{s:LE}.

Let us explain a little bit more 
the motivation for introducing derived algebraic spaces.
Our definition of the lisse-\'etale $\infty$-site for a derived stack
will be a direct  analogue of the lisse-\'etale site for an algebraic stack
(see \S \ref{ss:Cl:LE} for a brief recollection).
In the non-derived case, one can describe every object on the lisse-\'etale site
of an algebraic stack $X$ in terms of an object on the big \'etale site of 
the simplicial algebraic space $X_{\bullet}$ making out of 
a smooth \'etale covering $U \to X$, as explained in \cite{O, O:book}.
Thus it will be useful for the study of the lisse-\'etale $\infty$-site 
{}to introduce a derived analogue of algebraic spaces.

In this section we work over a fixed base commutative ring $k$
unless otherwise stated.
We suppress the notation of $k$ unless confusion would arise.
For example, $\idSt$ means $\idSt_k$.

\subsection{Derived algebraic spaces}
\label{ss:dAS:dAS}


\subsubsection{Geometric derived stacks as quotients}

We begin with the citation from \cite[\S 1.3.4]{TVe2} 
on a characterization of a geometric 
derived stack as a quotient of Segal groupoid.
Our presentation is a translation of the model-theoretic statements 
in loc.\ cit.\ to the $\infty$-theoretic language.

Recall the category $\CS$ of combinatorial simplices 
(Definition \ref{dfn:ic:CS}), and let $X_\bullet: \CS^{\op} \to \iC$ be 
a simplicial object in an $\infty$-category $\iC$.
For $n \in \bbZ_{>0}$ and $i=0,1,\ldots,n-1$, 
we define a morphism 
\[
 \sigma_i: X_n \longto X_1
\]
in $\iC$ to be the pullback
by the order-preserving map $[1] \longto [n]$ given by
$0 \mapsto i$ and $1 \mapsto i+1$.
We denote the face map (\S \ref{ss:ic:ss}) of $X_\bullet$ by 
\[
 d_j: X_n \longto X_{n-1}
\]
for $n \in \bbN$ and $j=0,\ldots,n-1$.
It is defined to be the pullback by the order-preserving map $[n-1] \to[n]$ with
$i \mapsto i$ for $i<j$ and $i \mapsto i+1$ for $i \ge j$.

\begin{dfn}\label{dfn:dAS:Segal}
A \emph{Segal groupoid object} in an $\infty$-category $\iC$ is a simplicial object
$X_\bullet: \CS^{\op} \to \iC$ satisfying the following two conditions.
\begin{itemize}[nosep]
\item 
For any $n \in \bbZ_{>0}$, the following morphism
is an equivalence in $\iC$.
\[
 \prod_{i=0}^{n-1} \sigma_i: X_n \longto
 \underbrace{X_1 \times_{d_0,X_0,d_0} X_1 \times_{d_0,X_0,d_0} \cdots 
                 \times_{d_0,X_0,d_0} X_1}_{\text{$n$-times}}.
\]
\item
The morphism $d_0 \times d_1: X_2 \to X_1 \times_{d_0,X_0,d_0} X_1$ 
is an equivalence in $\iC$.
\end{itemize}
\end{dfn}


\begin{dfn*}
For $n \in \bbZ_{\ge -1}$, a Segal groupoid object $\stX_\bullet$ in $\idSt_k$ 
is \emph{$n$-smooth} if it satisfies the following two conditions.
\begin{itemize}[nosep]
\item 
The derived stacks $\stX_0$ and $\stX_1$ are (small) coproduct of 
$n$-geometric derived stacks.
\item
The morphism $d_0: \stX_1 \to \stX_0$ of derived stacks is $n$-smooth 
in the sense of Definition \ref{dfn:dSt:geom}.
\end{itemize}
\end{dfn*}

\begin{fct}[{\cite[Proposition 1.3.4.2]{TVe2}}]\label{fct:dAS:Segal}
For a derived stack $\stX$ and $n \in \bbN$, 
the following two conditions are equivalent.
\begin{enumerate}[nosep, label=(\roman*)]
\item 
$\stX$ is $n$-geometric.
\item
There exists an $(n-1)$-smooth Segal groupoid object $\stX_\bullet$ in $\idSt_k$
such that $\stX_0$ is a coproduct 
of affine derived schemes
and $\stX \simeq \iclim_{[m] \in \CS} \stX_m$ in $\idSt_k$.
\end{enumerate}
\end{fct}

\begin{ntn*}
If one of the conditions in Fact \ref{fct:dAS:Segal} is satisfied, 
then we say \emph{$\stX$ is the quotient stack of 
the $(n-1)$-smooth Segal groupoid $\stX_{\bullet}$}.
\end{ntn*}

Note that the expression $\iclim_{[n] \in \CS} \stX_n$ in (ii) makes sense,
for $\idSt_k$ is an $\infty$-topos (Fact \ref{fct:idSt:hyper}) 
so that it admits small limits and colimits (Corollary \ref{cor:pr:lcl}).

\begin{rmk}\label{rmk:dAS:Segal}
Let us explain the indication (i) $\Longrightarrow$ (ii), i.e., 
a construction of a Segal groupoid object $X_\bullet$ from 
a given $n$-geometric derived stack $\stX$.
Let $\{V_i\}_{i \in I}$ be an $n$-atlas of $\stX$.
Then we put $\stX_0 := \coprod_{i \in I} V_i$, and 
let $p: \stX_0 \to \stX$ be the natural effective epimorphism of derived stacks.
For $n \ge 1$, we set 
\[
 \stX_n := \underbrace{\stX_0 \times_{p, \stX, p} \stX_0 \times_{p, \stX, p} 
                 \cdots \times_{p, \stX, p} \stX_0}_{\text{$n$-times}}.
\]
\end{rmk}


\subsubsection{Derived algebraic spaces}
\label{sss:dAS:dAS}

Recall Fact \ref{fct:dSt:mAS=AS} which characterizes
schemes and algebraic spaces among higher Artin stacks.
Considering its simple analogue, we introduce

\begin{dfn}\label{dfn:dAS:dAS}
\begin{enumerate}[nosep]
\item
A \emph{derived scheme} $X$ (over $k$) is derived stack having an $n$-atlas 
$\{U_i\}_{i \in I}$ with some $n \in \bbZ_{\ge -1}$ such that each morphism 
$U_i \to X$ is a monomorphism in $\idSt$ (Definition \ref{dfn:dSt:epi-mono}).
We denote by $\idSch \subset \idSt$ the sub-$\infty$-category 
spanned by derived schemes.

\item 
A \emph{derived algebraic space} $\stU$ (over $k$) 
is a derived stack satisfying the following conditions.
\begin{itemize}
\item 
There exists an $n$-atlas $\{V_i\}_{i \in I}$ of $\stU$ 
for some $n \in \bbZ_{\ge -1}$
such that each morphism $V_i \to \stU$ of derived stacks is \'etale.
\item
The diagonal morphism $\stU \to \stU \times \stU$ 
is a monomorphism in $\idSt$.
\end{itemize}
We denote by $\idAS \subset \idSt$ the sub-$\infty$-category 
spanned by derived algebraic spaces.
\end{enumerate}
\end{dfn}

We have the sequence $\idAff \subset \idSch \subset \idAS \subset \idSt$
of $\infty$-categories.
Moreover we have an analogous result to Fact \ref{fct:dSt:ase}.

\begin{prp*}
A derived algebraic space is $1$-geometric.
In particular, a derived scheme is $1$-geometric.
\end{prp*}

\begin{proof}
Let $\stU$ be a derived algebraic space 
and take an $n$-atlas $\{V_i\}_{i \in I}$.
By Remark \ref{rmk:dSt:geom} (2), we see that $\stU$ is $n$-geometric.
By Fact \ref{fct:dAS:Segal} and Remark \ref{rmk:dAS:Segal},
we may assume that $\stU$ is a quotient stack of 
$(n-1)$-smooth Segal groupoid $U_{\bullet}$
with $U_0:= \coprod_{i \in I} V_i$ and $U_1 := U_0 \times_{\stU} U_0$. 
Then the morphism $d_0 \times d_1: U_1 \to U_0 \times U_0$ is a monomorphism,
where we denoted by $d_0,d_1: U_1 \to U_0$ 
the structure morphisms in $U_{\bullet}$.
In fact, the diagram
\[
 \xymatrix{
  U_1 \ar[r]^(0.4){d_0 \times d_1} \ar[d] & U_0 \times U_0 \ar[d] \\
  \stU \ar[r]_(0.4){\Delta_{\stU}} & \stU \times \stU}
\]
is a cartesian square in the $\infty$-category $\idSt$ 
(Definition \ref{dfn:ic:pb-po}).
Since the diagonal morphism $\Delta_{\stU}$ is assumed to be a monomorphism,
the morphism $d_0 \times d_1$ is also a monomorphism as claimed above.
We can also check that $\{V_i\}_{i \in I}$ is an $n$-atlas of $U_1$ 
with each $V_i \to U_1$ \'etale.
Thus $U_1$ is a derived algebraic stack.
Now note that $U_1$ is an $(n-1)$-geometric derived stack 
since $U_{\bullet}$ is an $(n-1)$-smooth Segal groupoid.
Thus $U_1$ is an $(n-1)$-geometric derived algebraic space,
A similar argument shows that 
$U_0$ is also an $(n-1)$-geometric derived algebraic spaces.
Moreover, as $U_0$ being a coproduct of affine derived schemes,
it is $0$-geometric by Lemma \ref{lem:dSt:0-geom}.
Then from $U_1 \subset U_0 \times U_0$, we see that $U_1$ is also $0$-geometric.
Thus $U_\bullet$ is a $0$-smooth Segal groupoid,
so that the quotient stack $\stU$ of $U_\bullet$ is $1$-geometric.
%
\end{proof}

Thus the following definition makes sense.

\begin{dfn*}
A morphism $f: \stU \to \stV$ in $\idAS$ is \emph{smooth}
if it is $1$-smooth in the sense of Definition \ref{dfn:dSt:geom}.
\end{dfn*}

\begin{rmk}\label{rmk:dAS:diag}
Recall the fully faithful functor $\Dex: \iSt \to \idSt$ 
(Definition \ref{dfn:dSt:trunc}).
The image of an algebraic space under $\Dex$ is obviously a derived algebraic space,
and restricting $\Dex$ to $\iAlgSp \subset \iSt$ we have a fully faithful functor 
\[
 \Dex: \iAlgSp \longto \idAS.
\]
%
%
Similarly, the image of a scheme under $\Dex$ is a derived scheme.
Thus 
we have the following diagram (compare with Remark \ref{rmk:dSt:mAS=AS}).
\begin{align*}
 \xymatrix{
    \iAff        \ar@{^{(}->}[r] \ar@{^{(}->}[d]
  & \iSch        \ar@{^{(}->}[r] \ar@{^{(}->}[d]
  & \iAlgSp      \ar@{^{(}->}[r] \ar@{^{(}->}[d]
  & \iSt_{\geom} \ar@{^{(}->}[d]^{\Dex} \\
    \idAff       \ar@{^{(}->}[r] 
  & \idSch       \ar@{^{(}->}[r] 
  & \idAS        \ar@{^{(}->}[r] 
  & \idSt_{\geom}
 }
\end{align*}
\end{rmk}

\subsection{\'Etale $\infty$-site on a derived algebraic space}
\label{ss:dAS:eis}

\subsubsection{$\infty$-sites and $\infty$-topoi}

Here we give some complementary explanation on 
$\infty$-topoi arising from $\infty$-sites. 
We will give a construction of geometric morphisms of $\infty$-sites from
a \emph{continuous} functor of $\infty$-sites.

\begin{dfn}\label{dfn:dAS:cont}
A \emph{continuous functor} $f: (\iC',\tau') \to (\iC,\tau)$ of $\infty$-sites
is a functor $\iC' \to \iC$ of $\infty$-categories 
satisfying the following two conditions.
\begin{itemize}[nosep]
\item
For every $X' \in \iC'$ and $\{X'_i \to X'\}_{i \in I} \in \Cov_{\tau'}(X')$,
the family $\{f(X'_i) \to f(X')\}_{i \in I}$ is in $\Cov_{\tau}(f(X'))$.
\item
$f$ commutes with finite limits (if they exist). 
\end{itemize}
\end{dfn}

Restating the argument in \cite[\S 2.2]{O:book}, we have

\begin{prp}[{\cite[Proposition 2.2.26, 2.2.31]{O:book}}]\label{prp:dAS:sgm}
Let $f: (\iC',\tau') \to (\iC,\tau)$ be a functor of $\infty$-sites.
\begin{enumerate}[nosep]
\item 
The composition with $f$ induces a functor 
$f_*: \iSh(\iC,\tau) \to \iSh(\iC',\tau')$ of $\infty$-topoi.

\item
If $f$ is continuous, then $f_*$ has a left adjoint 
$f^*: \iSh(\iC',\tau') \to \iSh(\iC,\tau)$.

\item
If $\iC$ and $\iC'$ admit finite limits and if $f$ is continuous and commutes 
with finite limits, then $f^*$ is left exact and the adjunction 
$f^*: \iSh(\iC',\tau') \rlto \iSh(\iC,\tau) :f_*$ determines a geometric morphism
$\iSh(\iC,\tau) \to \iSh(\iC',\tau')$.
\end{enumerate}
\end{prp}

We followed Remark \ref{rmk:is:gm} on the notation of geometric morphisms.

\begin{proof}
These statements are essentially shown in \cite[\S 2.2]{O:book},
but let us explain an outline of the proof.
The proof of (1) is standard, so we omit it.

Let us explain the construction of $f^*$ in (2).
Given a sheaf $\shG \in \iSh(\iC',\tau')$, 
we want to construct $f^* \shG \in \iSh(\iC,\tau)$.
For an object $U \in \iC$,
we denote by $\iI_{U}$ the full sub-$\infty$-category of 
the under-$\infty$-category $\iC_{U \! /}$ spanned by 
the essential image of the functor $f$.
An object of $\iI_{U}$ can be regarded as a pair $(V,\nu)$
where $V \in \iC'$ and $\nu: U \to f(V)$ is a morphism in $\iC$.
Now we define $\shF \in \iPSh(\iC)$ by 
$\shF(U) := \iclim_{(V,\nu) \in \iI_{U}} \shG(V)$.
It is immediate that $\shF$ is well-defined, 
and we define the sheaf $f^* \shG$ to be the sheafification of $\shF$.

As for the item (3), 
since the localization functor in a topological localization is left exact
\cite[Corollary 6.1.2.6]{Lur1},
the sheafification functor is left exact by Fact \ref{fct:pr:sh}.
Then the construction implies that $f^*$ is left exact.
By the same argument in the proof of \cite[Proposition 2.2.31]{O:book}, we have 
a counit transformation $f_* f_* \to \id$ (Definition \ref{dfn:ic:uc}).
Considering the evaluation at the object $(V,\id_{f(V)})$
of $\iI_{f}(V)$ for any $V \in \iSh(\iC',\tau')$,
we obtain the inverse of the counit transformation.
Thus we have the desired adjunction 
$f^*: \iSh(\iC',\tau') \rlto \iSh(\iC,\tau) :f_*$.
\end{proof}


\subsubsection{\'Etale $\infty$-site}
\label{sss:dAS:eis}

Here we introduce an $\infty$-theoretical analogue of \'etale topoi of schemes.
We continue to work over a commutative ring $k$.

\begin{dfn}\label{dfn:dAS:small-et}
Let $\stU$ be a derived algebraic space.
\begin{enumerate}[nosep]
\item
The (\emph{small}) \emph{\'etale $\infty$-site} of $\stU$ is the $\infty$-site
$\gsEt(\stU):=(\idAS_{\stU}^{\et},\et)$ consisting of 
\begin{itemize}[nosep]
\item 
The full sub-$\infty$-category $\idAS_{\stU}^{\et}$ 
of the over-$\infty$-category $\oc{\idAS}{\stU}$
spanned by \'etale morphisms $\stT \to \stU$ of derived stacks 
(Definition \ref{dfn:dSt:mor-Q}).

\item
The Grothendieck topology $\et$, called the (\emph{small}) \emph{\'etale topology},
on $\idAS_{\stU}^{\et}$ whose set $\Cov_{\et}(\stT)$ of covering sieves of 
$\stT \in \idAS_{\stU}^{\et}$ consists of families $\{\stT_i \to \stT\}_{i \in I}$
with $\coprod_{i \in I} \stT_i \to \stT$ an epimorphism of derived stacks
(Definition \ref{dfn:dSt:epi-mono}).
\end{itemize}

\item
We denote the associated $\infty$-topos (Fact \ref{fct:pr:sh}) 
by $\stU_{\et} := \iSh(\gsEt(\stU))$ 
and call it the (\emph{small}) \emph{\'etale $\infty$-topos} on $\stU$.
\end{enumerate}
\end{dfn}

The following statement is an analogue of the one 
in the \'etale topology of a scheme 
(see \cite[Example 2.2.10]{O:book} for example).
The proof is quite similar, and we omit it.

\begin{lem}\label{lem:dAS:small-et}
For a derived algebraic space $\stU$, 
let $\gsEt^{\aff}(\stU) := (\idAff_{\stU}^{\et},\et^{\aff})$ be the $\infty$-site 
consisting of 
\begin{itemize}[nosep]
\item 
The full sub-$\infty$-category $\idAff_{\stU}^{\et}$
of the over-$\infty$-category $\oc{\idAff}{\stU}$ of \emph{affine derived schemes}
over $\stU$ spanned by \'etale morphisms of derived stacks.

\item
The Grothendieck topology $\et^{\aff}$ 
where a covering sieve is set to be a covering sieve in $\et$
(Definition \ref{dfn:dAS:small-et} (1)).
\end{itemize}
Then the associated $\infty$-topos $\stU_{\et^{\aff}}$ 
is equivalent to $\stU_{\et}$.
\end{lem}

The \'etale $\infty$-topos on a derived algebraic space is functorial
in the following sense.
Let $f: \stU \to \stV$ be a morphism of derived algebraic spaces.
Taking fiber products, we have the following continuous functor 
of $\infty$-sites (Definition \ref{dfn:dAS:cont}).
\[
  f^{-1}: \gsEt(\stV) \longto \gsEt(\stU), \quad 
          \stV' \longmapsto \stV' \times_{\stV} \stU .
\]
Composition with $f^{-1}$ gives the following functor of $\infty$-topoi.
\[
 f^{\et}_*: \stU_{\et} \longto \stV_{\et}, \quad
 (f^{\et}_* \shF)(V) :=  \shF(f^{-1}(V)).
\]
Applying Proposition \ref{prp:dAS:sgm} to the present situation,
we have

\begin{lem}\label{lem:dAS:et-geom}
The functor $f^{\et}_*: \stU_{\et} \to \stV_{\et}$
has a left exact left adjoint $f_{\et}^{-1}$.
Thus we have a geometric morphism $f_{\et}: \stU_{\et} \to \stV_{\et}$ of 
$\infty$-topoi (Definition \ref{dfn:is:gm}) corresponding to the adjunction
\[
 f_{\et}^{-1}: \stV_{\et} \adjunc \stU_{\et} :f^{\et}_*.
\]
\end{lem}

Let us now introduce some basic notions on derived algebraic spaces.
These are simple analogue of the corresponding notions in the scheme theory.
For our definitions, morphisms of schemes and algebraic spaces 
will be replaced by \emph{geometric} morphisms of the \'etale $\infty$-topoi
on derived algebraic spaces.
In this part, $\stU$ and $\stV$ denote derived algebraic spaces over $k$,
and $f: \stU \to \stV$ denotes a morphism between them.

We begin with the definition of open immersions
by applying the discussion in \ref{ss:is:oci} for the summary.
Our definition is an analogue of \cite[Definition 9.5]{Lur7}.

\begin{dfn}\label{dfn:dAS:oi}
$f: \stU \to \stV$ is an \emph{open immersion}
if the geometric morphism $f_{\et}: \stU_{\et} \to \stV_{\et}$
of the \'etale $\infty$-topoi 
is an open geometric immersion (Definition \ref{dfn:is:oi}).
\end{dfn}

Let us spell out this definition differently.
For an $\infty$-topos $\tT$, we denote by $\Sub(\one_{\tT})$
the set of equivalence classes of $(-1)$-truncated objects in $\tT$
(Definition \ref{dfn:it:sub}).
Applying this notation to the \'etale $\infty$-topos $\tT=\stU_{\et}$,
each $U \in \Sub(\one_{\stU_{\et}})$ is represented by
an \'etale morphism $j_U: U' \to \stU$ of derived algebraic spaces
where $U'$ is an affine derived scheme.

Then we have that $f: \stU \to \stV$ is an open immersion
if there is a $V \in \Sub(\one_{\stV_{\et}})$ 
represented by an \'etale morphism $j_V: V' \to \stV$ such that 
$f$ is equivalent to the composition $\stU \xr{f'} V \xr{j_V} \stV$
with $f'$ an equivalence of derived algebraic spaces.

Next we introduce closed immersions 
by adapting the general notion in \cite[\S 7.2.3]{Lur1}.
See \ref{ss:is:oci} for the summary.
Let us also refer \cite[\S 4]{Lur9} for the relevant discussion 
in the spectral algebraic geometry.

\begin{dfn*}
$f: \stU \to \stV$ is a \emph{closed immersion}
if the geometric morphism $f_{\et}: \stU_{\et} \to \stV_{\et}$ 
of the \'etale $\infty$-topoi 
is a closed geometric immersion (Definition \ref{dfn:is:ci}).
\end{dfn*}

Following \cite[Definition 1.4.11]{Lur9}, we introduce

\begin{dfn}\label{dfn:dAS:sep}
\begin{enumerate}[nosep]
\item 
$f: \stU \to \stV$ is \emph{strictly separated} if 
the diagonal morphism $\stU \to \stU \times_{f,\stV,f} \stU$ is a closed immersion.

\item
$\stU$ is \emph{separated}
if the structure morphism $\stU \to \dSpec k$ is strictly separated.
\end{enumerate}
\end{dfn}

We now introduce \emph{quasi-compact} derived algebraic spaces.
Recall the notion of quasi-compact $\infty$-topoi (Definition \ref{dfn:pr:qcpt}).

\begin{dfn}\label{dfn:dAS:qcpt}
A derived algebraic space $\stU$ is \emph{quasi-compact}
if the $\infty$-topoi $\stU_{\et}$ is quasi-compact.
\end{dfn}

Next we introduce \emph{quasi-compact morphisms} of derived algebraic spaces.
For that, we need the notion of quasi-compact morphisms between $\infty$-topoi.
Recall Definition \ref{dfn:pr:qcpt} of quasi-compactness 
for objects in $\infty$-topoi.

\begin{dfn}\label{dfn:dAS:qcpt-top}
A geometric morphism $f: \tT \to \tT'$ of $\infty$-topoi is called 
a \emph{quasi-compact morphism} if for any quasi-compact object $U \in \tT'$ 
the object $f^* U \in \tT$ is quasi-compact.
\end{dfn}

\begin{rmk*}
Starting with quasi-compactness, one can introduce by induction 
the notion of $n$-coherence of $\infty$-topoi \cite[\S3]{Lur7}.
Using the $n$-coherence, Lurie introduced in  \cite[\S 1.4]{Lur8}
the notion of $n$-quasi-compactness for spectral schemes, 
spectral Deligne-Mumford stacks and morphisms between them.
Our definition of quasi-compact morphism is an adaptation of 
this $n$-quasi-compactness to the case $n=0$.
\end{rmk*}

\begin{dfn}\label{dfn:dAS:qcpt-mor}
A morphism $f: \stU \to \stV$ of derived algebraic spaces is \emph{quasi-compact}
if the geometric morphism $f_{\et}: \stU_{\et} \to \stV_{\et}$ 
of $\infty$-topoi (Lemma \ref{lem:dAS:et-geom}) is a quasi-compact morphism
(Definition \ref{dfn:dAS:qcpt-top}).
\end{dfn}

\begin{rmk*}
We have already introduced 
the notion of quasi-compact morphisms of derived stacks
in Definition \ref{dfn:dSt:qcpt}.
By a routine one can show that these two notions are equivalent.
\end{rmk*}

Let us now turn to the notion of \emph{quasi-separated} derived algebraic spaces.
It is a direct analogue of the notion of 
quasi-separated schemes and quasi-separated algebraic spaces.
Let us also refer \cite[Definition 1.3.1]{Lur12} for a relevant notion 
for spectral Deligne-Mumford stacks.

\begin{dfn}\label{dfn:dAS:qsep}
\begin{enumerate}[nosep]
\item 
$f$ is \emph{quasi-separated} if the diagonal morphism 
$\stU \to \stU \times_{f,\stV,f} \stU$ is quasi-compact 
(Definition \ref{dfn:dAS:qcpt-mor}).

\item
$\stU$ is \emph{quasi-separated}
if the structure morphism $\stU \to \dSpec k$ is quasi-separated 
in the sense of (1).
\end{enumerate}
\end{dfn}

Finally we introduce proper morphisms.
Our definition is an analogue of the strongly proper morphism 
in spectral algebraic geometry \cite[\S 3]{Lur12}.

Recall the truncation functor $\Trc: \Ho \idSt \to \Ho \iSt$ 
(Definition \ref{dfn:dSt:trunc}).
For a derived algebraic space $\stU$,
the truncation $\Trc \stU$ is an algebraic space.

\begin{dfn}\label{dfn:dAS:prp}
\begin{enumerate}
\item 
A morphism $f: \stU \to \stV$ of derived algebraic spaces is \emph{proper}
if the corresponding morphism $\Trc \stU \to \Trc \stV$
of algebraic spaces is proper (Definition \ref{dfn:Cl:AS:prp}).

\item
A derived algebraic space $\stU$ is \emph{proper}
if the structure morphism $\stU \to \dSpec k$ is proper
in the sense of (1).
\end{enumerate}
\end{dfn}

One can check the ordinary properties of proper morphisms,
such as stable under composition and base change, hold in $\idAS$.

\begin{rmk}\label{rmk:dAS:table}
The notions on morphisms of derived algebraic spaces given above
and those of ordinary algebraic spaces (\S \ref{ss:Cl:AS}) are compatible
under the functor $\Dex$ in Remark \ref{rmk:dAS:diag}.
\begin{table}[htbp]
\begin{tabular}{c|c|c}
property of morphisms & derived algebraic spaces & algebraic spaces \\
\hline
separated       & Definition \ref{dfn:dAS:sep}  & Definition \ref{dfn:Cl:AS:cim}
\\ 
quasi-separated & Definition \ref{dfn:dAS:qsep} & Definition \ref{dfn:Cl:AS:qsep}
\\ 
proper          & Definition \ref{dfn:dAS:prp} & Definition \ref{dfn:Cl:AS:prp} 
\end{tabular}
\caption{Morphisms between derived and algebraic spaces}
\end{table}
\end{rmk}

\subsection{\'Etale sheaves of rings and modules}

In this subsection we collect notations for sheaves of commutative rings 
and modules on derived algebraic spaces in the \'etale topology.
We work over the base commutative ring $k$ as before.

\subsubsection{\'Etale structure sheaves}
\label{sss:dAS:str-sh}

This part will not be used in the later sections.
We record it for completeness of our presentation.

\begin{dfn*}
For a derived algebraic space $\stU$, 
the (\emph{small}) \emph{\'etale structure sheaf $\shO_{\stU}$ of $\stU$} 
is an object of $\iShv_{\isCom_{k}}(\stU_{\et})$ determined by 
\[
 \shO_{\stU}(U) := A \in \isCom_{k} \ 
 \text{ for } \ U = \dSpec A \in \stU_{\et}.
\] 
Here $U \in \stU_{\et}$ means that $U  \in \idAS^{\et}_{\stU}$ and 
it is identified with an object of $\stU_{\et}$
by the Yoneda embedding $\yon(U) \in \iSh(\idAS^{\et}_{\stU},\et)=\stU_{\et}$.
\end{dfn*}

Thus we obtain a ringed $\infty$-topos $(\stU_{\et},\shO_{\stU})$.
It is intimately related to \emph{spectral algebraic spaces} 
in Lurie's spectral algebraic geometry \cite{Lur7,Lur8}.
We also have the stable $\infty$-category
$\iMod_{\shO_{\stU}}(\iShv_{\isCom_k}(\tT))$
of stable \'etale sheaves of $\shO$-modules over $\stU$.

For a morphism $f: \stU \to \stV$ of derived algebraic spaces,
we can construct a functor 
$(f_{\et},f^{\sharp}): (\stU_{\et},\shO_{\stU}) \to (\stV_{\et},\shO_{\stV})$
of ringed $\infty$-topoi,
where $f_{\et}: \stU_{\et} \to \stV_{\et}$ is the geometric morphism 
in Lemma \ref{lem:dAS:et-geom},
and $f^{\sharp}: \shO_{\stV} \to f_* \shO_{\stU}$ is a morphism 
in $\iShv_{\isCom_k}(\stV_{\et})$.
We can also construct a geometric morphism corresponding to the adjunction
$f_{\et}^*: \iMod_{\shO_{\stV}}(\stV_{\et}) \rightleftarrows
      \iMod_{\shO_{\stU}}(\stU_{\et}) :f^{\et}_*$
of functors between stable $\infty$-categories,


\subsubsection{\'Etale sheaf of commutative rings}

Recall that in \S \ref{ss:is:crm} we introduced 
sheaves of commutative rings over an $\infty$-topos $\tT$.
Let us set $\tT = \stU_{\et}$, the \'etale $\infty$-topos of 
an a derived algebraic space $\stU$.
Thus, for a sheaf $\shA \in \iShv_{\iCom}(\stU_{\et})$, 
we have a ringed $\infty$-topos $(\stU_{\et},\shA)$.
We call such an $\shA$ an \emph{\'etale sheaf of commutative rings on $\stU$}.

Let us apply the notations in \S \ref{sss:is:smrf} to the present situation.
We call 
\[
 \iMod_{\shA}(\stU_{\et}) := \iMod_{\shA}(\iShv_{\iAb}(\stU_{\et})),
\]
the \emph{$\infty$-category of \'etale sheaves of $\shA$-modules on $\stU$}. 
It is equipped with internal Hom functor $\shHom_{\shA}$
and the tensor functor $\otimes_{\shA}$.
We also denote by
$\Mod_{\shA}(\stU_{\et}) := \Ho \iMod_{\shA}(\stU_{\et})$
its homotopy category,
which is a Grothendieck abelian category (Proposition \ref{prp:is:G}).

Next recall the notations in \S \ref{sss:is:dic}.
We denote by 
\[
 \iMod^{\stab}_{\shA}(\stU_{\et})  := 
 \iMod^{\stab}_{\shA}(\iShv_{\iSp}(\stU_{\et}))
\]
the $\infty$-category of sheaves of stable $\shA$-modules on $\stU_{\et}$.
It is stable and equipped with a $t$-structure.
It also has the internal Hom functor $\shHom_{\shA}$
and the tensor functor $\otimes_{\shA}$.
By Proposition \ref{prp:is:Da=Stab}, 
we also have a $t$-exact equivalence 
\[
 \iMod^{\stab}_{\shA}(\stU_{\et}) \simeq \iDa(\Mod_{\shA}(\stU_{\et}))
\]
of stable $\infty$-categories,
where the right hand side is the derived $\infty$-category 
of the Grothendieck abelian category $\Mod_{\shA}(\stU_{\et})$.
Hereafter we mainly discuss in terms of the derived $\infty$-category,
and use the following notation.

\begin{ntn*}
For a derived algebraic space $\stU$ and 
an \'etale sheaf $\shA$ of commutative rings on $\stU$, we set 
\[
 \iDu{*}(\stU_{\et},\shA) := \iDu{*}(\Mod_{\shA}(\stU_{\et}))
 \quad * \in \{\emptyset,+,-,b\}
\]
and call it the (resp.\ left bounded, resp.\ right bounded, resp.\ bounded)
\emph{derived $\infty$-category of \'etale sheaves of $\shA$-modules on $\stU$}.
For a commutative $\Lambda$, we denote by
\[
 \iDu{*}(\stU_{\et},\Lambda) := \iDu{*}(\Mod_{\Lambda}(\stU_{\et}))
 \quad * \in \{\emptyset,+,-,b\}
\]
for the derived $\infty$-category of \'etale sheaves of 
$\Lambda$-modules on $\stU$ (with some bound condition),
where we denote the constant sheaf by the same symbol $\Lambda$.
\end{ntn*}

Now assume that we are given a morphism $f: \stU \to \stV$
of derived algebraic spaces.
Then by Lemma \ref{lem:dAS:et-geom} we have a geometric morphism 
$f_{\et}: \stU_{\et} \to \stV_{\et}$,
so that the argument in \S \ref{ss:is:crm} 
gives rise to the direct image functors 
\[
 f^{\et}_{*}:
 \iMod_{\shA}(\stU_{\et}) \longto \iMod_{\shA}(\stV_{\et}), \quad
 \iDa(\stU_{\et}, \shA)  \longto \iDa(\stV_{\et}, \shA)
\]
and inverse image functors 
\[
 f_{\et}^*:
 \iMod_{\shA}(\stV_{\et}) \longto \iMod_{\shA}(\stU_{\et}), \quad 
 \iDa(\stV_{\et}, \shA)  \longto \iDa(\stU_{\et}, \shA).
\]

\subsubsection{Proper base change}

Recall Definition \ref{dfn:dAS:prp} of proper morphisms
of derived algebraic spaces.
It enable us to translate to derived settings 
the proper base change theorem in the ordinary scheme theory.
In the following sections 
we will mainly discuss constant sheaves of commutative rings.
So, for a commutative ring $A$,
let us denote by the \emph{constant sheaf} valued in $A$ on $\stU_{\et}$ 
(Definition \ref{dfn:pr:cst-sh}) by the same symbol $A$.

\begin{lem}\label{lem:dAS:pbc1}
Let $\Lambda$ be a torsion ring and
\[
 \xymatrix{
  \stU' \ar[r]^{g'} \ar[d]_{f'} & \stU \ar[d]^{f} \\ \stV' \ar[r]_{g} & \stV}
\]
be a cartesian square in $\idAS$ with $f$ proper.
Then the base change morphism (Definition \ref{dfn:is:bc})
\[
 g_{\et}^* \, f^{\et}_* \longto {f'}^{\et}_* \, {g'}_{\et}^*
\]
of functors 
$\iMod_\Lambda(\stU'_{\et}) \to \iMod_\Lambda(\stV_{\et})$
is an equivalence.
\end{lem}

\begin{proof} 
By the equivalence $\stU_{\et} \simeq \stU_{\et^{\aff}}$ 
(Lemma \ref{lem:dAS:small-et})
and Definition \ref{dfn:dAS:prp} of proper morphisms,
we can reduce the problem to the proper base change 
for modules over sheaves of torsion rings on schemes 
\cite[XII, Th\'eor\`eme 5.1]{SGA4}.
\end{proof}

The proper base change naturally extends to 
the derived $\infty$-categories, and we have

\begin{lem}\label{lem:dAS:proper-bc}
Under the same assumption with Lemma  \ref{lem:dAS:pbc1},
the base change morphism
$ g_{\et}^* \, f^{\et}_* \to {f'}^{\et}_* \, {g'}_{\et}^*$
of functors between derived $\infty$-categories 
$\iDa(\stU'_{\et},\Lambda) \to \iDa(\stV_{\et},\Lambda)$
is an equivalence.
\end{lem}

\subsection{Direct image functor with proper support}
\label{ss:dAS:dip}

In \S \ref{ss:dAS:dual} we will introduce dualizing complexes 
for derived algebraic spaces.
For that, we need the shriek functors $f_!$ and $f^!$,
which will be defined in this and the next subsections.
Our argument is a simple analogue of that for schemes \cite[XVII]{SGA4}.
We continue to work over a commutative ring $k$.

\subsubsection{Open immersions}
\label{sss:dAS:oi}

Let $j: \stU \to \stV$ be an open immersion of derived algebraic spaces 
(Definition \ref{dfn:dAS:oi}).
Recall that we defined an open immersion using the more general definition of 
an open geometric immersion of $\infty$-topoi (Definition \ref{dfn:is:oi}).
By the argument in \S \ref{ss:is:oci}, associated to $j$ 
we have two geometric morphisms of $\infty$-topoi:
\[
 j_!: \stU_{\et}  \adjunc \stV_{\et} :j^*, \quad 
 j^*: \stV_{\et}  \adjunc \stU_{\et} :j_*.
\]

Let $(\stU_{\et},\shA)$ be a ringed $\infty$-topos, where $\stU_{\et}$ is 
the \'etale $\infty$-topos of the derived algebraic space $\stU$.
Then we have another ringed $\infty$-topos $(\stV_{\et},j^*\shA)$.
The above geometric morphisms induce the adjunctions 
\[
 j_!: \iMod_{   \shA}(\stU_{\et})  \adjunc \iMod_{j^*\shA}(\stV_{\et}) :j^*, \quad 
 j^*: \iMod_{j^*\shA}(\stV_{\et})  \adjunc \iMod_{   \shA}(\stU_{\et}) :j_*.
\]
We call $j_!$ the \emph{extension by zero}, and $j^*$ the \emph{restriction}.
We also have a morphism 
\[
 j_! \longto j_*
\]
in $\iFun(\iMod_{\shA}(\stU_{\et}),\iMod_{j^*\shA}(\stV_{\et}))$.

\subsubsection{Compactifiable morphism}

Recall that an $S$-morphism $f: X \to Y$ of schemes over a base scheme $S$ 
is \emph{$S$-compactifiable} \cite[XVII, Definition 3.2.1]{SGA4}
if there exists an $S$-scheme $P$ which is proper over $S$
and a factorization $f = (X \xr{j} P \times_S Y \xr{p_Y}Y)$ 
with $j$ quasi-finite and separated.
If moreover $S=Y$, then $f$ is called \emph{compactifiable}.

Let us define the corresponding notion for algebraic spaces as follows:

\begin{dfn*}
An $S$-morphism $X \to Y$ of algebraic spaces over a base scheme $S$
is \emph{$S$-compactifiable}
if there exists an algebraic space $P$ over $S$
which is proper  (Definition \ref{dfn:Cl:AS:prp}) 
and a factorization $f= (X \xr{j} P \times_S Y \xr{p_Y}Y)$ 
with $j$ quasi-compact (Definition \ref{dfn:Cl:AS:q-cpt}), 
locally quasi-finite (Definition \ref{dfn:Cl:AS:P})
and separated (Definition \ref{dfn:Cl:AS:cim} (2)).
\end{dfn*}

Recalling Definition \ref{dfn:dAS:prp} of proper morphisms 
of derived algebraic spaces, we introduce a derived analogue.
Let us write the base commutative ring $k$ explicitly for a while.

\begin{dfn*}
A morphism $f: \stU \to \stV$ of derived algebraic spaces 
over $k$ is \emph{compactifiable}
if there exists a proper derived algebraic space $\stP$ 
(Definition \ref{dfn:dAS:prp}) and a factorization
$f=  (\stU \xr{j} \stP \times_{\dSpec k} \stV \xr{p_{\stV}}\stV)$ 
such that the truncation
$\Trc f  =  (\Trc \stU \xr{\Trc j} \Trc \stU \times_{\Spec k} \Trc \stV 
 \xr{p_{\Trc \stV}} \Trc \stV)$ 
makes $\Trc f$ a compactifiable morphism of algebraic spaces.
\end{dfn*}

Next we introduce an analogue of the category $(S)$ 
of $S$-compactifiable morphisms \cite[XVII.3.2]{SGA4}
in the context of derived algebraic spaces.

\begin{dfn*}
We define $\idAS^{\cpt}_k$ to be the sub-$\infty$-category of $\idAS_k$
whose objects are derived algebraic spaces $\stU$ over $k$
whose truncations $\Trc \stU$ are quasi-compact and quasi-separated 
algebraic spaces over $\Spec k$
and whose morphisms are compactifiable morphisms.
\end{dfn*}

Then it is natural to set 

\begin{dfn*}
Let $f: \stU \to \stV$ be a morphism in $\idAS^{\cpt}_k$.
A \emph{compactification} of $f$ is a triangle
\[
 \xymatrix{ 
  \stU \ar@{^{(}->}[r]^i \ar[d]_{f} & 
  \ol{\stU} \ar[ld]^{\ol{f}} \\ \stV}
\] 
in $\idAS_k$ with $i$ an open immersion and $\ol{f}$ proper.
\end{dfn*}

Since our definitions for derived algebraic spaces 
refer only to the truncated algebraic spaces,
the argument in \cite[XVII Proposition 3.23]{SGA4} works as it is,
and we have

\begin{lem*}
Any morphism $f$ in $\idAS^{\cpt}_k$ has a compactification.
\end{lem*}

Note that an open immersion of derived algebraic spaces 
lives in $\idAS^{\cpt}_k$.
Recall also that we have base change theorems
for open immersions (Lemma \ref{lem:is:j_!-bc})
and for proper morphisms (Lemma \ref{lem:dAS:proper-bc}).
Then by the construction in \cite[XVII.3.3, 5.1]{SGA4} we have

\begin{lem}
Let $\Lambda$ be a torsion commutative ring.
For each morphism $f: \stU \to \stV$ in the $\infty$-category $\idAS^{\cpt}_k$, 
we have a $t$-exact functor
\[ 
 f_!: \iDa(\stU_{\et},\Lambda) \longto \iDa(\stV_{\et},\Lambda)
\]
which extends $i_!: \iMod(\stU_{\et},\Lambda) \to \iMod(\stV_{\et},\Lambda)$
for open immersion $i: \stU \to  \stV$ in $\idAS^{\cpt}_k$.
\end{lem}

\subsection{Extraordinary inverse image functor}
\label{ss:dAS:extra}

We continue to use the notations in the previous \S \ref{ss:dAS:dip}.
By \cite[XVIII Th\'eor\`eme 3.1.4]{SGA4} we have

\begin{lem*}
Let $f: \stU \to \stV$ be a morphism in $\idAS^{\cpt}_k$,
and $\Lambda$ be a torsion commutative ring.
Then the functor 
$f_!: \iDa(\stU_{\et},\Lambda) \to \iDa(\stV_{\et},\Lambda)$
admits a right adjoint $t$-exact functor 
\[ 
 f^!: \iDp(\stV_{\et},\Lambda) \longto \iDp(\stU_{\et},\Lambda).
\]
We call it the \emph{extraordinary inverse image functor}.
\end{lem*}

\subsection{Dualizing objects on derived algebraic spaces}
\label{ss:dAS:dual}

Having introduced functors between sheaves of modules,
we can now discuss the dualizing complexes for derived algebraic spaces.
Let us apologize that our discussion is not of full generality:
we only discuss a rather restricted situation over the base ring $k$
(see Assumption \ref{asp:dAS:kL} below).
It might be possible to take a more general base scheme as in \cite{LO1,LO2},
but we will not pursue this point.

We will define dualizing objects 
for derived algebraic spaces by some gluing argument,
following the discussion for algebraic spaces in \cite[3.1]{LO1}.

Let $\stW$ be a derived algebraic space over $k$.
By Lemma \ref{lem:dAS:small-et},
we can replace the \'etale $\infty$-site $\gsEt(\stW)=(\idAS^{\et}_{\stW},\et)$ 
by the $\infty$-site $(\idAff^{\et}_{\stW},\et^{\aff})$ 
consisting only of \'etale morphisms $U \to \stW$ from 
\emph{affine derived schemes} $U$, and have an equivalence
\[
 \rst{\stW_{\et}}{U} \simeq U_{\et^{\aff}}.
\]
Here the left hand side is the localized $\infty$-topos (Fact \ref{fct:pr:ovc}).

Hereafter we assume the following conditions.

\begin{asp}\label{asp:dAS:kL}
\begin{itemize}[nosep]
\item 
The base ring $k$ is a finite field or a separably closed field.
\item
The commutative ring $\Lambda$ is a torsion noetherian ring 
annihilated by an integer invertible in $k$.
\item
The derived algebraic space $\stW$ is separated (Definition \ref{dfn:dAS:sep})
and of finite presentation as a geometric derived stack 
(Definition \ref{dfn:dSt:fp}).
\end{itemize}
\end{asp}

Then we can further replace the underlying-$\infty$-category $\idAff^{\et}_{\stW}$ 
of the \'etale $\infty$-site by the full sub-$\infty$-category 
$\idAff^{\et,\fp}_{\stW}$ spanned by affine derived schemes 
of finite presentation (Definition \ref{dfn:dSt:fin-pr}).
By this observation, we regard 
\[
 U_{\et^{\aff}} = \iSh(\idAff^{\et,\fp}_{\stW},\et)
\]
in the following discussion.

A complex of $k$-module can be regarded as an object 
of the derived $\infty$-category $\iDa((\dSpec k)_{\et^{\aff}},\Lambda)$ 
of \'etale sheaves $\Lambda$-modules over $\dSpec k$.
Here $\dSpec k$ is seen as a derived algebraic space.
We fine the dualizing complex over $\dSpec k$ to be
\[
 \Omega_k := \Lambda \in \iDa((\dSpec k)_{\et^{\aff}},\Lambda).
\]
Recalling the extraordinary inverse image functor $\mu^!$ 
in \S \ref{ss:dAS:extra}, we introduce 

\begin{dfn*}
Let $u: U \to \stW$ be an object of $\idAff_{\stW}^{\et}$,
i.e., an \'etale morphism from an affine derived scheme $U$ over $k$.
Let also $\mu: U \to \dSpec k$ be the structure morphism.
We define the \emph{relative dualizing object} $\Omega_{u}$ to be 
\[ 
 \Omega_{u} := \mu^! \Omega_k \in \iDa(U_{\et^{\aff}},\Lambda).
\]
\end{dfn*}

\begin{lem}
The above construction 
$(u: U \to \stW) \mapsto \Omega_u$ is functorial in 
the $\infty$-topos $\stW_{\et^{\aff}}$. 
In other words, for any morphism $f: U \to V$ in $\stW_{\et^{\aff}}$,
we have a functorial isomorphism $f^* \Omega_v \simeq \Omega_u$.
\end{lem}

\begin{proof}
In the square 
\[
 \xymatrix@C=1em{
  U \ar[rr]^{f} \ar[rd]^u \ar[rdd]_{\mu} & & V \ar[ld]_v \ar[ldd]^{\nu} \\ 
  & \stW \ar[d]^(0.4){w} \\ & \dSpec k
 }
\]
$f$ is \'etale since $u$ and $v$ are so.
Thus we have the inverse image functor 
$f^*: \iDa(V_{\et^{\aff}},\Lambda) \to \iDa(U_{\et^{\aff}},\Lambda)$,
which is the desired one.
\end{proof}

Recall the restriction of sheaves (Notation \ref{ntn:is:rst}).

\begin{prp}
There exists an object $\Omega_w \in \iDa(\stW_{\et},\Lambda)$,
uniquely up to contractible ambiguity,
such that $\rst{\Omega_w}{U} = \Omega_u$.
\end{prp}

\begin{proof}
By the discussion in \cite[Th.\ finitude, Th\'eor\`eme 4.3]{SGA4.5}
we have $\shHom_{\Lambda}(\Omega_u,\Omega_u)=\Lambda$, which implies 
$\Ext^i_{\Lambda}(\Omega_u,\Omega_u)=0$ 
for any $i<0$.
Then the gluing lemma (Fact \ref{fct:is:glue}) gives the consequence.
\end{proof}

The object $\Omega_w$ satisfies the following properties.

\begin{lem}\label{lem:dAS:dual}
\begin{enumerate}[nosep]
\item 
$\Omega_w$ is of finite injective dimension.
In other words, $\pi_n \Omega_w =0$ for $n \gg 0$.

\item
For every $\shM \in \iDa(\stW_{\et},\Lambda)$, the canonical map
\[
 \Map_{\iDa(\stW_{\et},\Lambda)}(\shM,\Lambda) \longto 
 \Map_{\iDa(\stW_{\et},\Lambda)}(\shM \otimes_{\Lambda}\Omega_w,\Omega_w) 
\]
is an equivalence in $\topH$.
\end{enumerate}
\end{lem}

\begin{proof}
(1) is the consequence of Assumption \ref{asp:dAS:kL}.
(2) is by construction.
\end{proof}

Following the terminology in \cite[Definition 4.2.5]{Lur14}, we name

\begin{ntn}\label{ntn:dAS:dual}
We call $\Omega_w \in \iDa(\stW_{\et},\Lambda)$ the \emph{dualizing object}. 
\end{ntn}

In the later \S \ref{ss:LE:dual} 
we will discuss lisse-\'etale sheaves on derived stacks.
For that, we need functionality of dualizing objects 
with respect to smooth morphisms.
Now the Tate twist comes into play 
as in the cases of schemes and of algebraic spaces.
Let us set a notation for the Tate twist.

\begin{ntn}\label{ntn:dAS:Tate}
Under Assumption \ref{asp:dAS:kL} on $k$ and $\Lambda$, 
for $\shM \in \iDa(\stW_{\et},\Lambda)$ and $d \in \bbZ$, 
we denote by $\shM(d)$ the $d$-th Tate twist and set
\[
 \shM \tsh{d} := \shM(d)[2d].
\]
\end{ntn}

%

Now the proof of \cite[3.1.2 Lemma]{LO1} works for derived algebraic spaces,
and we have 

\begin{lem}\label{lem:dAS:3.1.2}
Let $\stW_1$ and $\stW_2$ be derived algebraic spaces over $k$ 
which are separated (Definition \ref{dfn:dAS:sep}) 
and finitely presented (Definition \ref{dfn:dSt:mor-Q}),
and let $f: \stW_1 \to \stW_2$ be a smooth morphism of relative dimension $d$ 
(Definition \ref{dfn:dSt:reldim}).
We denote by $\Omega_i$ the dualizing objects of $\stW_i$ ($i=1,2$).
Then we have an equivalence
\[
 f^* \Omega_2 \simeq \Omega_1 \tsh{-d}.
\]
\end{lem}

\section{Lisse-\'etale sheaves on derived stacks}
\label{s:LE}

The purpose of this section is to introduce the lisse-\'etale site 
on a derived stack of certain type and build a theory of 
constructible sheaves with finite coefficient on the lisse-\'etale site.

Let us recall the situation in the ordinary setting.
For an algebraic stack $X$, we consider the lisse-\'etale site 
whose underlying category consists of algebraic spaces smooth over $X$
and whose Grothendieck topology is given by the \'etale coverings.
In \cite{LM} the theory of derived category of constructible sheaves 
on this lisse-\'etale site is developed.
As explained in \cite{O}, some arguments in \cite{LM} are not correct 
due to the non-functoriality of the lisse-\'etale site.
Using cohomological descent those errors are fixed in \cite{O}.
Based on this correction, the theory of derived functors of 
$\ol{\bbQ}_{\ell}$-sheaves on algebraic stacks 
is constructed in \cite{LO1, LO2}.

As in the previous \S \ref{s:dAS}, 
we work over a fixed base commutative ring $k$
unless otherwise stated, and suppress the notation of $k$.

\subsection{The lisse-\'etale site}
\label{ss:LE:lis-et}

The lisse-\'etale site for algebraic stacks 
is introduced in \cite[Chap.\ 12]{LM} 
in order to build a reasonable theory of sheaves on algebraic stacks.
Its definition is a mixture of \'etale and smooth morphisms,
reflecting Definition \ref{dfn:Cl:AlgSt} of algebraic stacks 
which uses both \'etale and smooth morphisms. 
Our definition of the lisse-\'etale $\infty$-site for a derived stack
will be a simple analogue of this lisse-\'etale site.

We begin with somewhat similar argument to \S \ref{ss:dAS:eis}.
For a derived stack $\stX$, 
we have the over-$\infty$-category $\oc{\idSt}{\stX}$,
which can be regarded as the $\infty$-category of pairs $(\stY,y)$
consisting of $\stY \in \idSt$ and $y: \stY \to \stX$ a morphism in $\idSt$.
We denote by $\oc{\idAS}{\stX}$ the full sub-$\infty$-category of 
$\oc{\idSt}{\stX}$ spanned by those pairs $(\stU,u)$
with $\stU$ a derived algebraic space and $u: \stU \to \stX$ a morphism in $\idSt$.

\begin{dfn*}
Let $\stX$ be a geometric derived stack.
\begin{enumerate}[nosep]
\item
The \emph{lisse-\'etale $\infty$-site} of $\stX$ is the $\infty$-site
$\gsLE(\stX):=(\idAS_{\stX}^{\lis},\txle)$ consisting of 
\begin{itemize}[nosep]
\item 
The full sub-$\infty$-category $\idAS_{\stX}^{\lis}$ of $\oc{\idAS}{\stX}$
spanned those $(\stU,u)$ with $u: \stU \to \stX$ 
a smooth morphism of derived stacks (Definition \ref{dfn:dSt:mor-Q}).

\item
The Grothendieck topology $\txle$, called the \emph{lisse-\'etale topology},
on $\idAS_{\stU}^{\lis}$ whose set $\Cov_{\txle}(\stU)$ of covering sieves of 
$\stU \in \idAS_{\stU}^{\lis}$ consists of families 
$\{\stU_i \to \stU\}_{i \in I}$
such that each $\stU_i \to \stU$ is an \'etale morphism derived stacks
(Definition \ref{dfn:dSt:mor-Q}),
and the induced $\coprod_{i \in I} \stU_i \to \stU$ 
is an epimorphism of derived stacks (Definition \ref{dfn:dSt:epi-mono}).
\end{itemize}

\item
We denote the associated $\infty$-topos (Fact \ref{fct:pr:sh}) 
by $\stX_{\txle} := \iSh(\gsLE(\stX))$ 
and call it the \emph{lisse-\'etale $\infty$-topos} on $\stX$.
\end{enumerate}
\end{dfn*}

As in the non-derived case \cite[Lemma (12.1.2)]{LM} we have 

\begin{lem*}
$\stX_{\txle}$ is equivalent to the $\infty$-topos 
arising from the $\infty$-site $(\idAS_{\stX}^{\lis},\txle')$
where $\txle'$ is the Grothendieck topology whose covering sheaves 
consist only of \emph{finite} families of \'etale morphisms.
\end{lem*}

\begin{proof}
The same argument as in the proof of \cite[Lemma (12.1.2)]{LM} works 
since $\idSt$ is quasi-compact (Fact \ref{fct:dSt:et:qcpt}).
\end{proof}

We also have the following obvious statement.

\begin{lem}\label{lem:LE:et-le}
For a derived algebraic space $\stU$, 
the identity functor 
$\idAS_{\stU}^{\et} \ni (\stU',u') \mapsto (\stU',u') \in \idAS_{\stU}^{\txle}$ 
gives a continuous functor $\gsEt(\stU) \to \gsLE(\stU)$ of $\infty$-sites.
It induces an equivalence of $\infty$-topoi
\[
 \ve: \stU_{\txle} \longto \stU_{\et}.
\]
\end{lem}

We have a similar statement to Lemma \ref{lem:dAS:small-et}.

\begin{lem*}
Let $\stX$ be a geometric derived stack, and 
$\gsLE^{\aff}(\stX):=(\idAff^{\et}_{\stX},\txle^{\aff})$ 
be the $\infty$-site consisting of 
\begin{itemize}[nosep]
\item
The full-$\infty$-subcategory $\idAff^{\et}_{\stX}$
of $\oc{\idAff}{\stX}$ 
spanned by affine derived $\stX$-schemes $(U,u)$ 
with $u: U \to \stX$ a smooth morphism of derived stacks.
Here $\oc{\idAff}{\stX}$ denotes the full sub-$\infty$-category of 
$\oc{\idSt}{\stX}$ spanned by affine derived schemes over $\stX$.
\item
The Grothendieck topology $\txle^{\aff}$
where a covering sieve is defined to be a covering sieve in $\et$ 
on $\idAS_{\stX}^{\lis}$.
\end{itemize} 
Then the associated $\infty$-topos
$\stX_{\txle^{\aff}} := \iSh(\gsLE^{\aff}(\stX))$
is equivalent to $\stX_{\txle}$
\end{lem*}


We have the following relationship between the lisse-\'etale $\infty$-topos 
introduced above and the lisse-\'etale topos on an algebraic stack 
in the ordinary sense (Definition \ref{dfn:Cl:AlgSt}).
Let $X$ be an algebraic stack.
Then applying the functor $\iota = \Dex \circ a$ (Definition \ref{dfn:dSt:St-dSt})
we have a derived stack $\iota(X)$, 
which is $1$-geometric by Fact \ref{fct:dSt:ase}.
Thus we have the lisse-\'etale $\infty$-topos ${\iota(X)}_{\txle}$.
Then by Lemma \ref{lem:pr:cl}, 
we have a topos ${\iota(X)}_{\txle}^{\cl}$ in the ordinary sense.

\begin{lem}\label{lem:LE:cl}
For an algebraic stack $X$,
the topos ${\iota(X)}_{\txle}^{\cl}$ is equivalent to the lisse-\'etale topos
$X_{\txle}$ (Definition \ref{dfn:Cl:lisse-etale}).
\end{lem}

\begin{proof}
Since $\iota(X)$ is $1$-geometric, the underlying $\infty$-category 
$\idAS_{\iota(X)}^{\lis}$ of $\gsLE(\stX)$ 
is equivalent to the nerve of the underlying category 
of the lisse-\'etale site on $X$ in Definition \ref{dfn:Cl:lisse-etale}.
It implies the conclusion.
\end{proof}

In the remaining part of this subsection,
we give a preliminary discussion on functors of lisse-\'etale sheaves.
%
%
Let $f: \stX \to \stY$ be a morphism of geometric derived stacks.
It gives rise to a continuous functor 
\[
 \gsLE(\stY) \longto \gsLE(\stX), \quad 
 U \longmapsto U \times_{\stY} \stX
\]
of $\infty$-sites.
In fact, the morphism $U \times_{\stY}\stX \to \stX$ is smooth 
by \cite[Lemma 2.2.3.1 (2)]{TVe2}
and if $\{U_i \to U\}_{i \in I}$ is an \'etale covering then 
$\{U_i \times_{\stY} \stX \to U \times_{\stY} \stX\}_{i \in I}$ 
is also an \'etale covering. 
Thus $U \times_{\stY} \stX$ belongs to $\gsLE(\stX)$.

\begin{lem}\label{lem:LE:nonadj}
For a morphism $f: \stX \to \stY$ of $n$-geometric derived stacks,
we have a pair of functors 
\[
 f_{\txle}^{-1}: \stY_{\txle} \longto \stX_{\txle}, \quad
 f^{\txle}_*   : \stX_{\txle} \longto \stY_{\txle}
\]
of $\infty$-topoi by the following construction.
\begin{enumerate}[nosep]
\item
For $\shF \in \stX_{\txle}$, we define $f_* \shF \in \stY_{\txle}$ by 
\[
 (f^{\txle}_* \shF)(U) := \shF(U \times_{\stY} \stX).
\]
\item
For $\shG \in \stY_{\txle}$,
we define $f_{\txle}^{-1} \shG \in \stY_{\txle}$ to be the sheafification
of the presheaf given by
\[
 \gsLE(\stX) \ni V \longmapsto \iclim_{V \to U} \shF(U).
\]
Here the colimit is taken in the sub-$\infty$-category 
of the over-$\infty$-category $\oc{\idSt}{f}$ 
spanned by morphisms $V \to U$ over $f$ with $U \in \gsLE_{}(\stY)$.
In other words, we are considering a square
\begin{align}\label{diag:f^-1}
 \xymatrix{V \ar[r] \ar[d] & U \ar[d] \\ \stX \ar[r]_{f} & \stY}
\end{align}
where vertical morphisms are smooth.
\end{enumerate}
\end{lem}

As in the ordinary case indicated by \cite[\S 3.3]{O}, 
the functor $f^{-1}$ is not left exact since the colimit in its definition 
is not filtered (recall Definition \ref{dfn:ic:exact} of left exactness).
In other words, the pair $(f^{-1},f_*)$ does not give a geometric morphism 
of $\infty$-topoi (Definition \ref{dfn:is:gm}).
However, if $f$ is smooth, then $f^{-1}$ is just the restriction functor
as the diagram \eqref{diag:f^-1} suggests, so that it is exact.
We take the same strategy as in \cite{O} to remedy this non-exactness problem,
which will be explained in the following subsections.

Let us close this subsection by introducing a terminology 
for sheaves over $\stX_{\txle}$.

\begin{ntn*}
Let $\stX$ be a geometric derived stack and 
$\iC$ be an $\infty$-category admitting small limits.
An object of $\iShv_{\iC}(\stX_{\txle})$ will be called 
a \emph{lisse-\'etale sheaf valued in $\iC$ on $\stX$}.
\end{ntn*}

Let us close this subsection by introducing the lisse-\'etale $\infty$-topos 
for \emph{locally geometric} derived stacks (Definition \ref{dfn:dSt:lcgm}).
Recall that a locally geometric derived stack $\stX$ 
can be expressed as a colimit $\iclim_{i \in I} \stX_i$
of a filtered family of geometric derived stacks $\stX_i$.

\begin{dfn}
Let $\stX \simeq \iclim_{i \in I} \stX_i$ be a locally geometric derived stack
with $\{\stX_i\}$ a filtered family of geometric derived stacks.
Then we define the \emph{lisse-\'etale $\infty$-topos $\stX_{\txle}$} 
of $\stX$ to be the colimit
\[
 \stX_{\txle} := \iclim_{i \in I} (\stX_i)_{\txle}
\]
of the lisse-\'etale $\infty$-topoi of $\stX_i$.
Here we take the colimit in the $\infty$-category $\iRTop$ of $\infty$-categories,
which admits small colimits (Fact \ref{fct:is:RTop}).
\end{dfn}

\subsection{Lisse-\'etale sheaves on derived stacks}

In this subsection we introduce the notion of lisse-\'etale sheaves of modules
on geometric derived stacks.
As before we work on a fixed commutative ring $k$.

\subsubsection{Notations on lisse-\'etale sheaves} 

Throughout this section we will use various notions of 
sheaves on ringed $\infty$-topoi discussed in \S \ref{s:is}.
So let us recollect some of them as a preliminary.

Fix a geometric derived stack $\stX$ over $k$.
Then applying to $\tT = \stX_{\txle}$  the definitions in \S \ref{s:is},
we have the following $\infty$-categories of lisse-\'etale sheaves.

\begin{ntn*}
\begin{enumerate}[nosep]
\item 
We call $\iShv_{\iCom}(\stX_{\txle})$ the $\infty$-category of 
\emph{lisse-\'etale sheaves of commutative rings} on $\stX$.

\item
For an $\shA \in \iShv_{\iCom}(\stX_{\txle})$,
we call 
\[
 \iMod_{\shA}(\stX_{\txle})
\]
the $\infty$-category of 
\emph{lisse-\'etale sheaves of $\shA$-modules} on $\stX$.
We also call the full sub-$\infty$-category
\[
 \iMod^{\stab}_{\shA}(\stX_{\txle}) \subset \iMod_{\shA}(\stX_{\txle})
\]
the $\infty$-category of 
\emph{lisse-\'etale sheaves of stable $\shA$-modules on $\stX$}.
\end{enumerate}
\end{ntn*}

For an $\shA \in \iShv_{\iCom}(\stX_{\txle})$,
the homotopy category 
\[
 \Mod_{\shA}(\stX_{\txle}) := \Ho \iMod_{\shA}(\stX_{\txle})
\] 
is a Grothendieck abelian category (Proposition \ref{prp:is:G}).
So we can consider the associated derived $\infty$-category 
$\iDa(\Mod_{\shA}(\stX_{\txle}))$ (\S \ref{sss:is:dic}).
Recalling the equivalence in Proposition \ref{prp:is:Da=Stab}, 
we set 

\begin{ntn}
For an $\shA \in \iShv_{\iCom}(\stX_{\txle})$,
we denote
\[
 \iDa(\stX_{\txle},\shA) := 
 \iDa(\Mod_{\shA}(\stX_{\txle})) \simeq \iMod_{\shA}^{\stab}(\stX_{\txle})
\]
and call it the \emph{derived $\infty$-category of lisse-\'etale sheaves of 
$\shA$-modules} on $\stX$.
\end{ntn}

In the rest of this part, we introduce a derived analogue 
of \emph{flat} sheaves of commutative rings \cite[Chap.\ 12, (12.7)]{LM}.
In fact, the contents will not be used essentially 
in our main argument, so the readers may skip it.

In order to introduce the flatness, 
let us consider the following situation:
Let $\stX$ be a geometric derived stack. 
Then for each object $(\stU, u: \stU \to \stX)$
of the underlying $\infty$-category $\idAS^{\lis}_{\stX}$ of $\gsLE(\stX)$,
we have the pair of functors 
\[
 u_{\txle}^{-1}: \stX_{\txle} \longto \stU_{\txle}, \quad
 u^{\txle}_*   : \stU_{\txle} \longto \stX_{\txle}
\]
of $\infty$-topos by Lemma \ref{lem:LE:nonadj}.

On the other hand, we have the equivalence $\ve: \stU_{\txle} \simto \stU_{\et}$ 
of $\infty$-topoi by Lemma \ref{lem:LE:et-le}.
Thus we can introduce

\begin{ntn}\label{ntn:LE:FU}
For $\shF \in \stX_{\txle}$ and $(\stU,u) \in \idAS^{\lis}_{\stX}$, 
we denote
\[
 \shF_{(\stU,u)} := \ve u_{\txle}^{-1}(\shF) \in \stU_{\et}.
\]
If no confusion would occur, then we simply denote $\shF_U := \shF_{(U,u)}$.
\end{ntn}

Next, let $f: \stU \to \stV$ be a morphism in the underlying $\infty$-category 
$\idAS^{\lis}_{\stX}$ of $\gsLE(\stX)$.
By Lemma \ref{lem:LE:nonadj}, we have a functor
$f^{-1}: \stV_{\txle} \to \stU_{\txle}$.
Using the equivalence $\ve$ in Lemma \ref{lem:LE:et-le}, we have
a functor $\stV_{\et} \to \stU_{\et}$.
Abusing symbols, we denote it by 
\[
 f^{-1}: \stV_{\et} \longto \stU_{\et}.
\]
Then, for $\shF \in \stX_{\txle}$,
we have a morphism $f^{-1}\shF_{\stV} \to \shF_{\stU}$ in $\stU_{\et}$
since $\shF$ is a sheaf on $\idAS^{\lis}_{\stX}$.

Note that we can replace the $\infty$-category 
$\stX_{\txle} \simeq \iShv_{\iS}(\stX_{\txle})$
by the $\infty$-category $\iShv_{\iCom}(\stX_{\txle})$ 
of lisse-\'etale sheaves of commutative rings  in the argument so far.
We summarize it in 

\begin{ntn}\label{ntn:LE:nat} 
Let $\shA \in \iShv_{\iCom}(\stX_{\txle})$ be a lisse-\'etale sheaf of 
commutative rings on $\stX$, 
and $f: (\stU,u) \to (\stV,v)$ be a morphism in $\idAS^{\lis}_{\stX}$. 
We denote
\[
 \shA_{\stU} = \shA_{(\stU,u)} := \ve u^{-1}(\shA) \in \iShv_{\iCom}(\stU_{\et})
\]
and call it the \emph{restriction of $\shA$ on $(\stU,u)$}.
Then we have a natural morphism $f^{-1}(\shA_{\stV}) \to \shA_{\stU}$
in the $\infty$-category $\iShv_{\iCom}(\stU_{\et})$
of \'etale sheaves of commutative rings on the derived algebraic space $\stU$.

We use a similar notation $\shM_U \in \iMod_{\shA_U}(\stU_{\et})$ 
for $\shM \in \iMod_{\shA}(\stX_{\txle})$.
\end{ntn}

Now we introduce a derived analogue of the flatness of 
a sheaf of commutative rings \cite[Definition 3.7]{O}.

\begin{dfn}\label{dfn:LE:flat}
Let $\stX$ be a geometric derived stack,
and $\shA$ be a lisse-\'etale sheaf of commutative rings on $\stX$.
Then $\shA$ is \emph{flat} if for any morphism $f: \stU \to \stV$ 
in $\idAS^{\lis}_{\stX}$, the morphism $f^{-1}(\shA_V) \to \shA_U$ 
in $\iShv_{\iCom}(\stU_{\et})$ is faithfully flat (in the ordinary sense).
\end{dfn}

\subsubsection{Cartesian sheaves}

In \cite[D\'{e}finition (12.3), (12.7.3)]{LM}, 
the notion of \emph{cartesian sheaves} on algebraic stacks
is introduced and used to define quasi-coherent sheaves, derived categories 
of quasi-coherent sheaves and derived functors between them.
Let us briefly recall its definition.

\begin{dfn}[{\cite[D\'{e}finition (12.7.3)]{LM}}]
\label{dfn:LE:cl-cart}
Let $X$ be an algebraic stack (Definition \ref{dfn:Cl:AlgSt})
and $\shA$ be a flat sheaf of commutative rings on 
the lisse-\'etale topos $X_{\txle}$ (Definition \ref{dfn:Cl:lisse-etale}).
A sheaf $\shM$ of $\shA$-modules on $X_{\txle}$ is \emph{cartesian}
if for any morphism $f: U \to V$ in the lisse-\'etale site of $X$ 
the natural morphism 
$\shA_{U} \otimes_{f^{-1}(\shA_V)} f^{-1}(\shM_V) \to \shM_U$ 
of sheaves on $U_{\et}$ is an isomorphism.
\end{dfn}

The essence of cartesian sheaves is shown in the following fact.

\begin{fct}[{\cite[Lemma 3.8]{O}}]\label{fct:LE:cl-cart}
Let $X$ and $\shA$ be as in Definition \ref{dfn:LE:cl-cart}.
A sheaf of $\shA$-modules $\shM$ on $X_{\txle}$ is cartesian if and only if 
for every \emph{smooth} morphism $f: U \to V$ in the lisse-\'etale site of $X$
the natural morphism 
$\shA_{U} \otimes_{f^{-1}(\shA_V)} f^{-1}(\shM_V) \to \shM_U$ is an isomorphism.
\end{fct}

Thus a cartesian sheaf is a sheaf of modules which is totally characterized
by the behavior under smooth morphisms.

Now let us introduce cartesian lisse-\'etale sheaves.
We continue to use the notations in the previous parts,
so that $\stU$ and $\stV$ denote derived algebraic spaces
in the underlying $\infty$-category $\idAS_{\stX}^{\lis}$
of the lisse-\'etale $\infty$-site $\gsLE(\stX)$ of a derived stack $\stX$.
Recall also Notation \ref{ntn:LE:nat} on the restriction of sheaves.

\begin{dfn}\label{dfn:LE:cart}
Let $\stX$ be a geometric derived stack and 
$\shA$ be a flat lisse-\'etale sheaf of commutative rings on $\stX$.
An $\shA$-modules $\shM \in \iMod_{\shA}(\stX_{\txle})$ is called \emph{cartesian} 
if for any morphism $f: \stU \to \stV$ in $\idAS_{\stX}^{\lis}$, the morphism 
\[
 \shA_{\stU} \otimes_{f^{-1}(\shA_{\stV})} f^{-1}(\shM_{\stV}) 
 \longto \shM_{\stU}
\]
is an equivalence in $\iMod_{\shA_{\stU}}(\stU_{\et})$.
We denote by 
\[
 \iMod^{\cart}_{\shA}(\stX_{\txle}) \subset \iMod_{\shA}(\stX_{\txle})
\]
the full sub-$\infty$-category spanned by cartesian sheaves.
\end{dfn}

We have an analogue of Fact \ref{fct:LE:cl-cart}.

\begin{lem}\label{lem:LE:3.8}
Let $\stX$ be a geometric derived stack, 
$\shA \in \iShv_{\iCom}(\stX_{\txle})$ be flat, 
and $\shM \in \iMod_{\shA}(\stX_{\txle})$.
Then $\shM$ is cartesian if and only if for every morphism 
$f: \stU \to \stV$ in $\idAS_{\stX}^{\lis}$
the natural morphism 
$\shA_{\stU} \otimes_{f^{-1}(\shA_{\stV})} f^{-1}(\shM_{\stV}) \to \shM_{\stU}$ 
is an equivalence in $\iMod_{\shA}(\stU_{\et})$.
\end{lem}

\begin{proof}
We follow the proof of \cite[Lemma 3.8]{O}.
By the definition of cartesian sheaf it is enough to show the ``if" part.
Let $\stX$ be an $n$-geometric derived stack.
The proof is by induction on $n$.
Assume $n=1$.
Take a $1$-atlas $\{X_i\}_{i \in I}$ of $\stX$
and consider the $1$-smooth effective epimorphism 
$X' := \coprod_{i \in I}X \surj \stX$.
Given a morphism $f: \stU \to \stV$ in $\idAS_{\stX}^{\lis}$, 
we set $\stV':= \stV \times_{\stX} X'$ and $\stU':= \stU \times_{\stX} X'$. 
Then we have a pullback square
\[
 \xymatrix{\stU' \ar[r]^{f'} \ar[d]_{u} & \stV' \ar[d]^{v} \\ 
           \stU \ar[r]_{f} & \stV}
\]
in $\oc{\idAS}{\stX}$.
We also find that $u$ and $v$ are smooth and surjective 
and that the composition $x: \stU' \xr{f'} \stV' \to \stX$ is smooth.
In order to prove that 
$\shA_{\stU} \otimes_{f^{-1}(\shA_{\stV})} f^{-1}(\shM_{\stV}) \to \shM_{\stU}$
is an equivalence, it is enough to show that 
\[
 i: \shA_{\stU'} \otimes_{u^{-1}f^{-1}(\shA_V)} u^{-1}f^{-1}(\shM_V) 
    \longto \shM_{\stU'}
\]
is an equivalence
since $u$ is smooth and surjective.
Hereafter let us suppress the change of ring,
and denote $i: u^{-1}f^{-1}(\shM_V) \to \shM_{\stU'}$.
By the above square and the assumption 
we find $u^{-1}f^{-1}(\shM_{\stV}) \simeq {f'}^{-1}(\shM_{\stV'})$, 
and by the smoothness of $u$
and the assumption we find $u^{-1}(\shM_{\stU}) \simeq \shM_{\stU'}$.
Thus we have $i: f^{-1}(\shM_{\stV'}) \to \shM_{\stU'}$.
On the other hand, $i$ is equivalent to the pullback by $x$ of $\id_{\shF_X}$.
Then by the smoothness of $x$ and 
the assumption we find that $i$ is an equivalence.
\end{proof}

As a corollary, we have

\begin{cor}\label{cor:LE:cart}
The homotopy category 
\[
 \Ho \iMod^{\cart}_{\shA}(\stX_{\txle}) \subset \Ho \iMod_{\shA}(\stX_{\txle})
\]
is a Serre subcategory.
\end{cor}

Now we introduce notations on the derived $\infty$-category of 
lisse-\'etale sheaves with cartesian homotopy groups.

%

%

\begin{ntn}\label{ntn:LE:iDcart}
Let $\stX$ be a geometric derived stack 
and $\shA \in \iShv_{\iCom}(\stX_{\txle})$ be flat. 
We define 
\[
 \iDu{\cart}(\stX_{\txle},\shA) \subset 
 \iDa(\stX_{\txle},\shA) \simeq \iMod^{\stab}_{\shA}(\stX_{\txle})
\]
the full sub-$\infty$-category of the derived $\infty$-category spanned by 
those objects $\shM$ whose homotopy group 
$\pi_j(\shM) \in \iMod_{\shA}(\stX_{\txle})$ is cartesian for all $j \in \bbZ$.
\end{ntn}

Recall that for a stale $\infty$-category $\iC$ with a $t$-structure
we have the sub-$\infty$-categories $\iC^+$, $\iC^-$ and $\iC^b$ 
of left bounded, right bounded and bounded objects 
(Definition \ref{dfn:stb:+-b}).  

\begin{ntn*}
Let $\stX$ be a geometric derived stack, 
and $\shA \in \iShv_{\iCom}(\stX_{\txle})$ be flat.
We denote by
\[
 \iDu{\cart,*}(\stX_{\txle},\shA) \subset 
 \iDu{\cart  }(\stX_{\txle},\shA) \quad   (* \in \{+,-,b\})
\]
the sub-$\infty$-categories of left bounded, right bounded 
and bounded objects respectively.
\end{ntn*}

\subsection{Lisse-\'etale $\infty$-topos and simplicial \'etale $\infty$-topos}
\label{ss:LE:spl}

As mentioned in the beginning of \S \ref{s:dAS}, 
the lisse-\'etale topos on an algebraic stack
can be described by the \'etale topos on a simplicial algebraic space.
See \cite{O} and \cite[\S 9.2]{O:book} for the detailed explanation 
of the non-derived case.
In this subsection we give its derived analogue.

%


Recall Fact \ref{fct:dAS:Segal}.
For an $n$-geometric derived stack $\stX$, 
there exists a family $\{\stU_i\}_{i \in I}$ of derived algebraic spaces
such that each of them is equipped with an $(n-1)$-smooth morphism $\stU_i \to \stX$
and $\stX \simeq \iclim_j \stX_j$ with $\stX_0 := \coprod_{i \in I} \stU_i$ and 
$\stX_j := \stX_0 \times_{\stX} \cdots \times_{\stX} \stX_0$
for $j \in \bbZ_{\ge 1}$ ($j$-times fiber product).
These $\stX_j$'s gives rise to a simplicial object $\stX_{\bullet}$
in $\idAS_{\stX}^{\lis}$,
the underlying $\infty$-category of the lisse \'etale $\infty$-site $\gsLE(\stX)$.
Indeed, $\stX_n$ is a derived algebraic space for each $n \in \bbN$.
We also note that the morphism $\stX_n \to \stX_m$ is smooth
for each $[m] \to [n]$ in $\CS$.

\begin{dfn}\label{dfn:LE:pres}
We call the simplicial object $\stX_{\bullet}$ in $\idAS_{\stX}^{\lis}$
a \emph{smooth presentation of $\stX$}, and denote this situation by 
\[
 \stX_{\bullet} \to \stX.
\]
\end{dfn}

Note that a smooth presentation is nothing but 
the ($0$-)coskeleton of the surjection $\stX_1 \to \stX$
in the sense of \cite[9.2.1]{O:book},
and is also equivalent to the hypercovering of $\stX$
in the sense of Definition \ref{dfn:ic:top:hc}.

Restricting to the subcategory $\CS^{\str} \subset \CS$
(Definition \ref{dfn:is:CSinj}),
we have a strictly simplicial object $\stX^{\str}_{\bullet}$ 
of derived algebraic spaces.

Since each morphism in $\stX^{\str}_{\bullet}$ is smooth,
we have the strictly simplicial $\infty$-topoi $\stX^{\str}_{\bullet,\txle}$
and $\stX^{\str}_{\bullet,\et}$ (Definition \ref{dfn:is:CSinj}).
Explicitly, denoting $\tau = \txle$ or $\et$, 
for each $n \in \bbN$, we define $\stX_{n,\tau}$ to be 
the $\infty$-topos associated to the $\infty$-site 
$\gsLE(\stX_n)$ or $\gsEt(\stX_n)$, 
and for each morphism $\delta: [m] \to [n]$ 
we have the corresponding geometric morphism 
$\delta_\tau^{-1}: \stX_{n,\tau} \to \stX_{m.\tau}$.

By restriction, we have an equivalence
$\ve_{n}: \stX^{\str}_{\bullet,\txle} \simto \stX^{\str}_{\bullet,\et}$ 
of $\infty$-topoi for each $n \in \bbN$.
These induce an equivalence 
\[
 \ve: \stX^{\str}_{\bullet,\txle} \longsimto \stX^{\str}_{\bullet,\et}.
\]

Now let $\shA \in \iShv_{\iCom}(\stX_{\txle})$ 
be a flat lisse-\'etale sheaf of commutative rings.
Then the above argument works for the ringed $\infty$-topos 
$(\stX_{\txle},\shA)$.
Thus we have a strictly simplicial ringed $\infty$-topos 
$(\stX^{\str}_{\bullet,\txle},\shA_{\bullet})$
and its restriction 
$(\stX^{\str}_{\bullet,\et},\shA_{\bullet})$.
We have the equivalence induced by restriction:
\[
 \ve: 
 \bigl(\stX^{\str}_{\bullet,\txle},\shA_{\bullet}\bigr) 
 \longto 
 \bigl(\stX^{\str}_{\bullet,\et},\shA_{\bullet}\bigr).
\]
We also have the simplicial ringed $\infty$-topos
$(\stX_{\bullet,\et},\shA_{\bullet})$.

We will analyze objects in $\iMod^{\cart}_{\shA}(\stX_{\txle})$
via these (strictly) simplicial $\infty$-topoi 
$\stX^{\str}_{\bullet,\txle}$, $\stX^{\str}_{\bullet,\et}$.
and $\stX_{\bullet,\et}$.
For that, recalling Corollary \ref{cor:LE:cart}
which says that 
$\Ho \iMod^{\cart}_{\shA}(\stX_{\txle}) \subset \Ho \iMod_{\shA}(\stX_{\txle})$
is a Serre subcategory, 
we apply the notations in \S \ref{sss:is:des} to the present situation.

\begin{dfn}\label{dfn:LE:cart:spl}
Let $\stX$ and $\shA$ be as above.
\begin{enumerate}[nosep]
\item
Let $\tau = \txle$ or $\et$.
An object $\shM_\bullet$ of $\iMod_{\shA_{\bullet}}(\stX^{\str}_{\bullet,\txle})$
is called \emph{cartesian} if $\shM_n \in \iMod_{\shA_{n}}(\stX_{n,\txle})$
is cartesian (Definition \ref{dfn:LE:cart}) for each $[n] \in \CS^{\str}$,
and if the morphism $\delta^* \shM_n \to \shM_m$ is an equivalence 
in $\iMod^{\cart}_{\shA_{m}}(\stX_{m,\txle})$
for each morphism $\delta: [m] \to [n]$ in $\CS^{\str}$.
We denote by  
\[
 \iMod^{\cart}_{\shA_{\bullet}}(\stX_{\bullet,\txle}) \subset 
 \iMod_{\shA_{\bullet}}(\stX_{\bullet,\txle})
\]
the full sub-$\infty$-category spanned by cartesian objects.

\item
A \emph{cartesian} object of $\iMod_{\shA_{\bullet}}(\stX_{\bullet,\et})$ 
is defined by the same condition as (1) but replacing $\CS^{\str}$ by $\CS$.
We denote by 
\[
 \iMod^{\cart}_{\shA_{\bullet}}(\stX_{\bullet,\et}) \subset 
 \iMod_{\shA_{\bullet}}(\stX_{\bullet,\et})
\]
the corresponding full sub-$\infty$-category.
\end{enumerate}
\end{dfn}

We then have the following  functors of $\infty$-categories.
\[
 \iMod^{\cart}_{\shA_{\bullet}}(\stX_{\bullet,\et})         \xrr{\Rst}
 \iMod^{\cart}_{\shA_{\bullet}}(\stX^{\str}_{\bullet,\et})  \xrr{\ve^*}
 \iMod^{\cart}_{\shA_{\bullet}}(\stX^{\str}_{\bullet,\txle}).
\]
Here $\Rst$ denotes the functor induced by the restriction to 
$\CS^{\str} \subset \CS$,
and $\ve^*$ is the one induced by the equivalence $\ve$.

We have the following analogue of \cite[Proposition 4.4]{O}.

\begin{prp}\label{prp:LE:cart}
The functors $\Rst$ and $\ve^*$ are equivalences,
and the $\infty$-categories
\[
 \iMod^{\cart}_{\shA}(\stX_{\txle}), \quad 
 \iMod^{\cart}_{\shA_{\bullet}}(\stX_{\bullet,\et}). \quad
 \iMod^{\cart}_{\shA_{\bullet}}(\stX^{\str}_{\bullet,\et}), \quad
 \iMod^{\cart}_{\shA_{\bullet}}(\stX^{\str}_{\bullet,\txle}).
\]
are all equivalent.
\end{prp}

\begin{proof}
We follow the proof of \cite[Proposition 4.4]{O}.
Since $\ve$ is an equivalence, the induced $\ve^*$ is also an equivalence.
The fully faithfulness of $\Rst$ is obtained by unwinding the definition, 
so we focus on the essential surjectivity of $\Rst$.

Recall that the $0$-th part of $\stX_{\bullet}$ is given by 
$\stX_0 = \coprod_{i \in I} \stX_i$,
and it is a derived algebraic space.
Recalling also that $\stX_1 = \stX_0 \times_{\stX} \stX_0$
and $\stX_2 = \stX_0 \times_{\stX} \stX_0 \times_{\stX} \stX_0$,
let us denote by $d^1_0,d^1_1: \stX_1 \to \stX_0$ 
and $d^2_0,d^2_1,d^2_2: \stX_2 \to \stX_1$ the projections.
Compared to the ordinary symbols,
we have $d^1_0=\pr_2$, $d^1_1=\pr_1$ and 
$d^2_0=\pr_{23}$, $d^2_1=\pr_{13}$, $d^2_2=\pr_{12}$.

Consider the $\infty$-category $\iDes(\stX_0/\stX)$
whose object is a pair $(\shG,\iota)$ of $\shG \in \iMod_{\shA_0}(\stX_{0,\et})$
and an equivalence $\iota: (d^1_1)^* \shG \simto (d^1_0)^* \shG$ 
in $\iMod_{\shA}(\stX_{1,\et})$ such that the two equivalences 
$(d^2_1)^*(\iota)$ and $(d_2^0)^* \circ (d_2^2)^*$ seen as 
$(d^1_1 \circ d^2_2)^*\shG \to (d^1_0 \circ d^2_0)^*\shG $ are equivalent.
We have a functor 
$\iMod^{\cart}_{\shA}(\stX_{\txle}) \to \iDes(\stX_0/\stX)$
by sending $\shM$ to the pair of $\shG := \shM_{\stX_0}$ 
(using Notation \ref{ntn:LE:FU}) and 
$\iota$ defined to be the composition of 
$(d^1_1)^* \shM_{\stX_0} \simto \shM_{\stX_1}$ and the inverse of  
$(d^1_0)^* \shM_{\stX_0} \simto \shM_{\stX_1}$.
Then the argument in \cite[Lemma 4.5]{O} works and 
this functor gives an equivalence.

On the other hand, we have a functor
$\iMod^{\cart}_{\shA_{\bullet}}(\stX^{\str}_{\bullet,\et}) \to \iDes(\stX_0/\stX)$
by sending $\shM_\bullet$ to the pair of $\shG := \shM_{0}$ and 
$\iota$ defined to be the composition of 
$(d^1_1)^* \shM_0 \simto \shM_1$ and the inverse of  
$(d^1_0)^* \shM_0 \simto \shM_1$.
Note that the last equivalences come from the simplicial structure.
Then, by the argument in \cite[Proposition 4.4]{O},
this functor is also an equivalence.
The same construction gives 
$\iMod^{\cart}_{\shA_{\bullet}}(\stX_{\bullet,\et}) \to \iDes(\stX_0/\stX)$.
Finally, the composition 
$\iMod^{\cart}_{\shA_{\bullet}}(\stX^{\str}_{\bullet,\et}) 
 \simto \iDes(\stX_0/\stX)
 \xleftarrow{\sim} \iMod^{\cart}_{\shA_\bullet}(\stX_{\bullet,\et})
 \xr{\Rst} \iMod^{\cart}_{\shA_{\bullet}}(\stX^{\str}_{\bullet,\et})$
is equivalent to the identity functor. 
Thus $\Rst$ is an equivalence,and we have shown that 
$\iMod^{\cart}_{\shA_{\bullet}}(\stX_{\bullet,\et})$.
$\iMod^{\cart}_{\shA_{\bullet}}(\stX^{\str}_{\bullet,\et})$ and 
$\iMod^{\cart}_{\shA_{\bullet}}(\stX^{\str}_{\bullet,\txle})$
are equivalent to $\iDes(\stX_0/\stX) \simeq \iMod^{\cart}_{\shA}(\stX_{\txle})$.
\end{proof}

We next discuss the derived $\infty$-category.
For $\tau = \txle$ or $\et$, 
we denote by 
\[
 \iDa(\stX^{\str}_{\bullet,\tau},\shA_{\bullet}) 
\]
the derived $\infty$-category of the Grothendieck abelian category
$\Ho \iMod_{\shA_{\bullet}}(\stX^{\str}_{\bullet,\tau})$ 
(Notation \ref{ntn:is:sDa}).
Similarly we define $\iDa(\stX_{\bullet,\et},\shA_{\bullet})$.
An object $\shM_{\bullet}$ of $\iDa(\stX^{\str}_{\bullet,\tau},\shA_{\bullet})$
consists of $\shM_n \in \iDa(\stX_{n,\tau},\shA_{n})$ for each $n \in \bbN$
and $\delta^* \shM_n \to \shM_m$ for each $\delta: [m] \to [n]$ in $\CS^{\str}$.

\begin{dfn*}
Let $\stX$ and $\shA$ be as before.
\begin{enumerate}[nosep]
\item
Let $\tau = \txle$ or $\et$.
An object $\shM_{\bullet}$ of $\iDa(\stX^{\str}_{\bullet,\tau},\shA_{\bullet})$ 
is called \emph{cartesian} if $\shM_n$ belongs to 
$\iDu{\cart}(\stX^{\str}_{n,\tau},\shA_{n})$ 
for each $[n] \in \CS^{\str}$
and the morphism $\delta^* \shM_n \to \shM_m$ is an equivalence 
in $\iDa(\stX^{\str}_{m,\tau},\shA_{m})$
for each morphism $\delta: [m] \to [n]$ in $\CS^{\str}$.
We denote by  
\[
 \iDu{\cart}(\stX^{\str}_{\bullet,\tau},\shA_{\bullet}) \subset 
 \iDa(\stX^{\str}_{\bullet,\tau},\shA_{\bullet})
\]
the full sub-$\infty$-category spanned by cartesian objects.

\item
A \emph{cartesian} object of $\iDa(\stX_{\bullet,\et},\shA_{\bullet})$ 
is defined by the same condition as (1)
with replacing $\CS^{\str}$ by $\CS$.
We denote by 
\[
 \iDu{\cart}(\stX_{\bullet,\et},\shA_{\bullet}) \subset 
 \iDa(\stX_{\bullet,\et},\shA_{\bullet})
\]
the corresponding full sub-$\infty$-category.
\end{enumerate}
\end{dfn*}

%
%
%
%
Let us denote by $\pi: \stX^{\str}_{\bullet,\txle} \to \stX_{\txle}$
the functor induced by the strictly simplicial structure,
which is indeed a geometric morphisms of $\infty$-topoi.
Recall the equivalence 
$\ve: \stX^{\str}_{\bullet,\txle} \simto \stX^{\str}_{\bullet,\et}$.
We denote by the same symbols the induced functors 
\[
 \pi: (\stX^{\str}_{\bullet,\txle},\shA_{\bullet}) \longto 
      (\stX_{\txle},\shA), \quad
 \ve: (\stX^{\str}_{\bullet,\txle},\shA_{\bullet}) \longto 
      (\stX^{\str}_{\bullet,\et}, \shA_{\bullet})
\]
of flat ringed $\infty$-topoi, and denote the corresponding inverse images by
\[
 \pi^*: \iDu{\cart}(\stX_{\txle},\shA) \longto 
        \iDu{\cart}(\stX^{\str}_{\bullet,\txle},\shA_{\bullet}), \quad
 \ve^*: \iDu{\cart}(\stX^{\str}_{\bullet,\et},\shA_{\bullet}) \longto 
        \iDu{\cart}(\stX^{\str}_{\bullet,\txle},\shA_{\bullet}).
\]
Let us also define a functor 
\[
 \iDu{\cart}(\stX_{\bullet,\et},\shA_{\bullet})         \xrr{\Rst}
 \iDu{\cart}(\stX^{\str}_{\bullet,\et},\shA_{\bullet}) 
\]
by the restriction with respect to $\CS^{\str} \subset \CS$.
Now we can state the main result of this part,
which is a derived analogue of \cite[Theorem 4.7]{O}.

\begin{thm}\label{thm:LE:pi-ve}
All of the functors $\pi^*$, $\ve^*$ and $\Rst$ are equivalences.
In particular, the $\infty$-categories 
\[
 \iDu{\cart}(\stX_{\txle},\shA), \quad
 \iDu{\cart}(\stX^{\str}_{\bullet,\txle},\shA_{\bullet}), \quad
 \iDu{\cart}(\stX^{\str}_{\bullet,\et},\shA_{\bullet}), \quad 
 \iDu{\cart}(\stX_{\bullet,\et},\shA_{\bullet})
\]
are all equivalent,
and hence each of them is stable and equipped with a $t$-structure.
\end{thm}

\begin{proof}
$\pi^*$ is an equivalence by Proposition \ref{prp:is:223}.
$\ve^*$ is an equivalence since $\ve$ is an equivalence.
As for $\Rst$, the argument in Proposition \ref{prp:LE:cart} works.
\end{proof}

\subsection{Constructible sheaves}

In this subsection 
$\Lambda$ denotes a torsion noetherian ring annihilated by an integer
invertible in the base commutative ring $k$.

Recall the notion of a \emph{constructible sheaf} on an ordinary scheme 
\cite[IX \S 2]{SGA4}: 

\begin{dfn*}
A sheaf $\shF$ of sets on a scheme $X$ is \emph{constructible} 
if for any affine Zariski open $U \subset X$ there is a finite decomposition 
$U=\sqcup_i U_i$ into constructible locally closed subschemes $U_i$ such that 
$\rst{\shF}{U_i}$ is a locally constant sheaf with value in a finite set.
\end{dfn*}

As noted in \cite[Remarque (18.1.2) (1)]{LM},
the phrase ``for any affine Zariski open $U \subset X$" can be replaced by
``for each affine scheme belonging to any \'etale covering of $X$".
In \cite[Chap.\ 18]{LM} the corresponding notion is introduced 
for the lisse-\'etale topos of an algebraic stacks.
In this subsection we introduce an analogous notion for derived stacks.

We cite from \cite[\S A.1]{Lur2} 
the definition of locally constant sheaf on an $\infty$-topos.

\begin{dfn}[{\cite[Definintion A.1.12]{Lur2}}]\label{dfn:LE:lc}
Let $\tT$ be an $\infty$-topos and $\shF \in \tT$.
\begin{enumerate}[nosep]
\item
$\shF$ is \emph{constant} if it lies 
in the essential image of a geometric morphism $\pi^*: \iS \to \tT$
(see Remark \ref{rmk:is:gm} for the notation on a geometric morphism).
\item
$\shF$ is \emph{locally constant} if there exists a 
collection $\{U_i\}_{i \in I}$ of objects $U_i \in \tT$
satisfying the following conditions.
\begin{itemize}[nosep]
\item 
$\{U_i\}_{i \in I}$ covers $\tT$, i.e., there is an effective epimorphism 
$\coprod_{i \in I} U_i \to \one_{\tT}$, 
where $\one_{\tT}$ denotes a final object of $\tT$
(see Definition \ref{dfn:it:eff-epi} 
 for the definition of  \emph{effective epimorphism}).
\item
The product $\shF \times U_i$ is a constant object 
of the $\infty$-topos $\oc{\tT}{U_i}$ (Fact \ref{fct:pr:ovc}).
\end{itemize}
\end{enumerate}
\end{dfn}

In order to introduce constructible sheaves on a derived stack,
we start with those on a derived algebraic space.

\begin{dfn}\label{dfn:LE:cstr:dAS}
Let $\stU$ be a derived algebraic space and 
$\shF \in \iShv_{\iS}(\stU_{\et}) \simeq \stU_{\et}$.
\begin{enumerate}[nosep]
\item
$\shF$ is \emph{locally constant} if 
it is locally constant as an object of the $\infty$-topos $\stU_{\et}$
in the sense of Definition \ref{dfn:LE:lc}.

\item
$\shF$ is  \emph{constructible} if for any \'etale covering 
$\{\stU_i \to \stU\}_{i \in I}$ by derived algebraic spaces $\stU_i$,
there is a finite decomposition 
$\stU_i \simeq \coprod_j U_{i,j}$ into affine derived schemes $U_{i,j}$ 
for each $i$ such that 
\begin{itemize}
\item
The (non-derived) affine scheme $\Trc(U_{i,j})$ is 
a constructible locally closed subscheme of $\Trc(\stU_i)$. 
\item
The restriction 
$\shF_{U_{i,j}}$ is a locally constant sheaf on $U_{i,j}$ for all $j$
in the sense of (1).
\end{itemize}
\end{enumerate}
\end{dfn}

Let us now introduce \emph{constructible sheaves of $\Lambda$-modules}.
As before, for a commutative ring $\Lambda$,
we denote the constant lisse-\'etale sheaf valued in $\Lambda$
by the same symbol. 

\begin{dfn}\label{dfn:LE:cstr:L}
Let $\stU$ be a derived algebraic space and $\Lambda$ be a commutative ring.
An object $\shM$ of $\iMod_{\Lambda}(\stU_{\et})$ is called 
\emph{constructible} if is is constructible in the sense of 
Definition \ref{dfn:LE:cstr:dAS} as an object of $\iShv_{\iS}(\stU_{\et})$.
\end{dfn}

The following statement is an analogue of \cite[Lemma 9.1]{O}.
The proof is the same with loc.\ cit.,
and we omit it.

\begin{lem}\label{lem:LE:cstr}
Let $\stX$ be a geometric derived stack 
\begin{enumerate}[nosep]
\item
Let $\shF \in \iShv_{\iS}(\stX_{\txle}) \simeq \stX_{\txle}$ be cartesian.
Then the following conditions are equivalent.
\begin{enumerate}[nosep,label=(\roman*)]
\item 
For any $\stU \in \oc{\idAS}{\stX}$, the sheaf $\shF_{\stU}$ 
is a locally constant (resp.\ constructible) sheaf on $\stU_{\et}$.
\item
There exists a smooth presentation $\stX_{\bullet} \to \stX$
(Definition  \ref{dfn:LE:pres}) such that 
$\shF_{\stX_0}$ is locally constant (resp.\ constructible).
\end{enumerate}
\item
Let $\Lambda$ be a commutative ring.
Then the same equivalences as (1) hold for 
$\shM \in \iMod^{\cart}_{\Lambda}(\stX_{\txle})$.
\end{enumerate}
\end{lem}

Note that  the symbol $\iMod^{\cart}_{\Lambda}(\stX_{\txle})$ in the item (2)
makes sense since the constant sheaf $\Lambda$ is flat 
(Definition \ref{dfn:LE:flat})
so that the notion of cartesian sheaves is well-defined.

Now we can state the definition of constructible lisse-\'etale sheaves
on a derived stack.

\begin{dfn}\label{dfn:LE:cstr}
Let $\stX$ be a geometric derived stack
and $\Lambda$ be a commutative ring.
Then a sheaf $\shM \in \iMod^{\cart}_{\Lambda}(\stX_{\txle})$
is \emph{constructible} if it satisfies one of the conditions 
in Lemma \ref{lem:LE:cstr}.
We denote by
\[
 \iMod_{\cstr}(\stX_{\txle},\Lambda) \subset
 \iMod^{\cart}_{\Lambda}(\stX_{\txle})
\]
the full sub-$\infty$-category spanned by constructible sheaves.
\end{dfn}

Similarly to the case of cartesian sheaves (\S \ref{ss:LE:spl}), 
we can describe constructible sheaves by (strictly) simplicial $\infty$-topoi.
Let us take a smooth presentation $\stX_{\bullet} \to \stX$,
and consider the strictly simplicial $\infty$-topoi 
$\stX^{\str}_{\bullet,\txle}$, $\stX^{\str}_{\bullet,\et}$
and the simplicial $\infty$-topoi $\stX_{\bullet,\et}$.
These can be ringed by the constant sheaf $\Lambda$.

\begin{dfn*}
\begin{enumerate}[nosep]
\item 
Let $\tau = \txle$ or $\et$.
An object $\shM_{\bullet}$ of 
$\iMod^{\cart}_{\Lambda}(\stX^{\str}_{\bullet,\tau})$ is \emph{constructible}
if $\shM_n$ is constructible in the sense of Definition \ref{dfn:LE:cstr}
for each $[n] \in \CS^{\str}$.
We denote by 
\[
 \iMod_{\cstr}(\stX^{\str}_{\bullet,\tau}, \Lambda) \subset 
 \iMod^{\cart}_{\Lambda}(\stX^{\str}_{\bullet,\tau})
\]
the full sub-$\infty$-category spanned by constructible objects.

\item
An object of 
$\iMod^{\cart}_{\Lambda}(\stX_{\bullet,\et})$ is defined to be 
\emph{constructible} in the same way but replacing $\CS^{\str}$ by $\CS$.
The corresponding full sub-$\infty$-category is denoted by 
\[
 \iMod_{\cstr}(\stX_{\bullet,\et}, \Lambda) \subset 
 \iMod^{\cart}_{\Lambda}(\stX_{\bullet,\et}).
\]
\end{enumerate}
\end{dfn*}

\begin{prp}\label{prp:LE:Modc}
The natural restriction functors
\[
 \iMod_{\cstr}(\stX_{\txle},\Lambda) \longto 
 \iMod_{\cstr}(\stX^{\str}_{\bullet,\txle},\Lambda) \longto 
 \iMod_{\cstr}(\stX^{\str}_{\bullet,\et},\Lambda), 
\]
and
\[
 \iMod_{\cstr}(\stX_{\txle},\Lambda) \longto 
 \iMod_{\cstr}(\stX_{\bullet,\et},\Lambda) \longto 
 \iMod_{\cstr}(\stX^{\str}_{\bullet,\et},\Lambda) 
\]
are equivalences.
\end{prp}

\begin{proof}
Let $\shM_{\bullet} \in \iMod^{\cart}_{\Lambda}(\stX^{\str}_{\bullet,\tau})$
with $\tau=\txle$ or $\et$.
Recall that by the cartesian condition 
the morphism $\delta^* \shM_m \to \shM_n$
is an equivalence for every $\delta: [m] \to [n]$.
Thus $\shM_{\bullet}$ is constructible 
if and only if $\shM_{0} \in \iMod^{\cart}_{\Lambda}(\stX^{\str}_{0,\txle})$ 
is constructible.
The same criterion holds for 
$\shM_{\bullet} \in \iMod^{\cart}_{\Lambda}(\stX_{\bullet,\et})$.
Now the statement follows from Proposition \ref{prp:LE:cart}.
\end{proof}

Let us turn to derived $\infty$-categories.
Recall the stable $\infty$-category $\iDu{\cart}(\stX_{\txle},\shA)$
of cartesian sheaves in Notation \ref{ntn:LE:iDcart}.
Since the cartesian sheaves form a Serre subcategory 
(Corollary \ref{cor:LE:cart}), the following definition makes sense.

\begin{dfn}\label{dfn:LE:iDc:Lambda}
Let $\stX$ be a geometric derived stack, 
and $\Lambda$ be a commutative ring. 
We denote by
\[
 \iDc(\stX_{\txle},\Lambda) \subset \iDa(\stX_{\txle},\Lambda) 
\]
the full sub-$\infty$-category spanned by those objects $\shM$
whose homotopy sheaf $\pi_j \shM \in \iMod_{\Lambda}(\stX_{\txle})$
is constructible (Definition \ref{dfn:LE:cstr}) for any $j \in \bbZ$.
We also denote by 
\[
 \iDcu{*}(\stX_{\txle},\Lambda) \subset 
 \iDc(\stX_{\txle},\Lambda) \quad (* \in \{+,-,b,[m,n]\})
\]
the full sub-$\infty$-category spanned by constructible objects 
which are right bounded, left bounded, bounded and bounded in the range $[m,n]$.
\end{dfn}

%

As in the previous \S \ref{ss:LE:spl}, we can describe 
$\iDc(\stX_{\txle},\Lambda)$ by (strictly) simplicial $\infty$-topoi.

\begin{dfn*}
\begin{enumerate}[nosep]
\item 
Let $\tau = \txle$ or $\et$.
An object $\shM_{\bullet}$ of 
$\iDu{\cart}(\stX^{\str}_{\bullet,\tau},\Lambda)$ is \emph{constructible}
if $\shM_n$ is constructible in the sense of Definition \ref{dfn:LE:iDc:Lambda}
for each $[n] \in \CS^{\str}$.
We denote by 
\[
 \iDc(\stX^{\str}_{\bullet,\tau},\Lambda) \subset 
 \iDc(\stX^{\str}_{\bullet,\tau},\Lambda)
\]
the full sub-$\infty$-category spanned by constructible objects.

\item
An object of 
$\iDu{\cart}(\stX_{\bullet,\et},\Lambda)$ is defined to be 
\emph{constructible} in the same way but replacing $\CS^{\str}$ by $\CS$.
The corresponding full sub-$\infty$-category is denoted by 
\[
 \iDc(\stX_{\bullet,\et},\Lambda) \subset \iDc(\stX_{\bullet,\et},\Lambda).
\]
\end{enumerate}
\end{dfn*}

As in \S \ref{ss:LE:spl}, we denote by
\[
 \pi: \stX^{\str}_{\bullet,\txle} \longto \stX_{\txle}, \quad
 \ve: \stX^{\str}_{\bullet,\txle} \longto \stX^{\str}_{\bullet,\et}
\]
the natural geometric morphism of $\infty$-topoi.
We denote the induce functors of derived $\infty$-categories by
\[
 \pi^*: \iDc(\stX_{\txle},\Lambda) \longto 
        \iDc(\stX^{\str}_{\bullet,\txle},\Lambda), \quad
 \ve^*: \iDc(\stX^{\str}_{\bullet,\et},\Lambda) \longto  
        \iDc(\stX^{\str}_{\bullet,\txle},\Lambda).
\]

\begin{thm}\label{thm:LE:iDc-stb}
Both $\pi^*$ and $\ve^*$ are equivalences.
In particular, $\iDcu{*}(\stX_{\txle},\Lambda)$ is stable 
and equipped with a $t$-structure.
\end{thm}

\begin{proof}
This is a direct consequence of Theorem \ref{thm:LE:pi-ve} 
and definitions.
\end{proof}

\subsection{Dualizing objects for derived stacks}
\label{ss:LE:dual}

In this subsection we introduce dualizing complexes for 
the derived category of constructible lisse-\'etale sheaves on derived stacks.
which will be used in the construction of derived functors (\S \ref{s:6op}).
The strategy of defining dualizing complexes follows that in \cite[3.3--3.5]{LO1}:
gluing dualizing complexes over derived algebraic spaces (\S \ref{ss:dAS:dual}) to 
obtain the global one.

\subsubsection{Localized $\infty$-sites and localized $\infty$-topoi}

As a preliminary,
let us introduce \emph{localized $\infty$-sites} 
explain the relation to the localized $\infty$-topoi.

Let $(\iC,\tau)$ be an $\infty$-site, and $X \in \iC$ be an object.
Then we can attach to the over-$\infty$-category $\oc{\iC}{X}$ 
a new Grothendieck topology $\tau_X$ by setting the covering sieve to be 
\[
 \Cov_{\tau_X}(\varphi) :=
 \{\varphi^* \oc{\iC}{X}^{(0)} \mid \oc{\iC}{X}^{(0)} \in \Cov_{\tau}(X)\}
\]
for each $\varphi \in \oc{\iC}{X}$.
Let us name the obtained $\infty$-site.

\begin{dfn}\label{dfn:LE:ls}
The $\infty$-site $(\oc{\iC}{X},\tau_X)$ is called 
the \emph{localized $\infty$-site} on $X$.
\end{dfn}

In \cite[Definition 3.3.1]{TVe1} the corresponding $S$-site is 
called the comma $S$-site.

The following statement is checked directly by definition so we omit the proof.

\begin{lem}\label{lem:LE:lsf}
Let $(\iC,\tau)$ be an $\infty$-site and $f: X \to Y$ be a morphism in $\iC$.
Then composition with $f$ induces a continuous functor 
$(\oc{\iC}{X},\tau_X) \to (\oc{\iC}{Y},\tau_Y)$ 
of $\infty$-sites (Definition \ref{dfn:dAS:cont}).
\end{lem}

Recall that for an $\infty$-topos $\tT$ and $U \in \tT$, we have 
the localized $\infty$-topos $\rst{\tT}{U}$ of $\tT$ at $U$ 
(Fact \ref{fct:pr:ovc}).
Thus, given an $\infty$-site $(\iC,\tau)$ and an object $X \in \iC$, 
we can consider the localized $\infty$-topos $\rst{\iSh(\iC,\tau)}{\yon(X)}$.
Here $\yon: \iC \to \iShv(\iC,\tau)$ is the $\infty$-categorical Yoneda embedding
(Definition \ref{dfn:pr:Yoneda}).
On the other hand, we have the localized $\infty$-site $(\oc{\iC}{X},\tau_X)$ 
and the associated $\infty$-topos $\iSh(\oc{\iC}{X},\tau_X)$.

\begin{lem}\label{lem:LE:ls}
For an $\infty$-site $(\iC,\tau)$ and $X \in \iC$, we have an equivalence
$\iSh(\oc{\iC}{X},\tau_X) \simeq \rst{\iSh(\iC,\tau)}{\yon(X)}$.
\end{lem}

\begin{proof}
The forgetful functor $\oc{\iC}{X} \to \iC$ induces a continuous functor 
$(\oc{\iC}{X},\tau_X) \to \iSh(\iC,\tau)$ of $\infty$-sites.
Then by Proposition \ref{prp:dAS:sgm}, 
we have a geometric morphism $\iota^*: \iSh(\oc{\iC}{X},\tau_X) \to \iSh(\iC,\tau)$ 
of $\infty$-topoi (recall Notation \ref{rmk:is:gm}).
By definition of $(\oc{\iC}{X},\tau_X)$,
the geometric morphism $\iota^*$ factors through $\rst{\iSh(\iC,\tau)}{\yon(X)}$,
and the factored geometric morphism
$\iSh(\oc{\iC}{X},\tau_X) \to \rst{\iSh(\iC,\tau)}{\yon(X)}$ gives an equivalence.
\end{proof}

Recall that a continuous functor of $\infty$-sites induces a geometric morphism of 
the associated $\infty$-topoi (Proposition \ref{prp:dAS:sgm}).
Thus the continuous functor
$(\oc{\iC}{X},\tau_X) \to (\oc{\iC}{Y},\tau_Y)$ 
in Lemma \ref{lem:LE:lsf} induces a geometric morphism
$\iSh(\oc{\iC}{X},\tau_X) \to \iSh(\oc{\iC}{Y},\tau_Y)$.
Using the equivalence of Lemma \ref{lem:LE:ls}, we denote it by 
$\alpha: \rst{\iSh(\iC,\tau)}{\yon(X)} \to \rst{\iSh(\iC,\tau)}{\yon(Y)}$.

On the other hand, for any morphism $U \to V$ in a $\infty$-category $\iB$,
we have a functor $\oc{\iB}{U} \to \oc{\iB}{V}$ (Corollary \ref{cor:ic:ovc})
Thus we have a functor 
$\beta: \rst{\iSh(\iC,\tau)}{\yon(X)} \to \rst{\iSh(\iC,\tau)}{\yon(Y)}$.
Then we have  

\begin{lem}\label{lem:LE:a=b}
The functors $\alpha$ and $\beta$ are equivalent.
\end{lem}

\subsubsection{Gluing \'etale dualizing objects}
\label{sss:LE:dual}

We impose Assumption \ref{asp:dAS:kL} 
on the base ring $k$ and the commutative ring $\Lambda$.
Namely,
\begin{itemize}[nosep]
\item 
The base ring $k$ is a noetherian ring and 
has a dualizing complex $\Omega_k$.
\item
The commutative ring $\Lambda$ is a torsion noetherian ring 
annihilated by an integer invertible in $k$.
\end{itemize}

Let $\stX$ be a geometric derived stack locally of finite presentation 
(Definition \ref{dfn:dSt:lfp}).
Let $A: \stU \to \stX$ be an object of $\idAS_{\stX}^{\lis}$,
the underlying $\infty$-category of the $\infty$-site $\gsLE(\stX)$.
Then the derived algebraic space $\stU$ satisfies Assumption \ref{asp:dAS:kL}:
\begin{itemize}[nosep]
\item
The derived algebraic space $\stW$ is separated,
quasi-compact and locally of finite presentation.
\end{itemize}
Thus we can apply the argument in \S \ref{ss:dAS:dual} to the composition 
$\alpha:= (\stU \xr{A} \stX \to \dSpec k)$,
and have the dualizing object $\Omega_{\alpha} \in \iDa(\stU_{\et},\Lambda)$
(Notation \ref{ntn:dAS:dual}).
Recalling also Notation \ref{ntn:dAS:Tate} on the shift and the Tate twist, we define
\[
 K_A := \Omega_{\alpha}\tsh{-d_A} \in \iDc(\stU_{\et},\Lambda),
\]
where $d_A$ is the relative dimension of the smooth morphism $A$ 
(Definition \ref{dfn:dSt:reldim}).
Note that $d_A$ is locally constant,
and that $K_A$ is of finite injective dimension by Lemma \ref{lem:dAS:dual}.

Now the proof of \cite[3.2.1 Lemma]{LO1} works 
with the help of Lemma \ref{lem:dAS:3.1.2}, and we have

\begin{lem}\label{lem:LE:3.2.1}
Given a triangle
\[
 \xymatrix{ \stV \ar[rr]^{\sigma} \ar[rd]_B & & \stU  \ar[ld]^A \\ & \stX}
\]
in $\idAS_{\stX}^{\lis}$,
we have an equivalence $\sigma^* K_A \simeq K_B$.
\end{lem}

We now want to construct a dualizing object for the lisse-\'etale 
by gluing \'etale dualizing data using Lemma \ref{lem:LE:3.2.1}.
Let $A: \stU \to \stX$ be an object of $\gsLE(\stX)$.
We denote by $\rst{\gsLE(\stX)}{\stU}$ the localized $\infty$-site
(Definition \ref{dfn:LE:ls}).
Then the inclusion $\gsEt(\stU) \inj \rst{\gsLE(\stX)}{\stU}$
of $\infty$-sites is a continuous functor,
and we can apply Proposition \ref{prp:dAS:sgm} to it.
Thus we have

\begin{lem}\label{lem:LE:ve}
The inclusion $\gsEt(\stU) \inj \rst{\gsLE(\stX)}{\stU}$
induces a geometric morphism 
\[
 \ve_\stU: \rst{\stX_{\txle}}{\stU} \longto \stU_{\et}
\]
of $\infty$-topoi.
We denote the corresponding adjunction by
$\ve_{\stU}^{-1}: \stU_{\et} \rlto \rst{\stX_{\txle}}{\stU} :\ve^{\stU}_*$.
\end{lem}

We can describe these functors more explicitly.
Note that giving a sheaf $\shG \in \rst{\stX_{\txle}}{\stU}$ is equivalent to 
giving sheaves $\shG_V \in V_{\et}$ for each affine derived scheme $V$ over $\stU$ 
such that the composite $V \to \stU \xr{A} \stX$ is smooth,
which satisfy the gluing condition.
Then, for a given $\shF \in \stU_{\et}$, 
the sheaf $\shG=\ve_{\stU}^{-1}\shF \in \rst{\stX_{\txle}}{\stU}$ corresponds to 
$\shG_V = \pi_V^{-1}\shF$ for each $(\pi_V: V \to \stU) \in \rst{\gsLE(\stX)}{\stU}$.
On the other hand, for a sheaf $\shG \in \rst{\stX_{\txle}}{\stU}$ corresponding to 
$\{\shG_V\}_{V \to \stU}$, the sheaf $\ve^{\stU}_* \shG \in \stU_{\et}$ is given by
$\ve^{\stU}_* \shG = \shG_{\stU}$.
In particular, the functor $\ve^{\stU}_*$ is left and right exact.

Given a morphism $f: \stU \to \stV$ in the underlying $\infty$-category 
$\idAS_{\stX}^{\lis}$ of the $\infty$-site $\gsLE(\stX)$,
we have a square
\begin{align}\label{diag:LE:dual}
 \xymatrix{
  \rst{\stX_{\txle}}{\stU} \ar[r]^(0.6){\ve_{\stU}} \ar[d]_{f} 
  & \stU_{\et} \ar[d] \\
  \rst{\stX_{\txle}}{\stV} \ar[r]_(0.6){\ve_{\stU}} & \stV_{\et}  }
\end{align}
in the $\infty$-category $\iRTop$ of 
$\infty$-topoi and geometric morphisms.
Here $f: \rst{\stX_{\txle}}{\stU} \to \rst{\stX_{\txle}}{\stV}$ denotes 
the functor induced by $f$ (Corollary \ref{lem:LE:lsf}, Lemma \ref{lem:LE:a=b}).

Recalling that we are given $A: \stU \to \stX$, let us now define
\[
 \kappa_A := \ve_{\stU}^*(K_A) \in \iDa(\rst{\stX_{\txle}}{\stU},\Lambda),
\]
where we denote by 
$\ve_{\stU}^*: \iDa(\stU_{\et},\Lambda) \rlto
 \iDa(\rst{\stX_{\txle}}{\stU},\Lambda): f_* \ve^{\stU}$
the geometric morphism of the derived $\infty$-categories induced by $\ve_{\stU}$.
Then, given a triangle as in Lemma \ref{lem:LE:3.2.1},
we have $f^*(\kappa_B) \simeq \kappa_A$ since \eqref{diag:LE:dual} is a square.
Recall Notation \ref{ntn:is:shExt} of $\shExt$.
We denote by $\shExt_{\stU_{\et}}^i$ 
the functor $\shExt_{\Lambda}$ on $\iDa(\stU_{\et})$,
and by $\shExt^i_{\rst{\stX_{\txle}}{\stU}}$
that on $\iDa(\rst{\stX_{\txle}}{\stU})$.
Then we have

\begin{lem}\label{lem:6op:341}
\begin{enumerate}[nosep]
\item 
For any $\shM, \shN \in \iDa(\stU_{\et})$ and $i \in \bbZ$, we denote by 
$\shExt^i(\ve_{\stU}^* \shM,\ve_{\stU}^* \shN)_{\stU} 
 \in \iMod_{\Lambda}(\stU_{\et})$
the restriction of 
$\shExt^i_{\rst{\stX_{\txle}}{\stU}}(\ve_{\stU}^* \shM,\ve_{\stU}^* \shN)$ to 
$\stU_{\et}$.
Then we have an equivalence
\[
 \shExt^i(\ve_{\stU}^* \shM,\ve_{\stU}^* \shN)_{\stU} \simeq 
 \shExt^i_{\stU_{\et}}(\shM,\shN).
\]

\item
For any $A: \stU \to \stX$ in $\idAS_{\stX}^{\lis}$ we have
\[
 \shHom_{\Lambda}(\kappa_A,\kappa_A) = \Lambda  
\]
in $\iDa(\rst{\stX_{\txle}}{\stU},\Lambda)$.
In particular, we have $\shExt^i_{\Lambda}(\kappa_A,\kappa_A) = 0$
in $\iMod_{\Lambda}(\rst{\stX_{\txle}}{\stU})$ for any $i \in \bbZ_{<0}$.
\end{enumerate}
\end{lem}

We omit the detail of the proof since 
(1) can be shown by the same argument of \cite[3.4.1. Lemma]{LO1}
and (2) is a corollary of (1).

Now choose a smooth presentation $p: \stX_{\bullet} \to \stX$.
Then we have an object $\kappa_p \in \iDa(\rst{\stX_{\txle}}{\stX_0},\Lambda)$ 
together with the descent data to $\stX_{\txle}$.
Thus the gluing lemma (Fact \ref{fct:is:glue}) can be applied,
and we have an object $\Omega_\stX(p) \in \iDa(\stX_{\txle},\Lambda)$.
Since $\kappa_A$ is of finite injective dimension for any $A: \stU \to \stX$ 
in $\idAS_{\stX}^{\lis}$, the restriction $\rst{\kappa_p}{\stU} \simeq \kappa_A$
is bounded in both direction.

\begin{ntn}\label{ntn:LE:(*)}
We denote by
\[
 \iDcu{(*)}(\stX_{\txle},\Lambda) \subset \iDc(\stX_{\txle},\Lambda)
 \quad (* \in \{+,-,b\})
\]
the full sub-$\infty$-category spanned by those objects $\shM$
such that the restriction $\rst{\shM}{\stU}$ is in $\iDcu{*}(\stX_{\txle},\Lambda)$
for any quasi-compact open immersion $\stU \inj \stX$ 
(Definition \ref{dfn:dSt:qcpt}, \ref{dfn:dSt:open}).
\end{ntn}

By the above argument, we have 

\begin{lem*}
There exists $\Omega_\stX(p) \in \iDa(\stX_{\txle},\Lambda)$
inducing $\kappa_A$ for any $A \in \rst{\gsLE(\stX)}{\stX_0}$.
Moreover it is unique up to contractible ambiguity.
\end{lem*}

We can show that independence of the presentation $p$, for example, 
by constructing a new presentation from given two presentations.
Thus the following definition makes sense. 

\begin{dfn}\label{dfn:LE:dual}
The \emph{dualizing object} 
\[
 \Omega_{\stX} \in \iDa(\stX_{\txle},\Lambda)
\]
of $\stX$ is defined to be $\Omega_{\stX}(p)$
with $p$ a smooth presentation of $\stX$.
It is well-defined up to contractible ambiguity,
and is characterized by $\rst{\Omega_{\stX}}{\stU}=\ve_{\stU}^* K_A$
for $(A: \stU \to \stX) \in \idAS_{\stX}^{\lis}$.
\end{dfn}

\subsubsection{Biduality}

We impose the same conditions on $k$ and $\Lambda$ 
as in the previous \S \ref{sss:LE:dual} (see also Assumption \ref{asp:dAS:kL}).
Let $\stX$ be a geometric derived stack which is locally of finite presentation.

For $(A: \stU \to \stX) \in \idAS_{\stX}^{\lis}$,
we denote by $K_A \in \iDa(\stU_{\et},\Lambda)$ the dualizing object 
for the derived algebraic space $\stU$.
Also, for $\shM \in \iDa(\stX_{\txle},\Lambda)$, we use Notation \ref{ntn:LE:FU} to 
denote by $\shM_{\stU} \in \iDa(\stU_{\et},\Lambda)$ the restriction of $\shM$.
Then by the argument in \cite[3.5]{LO1} we have 

\begin{fct}[{\cite[3.5.2. Lemma]{LO1}}]\label{fct:LE:3.5.2}
Let
\[
 \xymatrix{ \stV \ar[rr]^{f} \ar[rd]_B & & \stU  \ar[ld]^A \\ & \stX}
\]
be a commutative triangle in $\idAS_{\stX}^{\lis}$,
and let $\shM \in \iDa(\stX_{\txle},\Lambda)$.
\begin{enumerate}[nosep]
\item 
In $\iDa(\stV_{\et},\Lambda)$ we have 
$f^* \! \shHom_{\Lambda}(\shM_{\stU},K) 
      = \shHom_{\Lambda}(f^* \shM_{\stU}, f^* K_A)
      = \shHom_{\Lambda}(f^* \shM_{\stU},K_B)$.

\item
For any $\shM \in \iDa(\stX_{\txle},\Lambda)$, we have 
$\shHom_{\Lambda}(\shM_{\stU},K_A) \in \iDc(\stU_{\et},\Lambda)$.
\end{enumerate}
\end{fct}

For $\stU \in \idAS_{\stX}^{\lis}$,
we denote by $\ve: \rst{\stX_{\txle}}{\stU} \to \stU_{\et}$
the geometric morphism in Lemma \ref{lem:LE:ve}.
Let also $\Omega_{\stX}$ be the dualizing object for the derived stack $\stX$ 
(Definition \ref{dfn:LE:dual}).

\begin{fct}[{\cite[3.5.3. Lemma]{LO1}}]\label{fct:LE:3.5.3}
Let $(A: \stU \to \stX) \in \idAS_{\stX}^{\lis}$ and 
$\shM \in \iDc(\stX_{\txle})$. 
Then we have an equivalence 
$\ve^* \! \shHom_{\Lambda}(\shM_{\stU},K_A) \simeq 
 \shHom_{\Lambda}(\shM,\Omega_{\stX})_{\stU}$.
\end{fct}

By Fact \ref{fct:LE:3.5.2} and \ref{fct:LE:3.5.3} we have

\begin{lem*}
For $\shM \in \iDc(\stX_{\txle},\Lambda)$, 
we have $\shHom_{\Lambda}(\shM,\Omega_{\stX}) \in \iDc(\stX_{\txle},\Lambda)$.
\end{lem*}

Now we can introduce

\begin{ntn}\label{ntn:LE:dual}
The \emph{dualizing functor} is defined to be
\[
 \funD_{\stX} := \shHom(-,\Omega_{\stX}): 
 \iDc(\stX_{\txle},\Lambda) \longto \iDc(\stX_{\txle},\Lambda)^{\op}.
\]
\end{ntn}

By the standard argument using Fact \ref{fct:LE:3.5.2} and \ref{fct:LE:3.5.3},
we have 

\begin{prp}[{\cite[3.5.8, 3.5.9. Proposition]{LO1}}]\label{prp:LE:dual}
\begin{enumerate}[nosep]
\item 
The natural morphism $\id \to \funD_{\stX} \circ \funD_{\stX}$
induced by the biduality morphism 
$\shM \to \shHom_{\Lambda}(\shHom_{\Lambda}(\shM,\shN),\shN)$
is an equivalence.

\item
For any $\shM,\shN \in \iMod^{\stab}_{\cstr}(\stX_{\txle},\Lambda)$
we have a canonical equivalence
\[
 \shHom_\Lambda(\shM,\shN) \simeq \shHom_\Lambda(\funD(\shN),\funD(\shM)).
\]
\end{enumerate}
\end{prp}

\section{Derived functors of constructible sheaves with finite coefficients}
\label{s:6op}

In this section we introduce $\infty$-theoretic analogue of 
the Grothendieck's six operations on the derived categories of constructible
sheaves with finite coefficients.
Let $k$ be the base commutative ring, and 
$\Lambda$ be a Gorenstein local ring of dimension $0$ and characteristic $\ell$.
We assume that $\ell$ is invertible in $k$.
All the derived stacks appearing in this section are defined over $k$.

\subsection{Derived category of constructible sheaves}

Let us begin with the recollection of the notations for derived $\infty$-categories
of constructible lisse-\'etale sheaves on derived stacks.
Let $\stX$  be a geometric derived stack over $k$ 
which is locally of finite presentation.
Then we have the full sub-$\infty$-categories
\[
 \iDcu{*}(\stX_{\txle},\Lambda)
 \subset \iDu{\cart,*}(\stX_{\txle},\Lambda)
 \subset \iDu{*}(\stX_{\txle},\Lambda)
 \quad (* \in \{\emptyset,+,-,b\})
\]
spanned by constructible (resp.\ cartesian) objects in 
the derived $\infty$-categories of lisse-\'etale $\Lambda$-modules on $\stX$
with prescribed bound conditions.
These are stable $\infty$-categories equipped with $t$-structure
by Theorem \ref{thm:LE:pi-ve} and \ref{thm:LE:iDc-stb}.

We also denote
\[
 \dD^{*}(\stX_{\txle},\Lambda) :=  \Ho \iDu{*}(\stX_{\txle},\Lambda), \quad 
 \dD^{*}_{\cstr}(\stX_{\txle},\Lambda) :=  \Ho \iDcu{*}(\stX_{\txle},\Lambda).
\]
These are triangulated categories by Fact \ref{fct:stb:stb},
and the obvious embedding 
$\dD_{\cstr}^*(\stX,\Lambda) \inj \dD^*(\stX,\Lambda)$
is a triangulated functor.

\begin{rmk*}
We have an obvious relation between the derived category
for a derived stack and that for an ordinary algebraic stack.
Let $X$ be an algebraic stack over $k$, and 
$\iota(X)$ be the associated derived stack (Definition \ref{dfn:dSt:St-dSt}).
Then by the construction we find that the category 
$\dD^*(\iota(X)_{\txle},\Lambda)$ 
(resp.\ $\dD^*_{\cstr}(\iota(X)_{\txle},\Lambda)$)
is equivalent to the derived category of complexes of $\Lambda$-modules 
(resp.\ the derived category of complexes of $\Lambda$-modules 
 with constructible cohomology sheaves).
\end{rmk*}

Let us also recall Notation \ref{ntn:LE:(*)}, where we denoted by
\[
 \iDcu{(*)}(\stX_{\txle},\Lambda) \subset 
 \iDc(\stX_{\txle},\Lambda) \quad (* \in \{+,-,b\})
\]
the full sub-$\infty$-category consisting of objects $\shM$
such that for any quasi-compact open immersion $\stU \inj \stX$
the restriction $\rst{\shM}{\stU}$ belongs to 
$\iDc^{*}(\stU_{\txle},\Lambda)$.
The homotopy categories will be denoted by 
$\dD_{\cstr}^{(*)}(\stX_{\txle},\Lambda) := 
 \Ho \iDcu{(*)}(\stX_{\txle},\Lambda)$.

Hereafter we often suppress the symbol $\Lambda$ to denote
$\iDc(\stX_{\txle}) := \iDc(\stX_{\txle},\Lambda)$ and so on.

\subsection{Derived direct image functor}

Let us recall the derived direct image functor for ordinary algebraic stacks
\cite[4.1]{LO1}.
Let $f: X \to Y$ be a morphism of finite type between algebraic stacks 
locally of finite type over $k$
(actually we can relax the condition on the base scheme).
We denote by $\dD(X)$ the derived category of complexes of $\Lambda$-modules 
on the algebraic stack $X$. 
Then the derived direct image functor $\dR f_*: \dD(X) \to \dD(Y)$
does not map the subcategory $\dD_{\cstr}(X)$ 
of complexes with constructible cohomology sheaves to $\dD_{\cstr}(Y)$.
However, we have $\dR f_*: \dD_{\cstr}^{(+)}(X) \to \dD_{\cstr}^{(+)}(Y)$,
where $\dD_{\cstr}^{(+)}(X) \subset \dD(X)$ is 
the full subcategory defined similarly as Notation \ref{ntn:LE:(*)}.

Now we consider the case of derived stacks.
Recall Proposition \ref{prp:is:df} which gives derived functors between
derived $\infty$-categories for general ringed $\infty$-topoi.
We then have the first half of the next proposition. 

\begin{prp}\label{prp:6op:pf}
Let $f: \stX \to \stY$ be a morphism of geometric derived stacks.
Then we have the direct image functor
\[
 f_*: \iDa(\stX_{\txle}) \longto \iDa(\stY_{\txle}),
 \quad (f_* \shF)(U)= \shF(U \times_{\stY}\stX)
\]
which is a $t$-exact functor of stable $\infty$-categories 
equipped with $t$-structures.
If moreover $f$ is quasi-compact (Definition \ref{dfn:dSt:qcpt}),
then $f_*$ restricts to a functor 
\[
 f_*: \iDcbp(\stX_{\txle}) \longto \iDcbp(\stY_{\txle}).
\]
\end{prp}


Recall Notation \ref{ntn:LE:FU} on the restriction of lisse-\'etale sheaves.
We need the following analogue of \cite[Proposition 9.8]{O}.
The proof is almost the same, and we omit it.

\begin{lem*}
Let $f: \stX \to \stY$ be a morphism of geometric derived stacks,
$\shM \in \iMod_{\Lambda}(\stX_{\txle})$, and $U \in \idAS_{\stY}^{\lis}$.
\begin{enumerate}[nosep]
\item
Assume $f$ is representable.
Then we have an equivalence
\[
 (f_* \shF)_{\stU} \simeq f_{\stU_{\et},*} \shM_{\stX \times_{\stY} \stU}
\]
in $\iDa(\stU_{\et})$.
Here $f_*: \iDa(\stX_{\txle}) \to \iDa(\stY_{\txle})$ 
denotes the derived direct image functor,
and $f_{\stU_{\et}}: (\stX \times_{\stY} \stU)_{\et} \to \stU_{\et}$
is the geometric morphism of $\infty$-topoi induced by 
the projection $\stX \times_{\stY} \stU \to \stU$.

\item
Let $\stX_{\bullet} \to \stX$ be a smooth presentation, and 
$f_{\stU,n}: \stX_n \times_{\stY} \stU \to \stU$ be the morphism induced by $f$.
Assume $\shM$ is cartesian.
Then there is a spectral sequence
\[
 E_1^{p q} = R^q (f_{\stU,n})_* \shM_{\stX_p \times_{\stY} \stU} 
 \Longrightarrow (R^{p+q} f_* \shM)_{\stU}.
\]
\end{enumerate}
\end{lem*}

\begin{proof}[{Proof of Proposition \ref{prp:6op:pf}}]
It is enough to show the second half. 
By Lemma \ref{lem:LE:3.8} of the criterion of cartesian property
and by Theorem \ref{thm:LE:pi-ve} of the equivalence 
$\iDa(\stY_{\txle}) \simeq \iDa(\stY_{\bullet,\et})$,
we can take a smooth presentation $\stY_{\bullet} \to \stY$
and replace $\stY$ by a quasi-compact derived algebraic space.
Take now a smooth presentation $\stX_{\bullet} \to \stX$ 
with $\stX_0$ quasi-compact.
Then the spectral sequence in the above Lemma implies that
it is enough to show the result for each morphism $\stX_n \to \stY$.
Thus we may assume that $\stX$ and $\stY$ are 
quasi-compact derived algebraic spaces.
Since Definition \ref{dfn:LE:cstr:L} of constructible sheaves
only depends on the truncated data, 
we can reduce to the non-derived setting,
where the result is shown in \cite[Proposition 9.9]{O}.
\end{proof}

On the homotopy category we denote the derived functor by 
\[
 \dR f_*: \dD^{(+)}_{\cstr}(\stX_{\txle}) \longto \dD^{(+)}_{\cstr}(\stY_{\txle}).
\]

\begin{rmk*}
If we set $\stX=\iota(X)$ and $\stY=\iota(Y)$ 
with $X$ and $Y$ algebraic stacks over $k$,
then we recover the construction in \cite{O,LO1} of the derived direct image 
\[
 \dR f_*: \dD^{(+)}_{\cstr}((\iota X)_{\txle},\Lambda) \longto 
          \dD^{(+)}_{\cstr}((\iota Y)_{\txle},\Lambda).
\]
\end{rmk*}

\subsection{Derived inverse image}

We construction the derived inverse image with the help of 
simplicial description of cartesian sheaves.
Let $f: \stX \to \stY$ a morphism of geometric derived stacks 
which are locally of finite presentation.
Take an $n$-atlas $\{Y_i\}_{i \in I}$ of $\stY$ with some $n$,
and denote by $\stY_0= \coprod_{i \in I} Y_i \to \stY$ the natural surjection.
Let us also take an $n$-atlas $\{X_j\}_{j \in J}$ of the fiber product 
$\wt{\stX} := \stX \times_{\stY} \stY_0$ and denote 
$\stX_0 := \coprod_{j \in J} X_j \to \wt{\stX}$.
Thus we have a square $\idSt_k$: 
\[
 \xymatrix{
  \wt{\stX} \ar[r] \ar[d] & \stX \ar[d]^{f} \\ 
  \stY_0    \ar[r]        & \stY}
\]
Let $e_{\stX}: \stX_{\bullet} \to \wt{\stX}$ and 
$e_{\stY}: \stY_{\bullet} \to \stY$ be the smooth presentations
with $\stX_0 := \coprod_{j \in J} X_j$ and $\stY_0= \coprod_{i \in I} Y_i$.
The morphism $f: \stX \to \stY$ induces 
a morphism $\stX_0 \to \stY_0$ of derived algebraic spaces,
and it further induces a morphism $f_{\bullet}: \stX_\bullet \to \stY_\bullet$
of simplicial derived algebraic spaces.
Restricting to $\CS^{\str} \subset \CS$, we have the following square
of strictly simplicial derived stacks:
\[
 \xymatrix{
  \stX^{\str}_\bullet  \ar[r]^{e_\stX} \ar[d]_{f_\bullet} & \stX \ar[d]^{f} \\
  \stY^{\str}_\bullet  \ar[r]_{e_\stY}                    & \stY}
\]
The morphism $f_{\bullet}$ induces a geometric morphism 
\[
 f_{\bullet,\et}: \stX^{\str}_{\bullet,\et} \longto \stY^{\str}_{\bullet,\et}
\]
of $\infty$-topoi, which induces a functor
\[
 f_\bullet^*: \iMod_{\cstr}(\stY^{\str}_{\bullet,\et}) \longto 
              \iMod_{\cstr}(\stX^{\str}_{\bullet,\et}).
\]
Here we denoted 
$\iMod_{\cstr}(\stY^{\str}_{\bullet,\et}) := 
 \iMod_{\cstr}(\stY^{\str}_{\bullet,\et},\Lambda)$.
On the other hand, by Proposition \ref{prp:LE:Modc},
we have equivalences 
\[
 r_{\stX}: 
 \iMod_{\cstr}(\stX_{\txle}) \longsimto 
 \iMod_{\cstr}(\stX^{\str}_{\bullet,\et}),
 \quad
 r_{\stY}: 
 \iMod_{\cstr}(\stY_{\txle}) \longsimto 
 \iMod_{\cstr}(\stY^{\str}_{\bullet,\et}).
\]

\begin{dfn*}
We define 
\[
 f^*: \iMod_{\cstr}(\stY_{\txle}) \longto \iMod_{\cstr}(\stX_{\txle})
\]
{}to be the composition $r_{\stX}^{-1} \circ f_\bullet^* \circ r_{\stY}$,
and call it the \emph{inverse image functor}.
\end{dfn*}

The functor $f^*$ is independent of the choice of $n$-atlases 
$\{Y_i\}_{i \in I}$ and $\{X_j\}_{j \in J}$ up to contractible ambiguity.

\begin{lem}
$f^*$ is a left adjoint of the functor 
$f_*: \iMod_{\cstr}(\stX_{\txle}) \to \iMod_{\cstr}(\stY_{\txle})$, and moreover 
$f^*$ is left exact in the sense of Definition \ref{dfn:ic:exact}.
\end{lem}

\begin{proof}
Since on the simplicial level we have a geometric morphism
$f_{\bullet,\et}: \stX^{\str}_{\bullet,\et} \to \stY^{\str}_{\bullet,\et}$,
the induced functor $f_\bullet^*$ actually sits in an adjunction
\[
 f_\bullet^*: \iMod_{\cstr}(\stY^{\str}_{\bullet,\et}) \longto 
              \iMod_{\cstr}(\stX^{\str}_{\bullet,\et}): 
 f_{\bullet,*},
\]
and $f_{\bullet,*}$ is equivalent to the functor induced by
the direct image functor $f^{\et}_*$.
Thus we have the conclusion.
\end{proof}

The same argument works for the derived $\infty$-category.
Namely, the geometric morphism 
$f_{\bullet,\et}: \stX^{\str}_{\bullet,\et} \to \stY^{\str}_{\bullet,\et}$ 
induces an adjunction 
\[
 f_\bullet^*: \iDc(\stY^{\str}_{\bullet,\et}) \adjunc
              \iDc(\stX^{\str}_{\bullet,\et}) :f_{\bullet,*},
\]
and by Theorem \ref{thm:LE:iDc-stb} we have also equivalences 
\[
 r_{\stX}: 
 \iDc(\stX_{\txle}) \longsimto \iDc(\stX^{\str}_{\bullet,\et}),
 \quad
 r_{\stY}: 
 \iDc(\stY_{\txle}) \longsimto \iDc(\stY^{\str}_{\bullet,\et}).
\]

\begin{dfn}\label{dfn:6op:pb}
We define 
\[
 f^*: \iDc(\stY_{\txle}) \longto \iDc(\stX_{\txle})
\]
{}to be the composition $r_{\stX}^{-1} \circ f_\bullet^* \circ r_{\stY}$,
and call it the \emph{derived inverse image functor}.
\end{dfn}

Obviously we have that 
the derived functor $f^*: \iDc(\stY_{\txle}) \to \iDc(\stX_{\txle})$
is a $t$-exact extension of the functor 
$f^*: \iMod_{\cstr}(\stY_{\txle},\Lambda) \to \iMod_{\cstr}(\stX_{\txle},\Lambda)$.
On the homotopy category we denote the derived functor by 
\[
 \dL f^*: \dD_{\cstr}(\stY_{\txle}) \longto \dD_{\cstr}(\stX_{\txle}).
\]

\begin{rmk*}
If we set $\stX=\iota(X)$ and $\stY=\iota(Y)$ 
with $X$ and $Y$ algebraic stacks over $k$,
then we recover the construction in \cite{O,LO1} of the derived inverse image 
\[
 \dL f_*: \dD_{\cstr}((\iota X)_{\txle},\Lambda) \longto 
          \dD_{\cstr}((\iota Y)_{\txle},\Lambda).
\]
\end{rmk*}

For later use, we record

\begin{lem}\label{lem:6op:fOm}
Let $\stX$ and $\stY$ be geometric derived stacks locally of finite presentation,
and $f: \stX \to \stY$ be a smooth morphism of relative dimension $d$.
Then $f^*\Omega_{\stY} \simeq \Omega_{\stX}\tsh{-d}$
\end{lem}

\begin{proof}
Take a smooth presentation $\stY_{\bullet} \to \stY$
and consider $\wt{\stX} := \stX \times_{\stY} \stY_0$.
Then we have a diagram 
\[
 \xymatrix{
  \stX_0 \ar[r] \ar[rd] & \wt{\stX} \ar[r] \ar[d] & \stX \ar[d]^{f} \\ 
                        & \stY_0    \ar[r]        & \stY}
\]
of derived stacks.
By the assumption on $f$, the morphism $\stX_0 \to \stY_0$ is smooth.
Consider the restrictions of $f^*\Omega_{\stY}$ and $\Omega_{\stX}\tsh{-d}$ to
$\rst{\stX_{\txle}}{\stX_0}$ 
coincide and have zero negative $\shExt$'s.
Then Lemma \ref{lem:is:234} implies the consequence.
\end{proof}

\subsection{Derived internal Hom functor}


Recall that in \S \ref{ss:is:sm}
we discussed the internal Hom functor 
\begin{align*}
 \shHom_{\shA}(-,-): 
 \iMod^{\stab}_{\shA}(\tT)^{\op} \times \iMod^{\stab}_{\shA}(\tT) 
 \longto \iMod^{\stab}_{\shA}(\tT)
\end{align*}
for a ringed $\infty$-topos $(\tT,\shA)$.
Applying it to $(\tT,\shA)=(\stX_{\txle},\Lambda)$, we have

\begin{ntn*}
For a geometric derived stack $\stX$.
we denote the internal Hom functor on $\iDa(\stX_{\txle})$ by
\[
 \shHom(-,-): 
 \iDa(\stX_{\txle})^{\op} \times \iD(\stX_{\txle}) \longto \iDa(\stX_{\txle}).
\]
We also denote it by $\shHom_{\stX_{\txle}}$ to 
emphasize the dependence on $\stX_{\txle}$.
The associated functor on the homotopy category is denoted by 
\[
 \dRHom(-,-): 
 \dD(\stX_{\txle})^{\op} \times \dD(\stX_{\txle}) \longto \dD(\stX_{\txle}).
\]
\end{ntn*}

As for the constructible objects, we have 

\begin{lem*}
Let $\stX$ be a geometric derived stack locally of finite presentation.
Then for $\shM \in \iDcbm(\stX_{\txle})$ and 
$\shN \in \iDcbp(\stX_{\txle})$, we have 
$\shHom(\shM, \shN) \in \iDcbp(\stX_{\txle})$.
\end{lem*}

\begin{proof}
We follow the argument in \cite[4.2.1. Lemma, 4.2.2. Corollary]{LO1}.
For $\shM,\shN \in \iDc(\stX_{\txle})$
and $\stU \in \idAS_{\stX}^{\lis}$, we have a functorial equivalence
\[
 \rst{\shHom_{\stX_{\txle}}(\shM,\shN)}{\stU_{\et}}
 \simeq \shHom_{\stU_{\et}}(\shM_\stU,\shN_\stU).
\]
Indeed, for the geometric morphism
$\ve: \rst{\stX_{\txle}}{\stU} \to \stU_{\et}$
in Lemma \ref{lem:LE:ve}, the counit transformation 
$\ve^* \ve_* \shM \to \shM$ and $\ve^* \ve_* \shN \to \shN$ are equivalences
in $\iDa(\rst{\stX_{\txle}}{\stU})$
since $\shM$ and $\shN$ have constructible homotopy groups
and by by Proposition \ref{prp:is:223}.
On the other hand, the unit transformation $\id \to \ve_* \ve^*$ is an equivalence,
so we have 
\begin{align*}
 \shHom_{\rst{\stX_{\txle}}{\stU}}(\shM,\shN) & \simeq 
 \shHom_{\rst{\stX_{\txle}}{\stU}}(\ve^* \ve_* \shM,\ve^* \ve_* \shN) \simeq 
 \shHom_{\stU_{\et}}(\ve^* \ve_* \shM,\ve^* \ve_* \shN) \\ & \simeq 
 \shHom_{\stU_{\et}}(\ve_* \shM,\ve_* \ve^* \ve_* \shN) \simeq 
 \shHom_{\stU_{\et}}(\ve_* \shM,\ve_* \shN) \simeq
 \shHom_{\stU_{\et}}(\shM_{\stU},\shN_{\stU}).
\end{align*}
Now consider the following triangle 
with $f$ a morphism in $\idAS_{\stX}^{\lis,\fp}$.
\[
 \xymatrix{ \stV \ar[rr]^{f} \ar[rd]_B & & \stU  \ar[ld]^A \\ & \stX}
\]
Recall the functor $f^*: \iDc(\stU_{\et}) \to \iDc(\stV_{\et})$
(Definition \ref{dfn:6op:pb}).
Then for $\shM \in \iDcbm(\stX_{\txle})$ and $\shN \in \iDcbp(\stX_{\txle})$, 
we have $\shM_{\stU} \in \iDcm(\stU_{\et})$, $\shM_{\stV} \in \iDcm(\stV_{\et})$,
$\shN_{\stU} \in \iDcp(\stU_{\et})$ and $\shN_{\stV} \in \iDcp(\stV_{\et})$.
We also have a morphism 
\[
 f_* \! \shHom_{\stU_{\et}}(\shM_{\stU},\shN_{\stU}) \longto 
        \shHom_{\stV_{\et}}(f^* \shM, f^* \shN) \simeq
        \shHom_{\stV_{\et}}(\shM_{\stV}, \shN_{\stV}).
\]
The consequence holds if we show this morphism is an equivalence.
But we have
\[
     f^* \! \shHom_{\stU_{\et}}(\shM_{\stU},\shN_{\stU}) \simeq
 f^* A^* \! \shHom_{\stX_{\txle}}(\shM,\shN) \simeq
     B^* \! \shHom_{\stX_{\txle}}(\shM,\shN) \simeq
            \shHom_{\stV_{\et}}(\shM_{\stV}, \shN_{\stV}).
\]
\end{proof}

Let $\stX_\bullet \to \stX$ be a smooth presentation of 
a geometric derived stack $\stX$, 
and $\rst{\stX_{\txle}}{\stX^{\str}_\bullet}$ be the strictly simplicial
$\infty$-topos obtained by restricting the lisse-\'etale $\infty$-topos to 
the strictly simplicial derived algebraic space $\stX^{\str}_\bullet$.
We have the associated geometric morphism denote by
$p: \rst{\stX_{\txle}}{\stX^{\str}_\bullet} \to \stX_{\txle}$.
Then we have 

\begin{lem*}
For $\shM,\shN \in \iDc(\stX_{\txle})$ 
we have a functorial equivalence
\[
 p^* \! \shHom_{\stX_{\txle}}(\shM,\shN) \simeq
 \shHom_{\rst{\stX_{\txle}}{\stX^\str_\bullet}}(p^*\shM,p^*\shN)
\]
\end{lem*}

\begin{proof}
We follow the argument in \cite[4.2.4. Lemma]{LO1}.
We can take an equivalence $\shN \to \shI$ in $\iDa(\stX_{\txle})$
with $\shI$ having injective homotopy groups.
Then we have
\[
 p^* \! \shHom_{\stX_{\txle}}(\shM,\shN) \simeq
 p^* \! \shHom_{\stX_{\txle}}(\shM,\shI) \simeq
 \shHom_{\rst{\stX_{\txle}}{\stX^\str_\bullet}}(p^* \shM,p^*\shI).
\]
We can also take an equivalence $p^* \shI \to \shJ_\bullet$ 
in $\iDa(\rst{\stX_{\txle}}{\stX^\str_\bullet})$,
where $\shJ_\bullet$ has injective homotopy groups.
This equivalence induces a morphism
\[
 \shHom_{\rst{\stX_{\txle}}{\stX^\str_\bullet}}(p^* \shM,p^*\shI)
 \to \shHom_{\rst{\stX_{\txle}}{\stX^\str_\bullet}}(p^* \shM,\shJ_\bullet).
\]
On the other hand, we have 
$\shHom_{\rst{\stX_{\txle}}{\stX^\str_\bullet}}(p^* \shM,p^*\shN)
 \simto \shHom_{\rst{\stX_{\txle}}{\stX^\str_\bullet}}(p^* \shM,\shJ_\bullet)$.
Composing these morphisms, we have
$p^* \! \shHom_{\stX_{\txle}}(\shM,\shN) \to
 \shHom_{\rst{\stX_{\txle}}{\stX_\bullet}}(p^*\shM,p^*\shN)$.
It is enough to show that this morphism is an equivalence.
For that, it suffices to show that each morphism
$\shHom_{\rst{\stX_{\txle}}{\stX^\str_n}}(p_n^* \shM,p_n^*\shI)
\to \shHom_{\rst{\stX_{\txle}}{\stX_n}}(p_n^* \shM,\shJ_n)$
is an equivalence in $\iDa(\rst{\stX_{\txle}}{\stX_n})$,
where $p_n: \rst{\stX_{\txle}}{\stX_n} \to \stX_{\txle}$
is the geometric morphism associated to the localization.
The last claim can be checked by routine
(see also the last part of the proof of \cite[4.2.4. Lemma]{LO1}).
\end{proof}

Using this Lemma, 
we can show the following restatement of \cite[4.2.3. Proposition]{LO1}.

\begin{lem}\label{lem:6op:423}
Let $\stX_\bullet \to \stX$ be a smooth presentation of a geometric stack $\stX$,
We denote by $\shM_{\et},\shN_{\et} \in \iDc(\stX^{\str}_{\bullet,\et})$
the restrictions of $\shM,\shN$ to 
the \'etale $\infty$-topos of $\stX^{\str}_{\bullet}$.
Then we have a functorial equivalence
\[
 \rst{\shHom_{\stX_{\txle}}}{(\shM,\shN)}{\stX^{\str}_{\bullet,\et}}
 \longsimto
 \shHom_{\stX^{\str}_{\bullet,\et}}(\shM_{\et},\shN_{\et}).
\]
\end{lem}

\begin{proof}
We follow the proof of \cite[4.2.3. Proposition]{LO1}.
Let us denote by
$e: \rst{\stX_{\txle}}{\stX^{\str}_\bullet} \to \stX^{\str}_{\bullet,\et}$
the geometric morphism of $\infty$-topoi induced by 
$\gsEt(\stX^{\str}_{\bullet}) \inj \gsLE(\stX^{\str}_{\bullet})$.
Then we have $\shM_{\et} = e_* p^* \shM$, $\shN_{\et} = e_* p^* \shN$ and 
$\rst{\shHom_{\stX_{\txle}}(\shM,\shN)}{\stX^{\str}_{\bullet,\et}}
 = e_* p^* \! \shHom_{\stX_{\txle}}(\shM,\shN)$ by definition.
We also have equivalences 
$e^* \shM_{\et} \simto \pi^* \shM$ and $e^* \shN_{\et} \simto \pi^* \shN$
since the counit transformation $e^* e_* \to \id$ is an equivalence 
by Proposition \ref{prp:is:223}.
Then
\begin{align*}
\rst{\shHom_{\stX_{\txle}}(\shM,\shN)}{\stX^{\str}_{\bullet,\et}}
&=  e_* p^* \! \shHom_{\stX_{\txle}}(\shM,\shN)
 \simeq e_* \! \shHom_{\rst{\stX_{\txle}}{\stX^{\str}_\bullet}}(p^* \shM, p^* \shN)
\\
&\simeq e_* \! \shHom_{\rst{\stX_{\txle}}{\stX^{\str}_\bullet}}
               (e_* \shM_{\et}, e_* \shN_{\et})
\simeq \shHom_{\stX_{\et,\bullet}}(\shM_{\et}, \shN_{\et}).
\end{align*}
\end{proof}

Now we have the following type of projection formula.
The proof is the same as in \cite[4.3.1. Proposition]{LO1} 
with the help of Lemma \ref{lem:6op:423}, so we omit it.

\begin{prp}\label{prp:6op:fHom}
Let $\stX$, $\stY$ be geometric derived stacks locally of finite presentation,
and $f: \stX \to \stY$ be a morphism of of finite presentation.
Then, for $\shM \in \iDc(\stY_{\txle})$ and 
$\shN \in \iDcbp(\stX_{\txle})$, we have a functorial equivalence
\[
 \shHom(\shM,f_*\shN) \longsimto f_* \! \shHom(f^*\shM,\shN)
\]
in $\iDc(\stY_{\txle})$.
\end{prp}

For later use, we record the following lemma.
Recall the dualizing object $\Omega_{\stX}$ 
for a derived stack $\stX$ (Definition \ref{dfn:LE:dual}).

\begin{lem}\label{lem:6op:461}
Let $\stX$, $\stY$ and $f: \stX \to \stY$ be as Proposition \ref{prp:6op:fHom}.
For $\shM \in \iDc(\stY_{\txle})$, the canonical morphism
\[
 f^* \! \shHom(\shM,\Omega_{\stY}) \longto 
        \shHom(f^* \shM, f^* \Omega_{\stY})
\]
is an equivalence.
\end{lem}

\subsection{Derived tensor functor}

Let $\stX$ be a geometric derived stack locally of finite presentation.
In the rest part of this section,
we denote by $\otimes$ the tensor functor $\otimes_{\Lambda}$
on $\iDa(\stX_{\txle})$.
Recall also the dualizing functor 
$\funD_{\stX}: \iDa(\stX_{\txle}) \to  \iDa(\stX_{\txle})^{\op}$
(Notation \ref{ntn:LE:dual}).

\begin{lem}\label{lem:6op:tensor}
\begin{enumerate}[nosep]
\item 
For $\shM,\shN \in \iDc(\stX_{\txle})$, we have
$\shHom(\shM,\shN) \simeq \funD_{\stX}(\shM \otimes \funD_{\stX}(\shN))$.

\item
For $\shM,\shN \in \iDcbm(\stX_{\txle})$, we have
$\shM \otimes \shN \in \iDcbm(\stX_{\txle})$.
\end{enumerate}
\end{lem}

\begin{proof}
The  argument in \cite[4.5.1 Lemma]{LO1} works
with the help of Proposition \ref{prp:LE:dual}.
\end{proof}

\subsection{Shriek functors}
\label{ss:6op:!}

Following the ordinary algebraic stack case \cite[\S 4.6]{LO1}, 
we introduce the derived functors $f_!$ and $f^!$ 
using the dualizing functor $\funD_{\stX}$.


\begin{dfn}\label{dfn:6op:!}
Let $\stX$ and $\stY$ be geometric derived stacks 
locally of finite presentation,
and let $f: \stX \to \stY$ be a morphism of finite presentation.
We define the functor
$f_!$ by
\[
 f_! := \funD_{\stY} \circ f_* \circ \funD_{\stX}: 
 \iDcbm(\stX_{\txle}) \longto \iDcbp(\stY_{\txle}).
\]
Here $f_*$ is the derived direct image functor (Proposition \ref{prp:6op:pf}).
We also define the functor $f^!$ by
\[
 f^! := \funD_{\stX} \circ f^* \circ \funD_{\stY}:
 \iDc(\stY,\Lambda) \longto \iDc(\stX,\Lambda),
\]
where $f^*$ is the derived inverse image functor (Definition \ref{dfn:6op:pb}).
\end{dfn}

We immediately have 

\begin{prp}
Let $\stX$, $\stY$ and $f: \stX \to \stY$ be as in Definition \ref{dfn:6op:!}.
For $\shM \in \iDcbm(\stX_{\txle})$ and $\shN \in \iDc(\stY_{\txle})$,
we have a functorial equivalence
\[
 f_* \! \shHom(\shM, f^! \shN) \simeq \shHom(f_! \shM,\shN).
\]
\end{prp}

\begin{proof}
The statement follows from the sequence of equivalences
\begin{align*}
f_* \! \shHom(\shM, f^! \shN)
&\stackrel{(*1)}{\simeq} f_* \! \shHom(\funD_{\stX} f^! \shN, \funD_{\stX}\shM)
 \stackrel{(*2)}{\simeq} f_* \! \shHom(f^* \funD_{\stY} \shN, \funD_{\stX}\shM) \\
&\stackrel{(*3)}{\simeq} \shHom(\funD_{\stY}\shN, f_* \funD_{\stX}\shM)
\stackrel{(*1)}{\simeq}
\shHom(\funD_{\stY}f_* \funD_{\stX}\shM, \funD_{\stY}\funD_{\stY} \shN)
\stackrel{(*4)}{\simeq} \shHom(f_! \shM, \shN).
\end{align*}
Here ($*1$) is by Proposition \ref{prp:LE:dual} (3),
($*2$) is by Definition \ref{dfn:6op:!},
($*3$) is by Proposition \ref{prp:6op:fHom},
and ($*4$) is by Definition \ref{dfn:6op:!} and Proposition \ref{prp:LE:dual} (2).
\end{proof}

We also have the projection formula.

\begin{prp}
Let $f: \stX \to \stY$ be of finite type.
\begin{enumerate}[nosep]
\item 
For $\shM \in \iDcbm(\stX_{\txle})$ and $\shN \in \iDcbm(\stY_{\txle})$
we have a functorial equivalence
\[
 f_!(\shM \otimes f^* \shN) \simeq (f_! \shM) \otimes \shN.
\]

\item
For $\shM \in \iDcbm(\stX_{\txle})$ and $\shN \in \iDcbp(\stY_{\txle})$
we have a functorial equivalence
\[
 f^! \shHom(\shM, \shN) \simeq \shHom(f^* \shM,f^! \shN).
\]
\end{enumerate}
\end{prp}

\begin{proof}
The argument in \cite[4.5.2, 4.5.3]{LO1} with Lemma \ref{lem:6op:tensor} works.
\end{proof}

Let us give a few properties of the functors $f^!$ and $f_!$.
The proof is now by a standard argument, so we omit it.

\begin{lem}\label{lem:6op:462}
Let $\stX$ and $\stY$ be geometric derived stacks locally of finite presentation.
\begin{enumerate}[nosep]
\item 
If $f:\stX \to \stY$ is a smooth morphism of relative dimension $d$, 
then $f^! = f^*\tsh{d}$.

\item
If $j: \stX \to \stY$ is an open immersion,
then $j^! = j^*$ and $j_!$ is equivalent to the extension by zero functor
(see \S \ref{sss:dAS:oi}).
In particular, we have a morphism $j_! \to j_*$.
\end{enumerate}
\end{lem}

In the rest part of this subsection, we discuss base change theorems
for lisse-\'etale sheaves of derived stacks under several situations.
Let us consider a cartesian square
\[
 \xymatrix{
  \stX' \ar[r]^{\pi} \ar[d]_{\varphi} & \stX \ar[d]^f \\ 
  \stY' \ar[r]^p & \stY}
\]
of derived stacks locally of finite presentation
with $f$ of finite presentation.
We have a morphism
\[
 p^* f_! \longto \varphi_! \pi^*
\]
in $\iFun(\iDcbm(\stX_{\txle}),\iDcbm(\stY'_{\txle}))$,
and 
\[
 p^! f_* \longto \phi_* \pi^!
\]
in $\iFun(\iDcbp(\stX_{\txle}),\iDcbp(\stY'_{\txle}))$.

\begin{prp}\label{prp:6op:sbc}
If $p$ is smooth, then the morphisms 
$p^* f_! \to \varphi_! \pi^*$ and 
$p^! f_* \to \phi_* \pi^!$
are equivalence.
\end{prp}

\begin{proof}
so that the functor $p^*$ is defined on all $\iDc(\stY_{\txle})$
and equal to the restriction from $\stY_{\txle}$ to $\stY'_{\txle}$.
\end{proof}

We expect to have simple extensions of the base change theorems 
for algebraic stacks shown in  \cite[\S 5]{LO1},
but we will not pursuit them.

\section{Adic coefficient case}
\label{s:adic}

In this section we give an extension of the results in the previous section to 
the adic coefficient case.

Let $k$ be a fixed commutative ring, and 
$\Lambda$ be a complete discrete valuation ring with characteristic $\ell$.
We assume that $\ell$ is invertible in $k$.
We write $\Lambda = \varprojlim_n \Lambda_n$, $\Lambda_n:=\Lambda/\frkm^n$
with $\frkm \subset \Lambda$ the maximal ideal.
We denote by $\Lambda_\bullet =(\Lambda_n)_{n \in \bbN}$
the projective system of commutative rings.

\subsection{Projective systems of ringed $\infty$-topoi}

Let $\tT$ be an $\infty$-topos.
Recall Definition \ref{dfn:is:prj} of the
$\infty$-topos of projective systems in $\tT$.
It is an $\iI$-simplicial $\infty$-topos $\tT^{\bbN}$ with $\iI=\Ner(\bbN)$,
equipped with a geometric morphism
$e_n: \tT \to \tT^{\bbN}$ of $\infty$-topoi for each $n \in \bbN$.
The adjunction
\[
 e_n^{-1}: \tT^{\bbN} \adjunc \tT :e_{n,*}.
\]
associated to $e_n$ is described as follows.
For $U_\bullet \in \tT^{\bbN}$, we have $e_n^{-1} U_{\bullet} = U_n$.
For $U \in\tT$, we have 
\[
 (e_{n,*} U)_m = \begin{cases} U & (m \ge n), \\ * & (m < n). \end{cases}
\]
Here $*$ denotes a final object of $\iS$.
We also have a geometric morphism
$p: \tT^{\bbN} \to \tT$ which corresponds to the adjunction
\[
 p^{-1}: \tT \adjunc \tT^{\bbN} : p_*
\]
with $p^{-1}(U) = (U)_{n \in \bbN}$.

Let $R_{\bullet}=(R_n)_{n \ge 0}$ be a projective system of commutative rings,
We ring the $\infty$-topos $\tT$ by the constant sheaf $R_n$,
and ring $\tT^{\bbN}$ by $R_{\bullet}$.
Then we have geometric morphisms
\[
 e_n: (\tT,R_n) \longto (\tT^{\bbN},R_{\bullet}), \quad
 p: (\tT^{\bbN},R_{\bullet}) \longto (\tT,R_n)
\]
induced by $e_n: \tT \to \tT^{\bbN}$ and $p: \tT^{\bbN} \to \tT$.

Let us apply this notation for $\tT=\stX_{\txle}$ and 
$R_\bullet=\Lambda_\bullet$,
where $\stX$ is a geometric derived stack over $k$ and 
$\Lambda_{\bullet}$ is the complete discrete valuation ring
in the beginning of this section.
Recall the notion of \emph{AR-nullity}:

\begin{ntn*}
A projective system $M_{\bullet}=(M_n)_{n \in \bbN}$ in an additive category
is \emph{AR-null} if there exists an $r \in \bbZ$ such that
the projection $M_{n+r} \to M_n$ is zero for every $n$.
\end{ntn*}

Following \cite[2.1.1. Definition]{LO2} we introduce 

\begin{dfn*}
Let $\shM_{\bullet} = (\shM_n)_{n \in \bbN}$ be an object of 
$\iDa(\stX_{\txle}^{\bbN},\Lambda_{\bullet})$.
\begin{enumerate}[nosep]
\item 
$\shM_{\bullet}$ is \emph{AR-null} if the projective system $\pi_j \shM_{\bullet}$
of the $j$-th homotopy groups is AR-null for any $j \in \bbZ$.

\item
$\shM_{\bullet}$ is \emph{constructible} if $\pi_j \shM_n$ is constructible
for any $j \in \bbZ$ and $n \in \bbN$.
\item
$\shM_{\bullet}$ is \emph{almost zero} if 
the restriction $\rst{(\pi_j \shM_{\bullet})}{\stU_{\et}}$ is AR-null
for any $\stU \in \idAS_{\stX}^{\lis}$ and $j \in \bbZ$.
\end{enumerate}
\end{dfn*}

Following \cite[2.2]{LO2}, we give a description of AR-null objects and 
almost zero objects by the restrictions to \'etale $\infty$-topos.
Let us take $\stU \in \idAS_{\stX}^{\lis}$.
We denote by $p: \stX_{\txle}^{\bbN} \to \stX$ the projection,
Identifying $\stU$ with the constant projective system $p^* \stU$, we have 
$(\rst{\stX_{\txle}}{\stU})^{\bbN} \simeq \rst{(\stX_{\txle}^{\bbN})}{\stU}$,
which will be denoted by $\rst{\stX_{\txle}^{\bbN}}{\stU}$.
Then we have a square
\[
 \xymatrix{ 
  \rst{\stX_{\txle}^{\bbN}}{\stU} \ar[r]^{j^{\bbN}} \ar[d]_{\rst{p}{\stU}} & 
  \stX_{\txle}^{\bbN} \ar[d]^p \\
  \rst{\stX_{\txle}}{\stU} \ar[r]_{j} & \stX_{\txle}}
\]
in $\iRTop$, where $j$ and $j^{\bbN}$ are the canonical functors 
(Corollary \ref{cor:ic:ovc}, Fact \ref{fct:is:biadj}).
By the exactness of $j^*$ and $(j^{\bbN})^*$, 
we have $(\rst{p}{\stU})_* (j^{\bbN})^* \simeq j^* p_*$
in $\iFun(\iDa(\stX_{\txle}^{\bbN}),\iDa(\rst{\stX_{\txle}}{\stU}))$.

On the other hand, denoting by
$\ve_{\stU}: \stX_{\txle} \to \stU_{\et}$ and 
$\ve^{\bbN}_{\stU}: \stX_{\txle}^{\bbN} \to \stU_{\et}^{\bbN}$
the geometric morphisms induced by the natural embedding 
$\gsEt(\stU) \to \rst{\gsLE(\stX)}{\stU}$ (Lemma \ref{lem:LE:ve}),
we have a square
\[
 \xymatrix{ 
  \rst{\stX_{\txle}^{\bbN}}{\stU} \ar[r]^(0.55){\ve^{\bbN}_{\stU}} 
  \ar[d]_{\rst{p}{\stU}} & \stU_{\et}^{\bbN} \ar[d]^{p_{\stU}} \\
  \rst{\stX_{\txle}}{\stU} \ar[r]_(0.55){\ve_{\stU}} & \stU_{\et}}
\]
in $\iRTop$.
Since $(\ve_{\stU})_*$ and $(\ve_{\stU}^{\bbN})_*$ are exact, we have 
$(p_{\stU})_* (\ve_{\stU}^{\bbN})_* \simeq 
 (\ve_{\stU})_* (\rst{p}{\stU})_*$.

For $\shM \in \iDa(\stX_{\txle},\Lambda)$, 
we denote by $\shM_{\stU} \in \iDa(\stU_{\et},\Lambda)$ 
the restriction (Notation \ref{ntn:LE:FU}).
By the identification of $\stU$ with $p^* \stU$, 
we can regard $\shM$ as an object of 
$\iDa(\rst{\stX_{\txle}^{\bbN}}{\stU},\Lambda_\bullet)$,
which will be denoted by the same symbol $\shM$.
Then by the argument above we immediately have 

\begin{lem*}[{\cite[2.2.1. Lemma]{LO2}}]
We have $(p_{\stU})_* \shM_{\stU} \simeq ((\rst{p}{\stU})_* \shM)_\stU$
in $\iDa(\stU_{\et},\Lambda)$.
\end{lem*}

Now AR-null objects and almost zero objects are described as follows.

\begin{fct}[{\cite[2.2.2. Proposition]{LO2}}]\label{fct:adic:222}
Let $\shM_{\bullet} \in \iDa(\stX_{\txle}^{\bbN},\Lambda_{\bullet})$.
\begin{enumerate}[nosep]
\item 
If $\shM_{\bullet}$ is AR-null, then $p_* \shM_{\bullet} = 0$ 
in $\iDa(\rst{\stX_{\txle}}{\stU},\Lambda)$.
\item
If $\shM_{\bullet}$ is almost zero, 
then $p_* \shM_{\bullet} = 0$ in $\iDa(\rst{\stX_{\txle}}{\stU},\Lambda)$.
\end{enumerate}
\end{fct}

\begin{dfn*}
Let $\shM_{\bullet}=(\shM_n)_{n \in \bbN}$ be an object of 
$\iDa(\stX_{\txle}^{\bbN},\Lambda_{\bullet})$.
\begin{enumerate}[nosep]
\item 
$\shM_{\bullet}$ is \emph{adic} if each $\shM_n$ is constructible
and each morphism $\Lambda_n \otimes_{\Lambda_{n+1}} \shM_{n+1} \to \shM_{n}$
is an equivalence.

\item 
$\shM_{\bullet}$ is \emph{almost adic} if each $\shM_n$ is constructible,
and if for each $\stU \in \idAS_{\stX}^{\lis}$ there is a morphism
$\shN \to \rst{\shM_{\bullet}}{\stU_{\et}}$ from an adic object
$\shN \in \iDa(\stU_{\et}^{\bbN},\Lambda_{\bullet})$
with almost zero kernel and cokernel in the homotopy category.

\item
$\shM_{\bullet}$ is a \emph{$\lambda$-object}
if $p_j \shM_{\bullet}$ is almost adic for each $j \in \bbZ$.
The full sub-$\infty$-category of $\iD(\stX_{\txle}^{\bbN},\Lambda_{\bullet})$
spanned by $\lambda$-objects is denoted by 
$\iDc(\stX_{\txle}^{\bbN},\Lambda_{\bullet})$
\end{enumerate}
\end{dfn*}

We focus on the localized $\infty$-category of 
$\iDc(\stX_{\txle}^{\bbN},\Lambda_{\bullet})$ by almost zero objects.

\begin{dfn*}
Let $W$ be the class of those morphisms in 
$\iDc(\stX_{\txle}^{\bbN},\Lambda_{\bullet})$
which are equivalences $\shM_{\bullet} \to \shN_{\bullet}$ 
with $\shN_{\bullet}$ an almost zero object.
We define
\[
 \iDDc(\stX,\Lambda) := 
 \iDc(\stX_{\txle}^{\bbN},\Lambda_{\bullet})[W^{-1}]
\]
\end{dfn*}

The $\infty$-category $\iDDc(\stX,\Lambda_{\bullet})$ is stable and equipped with
$t$-structure induced by that on  $\iDc(\stX_{\txle}^{\bbN},\Lambda_{\bullet})$.

By Definition \ref{dfn:ic:loc} and Fact \ref{fct:ic:loc} of $\infty$-localization,
we have a functor  
\[
 \iDc(\stX_{\txle}^{\bbN},\Lambda_{\bullet}) \longto \iDDc(\stX,\Lambda).
\]
On the other hand, the geometric morphism
$p: \stX_{\txle}^{\bbN} \to \stX_{\txle}$
induces an adjunction
\[
 p_*: \iDc(\stX_{\txle}^{\bbN},\Lambda_{\bullet}) \adjunc 
         \iDc(\stX_{\txle},\Lambda) :p^*.
\]
By Fact \ref{fct:adic:222}, the functor $p_*$ factors through 
$\iDDc(\stX_{\txle}^{\bbN},\Lambda_{\bullet})$.
The resulting functor is denoted by the same symbol as
\[
 p_*: \iDDc(\stX,\Lambda) \longto \iDc(\stX_{\txle},\Lambda).
\]

\begin{dfn*}
We define the \emph{normalization functor} to be
\[
 \Nrm: \iDDc(\stX,\Lambda) 
 \longto \iDc(\stX_{\txle}^{\bbN},\Lambda_{\bullet}), \quad
 \shM \longmapsto \Nrm(\shM) := p^* p_* \shM.
\]
An object $\shM \in \iDDc(\stX,\Lambda)$ is \emph{normalized}
if the counit transformation $\Nrm(\shM) \to \shM$ is an equivalence.
\end{dfn*}

Let us cite a useful criterion of normality.

\begin{fct}[{\cite[3.0.10. Proposition]{LO2}}]\label{fct:adic:3010}
An object $\shM \in \iDa(\stX_{\txle}^{\bbN},\Lambda_{\bullet})$
is normalized if and only if the morphism
$\Lambda_n \otimes_{\Lambda_{n+1}} \shM_{n+1} \to \shM_n$
is an equivalence for every $n$.
\end{fct}

By \cite[3.0.14. Theorem]{LO2}, 
if $\shM$ is a $\lambda$-object, then $\Nrm(\shM)$ is constructible
and the morphism $\Nrm(\shM) \to \shM$ has an almost zero cone.
Then we have

\begin{fct*}[{\cite[3.0.18. Proposition]{LO2}}]
The normalization functor sits in the adjunction
\[
 \Nrm: \iDDc(\stX,\Lambda) \adjunc 
       \iDc(\stX_{\txle}^{\bbN},\Lambda_{\bullet}) : p_*.
\]
Thus $\iDDc(\stX,\Lambda) \in \iCat$.
\end{fct*}

Using the $t$-structure on $\iDDc(\stX,\Lambda)$, let us introduce

\begin{ntn*}
We denote by 
\[
 \iDDcu{*}(\stX,\Lambda) \subset \iDDc(\stX,\Lambda), \quad * \in \{+,-.b\}
\]
the full sub-$\infty$-category spanned by bounded objects, and by 
\[
 \iDDcu{(*)}(\stX,\Lambda) \subset \iDDc(\stX,\Lambda), \quad * \in \{+,-.b\}
\]
the full sub-$\infty$-category spanned by those objects
whose restriction to any quasi-compact open immersion $\stU \inj \stX$
lies in $\iDDcu{*}(\stU,\Lambda)$. 
\end{ntn*}

Finally we give a definition of the derived $\infty$-category
of lisse-\'etale constructible $\Qlb$-sheaves.
We follow the $2$-categorical limit method 
taken in \cite[1.1.3]{D}, \cite[\S 6]{B}.

Let $\Lambda$ be a discrete valuation ring with residue characteristic $\ell$,
Denoting by $K$ the quotient field of $\Lambda$, we set
$\iDDc(\stX,K) := \iDDc(\stX,\Lambda) \otimes_\Lambda K$.
Running $K$ on the $\infty$-category $\iFE(\bbQ_\ell)$ 
of finite extensions of $\bbQ_\ell$,
these $\infty$-categories form a cartesian fibration on $\iFE(\bbQ_\ell)$.
Thus we can take
\[
 \iDDc(\stX,\Qlb) := 
 \varprojlim_{K \in \iFE(\bbQ_\ell)} \iDDc(\stX,K)
\]
It is a stable $\infty$-category equipped with a $t$-structure.

\begin{ntn}\label{ntn:adic:Qlb}
We call $\iDDc(\stX,\Qlb)$ 
the \emph{$\ell$-adic constructible derived $\infty$-category} of $\stX$.
An object of the heart is called 
an \emph{$\ell$-adic constructible sheaf} on $\stX$.
\end{ntn}

\subsection{Internal hom and tensor functors on $\iDDc$}

In the remaining part of this section,
we fix a geometric derived stack $\stX$ locally of finite presentation,
and denote 
\begin{align*}
 \iDc(\stX) :=\iDc(\stX_{\txle}.\Lambda) \text{ or } \iDc(\stX_{\txle},\Qlb),
 \quad
 \iDDc(\stX):=\iDDc(\stX.\Lambda) \text{ or } \iDDc(\stX,\Qlb).
\end{align*}
We also use $\iDa(\stX):=\iDa(\stX_{\txle},\Lambda)$
or $\iDa(\stX_{\txle},\Qlb)$.
Recall also the normalization functor 
$\Nrm:\iDDc(\stX,\Lambda) \to \iDc(\stX_{\txle}^{\bbN},\Lambda_{\bullet})$.

We define a bifunctor 
$\bsHom_{\Lambda}: \iDDc(\stX)^{\op} \times \iDDc(\stX) \to 
 \iDc(\stX_{\txle}^{\bbN},\Lambda_{\bullet})$ by 
\[
 \bsHom_{\Lambda}(\shM,\shN) := 
 \shHom_{\Lambda_{\bullet}}(\Nrm{\shM},\Nrm{\shN}).
\]
Then we have 

\begin{fct}[{\cite[4.0.8. Proposition]{LO2}}]\label{fct:adic:408}
$\bsHom$ gives a bifunctor
\[
 \bsHom_{\Lambda}: 
 \iDDcu{(-)}(\stX)^{\op} \times \iDDcu{(+)}(\stX) \longto 
 \iDDcu{(+)}(\stX).
\]
\end{fct}


Next we discuss the tensor functor.
For $\shM, \shN \in \iDDc(\stX)$, we set
\[
 \shM \otimes_{\Lambda} \shN :=
 \Nrm(\shM) \otimes_{\Lambda_{\bullet}} \Nrm(\shN).
\]
Thus we have a bifunctor
$\otimes_{\Lambda}: 
 \iDDc(\stX) \times \iDDc(\stX) \to \iD_{\cart}(\stX)$.
We then have

\begin{lem}[{\cite[6.0.12. Proposition]{LO2}}]\label{lem:adic:6012}
For $\shL,\shM,\shN \in \iDDcu{(-)}(\stX)$ we have 
\[
 \bsHom_{\Lambda}(\shL \otimes_{\Lambda} \shM,\shN) \simeq
 \bsHom_{\Lambda}(\shL,\bsHom_{\Lambda}(\shM,\shN))
\]
\end{lem}

\begin{proof}
Denoting $\wh{\shL} := \Nrm(\shL)$ and similarly for $\shM,\shN$, 
the usual adjunction yields
$\shHom_{\Lambda}(\wh{\shL} \otimes_{\Lambda} \wh{\shM},\wh{\shN})
\simeq \shHom_{\Lambda}(\wh{\shL},\shHom(\wh{\shM},\wh{\shN}))$.
Using Fact \ref{fct:adic:3010}, we can show that 
$\wh{\shL} \otimes_{\Lambda}\wh{\shM}$ is normalized.
Thus we have the consequence.
\end{proof}

\subsection{Dualizing object}

Here we explain the dualizing object with adic coefficients 
following \cite[7]{LO2}.

Recall that $\Lambda$ denotes a complete discrete valuation ring 
and $\Lambda_n =\Lambda/\frkm^n$.
Let $S$ be an affine excellent finite-dimensional scheme
where the residue characteristic $\ell$ of $\Lambda$ is invertible
and any $S$-schemes $f: U \to S$ of finite type has 
\emph{finite cohomological dimension}.
It means that there is an integer $d \in \bbN$ such that
for any abelian torsion \'etale sheaf $\shF$ over $U$
we have $R^i f_* \shF = 0$ for $i > d$.
By \cite[7.1.3]{LO2}, there exists a family 
$\{\Omega_{S,n},\iota_n\}_{n \in \bbZ_{>0}}$ with $\Omega_{S,n}$ 
a $\Lambda_n$-dualizing complex on $S$ and 
$\iota: \Lambda_n \otimes_{\Lambda_n} \Omega_{S,n+1} \to \Omega_{S,n}$
an isomorphism in the bounded derived category $\dD_{\cstr}^{b}(S,\Lambda_n)$
of complexes of $\Lambda$-modules with constructible cohomology groups.
We call it a \emph{compatible family of dualizing complexes}.

Let $k$ be an algebraic closure of the finite field $\bbF_q$ of order $q$
such that $\ell$ is invertible.
Then by \cite[1.0.1]{LO1} the affine scheme $S=\Spec k$ satisfies 
the above conditions, and we have a compatible family
$\{\Omega_{\Spec k,n},\iota_n\}_{n \in \bbZ_{>0}}$ of dualizing complexes.
Then, for a geometric derived stack $\stX$ locally of finite presentation
over $k$, we have the dualizing object 
\[
  \Omega_{\stX,n} \in \iD(\stX_{\txle},\Lambda_n)
\]
using $\Omega_{\Spec k,n}$ and the construction of \S \ref{ss:LE:dual}.
The isomorphism $\iota_n$ induces an equivalence
\[
 (\Omega_{\stX,n+1} \otimes_{\Lambda_{n+1}}\Lambda_{n})_{\stU_{\et}}
 \longto (\Omega_{\stX,n})_{\stU_{\et}}
\]
for any $\stU \in \idAS_{\stX,\fp}^{\lis}$.
By the gluing lemma (Fact \ref{fct:is:glue}),
we have an equivalence
\[
 \Omega_{\stX,n+1} \otimes^{\dL}_{\Lambda_{n+1}}\Lambda_n 
 \longsimto \Omega_{\stX.n}.
\]

In order to construct a dualizing object in $\iDDc(\stX)$,
let us construct a data of dualizing objects in $\iDc(\stX^{\bbN}_{\txle})$.
Note first that the $\bbN$-simplicial $\infty$-topos $\stX^{\bbN}_{\txle}$
is equivalent to the $\infty$-topos associated to the following $\infty$-site:
The underlying $\infty$-category of the nerve of the category 
whose objects are pairs $(\stU,m) \in \idAS_{\stX}^{\lis} \times \bbN$,
and whose set of morphisms from $(\stU,m)$ to $(\stV,n)$ is
empty if $n>m$, and is equal to $\Hom_{\Ho \idAS_{\stX}}(\stU,\stV)$.
A covering sieve is a collection $\{(\stU_i,m_i) \to (\stU,m)\}_{i \in I}$ with 
$m_i=m$ for all $i$ and $\{\stU_i \to \stU\}_{i \in I} \in \Cov_{\txle}(\stU)$.
Then, for each $\stU \in \idAS_{\stX}^{\lis,\fp}$ and $m \in \bbN$, 
we have a sequence 
\[
 \rst{\stX^{\bbN}_{\txle}}{(\stU,m)}
 \xrr{p_m} \rst{\stX_{\txle}}{\stU} \xrr{\ve} \stU_{\et} 
\]
in $\iRTop$.
Here $p_m$ is the geometric morphism defined by 
$p_m^{-1}(\shF) = (\shF)_{n \le m}$.
Now we set 
\[
 \Omega_{\stU,m} := p_m^*((\Omega_{\stX,m})_{\stU})
 \simeq (\ve \circ p_m)^* K_{\stU,m} \tsh{-d},
\]
where $K_{\stU,m} \in \iDc(\stU_{\et},\Lambda_m)$
is the dualizing object for the derived algebraic space $\stU$,
and $d$ is the relative dimension of the smooth morphism $\stU \to \stX$.
Then by the argument of \cite[7.2.3. Theorem]{LO2} 
we can apply again the gluing lemma (Fact \ref{fct:is:glue}) to 
$\Omega_{\stU,m}$'s.

\begin{fct*}[{\cite[7.2.3. Theorem]{LO2}}]
There exists a normalized object 
$\Omega_{\stX,\bullet} \in \iDc(\stX_{\txle}^{\bbN})$ inducing $\Omega_{\stX,n}$,
and it is unique up to contractible ambiguity.
\end{fct*}

\begin{ntn*}
We denote by 
\[
 \Omega_{\stX} \in \iDDc(\stX,\Lambda)
\]
the image of $\Omega_{\stX, \bullet}$ under the projection
$p_*: \iDc(\stX_{\txle}^{\bbN}) \to \iDDc(\stX,\Lambda)$,
and call it the \emph{dualizing object}.
\end{ntn*}

$\Omega_{\stX}$ is of locally finite quasi-injective dimension.
Namely, for each quasi-compact open immersion $\stU \to \stX$,
$(\Omega_{\stX.n})_{\stU}$ is of finite quasi-injective dimension,
and the bound depends only on $\stX$ and $\Lambda$, not on $n$.

Now we define the \emph{dualizing functor} $\funD_{\stX}$ by 
\[
 \funD_{\stX}(\shM) := \bsHom_{\Lambda}(\shM,\Omega_{\stX})
\]
for  $\shM \in \iDDc(\stX)^{\op}$.
By Fact \ref{fct:adic:408}, the image under $\funD_{\stX}$ lies in 
$\iDc(\stX_{\txle}^{\bbN},\Lambda_{\bullet})$.
Thus we have the induced functor
\[
 \funD_{\stX}: \iDDc(\stX)^{\op} \longto \iDDc(\stX).
\]
It is involutive: $\funD_{\stX}^2 \simeq \id$.

\begin{fct*}[{\cite[7.3.1. Theorem]{LO2}}]
The dualizing functor $\funD_{\stX}$ restricts to
\[
 \funD_{\stX}: \iDDcu{(-)}(\stX) \longto \iDDcu{(+)}(\stX)
\]
\end{fct*}

We record a corollary of this fact for later use.

\begin{lem}[{\cite[7.3.2. Corollary]{LO2}}]\label{lem:adic:732}
For $\shM,\shN \in \iDDc(\stX,\Lambda)$, we have an equivalence
\[
 \bsHom_{\Lambda}(\shM,\shN) \simeq 
 \bsHom_{\Lambda}\left(\funD_{\stX}(\shM),\funD_{\stX}(\shN)\right)
\]
which is unique up to contractible ambiguity.
\end{lem}

\subsection{Direct and inverse image functors}

Let $f: \stX \to \stY$ be a morphism of finite presentation
between derived stacks locally of finite presentation.
We have the induced geometric morphism 
$f_{\bullet}: \stX_{\txle}^{\bbN} \to \stY_{\txle}^{\bbN}$, 
and an adjunction
\[
 f^*: \iDa(\stX_{\txle}^{\bbN},\Lambda_{\bullet}) \adjunc
         \iDa(\stY_{\txle}^{\bbN},\Lambda_{\bullet}) : f^*
\]
of the derived direct image and inverse image functors.

By \cite[8.0.4. Proposition]{LO2},
if $\shM \in \iDa(\stX_{\txle}^{\bbN},\Lambda_{\bullet})$ 
is a left bounded $\lambda$-object, then $f_* \shM$ is a $\lambda$-object.
We can also check that AR-null objects are mapped to AR-null objects.
Thus the following definition makes sense.

\begin{ntn*}
The obtained functor
\[
 f_*: \iDDcu{(+)}(\stX,\Lambda) \longto \iDDcu{(+)}(\stY,\Lambda)
\]
is called the \emph{derived direct image functor}.
\end{ntn*}

On the other hand, one can check by definition 
that $f^*$ sends $\lambda$-objects to 
$\lambda$-objects and AR-null objects to AR-null objects.
Thus we have

\begin{ntn*}
The obtained functor
\[
 f^*: \iDDc(\stY,\Lambda) \longto \iDDc(\stX,\Lambda)
\]
is called the \emph{derived inverse image functor}.
\end{ntn*}

One can check the following by the standard argument using adjunction.

\begin{lem}\label{lem:adic:805}
For $\shM \in \iDDcu{(-)}(\stY,\Lambda)$ and 
$\shN \in \iDDcu{(+)}(\stX,\Lambda)$, we have an equivalence
\[
 f_* \bsHom_{\Lambda}(f^* \shM,\shN) \simeq
 \bsHom_{\Lambda}(\shM,f_* \shN)
\]
which is unique up to contractible ambiguity.
\end{lem}

\subsection{Shriek direct and inverse image functors}
\label{ss:pv:!}

As in the finite coefficient case,
we define shriek functors using dualizing functors.

Let $f: \stX \to \stY$ be a morphism of finite presentation
between derived stacks locally of finite presentation.
Let $\Omega_{\stX}$ be the dualizing object of $\stX$, and 
\[
 \funD_{\stX}:=\bsHom_{\Lambda}(-,\Omega_{\stX}):
 \iDDc(\stX,\Lambda)^{\op} \longto \iDDc(\stY,\Lambda)
\]
be the dualizing functor.
Similarly we denote $\funD_{\stY}:=\bsHom_{\Lambda}(-,\Omega_{\stY})$.

\begin{ntn*}
We define the functor $f_!$ by 
\[
 f_! := \funD_{\stY} \circ f_* \circ \funD_{\stX}:
 \iDDcu{(-)}(\stX,\Lambda) \longto \iDDcu{(-)}(\stY,\Lambda),
\]
and the functor $f^!$ by 
\[
 f^! := \funD_{\stY} \circ f^* \circ \funD_{\stX}:
 \iDDc(\stY,\Lambda) \longto \iDDc(\stX,\Lambda).
\]
\end{ntn*}

By Lemma \ref{lem:adic:805} and Lemma \ref{lem:adic:732},
we immediately have

\begin{lem*}
For $\shM \in \iDDcu{(-)}(\stX,\Lambda)$ and $\shN \in \iDDcu{(+)}(\stY,\Lambda)$,
we have an equivalence
\[
 f_* \bsHom_{\Lambda}(\shM,f^!\shN) \simeq \bsHom_{\Lambda}(f_!\shM,\shN)
\]
which is unique up to contractible ambiguity.
\end{lem*}

We also have

\begin{lem*}[{\cite[9.1.2. Lemma]{LO2}}]
If $f$ is smooth of relative dimension $d$,
then for any $\shM \in \iDDc(\stX,\Lambda)$ we have an equivalence
\[
 f^! \shM \simeq f^* \shM \tsh{d}
\]
which is unique up to contractible ambiguity. 
\end{lem*}

\begin{proof}
We have $\Omega_{\stX} \simeq f^* \Omega_{\stY}\tsh{d}$
by definition of the dualizing object and Lemma \ref{lem:6op:462} (1).
Then by the biduality $\funD_{\stX}^2 \simeq \id$ we have the consequence.
\end{proof}

Finally we explain the smooth base change with adic coefficients.
Let
\[
 \xymatrix{
  \stX' \ar[r]^{p'} \ar[d]_{f'} & \stX \ar[d]^f \\ \stY' \ar[r]_p & \stY
 }
\]
be a cartesian square of derived stacks
with $f$ of finite type.
Then we have a morphism
\[
 \alpha: (p')^* f_! \longto f'_! (p')^*
\]
of functors $\iDDcu{(+)}(\stX,\Lambda) \to \iDDcu{(-)}(\stY',\Lambda)$.
By Proposition \ref{prp:6op:sbc} we have

\begin{prp}\label{prp:adic:sbc}
If $p$ is smooth, then $\alpha$ is an equivalence.
\end{prp}


We close this section by 

\begin{rmk}\label{rmk:pv:cls}
\begin{enumerate}[nosep]
\item 
All the claims hold for the $\ell$-adic constructible derived $\infty$-category 
$\iDDc(\stX,\Qlb)$ (Notation \ref{ntn:adic:Qlb}).

\item
If we take $\stX$ to be an algebraic stack of finite presentation,
then we can recover the categories and functors in \cite{LO2}.
\end{enumerate}
\end{rmk}

\section{Perverse sheaves on derived stacks}
\label{s:pv}

In this section we introduce the perverse $t$-structure 
on the constructible derived $\infty$-category on a derived stack $\stX$,
and discuss perverse sheaves, the decomposition theorem and weights.
Our argument basically follows \cite{LO3},
where the theory of perverse sheaves on an algebraic stack is developed.

Let $k$ be a fixed field, and $\Lambda$ be a complete discrete valuation ring 
whose residue characteristic is invertible in $k$.
We denote $\Lambda_n := \Lambda/\frkm^{n+1}$ for $n \in \bbN$.
We fix a geometric derived stack $\stX$ locally of finite presentation over $k$.

\subsection{Gluing of $t$-structures}

We recollect standard facts on gluing of $t$-structures in \cite{BBD},
specializing to the constructible derived $\infty$-categories of derived stacks.
See also \cite[\S 2]{LO3}.

\begin{asp}\label{asp:pv:gl}
Let $\iD$, $\iD_U$ and $\iD_F$ be stable $\infty$-categories, and 
\[
 \iD_{F} \xrr{i_*} \iD \xrr{j^*} \iD_{U}
\]
be a sequence of exact functors (Definition \ref{dfn:stb:exact}).
Assume the following conditions hold.
\begin{enumerate}[nosep, label=(\roman*)]
\item 
$i_*$ has a left adjoint $i^*$ and a right adjoint $i^!$.
\item
$j^*$ has a left adjoint $j_!$ and a right adjoint $j_*$.
\item
$i^! j_* = 0$.
\item
For each $K \in \iD$,
there exist morphisms $i_* i^* K \to j_! j^* K[1]$
and $j_* j^* K \to i_* i^! K[1]$ in $\iD$ such that 
the induced triangles
\[
 j_! j^* K \to K \to i_* i^* K \to j_! j^* K[1], \quad
 i_* i^! K \to K \to j_* j^* K \to i_* i^! K[1]
\]
in $\Ho \iD$ is distinguished.
\item
All the unit and counit transformations
$i^* i_* \to \id \to i^! i_*$ and $j^* j_* \to \id \to j^* j_!$ 
are equivalences.
\end{enumerate}
\end{asp}

\begin{fct}\label{fct:pv:gl}
Under Assumption \ref{asp:pv:gl}, we further suppose that $\iD_F$ and $\iD_U$ 
are equipped with $t$-structures determined by $(\iD_F^{\le 0},\iD_F^{\ge 0})$ 
and $(\iD_U^{\le 0},\iD_U^{\ge 0})$ respectively.
Define the full sub-$\infty$-categories $\iD^{\le 0}, \iD^{\ge 0} \subset \iD$ by
\[
 \iD^{\le 0} := \{ K \in \iD \mid 
 j^* K \in \iD_{U}^{\le 0}, i^* K \in \iD_{F}^{\le 0}\}, \quad
 \iD^{\ge 0} := \{ K \in \iD \mid 
 j^* K \in \iD_{U}^{\ge 0}, i^! K \in \iD_{F}^{\ge 0}\}.
\]
Then the pair $(\iD^{\le 0},\iD^{\ge 0})$ determines a $t$-structure on $\iD$.
\end{fct}

We consider two cases which satisfy Assumption \ref{asp:pv:gl}.

Let $k$, $\Lambda$ and $\stX$ be as in the beginning of this section.
Let $i: \stF \to \stX$ be a closed immersion of a derived stack
and $\iota: \stU \to \stX$ be the open immersion 
of its complement (Notation \ref{ntn:dSt:c-o}).
Recall that we also have constructed derived functors 
$i_!: \iD_{\stF} \to \iD$ and $j^!: \iD \to \iD_{\stU}$,
By \S \ref{ss:6op:!} we have $i_!=i_*$ and $j^! \simeq j^*$,
so that they are compatible with the notation in Assumption \ref{asp:pv:gl}.
Then, by the argument in \S \ref{ss:6op:!}, we have

\begin{lem*}
Fix $n \in \bbN$.
Then the stable $\infty$-categories
\[
 \iD := \iDcu{b}(\stX_{\txle},\Lambda_n), \quad
 \iD_{\stF} := \iDcu{b}(\stF_{\txle},\Lambda_n), \quad 
 \iD_{\stU} := \iDcu{b}(\stU_{\txle},\Lambda_n),
\]
and the direct and inverse image functors 
\[
 \iD_{\stF} \xrr{i_*} \iD \xrr{j^*} \iD_{\stU}
\]
constructed in \S \ref{s:6op}
satisfy Assumption \ref{asp:pv:gl}.
\end{lem*}

We can also consider the adic coefficient case.

\begin{lem*}
The stable $\infty$-categories
\[
 \iDD := \iDDcu{b}(\stX_{\txle},\Lambda), \quad
 \iDD_{\stF} := \iDDcu{b}(\stF_{\txle},\Lambda), \quad 
 \iDD_{\stU} := \iDDcu{b}(\stU_{\txle},\Lambda),
\]
and the functors 
\[
 \iDD_{\stF} \xrr{i_*} \iDD \xrr{j^*} \iDD_{\stU}
\]
constructed in \S \ref{s:6op}
satisfy Assumption \ref{asp:pv:gl}.
\end{lem*}


\subsection{Perverse $t$-structure}
\label{ss:pv:pv}

\subsubsection{The case of derived algebraic spaces}

Here we will introduce the perverse $t$-structure for derived algebraic spaces.
Let $k$ and $\Lambda = \varprojlim_n \Lambda_n$ be as 
in the beginning of this section.

Consider a derived algebraic space $\stU$ of finite presentation over $k$.
Fix $n \in \bbN$ for a while, and let 
$\iD^{b} := \iD^{b}_{\cstr}(\stU,\Lambda_n)$ be 
the derived $\infty$-category of bounded complexes of 
\'etale $\Lambda_n$-sheaves with constructible cohomology groups.

Following the case of algebraic spaces (Definition \ref{dfn:AS:pt}),
let us introduce

\begin{dfn*}
A \emph{point} of a derived algebraic space $\stU$ is a 
monomorphism $\dSpec L \to \stU$ of derived stacks 
(Definition \ref{dfn:dSt:epi-mono}) with $L$ some field.
It will be denoted typically as $i_u: u \to \stU$.
\end{dfn*}

The equivalence of two points on $\stU$ is defined 
similarly as Definition \ref{dfn:AS:pt}.
The dimension of $\stU$ at its point $u$ is also defined 
in the standard manner, and it will be denoted by $\dim(u)$.

We denote by 
\[
 {}^{\pv}\iD^{b,\le 0} \subset \iD^{b}
\]
the full sub-$\infty$-category of objects $K$ such that $H^j(i_u^* K)=0$ 
for each point $u$ of $\stU$ and every $j > -\dim(u)$.
Here $i_u^* K := (i_{\ol{u}}^* K)_u$ with 
$i_{\ol{u}}: \ol{u}_{\txred} \to \stU$ the closed immersion.
Note that the symbol $\ol{X}_{\txred}$ makes sense here.
Similarly, we denote by 
\[
 {}^{\pv}\iD^{b,\ge 0} \subset \iD^b
\]
the sub-$\infty$-category of those $K$ such that 
$H^j(i_u^! K)=0$ for each point $u \in \stU$ and every $j < -\dim(u)$.
Then, as in the case of schemes \cite[2.2.11]{BBD}, the pair 
\[
 ({}^{\pv}\iD^{b,\le 0},{}^{\pv}\iD^{b,\ge 0})
\]
determines a $t$-structure on $\iD^{b}$,
which is called the (middle) \emph{perverse $t$-structure}.
We denote by 
\[
 \pH^{0}: \iD^{b} \longto \iD^{b}
\]
the perverse cohomology functor.

In \cite[\S 3]{LO3}, an extension is developed 
of the perverse $t$-structure on the bounded derived category to
the unbounded derived category.
Let us explain it in the context of derived algebraic spaces.
For $K \in \iD=\iD_{\cstr}(\stU,\Lambda_n)$ 
and $\alpha,\beta \in \bbZ$ with $\alpha \le \beta$,
we denote $\tau_{[\alpha,\beta]} K := \tau_{\ge \alpha} \tau_{\le \beta} K$.
Then we have

\begin{fct*}[{\cite[Lemma 3.3]{LO3}}]  
For any $K \in \iD^{b}$,
there exist $\alpha, \beta \in \bbZ$ such that $\alpha \le \beta$ and
$\pH^{0}(K) \simeq \pH^{0}(\tau_{\alpha,\beta}K)$.
\end{fct*}

For the completeness of presentation,
let us explain the outline of the proof.
It suffices to show that there exist $\alpha,\beta \in \bbZ$, 
$\alpha<\beta$ such that for any $K \in \iD^{<\alpha}$
or $K \in \iD^{>\beta}$ we have $\pH^0(K)=0$.
By the definition of the perverse sheaf,
one can take $\alpha$ to be an integer smaller then $-\dim X$.
Since the dualizing sheaf of a scheme of finite type over $k$
has finite quasi-injective dimension, there exists $c \in \bbN$
such that for any $d \in \bbZ$, any point $u$ of $\stU$ and $K \in \iD^{>d}$
we have $i_u^! K \in \iD^{>d+c}$.
Thus we can take $\beta$ to be an integer greater than $-c$.

Using the integers $\alpha<\beta$ in the above Fact, 
we have a well-defined functor 
\[
 \pH: \iD \longto \iD^{b}, \quad 
 K \longmapsto \pH(\tau_{[\alpha,\beta]}K).
\]
We now define ${}^{\pv}\iD^{\le 0}$ (resp.\ ${}^{\pv}\iD^{\ge 0}$) to be 
the full sub-$\infty$-category of $\iD$ spanned by those $K \in \catD$ with 
$\pvp_j(K) := \pH(K[j]) =0$ for any $j \in \bbZ_{\le 0}$ 
(resp.\ $j \in \bbZ_{\ge 0}$).
The pair
\[
 ({}^{\pv}\catD^{\le 0},{}^{\pv}\catD^{\ge 0})
\]
determines a $t$-structure on $\iD$, 
which is called the \emph{perverse $t$-structure}.

Similarly, for the derived $\infty$-category 
$\iDD^{b} := \iDD_{\cstr}^{b}(\stU,\Lambda)$ with adic coefficients,
we have a pair 
\[
 ({}^{\pv}\iDD^{b,\le 0},{}^{\pv}\iDD^{b,\ge 0})
\]
given by the same condition, 
and it determines a $t$-structure on $\iDD^{b}$
(see \cite[Proposition 3.1]{LO3} for the detail).
We also call the obtained $t$-structure the \emph{perverse $t$-structure}.

The same reasoning works for $\catDD:=\catDD_{\cstr}(X,\Lambda)$.

\subsubsection{The case of derived stacks}

Now we consider a geometric derived stack $\stX$.
We have the stable $\infty$-categories $\iDc(\stX,\Lambda_n)$ and 
$\iDDc(\stX,\Lambda)$. 
We denote either of them by $\iD(\stX)$.

We first assume that $\stX$ is of finite presentation.
Take a smooth presentation $\stX_{\bullet} \to \stX$.
We denote by $p: \stX_0 \to \stX$ the projection from 
the derived algebraic space $\stX_0$.
Then we may assume that $\stX_0$ is of finite presentation.

\begin{dfn*}
For a derived stack $\stX$ of finite presentation, we define 
\[
 \ipDm(\stX) \quad (\text{resp.\ } \ipDp(\stX))
\]
{}to be the full sub-$\infty$-category of $\iD(\stX)$ spanned by those objects $\shM$
such that $p^* \shM[d] \in \ipDm(\stX_0)$ (resp.\ $p^* \shM[d] \in \ipDp(\stX_0)$),
where $d$ is the relative dimension of $p:\stX_0 \to \stX$.
\end{dfn*}

\begin{lem*}
The $\infty$-categories $\ipDm(\stX)$ and $\ipDp(\stX)$ are independent of 
the choice of $\stX_{\bullet} \to \stX$ up to equivalence.
\end{lem*}

\begin{proof}
The proof of \cite[Lemma 4.1]{LO3} works.
\end{proof}

\begin{lem}\label{lem:pv:41}
For a geometric derived stack $\stX$ of finite presentation,
the pair $(\ipDm(\stX),\ipDp(\stX))$ determines a $t$-structure on $\iD(\stX)$.
\end{lem}

\begin{proof}
The non-trivial point is to check the condition (iii) of the $t$-structure 
(Definition \ref{dfn:stb:BBD}).
Using noetherian induction and gluing of $t$-structures (Fact \ref{fct:pv:gl}),
we can construct for each $\shM \in \iD(\stX)$ a fiber sequence 
$\shM' \to \shM \to \shM''$ with $\shM' \in \iD^{\ge 0}(\stX)$ 
and $\shM'' \in \iD^{\le 0}(\stX)$.
See \cite[Proposition 3.1]{LO3}
\end{proof}

Next we assume that $\stX$ is locally of finite presentation.
In this case, we have

\begin{prp}\label{prp:pv:lfp}
For a geometric derived stack $\stX$ locally of finite presentation,
we define  
\[
 \ipDm(\stX) \quad (\text{resp.\ } \ipDp(\stX))
\]
{}to be the full sub-$\infty$-category of $\iD(\stX)$ spanned by 
those objects $\shM$ such that for each open immersion $\stY \to \stX$ with 
$\stY$ a derived stack of of finite presentation,
the restriction $\rst{\shM}{\stY}$ belongs to $\ipDm(\stY)$ 
(resp.\ to $\ipDp(\stY)$). 
Then the pair 
\[
 (\ipDm(\stX),\ipDp(\stX))
\]
 determines a $t$-structure on $\iD(\stX)$,
which will be called the \emph{perverse $t$-structure}.
\end{prp}

\begin{proof}
We follow the proof of \cite[Theorem 5.1]{LO3}.
The non-trivial point is the condition (iii) of the $t$-structure.
We can write $\stX \simeq \iclim_{i \in I} \stX_i$ with $\{\stX_i\}_{i \in I}$
a filtered family of open derived substacks of finite presentation.
For each $i \in I$, the restriction $\rst{\shM}{\stX_i}$ sits in a fiber sequence 
$\shM'_{i} \to \rst{\shM}{\stX_i} \to \shM_i''$ by Lemma \ref{lem:pv:41}.
Denote by $j_i: \stX_i \to \stX$ the open immersion.
Then we have a sequence
\[
 (j_{i})_{!} \shM'_i \to (j_{i+1})_{!} \shM'_{i+1} \to \cdots
\]
in $\iD(\stX)$, so we can define $\shM'$ to be the colimit of this sequence.
We then have a morphism $\shM' \to \shM$, and also a fiber sequence 
$\shM' \to \shM \to \shM''$.
By \cite[Lemma 5.2]{LO3}, the restriction of this fiber sequence to $\stX_i$
is equivalent to $\shM'_{i} \to \rst{\shM}{\stX_i} \to \shM_i''$ for each $i \in I$.
Thus we have $\shM' \in \ipDm(\stX)$ and $\shM' \in \ipDp(\stX)$.
\end{proof}

\begin{dfn}
For a derived stack $\stX$ locally of finite presentation, we define 
the \emph{perverse $t$-structure} of $\iD(\stX)=\iDc(\stX,\Lambda_n)$ or 
$\iDDc(\stX,\Lambda)$ to be the $t$-structure given by Proposition \ref{prp:pv:lfp}.
Its heart is denoted by
\[
 \iPerv(\stX) \subset \iD(\stX),
\] 
and its object is called a \emph{perverse sheaf} on $\stX$.
\end{dfn}

We denote by
\[
 \pH^0: \iD(\stX) \longto \iPerv(\stX)
\]
the perverse cohomology functor.
The standard argument in \cite{BBD} gives 

\begin{lem*}
The homotopy category $\Ho \iPerv(\stX)$ is abelian
which is artinian and noetherian.
\end{lem*}

In particular, we have the notion of \emph{simple objects} in $\iPerv(\stX)$.
They are called \emph{simple perverse sheaves} on $\stX$.

Now let us introduce intermediate extensions.
Recall that for an open immersion $j$ of derived stacks
we have a morphism $j_! \to j_*$ of derived functors (Lemma \ref{lem:6op:462}).

\begin{dfn*}
Let $i: \stF \to \stX$ be a closed substack with the complement $j: \stU \to \stX$.
For a perverse sheaf $\shP$ on $\stU$,
we define $j_{! *} \shP \in \iPerv(\stX)$ to be the image in the abelian category 
$\Ho \iPerv(\stX)$ of the morphism $\pvp_0(j_! \shP) \to \pvp_0(j_* \shP)$:
\[
 j_{! *} \shP := \Img(\pvp_0(j_! \shP) \to \pvp_0(j_* \shP)).
\]
We call it the \emph{intermediate extension}.
\end{dfn*}

Note that as a perverse sheaf, or an object of $\iPerv(\stX)$,
$j_{! *} \shP$ is defined up to contractible ambiguity.
Hereafter we consider $j_{!*}$ as the functor
\[
 j_{!*}: \iPerv(\stU) \longto \iPerv(\stX).
\]

Recall the dualizing functor $\funD_{\stX}$ on $\iD(\stX)$
(see \S \ref{ss:pv:!}, Remark \ref{rmk:pv:cls}).
Here are standard properties of the intermediate extension.

\begin{lem*}
\begin{enumerate}[nosep]
\item 
We have $j^*(j_{!*}\shP) \simeq \shP$ and $\pvp_0(i^*(j_{!*}\shP))=0$,
and these properties determine $j_{!*}\shP$ as an object of $\iPerv(\stX)$
uniquely up to contractible ambiguity.

\item
Let $p: \stX_{\bullet} \to \stX$ be a smooth presentation of relative dimension $d$,
and let $\stF_0 \xr{i'} \stX_0 \xrightarrow{j'} \stU_0$ 
be the pullbacks of $\stF$ and $\stU$.
Then $p^*[d] j_{!*} \simeq j'_{!*}p^*[d]$
\end{enumerate}
\end{lem*}

\begin{proof}
The proofs of \cite[Lemma 6.1, 6.2]{LO3} works 
with the case of derived algebraic spaces,
the derived functors in \S \ref{s:adic} 
and the smooth base change (Proposition \ref{prp:adic:sbc}).
\end{proof}

For later use, let us introduce

\begin{dfn}\label{dfn:pv:sm}
\begin{enumerate}[nosep]
\item 
An object $\shM_{\bullet}=(\shM_n)_{n \in \bbN} \in \iMod_{\Lambda_\bullet}(\stX)$
is \emph{smooth} if all $\shM_n$ are locally constant (Definition \ref{dfn:LE:lc}).

\item
An object $\shM \in \iD(\stX)$ is \emph{smooth}
if for each $j \in \bbZ$ the homotopy group $\pi_j \shM$ is represented by 
a smooth object of $\iDc(\stX,\Lambda_\bullet)$,
and vanishes for almost all $j$.
\end{enumerate}
\end{dfn}

\subsection{Decomposition theorem}

We now discuss the decomposition theorem of perverse sheaves on derived stacks.
We fix a base field $k$.
Let $\stX$ be a geometric derived stack of finite presentation over $k$. 
We have the derived $\infty$-category $\iDDc(\stX,\bbZ_{\ell})$ 
of constructible $\bbZ_{\ell}$-sheaves with $\ell$ invertible in $k$.

Let us consider the full sub-$\infty$-category
$\iDDcu{b}(\stX,\bbZ_{\ell})$ spanned by bounded objects, and set
$\iDDcu{b}(\stX,\bbQ_{\ell}) := \iDDcu{b}(\stX,\bbZ_{\ell}) \otimes \bbQ_{\ell}$.
By the argument in the previous subsection,
we have the perverse $t$-structure on $\iDDcu{b}(\stX,\bbZ_{\ell})$.
It induces a $t$-structure on $\iDDcu{b}(\stX,\bbQ_{\ell})$,
which will also be called the perverse $t$-structure.

\begin{dfn*}
An object in the heart of this $t$-structure will be called 
a \emph{perverse $\bbQ_{\ell}$-sheaf}. 
\end{dfn*}

For a derived stack $\stV$ over the base field $k$,
we have a derived stack over the algebraic closure $\ol{k}$
with the \emph{reduced structure},
which will be denote by $(\stV \otimes_k \ol{k})_{\txred}$. 
Let us give a local explanation on the reduced derived stack:
For a derived $k$-algebra $A = \oplus_{n \in \bbN} A_n \in \isCom_k$,
we denote $A_{\txred}:= \oplus_{n \in \bbN} A_{\txred,n}$ 
with $A_{\txred,0} := (A_0)_{\txred}$ as a commutative $k$-algebra
and $A_{\txred,n} := A_n \otimes_{A_0} A_{\txred,0}$ for $n >0$.

We can now describe simple perverse $\bbQ_\ell$-sheaves on $\stX$.
Let us call a derived substack $\stV \inj \stX$ \emph{irreducible}
if $\stV$ is truncated (Definition \ref{dfn:dSt:trunc})
and the truncation $\Trc \stV$ is an irreducible algebraic stack.

\begin{prp}[{c.f.\ \cite[Theorem 8.2]{LO3}}]\label{prp:pv:dc}
Let $j: \stV \to \stX$ be the closed embedding of an irreducible substack
such that $( \stV \otimes_k \ol{k})_{\txred}$ is smooth. 
Let $\shL \in \iMod_{\bbQ_{\ell}}(\stV)$ be smooth (Definition \ref{dfn:pv:sm}) 
and a simple object (in the abelian category $\Ho \iMod_{\bbQ_{\ell}}(\stV)$). 
Then the intermediate extension $j_{!*}(\shL [\dim(\stV)])$ 
is a simple perverse $\bbQ_\ell$-sheaf on $\stX$,
and every simple perverse $\bbQ_\ell$-sheaf on $\stX$ is obtained in this way.
\end{prp}

The argument in \cite[\S 8]{LO3} works in our situation 
with obvious modifications, and we omit the detail of the proof.

\section{Moduli stack of perfect dg-modules}
\label{s:mod}

In this section we cite from \cite{TVa} the theory of moduli stacks
of modules over dg-categories via derived stacks.

\subsection{Dg-categories}
\label{ss:mod:dg}

In this subsection we collect basic notions on dg-categories.
We fix a commutative ring $k$,
and denote by $\C(k)$ the category of complexes of $k$-modules
with the standard monoidal structure $\otimes_k$.
For later use, we write the dependence on the universe explicitly.
See \S \ref{ss:intro:ntn} for our convention on the universe.

\begin{dfn}\label{dfn:mod:dg}
A $\bbU$-small \emph{dg-category} $\catD$ over $k$ 
consists of the following data.
\begin{itemize}[nosep]
\item 
A $\bbU$-small set $\Ob(\catD)$,
which is called the set of objects of $\catD$.
We denote $X \in \catD$ to mean that $X \in \Ob(\catD)$.
\item
For every $X,Y \in \catD$, a complex of $k$-modules
\[
 \cdots \xrr{d_{-2}} \Hom_{\catD}(X,Y)^{-1} \xrr{d_{-1}} \Hom_{\catD}(X,Y)^{0} 
        \xrr{d_0   } \Hom_{\catD}(X,Y)^{ 1} \xrr{d_1}    \cdots
\]
called the complex of morphisms from $X$ to $Y$,
which is denoted just by $\Hom_{\catD}(X,Y)$.

\item
For every $X,Y,Z \in \catD$, a morphism of complexes 
\[
 \circ: \Hom_{\catD}(Y,Z) \otimes_k \Hom_{\catD}(X,Y) \longto
        \Hom_{\catD}(X,Z)
\]
called the composition map,
which is to satisfy the unit and associativity conditions.

\end{itemize}
\end{dfn}

In other words, a dg-category over $k$ is nothing but a $\C(k)$-enriched category.
Hereafter the word ``a dg-category" means a $\bbU$-small dg-category over $k$.

\begin{ntn}\label{ntn:mod:dg-one}
We denote by $\kdg$ the dg-category over $k$ with one object $*$ 
and the complex of morphisms $* \to *$ given by $k$.
\end{ntn}

We will always consider the projective model structure on $\C(k)$,
where a fibration is defined to be an epimorphism and 
a weak equivalence is defined to be a quasi-isomorphism (see \cite[\S2.3]{H}).

Recall also the notion of \emph{dg-functors}.

\begin{dfn*}
A \emph{dg-functor} $f: \catD \to \catD'$ between 
dg-categories $\catD$ and $\catD'$ consists of 
\begin{itemize}[nosep]
\item 
A map $f: \Ob(\catD) \to \Ob(\catD')$ of sets.

\item
For every $x,y \in \catD$, a morphism
$f_{x,y}: \Hom_{\catD}(x,y) \to \Hom_{\catD'}(f(x),f(y))$
in $\C(k)$ satisfying the unit and associativity conditions.
\end{itemize}
\end{dfn*}

\begin{ntn}\label{ntn:mod:dgCat}
We denote by $\dgCat_{\bbU}$ the category of $\bbU$-small dg-categories
and dg-functors between them.
\end{ntn}

Let us recall some standard notions on dg-functors.
For that, we introduce 

\begin{ntn}\label{ntn:mod:bra}
For a dg-category $\catD$, we denote by $[\catD]$ the category with 
\[
 \Ob([\catD]) :=\Ob(\catD), \quad
 \Hom_{[\catD]}(x,y):=H^0(\Hom_{\catD}(x,y)).
\] 
Here $H^0$ denotes the $0$-th cohomology of a complex.
\end{ntn}

\begin{dfn}[{\cite[Definition 2.1]{T07}}]\label{dfn:mod:qe}
Let $f: \catD \to \catD'$ be a dg-functor between dg-categories.
\begin{enumerate}[nosep]
\item 
$f$ is \emph{quasi-fully faithful} if 
$f_{x,y}: \Hom_{\catD}(x,y) \to \Hom_{\catD'}(f(x),f(y))$
is a quasi-isomorphism in $\C(k)$ for any $x,y \in \catD$.

\item
$f$ is \emph{quasi-essentially surjective} if
the induced functor $[f]: [\catD] \to [\catD']$ is essentially surjective.

\item
$f$ is a \emph{quasi-equivalence} if it is quasi-fully faithful 
and quasi-essentially surjective.
\end{enumerate}
\end{dfn}

Finally we introduce

\begin{ntn*}
We denote by $\catD^{\op}$ the \emph{opposite} of a dg-category $\catD$.
It is a dg-category given by $\Ob(\catD^{\op}) := \Ob(\catD)$,
$\Hom_{\catD^{\op}}(x,y) := \Hom_{\catD}(y,x)$ and 
the composition map with an appropriate sign.
\end{ntn*}

\subsection{Perfect dg modules}
\label{ss:mod:pdg}

In this subsection we recall the notion of perfect dg-modules. 
The main references are \cite{T07} and \cite[\S 2]{TVa}.
We fix a dg-category $\catD$ over $k$.

\begin{dfn*}
A \emph{$\catD$-dg-module} is a $\C(k)$-enriched functor $\catD \to \C(k)$.
We denote the category of $\catD$-dg-modules over by $\dgMod(\catD)$.
\end{dfn*}

The category $\dgMod(\catD)$ has a model structure 
such that a morphism $f:F \to G$ is a weak equivalence (resp.\ fibration)
if for any $z \in \catD$ the morphism $f_z: F(z) \to G(z)$ is 
a weak equivalence (resp.\ fibration) in $\C(k)$.
We always regard $\dgMod(\catD)$ as a model category by this structure.
Then 

\begin{lem}\label{lem:mod:dgMod=t}
The model category $\dgMod(\catD)$ is stable in the sense of \cite[\S 7]{H},
so that the homotopy category $\mHo \dgMod(\catD)$ 
has a natural triangulated structure
whose triangles are the image of of homotopy fiber sequences.
\end{lem}

The dg structure and the model structure make $\dgMod(\catD)$ 
a $\C(k)$-model category in the sense of \cite[Definition 4.2.18]{H}.
So, let us give an interlude on $\C(k)$-model categories.

\begin{ntn*}
For a $\C(k)$-model category $\catM$, we denote by $\catM^{\circ} \subset \catM$ 
the model subcategory of fibrant-cofibrant objects 
(using the same symbol as \S \ref{ss:ic:uis}).
We endow $\catM^{\circ}$ with the dg structure by restriction of that on $\catM$,
and consider $\catM^{\circ}$ as a $\C(k)$-model category.
\end{ntn*}

In \cite{TVa}, $\catM^{\circ}$ is denoted by $\operatorname{Int}(M)$.
For a $\C(k)$-model category $\catM$, we have an equivalence 
\begin{equation}\label{eq:mod:mHo=bra}
 \mHo(\catM) \simeq [\catM^{\circ}]
\end{equation}
by \cite[Proposition 3.5]{T07}.
Here $[\catM^{\circ}]$ is the category 
arising from the dg structure (Notation \ref{ntn:mod:bra}),
and $\mHo(\catM)$ is the homotopy category arising from the model structure.

Let us return to the discussion on $\dgMod(\catD)$,
We denote by
\[
 \dgMod(\catD)^{\circ} \subset \dgMod(\catD)
\]
the full sub-dg-category of fibrant-cofibrant objects 
(using the same symbol as in \S \ref{ss:ic:uis}).
By the discussion in \cite[\S 2.2 pp.\ 399--400]{TVa},
any object of $\dgMod(\catD)$ is fibrant, so that $\dgMod(\catD)^{\circ}$
is actually equivalent to the full sub-dg-category of cofibrant objects.
Note that in \cite{TVa} it is denoted by $\wh{\catD^{\op}}$.

\begin{dfn*}
A $\catD$-dg-module $M \in \dgMod(\catD)^{\circ}$ is called \emph{perfect} if 
it is homotopically finitely presented in the model category $\dgMod(\catD)$.
We denote by 
\[
 \Perf(\catD) := \dgMod(\catD)^{\circ}_{\pe} \subset \dgMod(\catD)^{\circ}
\] 
the full sub-dg-category of perfect objects.
\end{dfn*}

In other words, for a perfect $\catD$-dg-module $M$,
the natural morphism 
\[
 \iclim_{i \in I} \Map_{\dgMod(\catD)}(M,N_i) \longto 
 \Map_{\dgMod(\catD)}(M, \iclim_{i \in I} N_i)
\]
is an isomorphism in the topological category $\topH$ of spaces
(Definition \ref{dfn:ic:topH})
for any filtered system $\{N_i\}_{i \in I}$ of $\catD$-dg-modules. 
%
Note that $\Perf(\catD^{\op})$ is denoted by $\wh{\catD}_{\pe}$ in \cite{TVa}.

\begin{rmk}\label{rmk:mod:Perf}
By Lemma \ref{lem:mod:dgMod=t} and the equivalence \eqref{eq:mod:mHo=bra},
we can regard $[\dgMod(\catD)^{\circ}] \simeq \mHo(\dgMod(\catD))$ 
as a triangulated category.
Then a perfect $\catD$-dg-module is nothing but a compact objects of 
this triangulated category.
\end{rmk}

The $\C(k)$-enriched version of the Yoneda lemma gives 
a quasi-fully faithful dg-functor 
\[
 \catD \longto \dgMod(\catD^{\op}), \quad x \longmapsto h_x.
\] 
For any $x \in \catD$, the $\catD^{\op}$-dg-module $h_x$ is perfect.
Thus the above dg-functor factors to
\[
 h: \catD \longto \Perf(\catD).
\]

\begin{ntn}\label{ntn:mod:dgY}
We call the dg-functor $h$ the \emph{dg Yoneda embedding}.
\end{ntn}

We close this part by recalling
the notion of pseudo-perfect dg-modules \cite[Definition 2.7]{TVa}.
We denote by $\dgCat_{\bbV}$ the category of $\bbV$-small dg-categories 
over $k$ and dg-functors (Notation \ref{ntn:mod:dgCat}).
By \cite{Tab}, it has a model structure in which 
weak equivalences are quasi-equivalences (Definition \ref{dfn:mod:qe}).
Recall that we have a tensor product $\catD \otimes \catD'$ of 
dg-categories $\catD$ and $\catD'$ with
$\Ob(\catD \otimes \catD'):= \Ob(\catD) \times \Ob(\catD')$ and 
\[
 \Hom_{\catD \otimes \catD'}((x,y),(x',y')) := 
 \Hom_{\catD}(x,y) \otimes_k \Hom_{\catD'}(y,y').
\]
See also \cite[\S 4]{T07}.
This tensor product gives rise to a derived tensor product $\otimes^{\bbL}$ 
in the homotopy category $\mHo(\dgCat_{\bbV})$.
Here the symbol $\mHo$ denotes the homotopy category of 
a model category (\S \ref{ss:intro:ntn}).

Let $\catD$ and $\catD'$ be two dg-categories.
By applying to $\catD \otimes^{\bbL} \catD'$ the construction 
$\Perf(-)$ which is well-defined on $\mHo(\dgCat_{\bbV})$,
we obtain $\Perf(\catD \otimes^{\bbL} \catD')$.
For each object $E$ of this $\C(k)$-model category, we can define a functor
$F_E: \catD \to \dgMod(\catD')^{\circ}$ by
sending $x \in \catD$ to 
\[
 F_E(x) : \catD' \longto \C(k), \quad y \longmapsto E(x,y).
\]

\begin{dfn*}
An object $E \in \Perf(\catD \otimes^{\bbL} \catD')$ is called 
\emph{pseudo-perfect relatively to $\catD'$} if the morphism $F_E$ factorizes 
through $\Perf(\catD')$ in $\mHo(\dgCat_{\bbV})$. 

If $\catD'= \kdg$  (Notation \ref{ntn:mod:dg-one}),
then such $E$ is called a \emph{pseudo-perfect $\catD$-dg-module}.
\end{dfn*}

\subsection{Moduli functor of perfect objects}
\label{ss:mod:mf}

We continue to use the symbols in the previous \S \ref{ss:mod:pdg}.
The main reference of this part is \cite[\S 3]{TVa}.

For a commutative simplicial $k$-algebra $A \in \sCom_k$, 
we denote by $N(A)$ the normalized chain complex
with the structure of a commutative $k$-dg-algebra
(see \S \ref{ss:sp:sm} for the detail).
We consider $N(A)$ as a dg-category with one object, and apply the argument 
in the previous \S \ref{ss:mod:pdg} to the dg-category $\catD = N(A)$.
Then we have the dg-category 
\[
 \dgMod(A) := \dgMod(N(A))
\]
of dg-modules over the dg-algebra $N(A)$.
We then have full sub-dg-category $\dgMod(A)^{\circ} \subset \dgMod(A)$
spanned by (fibrant-)cofibrant objects,
and denote by 
\[
 \Perf(A) := \Perf(N(A)) = \dgMod(N(A))^{\circ}_{\pe} 
\]
the full sub-dg-category of perfect objects in $\dgMod(A)^{\circ}$.

The correspondence $A \mapsto \dgMod(A)$ induces, after a strictification 
procedure, a functor from $\sCom_k$ to the category of $\C(k)$-model categories
and $\C(k)$-enriched left Quillen functors
(see \cite[\S 3.1, p.417--418]{TVa} for the detail).
Applying the construction $\catM \mapsto \catM^{\circ}$ 
for $\C(k)$-model-categories $\catM$ levelwise,
we obtain a functor $\sCom_k \to \dgCat_{\bbV}$, $A \mapsto \wh{A}$.
Here $\dgCat_{\bbV}$ denotes the model category of the $\bbV$-small dg-categories.
Taking the sub-dg-category of perfect objects.
we obtain a functor 
\[
 \sCom_k \longto \dgCat_{\bbV}, \quad 
 A \longmapsto \wh{A}_{\pe}, \quad 
 (A \to B) \longmapsto (N(B) \otimes_{N(A)} -: \wh{A}_{\pe} \to \wh{B}_{\pe}).
\]

Next we turn to the definition of the moduli functor.
Let $\catD$ be a dg-category over $k$, and consider the following functor.
\[
 \stM_{\catD}: \sCom_k \longto \sSet,\quad 
 \stM_{\catD}(A) := \dgMap_{\dgCat_{\bbV}}(\catD^{\op},\Perf(A)).
\]
Here $\sSet$ denotes the category of simplicial sets (Definition \ref{dfn:ic:CS})
and $\dgMap_{\dgCat_{\bbV}}$ denotes the mapping space 
in the model category $\dgCat_{\bbV}$,
which is regarded as a simplicial set.
For a morphism $A \to B$ in $\sCom_k$, 
the morphism $\stM_{\catD}(A) \to \stM_{\catD}(B)$ is given by 
composition with $N(B) \otimes_{N(A)} -: \Perf(A) \to \Perf(B)$. 

By Fact \ref{fct:ic:sag=Kan}, the value of the functor $\stM_{\catD}$ 
is actually in the category of Kan complexes.
Recall also that in \S \ref{ss:dSt:dSt} we set $\idAff_k = (\isCom_k)^{\op}$.
Thus we see that $\stM_{\catD}$ determines a presheaf of spaces over $\idAff_k$
in the sense of Definition \ref{dfn:ic:iPSh}.
We will denote the obtained presheaf by the same symbol:
\[
 \stM_{\catD} \in \iPSh(\idAff_{k}) = \iFun((\idAff_{k})^{\op},\iS).
\]
Now let us cite

\begin{fct*}[{\cite[Lemma 3.1]{TVa}}]
The presheaf $\stM_{\catD} \in \iPSh(\idAff_{k})$ 
is a derived stack over $k$.
\end{fct*}

Following \cite[Definition 3.2]{TVa}, we call $\stM_{\catD}$ 
the \emph{moduli stack of pseudo-perfect $\catD^{\op}$-dg-modules}.

As explained in loc.\ cit., the $0$-th homotopy $\pi_0(\stM_{\catD}(k))$ 
is bijective to the set of isomorphism classes of 
pseudo-perfect $\catD^{\op}$-dg-modules in $\mHo(\dgMod(\catD^{\op}))$.
For each $x \in \mHo(\dgMod(\catD^{\op}))$, we have
$\pi_1(\stM_{\catD^{\op}},x) \simeq \Aut_{\mHo(\dgMod(\catD^{\op}))}(x,x)$ and 
$\pi_i(\stM_{\catD^{\op}},x) \simeq \Ext^i_{\mHo(\dgMod(\catD^{\op}))}(x,x)$ 
for  $i \in \bbZ_{\ge 2}$, 
where $\mHo(\dgMod(\catD^{\op}))$ is regarded as a triangulated category.

Let us cite another observation from \cite{TVa}.

\begin{dfn}[{\cite[Definition 2.4]{TVa}}]\label{dfn:mod:sat}
Let $\catD$ be a dg-category over $k$.
\begin{enumerate}[nosep]
\item
$\catD$ is \emph{proper} if the triangulated category $[\dgMod(\catD^{\op})]$
has a compact generator, and if $\Hom_{\catD}(x,y)$ is a perfect complex
of $k$-modules for any $x,y \in \catD$.

\item
$\catD$ is \emph{smooth} if 
the $(\catD^{\op} \otimes^{\bbL})$-dg-module 
$\catD^{\op} \otimes^{\bbL} \catD \to \C(k)$, 
$(x,y) \mapsto \Hom_{\catD}(x,y)$ 
is perfect.

\item
$\catD$ is \emph{triangulated} if the dg Yoneda embedding 
$\catD \to \Perf(\catD^{\op})$ (Notation \ref{ntn:mod:dgY}) is a quasi-equivalence.

\item
$\catD$ is \emph{saturated} if it is proper, smooth and triangulated.
\end{enumerate}
\end{dfn}

\begin{fct}\label{fct:mod:sat}
If the dg-category $\catD$ is saturated, 
then we have an equivalence
\[
  \stM_{\catD}(k) \simeq \dgMap_{\dgCat}(\kdg,\catD)
\]
of simplicial sets.
In particular, $\stM_{\catD}(k)$ is a model for 
the classifying space of objects in $\catD$.
\end{fct}

%

\subsection{Geometricity of moduli stacks of perfect objects}
\label{ss:mod:g}

Now we can explain the main result in \cite{TVa}.

\begin{dfn}[{\cite[Definition 2.4.\ 7]{TVa}}]\label{dfn:mod:oft}
A dg-category $\catD$ over $k$ is \emph{of finite type}
if there exists a $k$-dg-algebra $B$  which is homotopically 
finitely presented in the model category $\dgA_k$ of $k$-dg-algebras 
such that $\Perf(\catD)=\dgMod(\catD)^{\circ}_{\pe}$ 
is quasi-equivalent (Definition \ref{dfn:mod:qe}) to $\dgMod(B)^{\circ}$ .
\end{dfn}

\begin{fct}[{\cite[Theorem 3.6]{TVa}}]\label{fct:TVa}
If $\catD$ is a dg-category over $k$ of finite type,
then the derived stack $\stM_{\catD}$ is 
locally geometric and locally of finite presentation.
\end{fct}

Since this fact is crucial for our study,
let us explain an outline of the proof.
We start with the definition of Tor amplitude of dg-modules.

\begin{dfn*}[{\cite[Definition 2.21]{TVa}}]
Let $A \in \sCom_k$, 
and $N(A)$ be the commutative $k$-dg-algebra in \S \ref{ss:mod:mf}. 
A dg-module $P$ over $N(A)$ is called \emph{of Tor amplitude contained in $[a,b]$}
if for any (non-dg) left module $M$ over $\pi_0(A)$
we have $H^i(P \otimes_{N(A)}^{\bbL} M)=0$ for any $i \notin [a,b]$.
Here $\otimes_{N(A)}^{\bbL}$ denotes the derived tensor product
arising from the tensor product $\otimes_{N(A)}$ of dg-modules over $N(A)$.
\end{dfn*}

\subsubsection{}

Now let us explain an outline of the proof.
We first consider the case $\catD=\kdg$, the trivial dg-category $\kdg$ 
over $k$ (Notation \ref{ntn:mod:dg-one}), which is of finite type.

The derived stack $\stM_{\kdg}$ can be regarded as 
the moduli stack of perfect $k$-modules.
Let us define a substack $\stM_{\kdg}^{[a,b]}$ of $\stM_{\kdg}$ 
in the following way:
For each $A \in \sCom_k$, define $\stM_{\kdg}^{[a,b]}(A)$ to be 
the full simplicial subset of $\stM_{\kdg}(A)$ spanned by
connected components corresponding to 
perfect dg-modules over $N(A)$ of Tor amplitude contained in $[a,b]$.
Then $\stM_{\kdg}^{[a,b]}$ is a derived stack and we have 
\[
 \stM_{\kdg} = \tbcup_{a \le b} \stM_{\kdg}^{[a,b]}.
\]
In case of $b=a$, we have 
\[
 \stM_{\kdg}^{[a,a]} \simeq \stVect := \tbcup_{r \in \bbN} \stVect_r,
\]
where $\stVect_r$ is the derived stack of rank $r$ vector bundles 
(Definition \ref{dfn:dSt:Vect}).
By Fact \ref{fct:dSt:Vect}, the derived stack $\stM_{\kdg}^{[a,a]}$ is truncated, 
$1$-geometric, smooth and of finite presentation.

In \cite[Proposition 3.7]{TVa}, it is shown that $\stM_{\kdg}^{[a,b]}$ is 
an $n$-geometric derived stack locally of finite presentation with $n=b-a+1$
by induction on $n$.
The proof is done by constructing a smooth surjection 
$\stU \to \stM_{\kdg}^{[a,b]}$ 
with $\stU$ $n$-geometric and locally of finite presentation,
so that it gives an $n$-atlas of $\stM_{\kdg}^{[a,b]}$. 

The derived stack $\stU$ is given as follows.
For a derived $k$-algebra $A \in \isCom_k$, 
consider the category $\dgMod(A)^I$ of morphisms in $\dgMod(A)$.
Here $I=\Delta^1$ means the $1$-simplex.
This category has a model structure induced levelwise by 
the projective model structure on $\dgMod(A)$.
We denote by $\dgMod(A)^{I,\cof} \subset \dgMod(A)^{I}$
the full subcategory spanned by cofibrant objects, and define
\[
 \E(A) \subset \dgMod(A)^{I,\cof}
\]
{}to be the full subcategory spanned by those objects 
$Q \to R$ in $\dgMod(A)^{I,\cof}$ such that 
$Q \in \stM_{\kdg}^{[a,b-1]}(A)$ and $R \in \stM_{\kdg}^{[b-1,b-1]}(A)$.
It has a levelwise model structure.
Then let $\E(A)_W \subset \E(A)$ be the subcategory of weak equivalences.
Now the derived stack $\stU$ is given by
\[
 \stU: \isCom_A \longto \iS, \quad A \longmapsto \Ner(\E(A)_W).
\]

We have a natural morphism $\stU \to \stM_{\kdg}^{[a,b-1]} \times \stVect$,
and by induction hypothesis the target is $(n-1)$-geometric 
and locally of finite presentation.
By \cite[Sub-lemma 3.9, 3.11]{TVa}, we can describe the fiber of $p$ explicitly, 
and find that $p$ is $(-1)$-representable and locally of finite presentation.
It implies that $\stU$ is $(n-1)$-representable and locally of finite presentation.

Next we construct a morphism $\stU \to \stM_{\kdg}^{[a,b]}$ of derived stacks.
For each $A \in \isCom_k$, consider the morphism $\stU(A) \to \stM_{\kdg}(A)$
sending a morphism $Q \to R$ of $N(A)$-dg-modules to its homotopy fiber.
The definition of $\stU$ yields that the homotopy fiber does belong to 
$\stM_{\kdg}^{[a,b]}(A)$.
Thus we have a morphism $\stU \to \stM_{\kdg}^{[a,b]}$.
This is obviously $(n-1)$-representable and locally of finite presentation.
A simple argument shows that it is a surjection.
The proof of smoothness is non-trivial but it is done in \cite[Lemma 3.12]{TVa}.

\subsubsection{}
\label{sss:mod:Mab}

We turn to the case of an arbitrary dg-category $\catD$ of finite type.
Let us define the derived stack $\stM_{\catD}^{[a,b]}$ 
by the following cartesian square in $\idSt$.
\[
 \xymatrix{
 \stM_{\catD}^{[a,b]} \ar[r] \ar[d] & \stM_{\kdg}^{[a,b]} \ar[d] \\
 \stM_{\catD} \ar[r] & \stM_{\kdg} 
 }
\]
Since $\stM_{\kdg} = \bigcup_{a \le b} \stM_{\kdg}^{[a,b]}$,
we have $\stM_{\catD} = \bigcup_{a \le b}\stM_{\catD}^{[a,b]}$.
In \cite[Proposition 3.13]{TVa} it is shown that 
the morphism $\stM_{\catD}^{[a,b]} \to \stM_{\kdg}^{[a,b]}$ 
is $n$-representable and locally of finite presentation for some $n$.
Then one can deduce that $\stM_{\catD}^{[a,b]}$ is a locally geometric derived stack
locally of finite presentation.
Here one needs to show inductively that $\stM_{\catD}^{[a,b]}$ 
is strongly quasi-compact in the sense of \cite[\S 2.3]{TVa}.

Let us close this subsection by the estimate on the geometricity $n$ 
appearing in the last argument.
Recall from the beginning of this subsection that $\Perf(\catD)$ 
is quasi-equivalent to $\dgMod(B)^{\circ}$ with some $k$-dg-algebra $B$
of homotopically finitely presentation.
By \cite[Proposition 2.2]{TVa}, $B$ is equivalent to a retract of 
a dg-algebra $B'$ which sits in a sequence
\[
 k = B_0 \longto B_1 \longto \cdots \longto B_m \longto B_{m+1} = B'
\]
such that for each $i=0,\ldots,m$, there exists a homotopy pushout square
\[
 \xymatrix{ C_i \ar[r] \ar[d] & k \ar[d] \\ B_i \ar[r] & B_{i+1}}
\]
with $C_i$ the free dg-algebra over the complex $k[p_i]$ for some $p_i \in \bbZ$.
Then the argument in \cite[Proposition 3.13]{TVa} shows 

\begin{rmk}\label{rmk:mod:n}
We can take
\[
 n = b-a-\min_{0 \le i \le m} p_i.
\]
\end{rmk}

\subsection{Moduli stack of complexes of quiver representations}
\label{sss:mod:Pf(Q)}

We cite an application of Fact \ref{fct:TVa} from \cite[\S 3.5]{TVa}.
Let $Q$ be a quiver and 
$B(Q) := k Q$ be the  path algebra of $Q$ over $k$.
Considering $B(Q)$ as a dg-algebra over $k$,
we have a dg-category with a unique object $B(Q)$,
which will be denoted by $\dgB(Q)$.
Under the notation in the previous \S \ref{ss:mod:pdg}, 
the dg-category $\dgMod(\dgB(Q))^{\circ}$ can be identified with 
the dg-category of complexes of representations of $Q$ over $k$.
A pseudo-perfect object in $\dgMod(\dgB(Q))$ is a complex of representations of $Q$
whose underlying complex of $k$-modules is perfect.

If $Q$ is finite and has no loops, 
then $B(Q)$ is a projective $k$-module of finite type,
and is a perfect complex of $B(Q) \otimes_k B(Q)^{\op}$-modules.
Moreover the dg-category $\dgB(Q)$ is smooth and proper.
Now we can apply Fact \ref{fct:TVa} to the dg-category $\dgB(Q)$.

\begin{dfn*}
Let $Q$ be a finite quiver with no loops.
We call the derived stack $\stM_{\dgB(Q)}$ 
the \emph{derived stack of perfect complexes of representations of $Q$}
and denote it by 
\[
 \stPf(Q) := \stM_{\dgB(Q)}.
\]
\end{dfn*}

\begin{fct}[{\cite[Corollary 3.34]{TVa}}]\label{fct:TVa:Q}
Let $Q$ be a finite quiver with no loops.
Then the derived stack $\stPf(Q)$
is locally geometric and locally of finite presentation over $k$.
\end{fct}

By \cite[Example 2.5.\ 3]{TVa}, 
for any finite dimensional $k$-algebra $B$ of finite global dimension,
the dg-category $\dgB$ is saturated (Definition \ref{dfn:mod:sat}).
In particular, $\dgB(Q)$ is saturated.
Thus Fact \ref{fct:mod:sat} works for $\stPf(Q)$, 
and we can regard $\pi_0(\stPf(Q)(k))$ as the set of isomorphism classes 
of representations of $Q$ over $k$.

\begin{rmk}\label{rmk:mod:ab}
If we consider the abelian category $\Mod(B(Q))$ of
representations of the quiver $Q$ instead of $\dgMod(\dgB(Q))$,
then the corresponding moduli stack is an algebraic stack locally of finite type
by Remark \ref{rmk:mod:n},
\end{rmk}

\section{Derived Hall algebra and its geometric formulation}
\label{s:H}


\subsection{Ringel-Hall algebra}
\label{ss:H:RH}

In this section we give a brief account on the Ringel-Hall algebra.

We call a category $\catA$ \emph{essentially small}
if the isomorphism classes of objects form a small set,
which is denoted by $\Iso(\catA)$.
For an object $M$ of $\catA$ 
its isomorphism class is denoted by $[M] \in \Iso(\catA)$.


Let $k=\bbF_q$ be a finite field with $\abs{k}=q$.
Let $\catA$ be a category satisfying the following conditions.
\begin{enumerate}[nosep, label=(\roman*)]
\item
$\catA$ is essentially small, abelian and $k$-linear.

\item
$\catA$ is of finite global dimension, 
and $\Ext^i_{\catA}(-,-)$ has finite dimension over $k$ for any $i \in \bbN$.
\end{enumerate}

We denote by $\bbQ_c(\catA)$ 
the linear space of $\bbQ$-valued functions on $\Iso(\catA)$ with finite supports. 
We have a basis $\{1_{[M]} \mid [M] \in \Iso(\catA)\}$ of $\bbQ_c(\catA)$,
where $1_{[M]}$ means the characteristic function of $[M]$.
The correspondence $1_{[M]} \mapsto [M]$ gives an identification
$\bbQ_c(\catA) \simto \bigoplus_{[M] \in \Iso(\catA)}\bbQ[M]$,
and we will always identify these two spaces.


For $M,N,R \in \Ob(\catA)$, we set 
$\Ext^1_{\catA}(M,N)_R := 
 \{0 \to N \to R \to M  \to 0 \mid \text{exact in } \catA\}$
and 
\begin{align*} 
&a_M := \abs{\Aut(M)}, \quad
 e_{M,N}^R :=  \abs{\Ext^1_{\catA}(M,N)_R}, \quad
 g_{M,N}^R :=  a_M^{-1} a_N^{-1} e_{M,N}^R.
\end{align*}
Note that $a_M,e_{M,N}^R,g_{M,N}^R$ are well-defined by the condition (i) above
and depend only on the isomorphism classes $[M],[N],[R] \in \Iso(\catA)$.

For $[M], [N] \in \Iso(\catA)$ we define $[M] * [N] \in \bbQ_c(\catA)$ by
\[
 [M] * [N] := 
 \sum_{[R] \in \Iso(\catA)} g^R_{M,N} [R],
\]
where we choose representatives $M,N,R \in \Ob(\catA)$ 
for the fixed isomorphism classes $[M],[N],[R] \in \Iso(\catA)$.
Denote by $[0] \in \Iso(\catA)$ the isomorphism class 
of the zero object $0$ in $\catA$.

\begin{fct}[{\cite{R}}]\label{fct:H:RH}
The triple
\[
 \Hall(\catA) := \bigl(\bbQ_c(\catA),*,[0]\bigr)
\]
is a unital associative $\bbQ$-algebra 
which has a grading with respect to the Grothendieck group $K_0(\catA)$ of $\catA$.
It is called the \emph{Ringel-Hall algebra} of $\catA$.
\end{fct}

Let us recall another definition of $g_{M,N}^R$.
For $M,N,R \in \Ob(\catA)$ we have
\[
 g^{R}_{M,N} = \abs{\calG^{R}_{M,N}}, \quad 
 \calG^{R}_{M,N} := \{N' \subset R \mid N' \simeq N, \ R/N \simeq M\}.
\]
One can prove this statement by considering a free action 
of $\Aut(M) \times \Aut(N)$ on $\Ext^1_{\catA}(M,N)_R$.

We have the following meaning of the multi-component product.
For $M,N_1,\ldots,N_r \in \Ob(\catA)$, we set 
\[
 \calF(M;N_1,\ldots,N_r) := 
 \{M_\bullet=(M=M_1 \supset \cdots \supset M_r \supset M_{r+1}=0) 
  \mid  M_i/M_{i+1} \simeq N_i \ (i=1,\ldots,r) \}.
\]
Then we have 
\[
 [N_1] * \cdots * [N_r] = 
 \sum_{[M] \in \Iso(\catA)} \sum_{M_\bullet \in \calF(M;N_1,\ldots,N_r)} [M].
\]

\subsection{Derived Hall algebra}
\label{ss:H:dgHall}

In this subsection we recall the derived Hall algebra 
introduced by To\"{e}n \cite{T06}.

Let $\catD$ be a small dg-category over $\bbF_q$
which is locally finite in the sense of \cite[Definition 3.1]{T06}.
In other words, for any objects $x,y \in \catD$, 
the complex $\Hom_{\catD}(x,y)$ is cohomologically bounded 
and with finite-dimensional cohomology groups.

As in \S \ref{ss:mod:pdg},  we denote by $\dgMod(\catD)$ 
the model category of $\catD$-dg-modules, and by
\[
 \Perf(\catD) := \dgMod(\catD)^{\circ}_{\pe} \subset \dgMod(\catD)
\] 
the full subcategory spanned by perfect (fibrant-)cofibrant objects. 
The category $\Perf(\catD)$ has an induced model structure, and 
we can consider the homotopy category $\mHo \Perf(\catD)$. 
For $x,y \in \Perf(\catD)$ and $i \in \bbZ$,
we denote by 
\[
 \Ext^i(x,y) := \Hom_{\mHo \Perf(\catD)}(x,y[i])
\]
where $[i]$ denotes the shift in the triangulated category 
$\mHo \Perf(\catD)$.

Let us consider the category 
$\dgMod(\catD)^I := \Fun(\Delta^1,\dgMod(\catD))$,
where $I=\Delta^1$ denotes the $1$-simplex (see \S \ref{ss:ic:ss}).
This category has a model structure determined levelwise 
by that on $\dgMod(\catD)$.
In particular, a cofibrant object in $\dgMod(\catD)^I$ 
is a morphism $u: x \to y$ in $\dgMod(\catD)$ where
$x$ and $y$ are cofibrant and $u$ is a cofibration.

For an object $u: x \to y$ in $\dgMod(\catD)^I$, we set
\[
 s(u) := x, \quad c(u) := y, \quad t(u) := y \coprod^x 0.
\]
These determine left Quillen functors 
\[
 s,c,t: \dgMod(\catD)^I \longto \dgMod(\catD).
\]

Next we denote by $w \Perf(\catD)^{\cof} \subset \Perf(\catD)$ 
the subcategory of cofibrant objects and weak equivalences between them.
We define the subcategory $w(\Perf(\catD)^I)^{\cof} \subset \Perf(\catD)^I$
in the same way,
By restricting the functors $s,c,t$ to these subcategories,
we obtain the following diagram.
\begin{align}\label{diag:dHall:toen}
 \xymatrix{
  w(\Perf(\catD)^I)^{\cof} \ar[r]^c \ar[d]_{s \times t} & 
  w \Perf(\catD)^{\cof} \\
  w \Perf(\catD)^{\cof} \times w \Perf(\catD)^{\cof} }
\end{align}
Let us then define objects $X^{(0)}(\catD)$ and $X^{(0)}(\catD)$ of $\topH$ 
(Definition \ref{dfn:ic:topH}) by 
\[
 X^{(0)}(\catD) := \Ho \Ndg(w \Perf(\catD)^{\cof}), \quad
 X^{(1)}(\catD) := \Ho \Ndg(w(\Perf(\catD)^I)^{\cof}).
\]
Here $\Ndg(-)$ denotes the dg-nerve functor (Definition \ref{dfn:stb:Ndg}),
and $\Ho: \sSet \to \topH$ denotes the functor 
giving the homotopy type of a simplicial set (Definition \ref{dfn:ic:Ho-type:sSet}).
Then, from the diagram \eqref{diag:dHall:toen},
we obtain the following diagram in $\topH$.
\[
 \xymatrix{
  X^{(1)}(\catD) \ar[r]^c \ar[d]_{s \times t} & X^0(\catD) \\
  X^{(0)}(\catD) \times X^{(0)}(\catD)}
\]

Now let us recall 

\begin{dfn*}
An object $X$ of $\topH$ is called \emph{locally finite} if 
for every $x \in X$,  each $\pi_i(X,x)$ is a finite group 
and there exists an $n \in \bbN$ such that $\pi_i(X,x)=0$ for each $i>n$.
We denote by $\topH^{\lf}$ the subcategory of $\topH$
spanned by locally finite objects.
\end{dfn*}

\begin{fct}[{\cite[Lemma 3.2]{T06}}]\label{fct:H:prp}
The homotopy types $X^{(0)}(\catD)$ and $X^{(1)}(\catD)$ are locally finite.
Moreover, for every (homotopy) fiber $F$ of the morphism $s \times t$,
the set $\pi_0(F)$ is finite.
\end{fct}

For $X \in \topH^{\lf}$, 
let us denote by $\bbQ_c(X)$ the linear space of 
$\bbQ$-valued functions on $X$ with finite supports.

\begin{dfn*}
Let $f: X \to Y$ be a morphism in $\topH^{\lf}$.
\begin{enumerate}[nosep]
\item
We define a linear map $f_!: \bbQ_c(X) \to \bbQ_c(Y)$ by
\[
 f_!(\alpha)(y):= \sum_{x \in \pi_0(X), \, f(x)=y}  \alpha(x) \cdot 
 \prod_{i >0}\Bigl(\abs{\pi_i(X,x)}^{(-1)^i} \, \abs{\pi_i(Y,y)}^{(-1)^{i+1}}\Bigr)
 \quad
 (\alpha \in \bbQ_C(X), \ y \in \pi_0(Y))
\]

\item
If $f$ has a finite fiber, 
then we define a linear map $f^*: \bbQ_c(Y) \to \bbQ_c(X)$ by
\[
 f^*(\alpha)(x):=\alpha(f(x)) \quad (\alpha \in \bbQ_C(Y), \ x \in \pi_0(X)).
\]

\end{enumerate}
\end{dfn*}

Now we can explain the definition of the derived Hall algebra.

\begin{fct}[{\cite[Definition 3.3, Theorem 4.1]{T06}}]\label{fct:H:DH}
For a locally finite dg-category $\catD$ over $\bbF_q$, we set
\[
 \mu := c_! \circ (s \times t)^*: 
 \bbQ_c(X^{(0)}(\catD)) \otimes \bbQ_c(X^{(0)}(\catD)) \longto 
 \bbQ_c(X^{(0)}(\catD)).
\]
Then the triple 
\[
 \Hall(\catD) := (\bbQ_c(X^{(0)}(\catD)), \mu,0)
\]
is a unital associative $\bbQ$-algebra.
It is called the \emph{derived Hall algebra} of $\catD$.
\end{fct}

If we take $\catD$ to be an abelian category $\catA$
satisfying the conditions (i) and (ii) in \S \ref{ss:H:RH},
then the derived Hall algebra $\Hall(\catA)$ is nothing but 
the Ringel-Hall algebra in Fact \ref{fct:H:RH}.
The grading is recovered by the Grothendieck group $K_0(\mHo \Perf(\catD))$
of the triangulated category $\mHo \Perf(\catD) \subset \mHo \dgMod(\catD)$
(Remark \ref{rmk:mod:Perf}).

\subsection{Geometric formulation of derived Hall algebra}

In this subsection we give the main content of this article.
We set $k=\bbF_q$, the finite field  of order $q$,
and take it as the base ring for the following discussion.
%
Let $\catD$ be a dg-category over $k$ which is of finite type.
We then have the moduli stack 
\[
 \stM_{\catD}: A \mapsto \dgMap(\catD^{\op},\Perf(A))
\]
of pseudo-perfect $\catD$-modules
locally geometric and locally of finite type.
By the discussion in \S \ref{sss:mod:Mab},
we have the stratification
\[
 \stM_{\catD} = \tbcup_{a \le b}\stM_{\catD}^{[a,b]}
\]
with each $\stM_{\catD}^{[a,b]}$ geometric and locally of finite presentation.
Note also that if we assume $\catD$ to be saturated,
then we have a decomposition 
\[
 \stM_{\catD}^{[a,b]} = 
 \tbscup_{\alpha \in K_0(\mHo \Perf(\catD))} \stM_{\catD}^{[a,b],\alpha},
\]
where $K_0(\mHo \Perf(\catD))$ denotes the Grothendieck group of 
the triangulated category $\mHo \Perf(D) \subset \mHo \dgMod(D)^{\circ}$ 
(Remark \ref{rmk:mod:Perf}).

\begin{dfn*}
We define the functor $\stEx_{\catD}:\isCom_k \to \iS$ by
\[
 \stEx_{\catD}(A) := 
 \dgMap_{\dgCat}\bigl(((\catD^{\op})^{I})^{\fib},\Perf(\catA)\bigr).
\]
Here $(\catD^{\op})^{I}$ denotes the $C(k)$-model category 
$\Fun(\Delta^1,\catD^{\op})$, and $(-)^{\fib}$ denotes 
the full sub-dg-category of fibrant objects.
We call it the \emph{moduli stack of cofibrations in $\catD$}.
\end{dfn*}

Note that we used $(-)^{\fib}$ on the opposite dg-category $\catD^{\op}$ to 
parametrize cofibrations in the original $\catD$.
Now the argument on $\stM_{\catD}$ (see \S \ref{ss:mod:g}) 
works for $\stEx_{\catD}$, and we have

\begin{lem}\label{lem:H:stEx}
The presheaf $\stEx_{\catD}$ is a derived stack over $k$
which is locally geometric and locally of finite type.
\end{lem}

Now we have a similar situation to \S \ref{ss:H:dgHall}.
There exist morphisms
\[
 s,c,t: \stEx_{\catD} \longto \stM_{\catD}
\]
of derived stacks sending a cofibration $u: N \to M$ to 
$s(u) = N$, $c(u) = M$ and  $t(u) = N \coprod^M 0$ respectively.

In the following we assume $\catD$ is saturated.
Similarly to $\stM_{\catD}$,
the derived stack $\stEx_{\catD}$ has a stratification 
\[
 \stEx_{\catD} = \tbcup_{a \le b}\stEx_{\catD}^{[a,b]},
\]
where $\stEx_{\catD}^{[a,b]}$ parametrizes $u: N \to M$
such that $M$ has Tor amplitude in $[a,b]$.
Each $\stEx_{\catD}^{[a,b]}$ has a decomposition
\[
 \stEx_{\catD}^{[a,b]} = 
 \tbscup_{\alpha,\beta \in K_0(\Perf(\catD))} 
 \stEx_{\catD}^{[a,b],\alpha,\beta},
\]
where $\stEx_{\catD}^{[a,b],\alpha,\beta}$ parametrizes $u: N \to M$
with $\alpha=\ol{N}$ and $\beta=\ol{t(u)}$.
Here we denoted by $\ol{N}$ the class of $N \in \Perf(D)$ in $K_0(\Perf(D))$.

Then the morphisms $s,t,u$ respect the stratification and decomposition.
The restrictions of $s,c,t$ give a diagram
\[
 \xymatrix{ 
  \stEx_{\catD}^{[a,b],\alpha,\beta} \ar[r]^{c} \ar[d]_{p} & 
   \stM_{\catD}^{[a,b],\alpha+\beta} \\ 
   \stM_{\catD}^{[a,b],\alpha} \times \stM_{\catD}^{[a,b],\beta} }
\]
of derived stacks with $p:= s \times t$.
The following is now obvious.

\begin{lem}\label{lem:H:sp}
The morphisms $s$ and $t$ are smooth, and the morphism $c$ is proper.
\end{lem}

Now let $\ol{\bbQ}_\ell$ be the algebraic closure of the field $\bbQ_\ell$ 
of $\ell$-adic numbers where $\ell$ and $q$ are assumed to be coprime.
We can apply the construction in \S \ref{s:adic} to the present situation,
and have the derived $\infty$-category $\iDDc(\stM_{\catD},\Qlb)$
of $\ell$-adic constructible sheaves.
We denote it by 
\[
 \iDDc(\stM_{\catD}) := \iDDc(\stM_{\catD},\Qlb).
\]
The stratification and decomposition of $\stM_{\catD}$ gives 
$\iDDc(\stM_{\catD}) = 
 \sum_{a \le b}\oplus_\alpha \iDDc(\stM_{\catD}^{[a,b],\alpha})$.
By Lemma \ref{lem:H:sp}, we have the derived functors
\[
 p^*: \iDDc(\stM_{\catD}^{\alpha} \times \stM_{\catD}^{\beta}) \longto 
      \iDDc(\stEx_{\catD}^{\alpha,\beta}), \quad
 c_!: \iDDc(\stEx_{\catD}^{\alpha,\beta}) \longto 
      \iDDc(\stM_{\catD}^{\alpha+\beta}).
\]
Here we suppressed Tor amplitude $[a,b]$ in the superscripts.
We also have
\[
 \iDDc(\stM_{\catD}^{\alpha} \times \stM_{\catD}^{\beta}) \simeq 
 \iDDc(\stM_{\catD}^{\alpha}, \Qlb) \times \iDDc(\stM_{\catD}^{\beta}),
\]
where the product in the right hand side denotes the product of simplicial sets.
Now we can introduce

\begin{dfn}\label{dfn:H:mu}
For $\alpha,\beta \in K_0(\Perf(\catD))$, 
we define a functor $\mu_{\alpha,\beta}$ by
\[
 \mu_{\alpha,\beta}: 
 \iDDc(\stM_{\catD}^{\alpha}) \times  \iDDc(\stM_{\catD}^{\beta }) \longto 
 \iDDc(\stM_{\catD}^{\alpha+\beta}), \quad
 M \longmapsto c_! p^*(M) [\dim p]
\]
They determine a functor 
\[
 \mu: \iDDc(\stM_{\catD}) \times \iDDc(\stM_{\catD}) \longto \iDDc(\stM_{\catD}).
\]
\end{dfn}

\begin{prp}
$\mu$ is associative.
In other words, we have in each component an equivalence
\[
 \mu_{\alpha,\beta+\gamma} \circ (\id \times \mu_{\beta,\gamma}) \simeq
 \mu_{\alpha+\beta,\gamma} \circ (\mu_{\alpha,\beta} \times \id)
\]
which is unique up to contractible ambiguity.
Here we suppressed Tor amplitude again.
\end{prp}

\begin{proof}
The following argument is standard, but let us write it down for completeness.
We follow the ``rough part" of the proof of \cite[Proposition 1.9]{S2} 
both for the argument and the symbols.

The functor $\mu_{\alpha,\beta+\gamma} \circ (\id \times \mu_{\beta,\gamma})$
in the left hand side 
corresponds to the rigid line part of the following diagram.
\[
 \xymatrix{
    \stEx^{\alpha,(\beta,\gamma)} \ar@{.>}[r]^{p_2''} \ar@{.>}[d]_(0.45){p_1''}
  & \stEx^{\alpha,\beta+\gamma}   \ar[r]^(0.45){p_2'} \ar[d]^(0.45){p_1'}
  & \stM^{\alpha+\beta+\gamma} \\
    \stM^{\alpha} \times \stEx^{\beta,\gamma} \ar[r]_{p_2} \ar[d]_{p_1}
  & \stM^{\alpha} \times  \stM^{\beta+\gamma} \\
    \stM^{\alpha} \times  \stM^{\beta} \times \stM^{\gamma} }
\]
Here we suppressed the symbol $\catD$ in the subscripts.
We can complete it into a commutative diagram by defining 
\[
 \stEx^{\alpha,(\beta,\gamma)} := 
 (\stM^{\alpha} \times \stEx^{\beta,\gamma}) 
 \times_{\stM^{\alpha} \times \stM^{\beta+\gamma}}
 \stEx^{\alpha,\beta+\gamma},
\] 
which is the moduli stack of pairs of cofibrations $(N \to M, M \to L)$ with
$\ol{N}=\gamma$, $\ol{M}=\beta+\gamma$ and $\ol{L}=\alpha+\beta+\gamma$.
Using the smooth base change (Proposition \ref{prp:adic:sbc})
in the completed diagram, we have
\begin{equation}\label{eq:H:sbc}
\begin{split}
       \mu_{\alpha,\beta+\gamma} \circ (\id \times \mu_{\beta,\gamma})
&\simeq (p_2')_! (p_1')^*  (p_2)_!   (p_1)^* [\dim p_1 + \dim p_1'] \\
&\simeq (p_2')_! (p_2'')_! (p_1'')^* (p_1)^* [\dim p_1 + \dim p_1'']
 \simeq (p_2' p_2'')_! (p_1 p_1'')^* [\dim (p_1 p_1'')].
\end{split}
\end{equation}

In the same way, the functor 
$\mu_{\alpha,\beta+\gamma} \circ (\id \times \mu_{\beta,\gamma})$
in the right hand side corresponds to the rigid line part of 
\[
 \xymatrix{
    \stEx^{(\alpha,\beta),\gamma} \ar@{.>}[r]^{p_2''} \ar@{.>}[d]_(0.45){p_1''}
  & \stEx^{\alpha+\beta,\gamma}   \ar[r]^(0.45){p_2'} \ar[d]^(0.45){p_1'}
  & \stM^{\alpha+\beta+\gamma}  \\
    \stEx^{\alpha,\beta} \times \stM^{\gamma} \ar[r]_{p_2} \ar[d]_{p_1}
  & \stM^{\alpha+\beta}  \times \stM^{\gamma} \\
    \stM^{\alpha}        \times \stM^{\beta}  \times \stM^{\gamma} }
\]
The completion is done with 
\[
 \stEx^{(\alpha,\beta,)\gamma} := 
 (\stEx^{\alpha,\beta} \times \stM^{\gamma}) 
 \times_{\stM^{\alpha+\beta} \times \stM^{\gamma}}
 \stEx^{\alpha+\beta,\gamma},
\] 
which is the moduli stack of pairs of cofibrations 
$(R \to L \coprod^M 0, M \to L)$ with
$\ol{M}=\gamma$, $\ol{R}=\beta$ and $\ol{L}=\alpha+\beta+\gamma$.
In this case the smooth base change yields 
the same calculation as \eqref{eq:H:sbc}.

Thus the conclusion holds if the derived stacks 
$\stEx^{\alpha,(\beta,\gamma)}$ and 
$\stEx^{(\alpha,\beta,)\gamma}$ are equivalent.
But this is shown in \cite{T06} on the $k$-rational point level.
The equivalence as derived stacks follows from the moduli property.
\end{proof}

Summarizing the arguments so far, we have 

\begin{thm}\label{thm:H:asc}
Let $\catD$ be a dg-category over $k=\bbF_q$ 
which is of finite type and saturated.
Assume that the positive integer $\ell$ is prime to $q$.
Then the derived $\infty$-category $\iDDc(\stM_{\catD},\Qlb)$ 
of $\ell$-adic constructible sheaves
on the moduli stack $\stM_{\catD}$ of pseudo-perfect $\catD^{\op}$-modules
has a unital associative ring structure with respect to 
the bifunctor $\mu$ (Definition \ref{dfn:H:mu}).
\end{thm}

Hereafter we denote 
\[
 \shM \star \shN := \mu(\shM,\shN).
\]
We use the ordinary symbol for an iterated multiplication:
$\shM_1 \star \cdots \star \shM_n := 
 \shM_1 \star (\shM_2 \star (\cdots \star \stM_n))$.

\subsection{Lusztig sheaves in the derived setting}

In this subsection we focus on the case
where $\catD$ is given by the path algebra of a quiver.
More precisely speaking, we consider the situation in \S \ref{sss:mod:Pf(Q)}.
Here is a list of the notations.
\begin{itemize}
\item
$Q$ denotes a finite quiver without loops.
\item
$B(Q):=(k Q)^{\op}$ denotes the opposite of the path algebra of $Q$
over $k=\bbF_q$.
\item
$\dgB(Q)$ denotes the dg-category with a unique object $B(Q)$.
\end{itemize}

Denoting by $\Rep(Q)$ the category of representation of $Q$, 
we know that $\Rep(Q)$ is a hereditary abelian category, i.e., 
its global dimension is $\le 1$.
We also know that $K_0(\Perf(\dgB(Q))) = K_0(\Rep(Q)) = \bbZ^{Q_0}$
with $Q_0$ the set of vertices in $Q$.
Its element is called a dimension vector.

Let us denote by
\[
 \stPf(Q) := \stM_{\dgB(Q)}
\]
the moduli stack of perfect complexes of representations of $Q$.
It has a stratification and decomposition
\[
 \stPf(Q) = \tbcup_{a \le b} \stPf(Q)^{[a,b]}, \quad
 \stPf(Q)^{[a,b]} = \tbscup_{\alpha \in \bbZ^{Q_0}} \stPf(Q)^{[a,b],\alpha}.
\]
Hereafter we denote 
\[
 \stPf(Q)^{n,\alpha} := \stPf(Q)^{[n,n],\alpha}.
\]

Recalling that $\dgB(Q)$ is of finite type and saturated,
we can apply Theorem \ref{thm:H:asc} to $\catD=\dgB(Q)$.
Then we have an associative ring structure on $\iDDc(\stPf(Q),\Qlb)$,
which is a geometric version of the derived Hall algebra for $Q$.

The $\infty$-category $\iDDc(\stPf(Q),\Qlb)$ is quite large,
and we should consider only an ``accessible" part of it.
Following the non-derived case established by Lusztig 
(see \cite[Chap.\ 9]{Lus} for example),
we consider the sub-$\infty$-category generated by constant perverse sheaves.
We also follow \cite[\S 1.4]{S2} for the argument and the symbols.

For each $\alpha \in \bbN^{Q_0}$ and $s \in \bbZ$, let 
\[
 \bbo_{\alpha,s} := 
 (\Qlb)_{\stPf(Q)^{n,\alpha}} [\dim \stPf(Q)^{s,\alpha}+s]
  \in \iDDcu{b}(\stPf(Q),\Qlb)
\]
be the shifted constant $\Qlb$-sheaf on $\stPf(Q)^{s,\alpha}$.
For a sequence $(\alpha_1,\ldots,\alpha_m)$ with $\alpha_j \in \bbN^{Q_0}$ 
and another $(s_1,\ldots,s_m)$ with $s_j \in \bbZ$, 
we define
\[
 L_{(\alpha_1,\ldots,\alpha_m)}^{(s_1,\ldots,s_m)} : = 
 \bbo_{\alpha_1,s_1} \star \cdots \star \bbo_{\alpha_m,s_m} 
 \in \iDDcu{b}(\stPf(Q),\Qlb).
\]
We call it a \emph{Lusztig sheaf}.

We call $\alpha \in \bbZ^{Q_0}$ \emph{simple}
if $\alpha = \ve_i$ with some $i \in Q_0$,
i.e., $\alpha$ has value $1$ at $i \in Q_0$ and $0$ at other $j \in Q_0$.
By Proposition \ref{prp:pv:dc}, 
the Lusztig sheaf $L_{\alpha}^{s}$ is a simple perverse sheaf.
For a fixed $\gamma \in \bbZ^{Q_0}$,
we denote by $\calP^\gamma$ the collection of all simple perverse sheaves on 
$\bigcup_{a \le b} \stPf(Q)^{[a,b],\gamma}$ arising, up to shift, 
as a direct summand of $L_{(\alpha_1,\ldots,\alpha_m)}^{(s_1,\ldots,s_m)}$
with $\alpha_1+\cdots+\alpha_n=\gamma$ and $\alpha_i$ simple for all $i$.
We also denote by $\iQ^{\gamma} \subset \iDDcu{b}(\stPf(Q),\Qlb)$ 
the full sub-$\infty$-category spanned by those objects 
which are equivalent to a direct sum $P_1 \oplus \cdots \oplus P_l$
with $P_j \in \calP^{\gamma}$,
and set $\iQ := \sqcup_{\gamma \in \bbZ^{Q_0}} \iQ^{\gamma}$.
We call it the \emph{derived Hall category}.

\begin{lem*}
$\iQ$ is preserved by the bifunctor $\mu$.
\end{lem*}

\begin{proof}
By the definition of Lusztig sheaves we have
\[
 L_{(\alpha_1,\ldots,\alpha_l)}^{(s_1,\ldots,s_l)} \star
 L_{(\beta_1,\ldots,\beta_m)}^{(t_1,\ldots,t_m)} \simeq
 L_{(\alpha_1,\ldots,\alpha_l,\beta_1,\ldots,\beta_m)}^{
    (s_1,\ldots,s_l),(t_1,\ldots,t_m)}.
\]
\end{proof}

\begin{rmk*}
If we restrict $s_i$'s to be zero,
then by Remark \ref{rmk:mod:ab} the related moduli stacks are algebraic stacks, 
so that we recover the \emph{Hall category} in the sense of \cite[\S 1.4]{S2}.
It has more properties such as the existence of a coproduct and a Hopf pairing.
We refer to loc.\ cit.\ and \cite{Lus} for the detail.
\end{rmk*}

\appendix
\section{Algebraic stacks}
\label{s:Cl}

In this appendix we recollect some basics on algebraic stacks
in the sense of \cite{LM,O:book}.

\subsection{Algebraic spaces}
\label{ss:Cl:AS}

We begin with the recollection of algebraic spaces.
Our main reference is \cite[\S 5.1]{O:book} and \cite{K}.
We will use the notion of Grothendieck topologies and sites freely. 
For a scheme $S$, we denote by $\Sch_S$ the category of $S$-schemes.
Let us fix the notation of the big \'etale topology 
since it will be used repeatedly.

\begin{dfn}\label{dfn:Cl:ET}
Let $S$ be a scheme.
The \emph{big \'etale site} $\gsET(S)$ over $S$ is defined to be 
the site $\gsET(S):=(\Sch_{S},\txET)$ consisting of 
\begin{itemize}[nosep]
\item
The category $\Sch_{S}$ of $S$-schemes.
\item
The  big \'etale topology $\txET$.
\end{itemize}
The Grothendieck topology $\txET$ is determined as follows:
the set $\Cov_{\txET}(U)$ of coverings of $U \in \Sch_S$ 
consists of a family $\{U_i \to U\}_{i \in I}$ of morphisms in $\Sch_S$
for which each $U_i \to U$ is \'etale and 
$\coprod_{i \in I}U_i \to U$ is surjective.
\end{dfn}

Recall that given a site $\gsS=(\catC,\tau)$
we have the notion of a \emph{sheaf} (\emph{of sets}) \emph{on $\gsS$} 
(see \cite[Chap.\ 2]{O:book} for example).
It is a functor $\catC^{\op} \to \Set$ satisfying the sheaf condition
with respect to the Grothendieck topology $\tau$.
Recall also that a morphism of sheaves on $\gsS$ 
is defined to be a morphism of functors.

Given a scheme $S$ and an $S$-scheme $T \in \Sch_S$, 
the functor $h_T(-):=\Hom_{\Sch_S}(-,T)$ defines a sheaf of sets 
on the big \'etale site $\gsET(S)$.
Hereafter we identify $T \in \Sch_S$ and the sheaf $h_T$ on $\gsET(S)$.

\begin{dfn*}
Let $S$ be a scheme.
An \emph{algebraic space} over $S$ is a functor $X: (\Sch_S)^{\op} \to \Set$
satisfying the following three conditions.
\begin{itemize}[nosep]
\item 
$X$ is a sheaf on the big \'etale site $\gsET(S)$.

\item
The diagonal map $X \to Y:=X \times_S X$ is represented by schemes,
i.e., the fiber product $X \times_Y T$ is a scheme  
for any $T \in \Sch_S$ and any morphism $T \to Y$ of sheaves on $\gsET(S)$.

\item 
There exists $U \in \Sch_S$ and a morphism $U \to X$ of sheaves on $\gsET(S)$ 
such that the morphism $p_T: U \times_X T \to T$ of schemes is surjective and 
\'etale for any $T \in \Sch_S$ and any morphism $T \to X$ of sheaves.
\end{itemize}
In this case we call either the scheme $U$ or the morphism $U \to X$ 
an \emph{\'etale covering} of $X$.

A \emph{morphism} of algebraic spaces over $S$ is a morphism of functors.
The category of algebraic spaces over $S$ is denoted by $\AlgSp_S$.
\end{dfn*}

Note that a scheme $U \in \Sch_S$ can be naturally regarded 
as an algebraic space over $S$.

Next we recall the description of an algebraic space as a quotient.

\begin{fct*}
Let $X$ be an algebraic space over a scheme $S$ 
and $U \to X$ be an \'etale covering.
\begin{enumerate}[nosep]
\item 
The fiber product $R:= U \times_X U$ in the category of sheaves
is representable by an $S$-scheme.
The corresponding scheme is denoted by the same symbol $R$.

\item
The natural projections $p_1,p_2: R \to U$ makes $R$ into 
an \'etale equivalence relation \cite[Definition 5.2.1]{O:book}.

\item
The natural morphism $U/R \to X$ from the quotient sheaf $U/R$
is an isomorphism.
\end{enumerate}
We denote this situation as 
$R \rightrightarrows U \to X$.
\end{fct*}

Using this quotient description 
we have a concrete way to give a morphism of algebraic spaces.

\begin{fct*}
Let $X_1$ and $X_2$ be algebraic spaces over a scheme $S$.
Take \'etale coverings $U_i \to X_i$ 
and denote $R_i:=U_i \times_{X_i} U_i$ ($i=1,2$).
\begin{enumerate}[nosep]
\item 
Assume that we have a diagram 
\[
 \xymatrix{
    R_1 
    \ar@<0.5ex>[r]^{p_{1,1}} \ar@<-0.5ex>[r]_{p_{1,2}} \ar[d]^g
  & U_1 \ar[r] \ar[d]^h & X_1 \\
    R_2 
    \ar@<0.5ex>[r]^{p_{2,1}} \ar@<-0.5ex>[r]_{p_{2,2}} 
  & U_2 \ar[r] & X_2
 }
\]
of sheaves on $\gsET(S)$ with $h \circ p_{1,i} = p_{2,i} \circ g$ ($i=1,2$).
Then there is a unique morphism $f: X_1 \to X_2$ 
making the resulting diagram commutative.

\item
Conversely, every morphism $X_1 \to X_2$ of algebraic spaces 
arises in this way for some choice of $U_1,U_2,g,h$.
\end{enumerate}
\end{fct*}

Using properties of schemes and their morphisms,
we can define properties of algebraic spaces and their morphisms.
We refer \cite[Definition 5.1.3]{O:book} for the definition of 
\emph{stability} and \emph{locality on domain}
of a property of morphisms in a site $\gsS$.

\begin{dfn}\label{dfn:Cl:AS:P}
Let $S$ be a scheme, $f: X \to Y$ be a morphism of algebraic spaces over $S$, 
and $\bfP$ be a property of morphisms in the big \'etale site $\gsET(S)$.
\begin{enumerate}[nosep]

\item
Assume $\bfP$ is stable.
Then we say $X$ \emph{has property $\bfP$} if it has an \'etale covering $U$ 
whose structure morphism $U \to S$ has property $\bfP$.

\item
Assume $\bfP$ is stable and local on domain.
We say $f$ \emph{has $\bfP$} if there exist \'etale coverings 
$u: U \to X$ and $v: V \to Y$ such that the morphism 
$p_V: U \times_{f\circ u,Y,v} V \to V$ in $\Sch_S$ has property $\bfP$.

\end{enumerate}
\end{dfn}

By \cite[Proposition 5.1.4]{O:book} and \cite[Chap.\ 2]{K},
the following properties $\bfP_1$ of morphisms are 
stable for the big \'etale site $\gsET(S)$.
\[
 \bfP_1 := \text{separated, universally closed, quasi-compact}.
\] 
By loc.\ cit., the following properties $\bfP_2$ are stable and local on domain.
\begin{align*}
 \bfP_2 := 
 &\text{surjective, \'etale, locally of finite type, 
        smooth, universally open,} \\
 &\text{locally of finite presentation, locally quasi-finite}. 
\end{align*}

Note that for an \'etale covering $f: U \to X$ of an algebraic scheme $X$
the morphism $f$ is \'etale and surjective 
in the sense of Definition \ref{dfn:Cl:AS:P}.

In the main text we need the \emph{properness} of a morphism of algebraic spaces.
In order to define it, 
we need to introduce some classes of morphisms between algebraic spaces
which are not given in Definition \ref{dfn:Cl:AS:P}.
Our main reference is \cite[Chap.\ 2]{K}.
Let us omit to mention the base scheme $S$ in the remaining part.

\begin{dfn}[{\cite[Chap.\ 2, Definition 1.6]{K}}]
\label{dfn:Cl:AS:q-cpt}
A morphism $f: X \to Y$ of algebraic spaces is \emph{quasi-compact} 
if for any \'etale morphism $U \to Y$
(Definition \ref{dfn:Cl:AS:P} (2))
with $U$ a quasi-compact scheme
the fiber product $X \times_Y U$ is a quasi-compact algebraic space 
(Definition \ref{dfn:Cl:AS:P} (1)). 
\end{dfn}

\begin{dfn}[{\cite[Definition 3.3]{K}}]
\label{dfn:Cl:AS:fin}
A morphism $f$ of algebraic spaces 
is \emph{of finite type} (resp.\ \emph{of finite presentation})
if it is locally of finite type (resp.\ locally of finite presentation)
in the sense of Definition \ref{dfn:Cl:AS:P} (2)
and quasi-compact (Definition \ref{dfn:Cl:AS:q-cpt}).
\end{dfn}

\begin{dfn}[{\cite[Chap.\ 2, Extension 3.8, Definition 3.9]{K}}]
\label{dfn:Cl:AS:cim}
Let $f: X \to Y$ be a morphism of algebraic spaces over a scheme $S$.
\begin{enumerate}[nosep]
\item 
$f$ is a \emph{closed immersion} if for any $U \in \Sch_S$
and any $U \to Y$
the fiber product $X \times_Y U$ is a scheme and $X \times_Y U \to U$
is a closed immersion in $\Sch_S$.
In this case, $X$ is called a \emph{closed subspace} of $Y$.

\item
$f$ is \emph{separated} if the induced  morphism $X \to X \times_Y X$
is a closed immersion in the sense of (1).
\end{enumerate}
\end{dfn}

Although it is not directly related to proper morphisms,
let us introduce now the notion of \emph{quasi-separated morphisms}.

\begin{dfn}[{\cite[Propoistion 5.4.7]{O:book}}]\label{dfn:Cl:AS:qsep}
A morphism $f: X \to Y$ of algebraic spaces over a scheme $S$ is 
\emph{quasi-separated} if the diagonal morphism
$\Delta_{X/Y}: X \to X \times_{f,Y,f} X$ is quasi-compact 
(Definition \ref{dfn:Cl:AS:q-cpt}).
An algebraic space $X$ over $S$ is \emph{quasi-separated} 
if the structure morphism $X \to S$ is quasi-separated. 
\end{dfn}

Going back to proper morphisms, we introduce 

\begin{dfn}[{\cite[Chap.\ 2, Definition 6.1, 6.9]{K}}]\label{dfn:AS:pt}
Let $X$ be an algebraic space.
\begin{enumerate}[nosep]
\item 
A \emph{point} of $X$ is a morphism $i: \Spec k \to X$ of algebraic spaces
where $k$ is  a field and $i$ is a categorical monomorphism.
Two points $i_j: p_j \to X$ ($j=1,2$) are \emph{equivalent}
if there exists an isomorphism $e:p_1 \to p_2$ such that $i_2 \circ e = i_1$.

\item
The \emph{underlying topological space} $\abs{X}$ of $X$
is defined to be the set of points of $X$ modulo equivalence
together with the topology where a subset $C \subset \abs{X}$ is closed
if $C$ is of the form $\abs{Y}$ for some closed subspace $Y \subset X$
(Definition \ref{dfn:Cl:AS:cim} (1))
\end{enumerate}
\end{dfn}

Note that a morphism $X \to Y$ of algebraic spaces
naturally induces a continuous map $\abs{X} \to \abs{Y}$
of the underlying topological spaces.

\begin{dfn}[{\cite[Chap.\ 2, Definition 6.9]{K}}]\label{dfn:Cl:AS:uni-cl}
A morphism $f: X \to Y$ of algebraic spaces is \emph{universally closed}
if for any morphism $Z \to Y$ of algebraic spaces 
the induced continuous map $\abs{X \times_Y Z} \to \abs{Z}$ is closed.
\end{dfn}

Let us also recall the notion of relaitive dimension.

\begin{dfn}\label{dfn:Cl:AS:reldim}
Let $f: X \to Y$ be a morphism of algebraic spaces over a scheme $S$,
and let $d \in \bbN \cup \{\infty\}$.
\begin{enumerate}[nosep]
\item 
For $x \in \abs{X}$, $f$ has \emph{relative dimension $d$ at the point $x$}
if for any commutative diagram 
\[
 \xymatrix{
  U \ar[r] \ar[d]_a & V \ar[d]^b & 
  u \ar@{|->}[r] \ar@{|->}[d] & v \ar@{|->}[d] \\
  X \ar[r]_f & Y & x \ar@{|->}[r] & y  }
\]
in $\Sch_S$ with $a$ and $b$ \'etale, 
the dimension of the local ring $\shO_{U_v,u}$ of the fiber $U_v$ at $u$  is $d$.

\item
$f$ has \emph{relative dimension} $d$ 
if it has relative dimension $d$ at any $x \in \abs{X}$.
\end{enumerate}
\end{dfn}

Finally we have 

\begin{dfn}[{\cite[Chap.\ 2, Definition 7.1]{K}}] \label{dfn:Cl:AS:prp}
A morphism $f: X \to Y$ of algebraic spaces is \emph{proper}
if it is separated (Definition \ref{dfn:Cl:AS:cim} (2)), 
of finite type (Definition \ref{dfn:Cl:AS:fin}) 
and universally closed (Definition \ref{dfn:Cl:AS:uni-cl}).
\end{dfn}

\subsection{Algebraic stacks}
\label{ss:Cl:AlgSt}

Next we recall the definition of algebraic stacks.
We assume some basics on categories fibered in groupoids \cite[Chap.\ 3]{O:book}.

Let $\gsS$ be a site.
In order to make terminology clear, let us call a stack  in the ordinary sense 
(\cite[Chap.\ 2]{LM}, \cite[Chap.\ 4]{O:book}) an \emph{ordinary stack}.
In other words, we have

\begin{dfn*}
Let $\gsS$ be a site.
An \emph{ordinary stack} over $\gsS$ is a category $\catF$ 
fibered in groupoids over $\gsS$ such that for any object $U \in \gsS$ 
and any covering $\{U_i \to U\}_{i \in I}$ of $U$
the functor $\catF(U) \to \catF(\{U_i \to U\}_{i \in I})$ 
is an equivalence of categories.
\end{dfn*}

Ordinary stacks over $\gsS$ form a $2$-category.

\begin{dfn}\label{dfn:Cl:ord-st}
An \emph{ordinary stack $X$} over a scheme $S$ means an ordinary stack $X$ 
over $\gsET(S)$, where $\gsET(S)$ is the big \'etale site on $S$
(Definition \ref{dfn:Cl:ET}).
\end{dfn}

A scheme and an algebraic space over $S$ 
can be naturally considered as an ordinary stack over $S$.
We further introduce

\begin{dfn}\label{dfn:Cl:St-mor}
Let $S$ be a scheme and 
$f:X \to Y$ be a $1$-morphism of ordinary stacks over $S$.
\begin{enumerate}[nosep]
\item
$f$ is called \emph{representable} 
if the fiber product $X \times_{f,Y,g} U$ is an algebraic space over $S$
for any $U \in \Sch_S$ and any $1$-morphism $g:U \to Y$. 

\item
$f$ is called \emph{quasi-compact} (resp.\ \emph{separated}) 
if it is representable and 
for any $U \in \Sch_S$ and any $1$-morphism $g: U \to Y$
the algebraic space $X \times_{f,Y,g} U$ is quasi-compact (resp.\ separated)
in the sense of Definition \ref{dfn:Cl:AS:P}.
\end{enumerate}
\end{dfn}

Here is the definition of an algebraic stack.

\begin{dfn}\label{dfn:Cl:AlgSt}
An \emph{algebraic stack over a scheme $S$} is defined to be 
an ordinary stack $X$ over $S$ satisfying the following two conditions.
\begin{itemize}[nosep]
\item
The diagonal $1$-morphism 
$\Delta: X \to X \times_{S} X$ is representable, quasi-compact and separated
in the sense of Definition \ref{dfn:Cl:St-mor}.
\item
There exists an algebraic space $U$ over $S$ and a smooth surjection $U \to X$,
i.e., for any $T \in\Sch_S$ and any $1$-morphism $T \to X$ 
the morphism $p_T: U \times_X T \to T$ of algebraic spaces 
(see Remark \ref{rmk:Cl:AlgSt} (1) below)
is a smooth surjection in the sense of Definition \ref{dfn:Cl:AS:P}.
\end{itemize}
In this case the algebraic space $U$ is called a \emph{smooth covering} of $X$.
\end{dfn}

\begin{rmk}\label{rmk:Cl:AlgSt}
\begin{enumerate}[nosep]
\item 
The fiber product $U \times_X T$ in the second condition is an algebraic space.
In order to see this, 
note first that any $1$-morphism $t: T \to X$ from a scheme is representable.
Indeed, for any $1$-morphism $u: U \to X$ from a scheme, the fiber product
$U \times_{u,X,t}T$ is isomorphic to the fiber product $Y$ of the diagram 
$X \xrightarrow{\Delta} X \times_S X \xleftarrow{u \times_S t} U \times_S T$.
This fiber product $Y$ is an algebraic space 
since $\Delta$ is representable by the first condition.

\item
One can replace the algebraic space $U$ by a scheme 
which is an \'etale covering of $U$.
The definition in \cite[Definition 8.1.4]{O:book} is given 
under this replacement.
\end{enumerate}
\end{rmk}

%

Let us recall a criterion of an algebraic stack being an algebraic space.

\begin{fct}[{\cite[Chap.\ 2, Proposition (4.4)]{LM}}]\label{fct:Cl:St-AS}
An algebraic stack $X$ over a scheme $S$ is an algebraic space
if and only if the following two conditions are satisfied.
\begin{enumerate}[nosep, label=(\roman*)]
\item 
$X$ is a Deligne-Mumford stack, i.e., 
there is an \'etale surjective $1$-morphism $U \to X$ from a scheme $U$.
\item
The diagonal $1$-morphism $X \to X \times_S X$ is a monomorphism.
\end{enumerate}
\end{fct}

We can introduce properties of algebraic stacks as the case of algebraic spaces.
We use stability and locality on domain \cite[Definition 5.1.3]{O:book}
of a property of morphisms in the smooth site.

\begin{dfn*}[{\cite[Definition 2.1.16]{O:book}}]
Let $S$ be a scheme.
The \emph{smooth site} $\gsSm(S)$ over $S$ is defined to be 
the site $\gsSm(S):=(\Sch^{\sm}_{S},\sm)$ consisting of the following data.
\begin{itemize}[nosep]
\item
The full subcategory $\Sch^{\sm}_{S} \subset \Sch_S$ 
spanned by smooth schemes $U$ over $S$.
\item
The smooth topology $\sm$.
\end{itemize}
The Grothendieck topology $\sm$ is determined as follows:
the set $\Cov_{\sm}(U)$ of coverings of $U \in \Sch^{\sm}_S$ 
consists of a family $\{U_i \to U\}_{i \in I}$ of morphisms in $\Sch_S^{\sm}$
for which each $U_i \to U$ is smooth and 
the morphism $\coprod_{i \in I}U_i \to U$ is surjective.
\end{dfn*}

\begin{dfn}[{\cite[Definition 8.2.1]{O:book}}]\label{dfn:Cl:St-p}
Let $S$ be a scheme, and $\bfQ_0$ be a property of $S$-schemes 
which is stable in $\gsSm(S)$.
An algebraic stack $X$ over $S$ \emph{has property $\bfQ_0$}
if there exists a smooth covering $U \to X$ from a scheme $U$
having property $\bfQ_0$.
\end{dfn}

Let us remark that the scheme $U$ in this definition can be replaced by 
an algebraic space $U$ \cite[Lemma 8.2.4]{O:book}.
We can apply this definition to 
\[
 \bfQ_0 := \text{locally noetherian}, \ 
           \text{locally of finite type over $S$},\ 
           \text{locally of finite presentation over $S$}.
\]

Let us also introduce some classes of morphisms between algebraic stacks.

\begin{dfn}[{\cite[Definition 8.2.6]{O:book}}]\label{dfn:Cl:St-mor1}
Let $S$ be a scheme, and $\bfQ_1$ be a property of morphisms of schemes
which is stable and local on domain with respect to $\gsSm(S)$.
A $1$-morphism $f: X \to Y$ of algebraic stacks 
\emph{has property $\bfQ_1$} if 
there exists a commutative diagram
\[
 \xymatrix{
 U \ar[r]^(0.4){g} \ar[rd]_{q} & X \times_{Y} V \ar[r]^(0.6){f'} \ar[d] & V \ar[d]^{p} \\ 
                  & X \ar[r]_{f} & Y
 }
\]
of $1$-morphisms between ordinary stacks where $U$ and $V$ are schemes, 
the square is cartesian, and $g$ and $p$ give smooth coverings 
such that the morphism $f' \circ g: U \to V$ of schemes has property $\bfQ_1$.
\end{dfn}

This definition is independent of the choice of smooth coverings 
$U \to X$ and $V \to Y$ \cite[Proposition 8.2.8]{O:book}.
We can apply this definition to
\[
 \bfQ_1 = \text{smooth}, \ \text{locally of finite type}, \ 
          \text{locally of finite presentation}, \ \text{surjective}.
\]

\begin{dfn}[{\cite[Definition 8.2.9]{O:book}}]\label{dfn:Cl:St-mor2}
Let $S$ be a scheme, and $\bfQ_2$ be a property of morphisms of algebraic spaces
over $S$ which is stable with respect to the smooth topology on $\AlgSp_S$.
A representable $1$-morphism $f: X \to Y$ of algebraic stacks over $S$
\emph{has property $\bfQ_2$} 
if for every $1$-morphism $V \to Y$ from some $Y \in \AlgSp_S$, 
the morphism $X \times_Y V \to V$ in $\AlgSp_S$ has property $\bfQ_2$.
\end{dfn}

In particular, we can apply this definition to 
\begin{align*}
 \bfQ_2 = \, 
&\text{\'etale}, \ \text{smooth of relative dimension $d$}, \ 
 \text{quasi-compact}, \\ \ 
&\text{quasi-separate}, \ \text{proper}, \ \text{being a closed immersion}.
\end{align*}

Since the diagonal $1$-morphism $\Delta_{X/Y}: X \to X \times_{f,Y,f} Y$ is 
representable for any $1$-morphism $f : X \to Y$ of algebraic stacks over $S$,
we can apply Definition \ref{dfn:Cl:St-mor} to introduce

\begin{dfn}[{\cite[Definition 8.2.12]{O:book}}]\label{dfn:Cl:St-mor3}
Let $f: X \to Y$ be a $1$-morphism of algebraic stacks.
\begin{enumerate}[nosep]
\item 
$f$ is \emph{separated} if $\Delta_{X/Y}$ is proper
in the sense of Definition \ref{dfn:Cl:St-mor2}.
\item
$f$ is \emph{quasi-separated} if 
$\Delta_{X/Y}$ is quasi-compact and quasi-separated 
in the sense of Definition \ref{dfn:Cl:St-mor2}.
\end{enumerate}
\end{dfn}

We finally introduce

\begin{dfn}\label{dfn:Cl:St-mor4}
A $1$-morphism $f$ of algebraic stacks
is \emph{of finite type} (resp.\ \emph{of finite presentation})
if it is quasi-compact (Definition \ref{dfn:Cl:St-mor2}) 
and locally of finite type (resp.\ locally of finite presentation)
in the sense of Definition \ref{dfn:Cl:St-mor3}.
\end{dfn}

\subsection{Lisse-\'etale site on algebraic stacks}
\label{ss:Cl:LE}

We cite from \cite[Chap.\ 12]{LM} 
the definition of the lisse-\'etale site on algebraic stacks.

\begin{dfn}\label{dfn:Cl:lisse-etale}
The \emph{lisse-\'etale site} on an algebraic stack $X$ over a scheme $S$ is given by
\begin{itemize}[nosep]
\item 
An object of the underlying category is a pair $(U,u)$
of an algebraic space $U$ over $S$ and 
a $1$-morphisms $u: U \to X$ of ordinary stacks over $S$.
A morphism from $(U,u:U \to X)$ to $(V,v: V \to X)$ is a pair $(\varphi,\alpha)$
of a smooth $1$-morphism $\varphi: U \to V$ of algebraic spaces over $S$
and a $2$-isomorphism $\alpha: u \simto v \circ \varphi$.

\item
As for the Grothendieck topology, the set $\Cov(u)$ of covering sieves
consists of families $\{(\varphi_i,\alpha_i): (U_i,u_i) \to (U,i)\}_{i \in I}$
such that the $1$-morphism 
$\coprod_{i \in I} \varphi_i: \coprod_{i \in I}U_i \to U$
of algebraic spaces over $S$ is \'etale and surjective.
\end{itemize}
The associated topos is denoted by $X_{\txle}$.
\end{dfn}

As shown in \cite[Lemme (12.1.2)]{LM},
one can replace ``an algebraic space $U$ over $S$" in the above definition
by ``an affine scheme $U$ over $S$",
and replace ``families $\{(\varphi_i,\alpha_i)\}_{i \in I}$"
by ``finite families $\{(\varphi_i,\alpha_i)\}_{i \in I}$".
The resulting topos is equivalent to $X_{\txle}$.

\section{$\infty$-categories}
\label{a:ic}

In this appendix we give some complementary accounts of
selected topics on $\infty$-categories.

\subsection{Kan model structure}
\label{ss:ic:Kan}

Let us explain the model structure on the category $\sSet$
of simplicial sets
which is called the Kan model structure in \cite[\S A.2.7]{Lur1}.
We begin with 

\begin{dfn}\label{dfn:ic:Kan-fib}
A simplicial map $f: X \to Y$ is called a \emph{Kan fibration} 
if $f$ has the right lifting property 
with respect to all horn inclusions $\Lambda^n_i \inj \Delta^n$,
i.e., if given any diagram
\[
 \xymatrix{ \Lambda_i^n \ar[r] \ar@{^{(}->}[d] & X \ar[d] \\ 
            \Delta^n    \ar[r]                    & Y}
\]
with arbitrary $n \in \bbN$ and any $i=0,1,\ldots,n$,
there exists a simplicial map $\Delta^n \to X$ 
making the diagram commutative.
\end{dfn}

Next, recall the classical theory telling that there exists an adjoint pair
\begin{align}\label{eq:ic:sSet-CG}
 \abs{-}: \sSet \adjunc \topCG : \fSing
\end{align}
of functors where $\topCG$ denotes 
the category of compactly generated Hausdorff topological spaces.
The functor $\abs{-}$ is called the \emph{geometric realization}.

Now we introduce

\begin{fct}[{\cite[Chap.\ 1 \S 11]{GJ}}]\label{fct:ic:Km}
The following data gives $\sSet$ a model structure.
\begin{itemize}[nosep]
\item 
A simplicial map $f: X \to Y$ is a cofibration 
if it is a monomorphism, i.e., the map $X_n \to Y_n$ is injective 
for each $n \in \bbN$.
\item
A simplicial map is a fibration if it is a Kan fibration.
\item
A simplicial map $f:X \to Y$ is a weak equivalence if
the induced map $\abs{X} \to \abs{Y}$ of geometric realizations
is a homotopy equivalence of topological spaces.
\end{itemize}
The obtained model structure is called the \emph{Kan model structure} on $\sSet$.
\end{fct}


\subsection{Homotopy category of an $\infty$-category}
\label{ss:ic:Ho}

In Definition \ref{dfn:ic:tc}, we denoted by $\topCG$ 
the category of compactly generated weakly Hausdorff topological spaces.
The classical theory tells us that for each $X \in \topCG$ 
there exists a CW complex $X'\in \topCW$,
giving rise to a well-defined functor
\[
 \theta: \topCG \longto \topH, \quad X \longmapsto [X] := X'.
\]
We call $[X] \in \topH$ the \emph{homotopy type} of $X$.
Now we can define the homotopy category of a topological category.

\begin{dfn*}[{\cite[pp.\ 16--17]{Lur1}}]
For a topological category $\topC$, we denote by $\Ho \topC$ 
the category enriched over $\topH$ defined as follows,
and call it the \emph{homotopy category} of $\topC$.
\begin{itemize}[nosep]
\item 
The objects of $\Ho \topC$ are defined to be the objects of $\topC$.
\item 
For $X,Y \in \topC$, we set $\Map_{\Ho \topC}(X,Y) := [\Map_{\topC}(X,Y)] \in \topH$.
\item 
Composition of morphisms in $\Ho \topC$ is given by the application of 
$\theta$ to composition of morphisms in $\topC$.
\end{itemize}
\end{dfn*}

Next we explain the homotopy category of a simplicial category.
Recall the adjunction 
\[
 \abs{-}: \sSet \adjunc \topCG : \fSing
\]
in \eqref{eq:ic:sSet-CG}.
Composing $\abs{-}: \sSet \to \topCG$ with $\theta:\topCG \to \topH$,
we have a functor 
\[
 [\cdot]: \sSet \longto \topH, \quad S \longmapsto [S] := \theta(\abs{S}).
\]

\begin{dfn}\label{dfn:ic:Ho-type:sSet}
For a simplicial set $S \in \sSet$, 
we call $[S]$ the \emph{homotopy type} of $S$.
\end{dfn}

Recall the category $\sCat$ of simplicial categories 
(Definition \ref{dfn:ic:sCat}).
Applying this functor $[\cdot]: \sSet \to \topH$ to the simplicial sets 
of morphisms, we obtain another functor
\[
 \Ho: \sCat \longto (\text{categories enriched over $\topH$}), \quad
 \catC \longmapsto \Ho \catC.
\]

\begin{dfn}[{\cite[p.\ 19]{Lur1}}]\label{dfn:ic:Ho-sCat}
For a simplicial category $\spC \in \sCat$, we call $\Ho \spC$ 
in the above construction the \emph{homotopy category} of $\spC$.
\end{dfn}

Now we explain the homotopy category of an $\infty$-category.
Recall that we denote by $\sCat$ the category of simplicial categories
(Definition \ref{dfn:ic:sCat}).
Then we can define a functor 
\[
 \frkC[-]: \sSet \longto \sCat
\]
as follows \cite[\S1.1.5]{Lur1}:
For a finite non-empty linearly ordered set $J$, 
we construct a simplicial category $\frkC[\Delta^J]$ as follows:
The objects of $\frkC[\Delta^J]$ are the elements of $J$.
For $i,j \in J$ with $i \le j$, 
the simplicial set $\Map_{\frkC[\Delta^J]}(i,j)$ 
is given by the nerve of the poset 
$\{I \subset J \mid i,j \in I, \ \forall \, k \in I \ i \le k \le j\}$.
For $i > j$ we set $\Map_{\frkC[\Delta^J]}(i,j) := \emptyset$.
The resulting functor $\frkC: \Delta \to \sCat$ 
extends uniquely to a functor $\frkC[-]: \sSet \to \sCat$,
and we denote by $\frkC[S]$ the image of $S \in \sSet$.

\begin{dfn}[{\cite[Definition 1.1.5.14]{Lur1}}]\label{dfn:ic:h:iC}
Let $S$ be a simplicial set.
The \emph{homotopy category} $\Ho S$ of $S$ is defined to be 
\[
 \Ho S := \Ho \frkC[S],
\]
the homotopy category of the simplicial category $\frkC[S]$.

The \emph{homotopy category} of an $\infty$-category $\iC$ 
is defined to be the homotopy category $\Ho \iC$ 
of $\iC$ as a simplicial set.
\end{dfn}

Noting that the homotopy category $\Ho S$ is enriched over $\topH$,
we denote by $\Hom_{\Ho S}(-,-) \in \topH$ its Hom space.
Then we can introduce

\begin{dfn}[{\cite[Definition 1.2.2.1]{Lur1}}]
\label{dfn:ic:Map}
For a simplicial set $S$ and its vertices $x,y \in S$,
we define 
\[
 \Map_{S}(x,y) := \Hom_{\Ho S}(x,y) \in \topH
\] 
and call it the \emph{mapping space} from $x$ to $y$ in $S$.
\end{dfn}

For use in the main text, we recall the construction of 
Kan complexes which represent mapping spaces.

\begin{dfn}\label{dfn:ic:HomR}
For a simplicial set $S$ and vertices  $x,y \in S$,
we defined a simplicial set $\HomR{S}(x,y)$ by
\begin{align*}
&\Hom_{\sSet}(\Delta^n,\HomR{S}(x,y)) =  \\
&\{z: \Delta^{n+1} \to S \mid 
   \text{simplicial maps}, \ \rst{z}{\Delta^{\{n+1\}}}=y, \ 
   \text{$\rst{z}{\Delta^{\{0,\ldots,n\}}}$ 
         is a constant complex at the vertex $x$}\}.
\end{align*}
The face and degeneracy maps are induced by those on $S_{n+1}$.
\end{dfn}

If $\iC$ is an $\infty$-category, then $\HomR{\iC}(x,y)$ is a Kan complex
by \cite[Proposition 1.2.2.3]{Lur1}.

By \cite[Proposition 2.2.4.1]{Lur1},
we have an adjunction
\[
 \abs{-}_{Q^{\bullet}}: \sSet \adjunc \sSet :\fSing_{Q^{\bullet}}
\]
which gives a Quillen autoequivalence on the category $\sSet$
equipped with the Kan model structure.
Now we have

\begin{fct*}[{\cite[Proposition 2.2.4.1]{Lur1}}]
For an $\infty$-category $\iC$ and objects $x,y \in \iC$,
there is a natural equivalence of simplicial sets
\[
 \abs{\HomR{\iC}(x,y)}_{Q^{\bullet}} \longto \Map_{\frkC[S]}(x,y).
\]
\end{fct*}

Here $\frkC:\sSet \to \sCat$ is the functor explained at \eqref{eq:ic:frkC}.

\subsection{Over-$\infty$-categories and under-$\infty$-categories}
\label{ss:ic:o/u}

This subsection is based on  \cite[\S1.2.9]{Lur1}.

For simplicial sets $S$ and $S'$, their \emph{join} \cite[Definition 1.2.8.1]{Lur1} 
is denoted by $S \join S'$.
The join of $\infty$-categories 
is an $\infty$-category \cite[Proposition 1.2.8.3]{Lur1}.

\begin{dfn}[{\cite[\S1.2.9]{Lur1}}]
\label{dfn:ic:o/u}
Let $p: K \to \iC$ be a simplicial map from 
a simplicial set $K$ to an $\infty$-category $\iC$.
\begin{enumerate}[nosep]
\item
Consider the simplicial set $\oc{\iC}{p}$ defined by
\[
 (\oc{\iC}{p})_n := \Hom_{/p}(\Delta^n \join K, \iC).
\]
The subscript $/p$ in the right hand side means that 
we only consider those morphisms $f: S \join K \to \iC$ such that $\rst{f}{K}=p$.
Then $\oc{\iC}{p}$ is an $\infty$-category, and 
called the \emph{over-$\infty$-category of objects over $p$}
\item
For $X \in \iC$, we denote by $\oc{\iC}{X}$ the over-$\infty$-category $\oc{\iC}{p}$ 
where $p: \Delta^0 \to \iC$ has $X$ as its image.
\item
Dually, the \emph{under-$\infty$-category} $\iC_{p/}$ is the $\infty$-category
defined by $(\iC_{p/})_n := \Hom_{p/}(K \join \Delta^n, \iC)$, 
where the subscript $p/$ means that 
we only consider those morphisms $f: K \join S \to \iC$ such that $\rst{f}{K}=p$.
\item
For $X \in \iC$, we denote by $\iC_{X/}$ the under-$\infty$-category $\iC_{p/}$ 
where $p: \Delta^0 \to \iC$ has $X$ as its image.
\end{enumerate}
\end{dfn}

The over-$\infty$-category $\oc{\iC}{p}$ is characterized by the universal property
\[
 \Hom_{\sSet}(S,\oc{\iC}{p}) = \Hom_{/p}(S \join K,\iC)
\]
for any simplicial set $S$.
We can characterize an under-$\infty$-category by a similar universal property.

Note that for an injective simplicial map $j: L \to K$ 
we have a natural functor $\oc{\iC}{p} \to \oc{\iC}{p \circ j}$.

By this universality one can deduce the following consequences.
We omit the proof.

\begin{cor}\label{cor:ic:ovc}
Let $\iC$ be an $\infty$-category.
\begin{enumerate}[nosep]
\item
For any $X \in \iC$ there exists a functor
\[
 \oc{\iC}{X} \longto \iC
\]
of $\infty$-categories.
It will be called the \emph{canonical functor} of $\oc{\iC}{X}$.

\item
For a morphism $f: X \to Y$ in $\iC$,
we have a functor $\oc{\iC}{X} \to \oc{\iC}{Y}$ of $\infty$-categories
induced by composition with $f$.
We also have functors $\oc{\iC}{f} \to \oc{\iC}{X}$ and $\oc{\iC}{f} \to \oc{\iC}{Y}$, where
$\oc{\iC}{f}$ is defined to be $\oc{\iC}{p}$ with $p:\Delta^1 \to \iC$ representing $f$. 
These functors form a commutative triangle 
\[
 \xymatrix{ & \oc{\iC}{f} \ar[ld] \ar[rd] \\ \oc{\iC}{X} \ar[rr]  && \oc{\iC}{Y}}
\]
in the $\infty$-category $\iCat$
of $\infty$-categories (Definition \ref{dfn:ic:iCat}).
\end{enumerate}
\end{cor}

\begin{rmk*}
In \cite{Lur1} and \cite{J} the canonical functor is called the projection,
but we avoid this terminology 
(see Definition \ref{dfn:is:oij} for the reason).
\end{rmk*}

Definition \ref{dfn:ic:o/u} works for any simplicial set $S$ 
instead for an $\infty$-category $S$.
The statement that $\oc{\iC}{p}$ for an $\infty$-category $\iC$
is indeed an $\infty$-category is shown in \cite[Proposition 2.1.2.2]{Lur1}.

\subsection{Limits and colimits in $\infty$-categories.}
\label{ss:ic:lim/colim}

The limits and colimits in $\infty$-categories are defined 
in terms of final and initial objects.

\begin{dfn*}[{\cite[Definition 1.2.13.4]{Lur1}}]
Let $\iC$ be an $\infty$-category, $K$ be a simplicial set 
and $p: K \to \iC$ be a simplicial map.
\begin{enumerate}[nosep]
\item
A \emph{colimit} of $p$ is an initial object of 
the under-$\infty$-category $\iC_{p/}$, and denoted by $\iclim p$.
\item
A \emph{limit} of $p$ is a final object of the over-$\infty$-category $\oc{\iC}{p}$,
and denoted by $\ilim p$.
\end{enumerate}
\end{dfn*}

\begin{rmk}\label{rmk:ic:colim:nu}
A (co)limit of $p$ is not unique if it exists,
but by \cite[Proposition 1.2.12.9]{Lur1} the full sub-$\infty$-category of (co)limits 
is either empty or is a contractible Kan complex.
\end{rmk}

We have the following restatement of (co)limit.

\begin{fct*}[{\cite[Remark 1.2.13.5]{Lur1}}]
A colimit of $p: K \to \iC$ can be identified with a simplicial map 
$\ol{p}: K^{\triangleright} \to \iC$ extending $p$.
Similarly, a limit of $p$ is identified with a simplicial map 
$\ol{p}: K^{\triangleleft} \to \iC$ extending $p$.
\end{fct*}

Here we used 

\begin{dfn}[{\cite[Notation 1.2.8.4]{Lur1}}] \label{dfn:ic:cone}
For a simplicial set $K$, 
we denote by $K^{\triangleright} := K \join \Delta^0$ the \emph{right cone} $K$,
where $\join$ denotes the join of simplicial sets.
We also denote by $K^{\triangleleft} := \Delta^0 \join K$ the \emph{left cone} of $K$.
\end{dfn}

\begin{rmk}\label{rmk:ic:colim}
As explained in \cite[\S 4.2.4]{Lur1}, colimits in the $\infty$-category 
are compatible with homotopy colimits in simplicial categories,
and limits are compatible with the homotopy limits.
\end{rmk}

Let us explain a few examples of limits and colimits in $\infty$-categories,
following \cite[\S4.4]{Lur1}.

\begin{itemize}
\item
We regard a set $A$ as a category by $\Hom_A(i,i) = *$ for $i \in A$ and 
$\Hom_A(i,j) = \emptyset$ for $i \neq j$.
We further consider $A$ as the simplicial set which is the nerve of this category.

\begin{dfn}[{\cite[\S 4.4.1]{Lur1}}]\label{dfn:ic:coprod}
Let $A$ be a set, $\iC$ an $\infty$-category and $p: A \to \iC$ be a map.
Thus $p$ is identified with the family $\{X_a \mid a \in A\}$ of objects in $\iC$.
Then a colimit $\iclim p$ is called a \emph{coproduct} 
of $\{X_a \mid a \in A\}$, and denoted by $\coprod_{a \in A} X_a$.
Dually, a limit $\ilim p$ is called a \emph{product} of $\{X_a \mid a \in A\}$,
and denoted by $\prod_{a \in A} X_a$.
\end{dfn}

Under Remark \ref{rmk:ic:colim}, the corresponding object 
in a simplicial category is the homotopy coproduct.

\item
Let $\iC$ be an $\infty$-category. 
A simplicial map $\Delta^1 \times \Delta^1 \to \iC$ 
is called a \emph{square} in $\iC$.
It will be typically depicted as 
\begin{align}\label{diag:square}
 \xymatrix{X' \ar[r]^{q'} \ar[d]_{p'} &X \ar[d]^{p} \\
 Y' \ar[r]_{q} & Y}
\end{align}

Since there are isomorphisms 
$(\Lambda^2_0)^{\triangleright} \simeq \Delta^1 \times \Delta^1 
 \simeq (\Lambda^2_2)^{\triangleleft}$
of simplicial sets,
we can introduce


\begin{dfn}[{\cite[\S 4.4.2]{Lur1}}]\label{dfn:ic:pb-po}
Let $\sigma: \Delta^1 \times \Delta^1 \to \iC$ be a square in an $\infty$-category $\iC$.
\begin{enumerate}[nosep]
\item
If $\sigma$ is a limit of $\rst{\sigma}{\Lambda^2_2}$ 
viewing $\Delta^1 \times \Delta^1  \simeq (\Lambda^2_2)^{\triangleleft}$, 
then it is called a \emph{pullback square} or a \emph{cartesian square}.
If the square \eqref{diag:square} is a pull-back square, then we write 
$X' = X \times_{p, Y, q} Y'$ or simply $X' = X \times_Y Y'$, 
and call $X'$ a \emph{pullback} or a \emph{base change} or a \emph{fiber product}.

\item
If $\sigma$ is a colimit of $\rst{\sigma}{\Lambda^2_0}$ 
viewing $\Delta^1 \times \Delta^1  \simeq (\Lambda^2_0)^{\triangleright}$, 
then it is called a \emph{pushout square} or a \emph{cocartesian square},
and we write $Y = X \coprod^{X'} Y'$ and call $Y$ a \emph{pushout}.
\end{enumerate}
\end{dfn}

Under Remark \ref{rmk:ic:colim}, the corresponding objects in a simplicial category
are the homotopy pullback and the homotopy pushout.

\item
Finally let $\catD$ be the category depicted by the diagram
\[
 \xymatrix{X \ar@<2pt>[r]^{F} \ar@<-2pt>[r]_{G} &  Y}
\]

\begin{dfn}\label{dfn:ic:coeq}
Let $p: \Ner(\catD) \to \iC$ be a simplicial map from the nerve of $\catD$ 
(Definition \ref{dfn:ic:nerve}) to an $\infty$-category $\iC$.
Set $f := p(F)$ and $g:=p(G)$.
A colimit of $p$ is called a \emph{coequalizer} of $f$ and $g$.
\end{dfn}
\end{itemize}

\subsection{Adjunctions}
\label{ss:ic:adjunc}

We now introduce the notion of adjoint functors of $\infty$-categories,
for which we prepare several definitions on simplicial maps.

\begin{dfn*}
Let $f: X \to S$ be a simplicial map.
\begin{enumerate}[nosep]
\item
$f$ is called a \emph{trivial fibration} if $f$ has the right lifting property 
with respect to all inclusions $\partial \Delta^n \inj \Delta^n$.

\item
$f$ is called a \emph{trivial Kan fibration} 
if it is a Kan fibration (Definition \ref{dfn:ic:Kan-fib}) and a trivial fibration.

\item
$f$ is called an \emph{inner fibration} if $f$ has the right lifting property 
with respect to horn inclusions $\Lambda^n_i \inj \Delta^n$ with $0 < i < n$.
\end{enumerate}
\end{dfn*}

\begin{dfn*}[{\cite[Definition 2.4.1.1]{Lur1}}]
Let $p: X \to S$ be an inner fibration of simplicial sets.
Let $f: x \to y$ be an edge in $X$. 
$f$ is called \emph{$p$-cartesian} if the induced map
$\oc{X}{f} \to \oc{X}{y} \times_{\oc{S}{p(y)}} \oc{S}{p(f)}$
is a trivial Kan fibration.
\end{dfn*}

\begin{dfn}[{\cite[Definition 2.4.2.1]{Lur1}}]\label{dfn:ic:cc-fib}
Let $p: X \to S$ be a simplicial map.
\begin{enumerate}[nosep]
\item
$p$ is called a \emph{cartesian fibration} 
if the following conditions are satisfied.
\begin{itemize}[nosep]
\item 
The map $p$ is an inner fibration.
\item 
For every edge $f: x \to y$ of $S$ and every vertex $\wt{y}$ of $X$ 
with $p(\wt{y})=y$, there exists a $p$-cartesian edge  
$\wt{f}: \wt{x} \to \wt{y}$ with $p(\wt{f}) = f$.
\end{itemize}
\item
$p$ is called a \emph{cocartesian fibration} 
if the opposite map $p^{\op}: X^{\op} \to S^{\op}$ is a cartesian fibration.
\end{enumerate}
\end{dfn}

Now we explain the main definition in this subsection.

\begin{dfn}[{\cite[Definition 5.2.2.1]{Lur1}}]
\label{dfn:ic:adj}
Let $\iB$ and $\iC$ be $\infty$-categories. 
An \emph{adjunction} between $\iB$ and $\iC$ is a simplicial map $p: M \to \Delta^1$
which is both a cartesian fibration and a cocartesian fibration 
together with equivalences $f: \iB \to p^{-1}\{0\}$ and $g: \iC \to p^{-1}\{1\}$.
In this case we say that $f$ is \emph{left adjoint} to $g$ and 
that $g$ is \emph{right adjoint} to $f$,
and denote 
\[
 f: \iB \adjunc \iC :g.
\]
\end{dfn}

As in the ordinary category theory,
an adjunction can be restated by a unit (and by a counit).

\begin{dfn}[{\cite[Definition 5.2.2.7]{Lur1}}]\label{dfn:ic:uc}
Let us given a pair of functors 
$(f: \iB \to \iC, g: \iC \to \iB)$ of $\infty$-categories.
A \emph{unit transformation} for $(f,g)$ is a morphism 
$u: \id_{\iB} \to g \circ f$ in $\iFun(\iB,\iB)$ 
such that for any $B \in \iB$ and $C \in \iC$ the composition
\[
 \Map_{\iC}(f(B),C) \longto \Map_{\iB}(g(f(B)),g(C)) 
 \xrr{ u(C) } \Map_{\iB}(B,g(C))
\]
is an isomorphism in the homotopy category $\topH$ of spaces.

Dually we have the notion of a \emph{counit transformation}
$c: f \circ g \to \id_{\iB}$.
\end{dfn}

\begin{fct}[{\cite[Proposition 5.2.2.8]{Lur1}}]
\label{fct:ic:uc}
Let $(f: \iB \to \iC, g: \iC \to \iB)$ be 
a pair of functors of $\infty$-categories.
Then the following conditions are equivalent.
\begin{itemize}[nosep]
\item 
The functor $f$ is a left adjoint to $g$.
\item
There exists a unit transformation $u: \id_{\iB} \to g \circ f$.
\end{itemize}
We have a dual statement for right adjoint and counit transformation.
\end{fct}

For an existence criterion of an adjoint functor via exactness,
see Fact \ref{fct:ic:adj}.

In the main text we will use following statement repeatedly.

\begin{fct*}[{\cite[Proposition 5.2.2.6]{Lur1}}]
Let $f_i: \iC_i \to \iC_{i+1}$ ($i=1,2$) be functors of $\infty$-categories.
Suppose that $f_i$ has a right adjoint $g_i$ ($i=1,2$).
Then $g_2 \circ g_1$ is right adjoint to $f_2 \circ f_1$.
\end{fct*}

\subsection{The underlying $\infty$-category of a simplicial model category}
\label{ss:ic:uis}

In the main text we often translate the model-categorical arguments 
in \cite{TVe1,TVe2} into the $\infty$-categorical arguments.
Such a translation is possible by 
the notion of the \emph{underlying $\infty$-category}, which is explained here.

Let us recall the monoidal model category structure \cite[Definition A.3.1.2]{Lur1} 
on the category $\sSet$ of simplicial sets given by 
the cartesian product and the Kan model structure (Fact \ref{fct:ic:Km}).
Here the cartesian product of $S,T \in \sSet$ is given by 
$(S \times T)_n := S_n \times T_n$, 
where the latter $\times$ means the cartesian product in the category $\Set$.

\begin{dfn}[{\cite[Definition A.3.1.5]{Lur1}, \cite[Definition 4.2.18]{H}}]
\label{dfn:ic:spl-model}
A \emph{simplicial model category} $\spA$ is a simplicial category equipped 
with a model structure satisfying the following conditions.
\begin{itemize}[nosep]
\item 
The category $\spA$ is tensored and cotensored over 
the monoidal model category $\sSet$ in the sense of \cite[Remark A.1.4.4]{Lur1}.
\item
The action map $\otimes: \spA \times \sSet \to \spA$ arising from 
the tensored structure is a left Quillen bifunctor.
\end{itemize}
\end{dfn}

One can construct an $\infty$-category from a simplicial model category $\spA$.
Let us denote by $\spA^{\circ} \subset \spA$ the full subcategory of 
fibrant-cofibrant objects, which is a fibrant simplicial category.
Taking the simplicial nerve (Definition \ref{dfn:ic:Nsp}),
we obtain an $\infty$-category $\Nsp(\spA^{\circ})$ by Fact \ref{fct:ic:Nsp}.

\begin{dfn}[{\cite[\S A.2]{Lur1}}]\label{dfn:ic:um}
We call the $\infty$-category $\Nsp(\spA^{\circ})$ 
the \emph{underlying $\infty$-category} of the simplicial model category $\spA$.
\end{dfn}

Let us cite a result on adjunctions.

\begin{fct}[{\cite[Proposition 5.2.4.6]{Lur1}}]\label{fct:ic:Qi}
Given a Quillen adjunction $\spA \adjunc \spA'$ of simplicial model categories,
there is a natural adjunction $\Nsp(\spA^{\circ}) \adjunc \Nsp(\spA')$ 
of the underlying $\infty$-categories.
\end{fct}

\subsection{$\infty$-localization}
\label{ss:ic:inf-loc}

We cite from \cite[\S 1.3.4]{Lur2}
terminologies on localization of $\infty$-categories.

\begin{dfn}[{\cite[Definition 1.3.4.1]{Lur2}}]
\label{dfn:ic:loc}
Let $\iC$ be an $\infty$-category and $W$ be a collection of morphisms in $\iC$. 
We say that \emph{a functor $f: \iC \to \iD$ exhibits $\iD$ as 
the $\infty$-category obtained from $\iC$ by inverting the set of morphisms $W$}
if, for every $\infty$-category $\iB$, composition with $f$ induces 
a fully faithful embedding $\iFun(\iD,\iB) \to \iFun(\iC,\iB)$, 
whose image is the collection of functors $F : \iC \to \iB$ 
mapping each morphism in $W$ to an equivalence in $\iB$. 
In this case we denote $\iC[W^{-1}]:=\iD$.
\end{dfn}

Note that $\iC[W^{-1}]$ is determined uniquely
up to equivalence by $\iC$ and $W$,
Note also that  $\iC[W^{-1}]$ exists for any $\iC$ and $W$.
See \cite[Remark 1.3.4.2]{Lur2} for an account.

Now let us recall

\begin{dfn}[{\cite[Definition 5.2.7.2]{Lur1}}]\label{dfn:ic:loc-func}
A functor $L: \iC \to \iD$ of $\infty$-categories is called 
a \emph{localization} (\emph{functor}) if $L$ has a fully faithful right adjoint.
\end{dfn}

Then by \cite[Proposition 5.2.7.12]{Lur1} we have 

\begin{fct}[{\cite[Example 1.3.4.3]{Lur2}}]\label{fct:ic:loc}
Let $\iC$ be an $\infty$-category and $L: \iC \to \iC_0$ be a localization.
We denote by $i: \iC_0 \inj \iC$ the fully faithful right adjoint of $L$.
Define $W$ to be the collection of those morphisms $\alpha$ in $\iC$ 
such that $L(\alpha)$ is an equivalence in $\iC_0$. 
Then the composite $\iC_0 \xr{i} \iC \xr{L} \iC[W^{-1}]$ 
is an equivalence of $\infty$-categories.
\end{fct}

\subsection{Presentable $\infty$-categories}
\label{ss:ic:presentable}

Most of the $\infty$-categories appearing in the main text 
are presentable in the sense of \cite[\S 5.5]{Lur1},
and enjoy a good property with respect to taking limits.
The notion of a presentable $\infty$-category 
is an $\infty$-theoretic analogue of 
the notion of a locally presentable category.

First we want to introduce the notion of a filtered $\infty$-category,
which is an $\infty$-theoretic analogue of filtered categories.
Recall the notation $K^{\triangleright}$ of the right cone of a simplicial set $K$
(Definition \ref{dfn:ic:cone}).

\begin{dfn}\label{dfn:ic:filtered}
Let $\kappa$ be a regular cardinal.
An $\infty$-category $\iC$ is called \emph{$\kappa$-filtered} 
if for any $\kappa$-small simplicial set $K$ and any simplicial map $f: K \to \iC$
there exists a simplicial map $\ol{f}: K^{\triangleright} \to \iC$ extending $f$.
\end{dfn}

Next we turn to the definition of ind-objects in an $\infty$-category.

\begin{dfn*}[{\cite[Definition 5.3.1.7]{Lur1}}]
For an $\infty$-category $\iC$ and a regular cardinal $\kappa$,
we denote by $\Ind_{\kappa}(\iC)$ the full sub-$\infty$-category of $\iPSh(\iC)$
spanned by the functors $f: \iC^{\op} \to \iS$ classifying 
right fibrations $\wt{\iC} \to \iC$, 
where the $\infty$-category $\wt{\iC}$ is $\kappa$-filtered.
An object of $\Ind_{\kappa}(\iC)$ is called an \emph{ind-object of $\iC$}.
\end{dfn*}

Here we used 

\begin{dfn*}[{\cite[Definition 2.0.0.3]{Lur1}}]
A simplicial map $f: X \to S$ is called a \emph{right fibration}
if $f$ has the right lifting property with respect to 
all the horn inclusions $\Lambda^n_i \inj \Delta^n$ for any $0 < i \le n$.
\end{dfn*}

Finally we introduce 

\begin{dfn*}[{\cite[Definition 5.4.2.1]{Lur1}}]
Let $\iC$ be an $\infty$-category.
\begin{enumerate}[nosep]
\item 
Let $\kappa$ be a regular cardinal.
We call $\iC$ \emph{$\kappa$-accessible} if there exists a small $\infty$-category
$\iC^0$ and an equivalence $\Ind_\kappa(\iC^0) \to \iC$.
\item
$\iC$ is called \emph{accessible} if it is $\kappa$-accessible 
for some regular cardinal $\kappa$.
\end{enumerate}
\end{dfn*}

Accessible $\infty$-categories have a nice property in terms of compact objects.
In order to state its precise definition, 
for an $\infty$-category $\iC$ and an object $C \in \iC$ 
let us denote by $j_C: C \to \wh{\iS}$ the composition
\[
 \iPSh(\iC) = \iFun(\iC^{\op},\wh{\iS}) \longto \iFun(\{s\},\wh{\iS}) 
 \longsimto \wh{\iS}
\]
and call it the \emph{functor corepresented by $C$}.
Here $\wh{\iS}$ denotes the $\infty$-category of (not necessarily small) spaces
(recall that $\iS$ denotes the $\infty$-category of small spaces).

\begin{dfn}[{\cite[Definition 5.3.4.5]{Lur1}}]\label{dfn:ic:cpt}
Let $\kappa$ be a regular cardinal, and
$\iC$ be an $\infty$-category admitting small $\kappa$-filtered colimits.
\begin{enumerate}[nosep]
\item
A functor $f: \iC \to \iD$ of $\infty$-categories is called 
\emph{$\kappa$-continuous} if it preserves $\kappa$-filtered colimits.
\item
Assume $\iC$ admits $\kappa$-filtered colimits.
Then an object $C \in \iC$ is called \emph{$\kappa$-compact} if the functor 
$j_C: C \to \wh{\iS}$ corepresented by $C$ is $\kappa$-continuous.
\end{enumerate}
\end{dfn}

\begin{dfn}[{\cite[Definition 5.5.0.1]{Lur1}}]\label{dfn:ic:pres}
An $\infty$-category $\iC$ is called \emph{presentable}
if it is accessible and admits arbitrary small colimits.
\end{dfn}

Let us cite an equivalent definition of presentable $\infty$-categories.
For that we need

\begin{dfn*}
Let $\iC$ be an $\infty$-category.
\begin{enumerate}[nosep]
\item 
$\iC$ is \emph{essentially small} if 
it is $\kappa$-compact as an object of $\iCat$
for some small regular cardinal $\kappa$.

\item
$\iC$ is \emph{locally small} if for any objects $X,Y \in \iC$ 
the mapping space $\Map_{\iC}(X,Y)$ is essentially small
as an $\infty$-category.
\end{enumerate}
\end{dfn*}

See \cite[Proposition 5.4.1.2]{Lur1} for equivalent definitions 
of essential smallness.

\begin{fct}[{\cite[Theorem 5.5.1.1]{Lur1}}]\label{fct:ic:pr=lpr}
For an $\infty$-category $\iC$,
the following conditions are equivalent.
\begin{enumerate}[nosep]
\item 
$\iC$ is presentable.
\item
$\iC$ is locally small and admits small colimits, 
and there exists a regular cardinal $\kappa$ and a small set $S$ of 
$\kappa$-compact objects of $\iC$ such that 
every object of $\iC$ is a colimit of a small diagram taking
values in the full sub-$\infty$-category of $\iC$ spanned by $S$.
\end{enumerate}
\end{fct}

Presentable categories enjoy much nice properties as explained in \cite[\S5.5]{Lur1}.

\begin{fct}\label{fct:ic:pres}
Let $\iC$ be a presentable $\infty$-category.
\begin{enumerate}[nosep]
\item 
$\iC$ admits arbitrary small limits
\cite[Corollary 5.5.2.4]{Lur1}.
\item
The product of presentable categories is presentable
\cite[Proposition 5.5.3.5]{Lur1}.
\item
The functor $\infty$-category of presentable categories is presentable
\cite[Proposition 5.5.3.6]{Lur1}.
\item
The over-$\infty$-category $\oc{\iC}{p}$ and under-$\infty$-category $\iC_{p/}$ 
with respect to a simplicial morphism $p$ of small $\infty$-categories
are presentable \cite[Proposition 5.5.3.10, Proposition 5.5.3.11]{Lur1}.
\end{enumerate}
\end{fct}

\subsection{Truncation functor}

We close this subsection by explaining 
the truncation functor for an $\infty$-category.
We begin with the truncation of objects.

\begin{dfn}[{\cite[Definition 5.5.6.1]{Lur1}}]\label{dfn:ic:n-trc}
Let $\iC$ be an $\infty$-category.
\begin{enumerate}[nosep]
\item 
Let $k \in \bbZ_{\ge -1}$.
An object $C \in \iC$ is called \emph{$k$-truncated} if 
for any $D \in \iC$ the space $\Map_{\iC}(D,C)$ is $k$-truncated,
i.e., $\pi_i \Map_{\iC}(D,C) = *$ for $i \in \bbZ_{\ge k+1}$.

\item
A \emph{discrete} object is defined to be a $0$-truncated object.

\item
An object $C$ is called \emph{$(-2)$-truncated} if it is a final object of $\iC$
(Definition \ref{dfn:ic:ini-fin}).
\item
For $k \in \bbZ_{\ge -2}$, we denote by $\tau_{\le k} \iC$ 
the full sub-$\infty$-category of $\iC$ spanned by $k$-truncated objects.
\end{enumerate}
\end{dfn}

The truncation of morphisms is given by 

\begin{dfn}[{\cite[Definition 5.5.6.8]{Lur1}}]\label{dfn:ic:tr-mor}
\begin{enumerate}[nosep]
\item 
Let $k \in \bbZ_{\ge -2}$.
A simplicial map $f: X \to Y$ of Kan complexes is \emph{$k$-truncated} 
if the homotopy fibers of $f$  taken over any base point of $Y$ are $k$-truncated.

\item
A morphism $f: C \to D$ in an $\infty$-category $\iC$ is \emph{$k$-truncated} 
if for any $E \in \iC$ the simplicial map $\Map_{\iC}(E,C) \to \Map_{\iC}(E,D)$
given by the composition with $f$ is $k$-truncated in the sense of (1).
\end{enumerate}
\end{dfn}

Now we have 

\begin{fct*}[{\cite[Proposition 5.5.6.18]{Lur1}}]
For a presentable $\infty$-category $\iC$ and $k \in \bbZ_{\ge-2}$, 
the inclusion $\tau_{\le k}\iC \inj \iC$ has an accessible left adjoint.
\end{fct*}

In the last line we used

\begin{dfn}[{\cite[Definition 5.4.2.5]{Lur1}}]\label{dfn:ic:acc-func}
Let $\iC$ be an accessible $\infty$-category.
A functor $F: \iC \to \iC'$ is called \emph{accessible}
if it is $\kappa$-continuous (Definition \ref{dfn:ic:cpt}) 
for some regular cardinal $\kappa$.
\end{dfn}

Now we may introduce

\begin{dfn}\label{dfn:ic:tau_k}
For a presentable $\infty$-category $\iC$ and $k \in \bbZ_{\ge-2}$,
a left adjoint to the inclusion $\tau_{\le k}\iC \inj \iC$ is denoted by 
\[
 \tau_{\le k}: \iC \longto \tau_{\le k} \iC
\]
and called the \emph{truncation functor}.
\end{dfn}

Here we used ``the" since it is unique up to contractible ambiguity
\cite[Remark 5.5.6.20]{Lur1}.
It is obviously a localization functor (Definition \ref{dfn:ic:loc-func}).

Another usage of truncation is

\begin{dfn}\label{dfn:ic:mono}
A morphism in an $\infty$-category is a \emph{monomorphism} 
if it is $(-1)$-truncated (Definition \ref{dfn:ic:tr-mor}).
\end{dfn}

Let us cite a redefinition of monomorphisms.

\begin{fct}
\label{fct:ic:mono-dfn}
A morphism $f: X \to Y$ in an $\infty$-category $\iC$ is a monomorphism
if and only if the functor $\oc{\iC}{f} \to \oc{\iC}{Y}$ 
is fully faithful (Definition \ref{dfn:ic:sfe}).
\end{fct}

For the use in the main text, let us also cite

\begin{fct}[{\cite[Lemma 5.5.6.15]{Lur1}}]\label{fct:ic:mono}
Let $\iC$ be an $\infty$-category which admits finite limits.
Then a morphism $f: X \to Y$ in $\iC$ is a monomorphism if and only if 
the diagonal $X \to X \times_{f,Y,f} X$
is an isomorphism.
\end{fct}

Strictly speaking, this statement is a special case $k=-1$ in loc.\ cit.

\subsection{Exact functors of $\infty$-categories}

We close this section by recalling 
left and right exact functors of $\infty$-categories.

\begin{dfn}[{\cite[Definition 2.0.0.3, Definition 5.3.2.1]{Lur1}}]
\label{dfn:ic:exact}
\begin{enumerate}
\item
A simplicial map $f: X \to S$ is a \emph{left fibration} 
(resp.\ \emph{right fibration}) if $f$ has the right lifting property 
(resp.\ \emph{left lifting property}) with respect to 
all the horn inclusions $\Lambda^n_i \inj \Delta^n$  for any $0 \le i < n$.

\item 
Let $F: \iB \to \iC$ be a functor of $\infty$-categories
and $\kappa$ be a regular cardinal $\kappa$,
$F$ is \emph{$\kappa$-left exact} (resp.\ \emph{$\kappa$-right exact})
if for any left fibration (resp.\ right fibration) $\iC' \to \iC$ 
where $\iC'$ is $\kappa$-filtered (Definition \ref{dfn:ic:filtered}),
the $\infty$-category $\iB' = \iB \times_{\iC} \iC'$ is also $\kappa$-filtered.

\item
A functor of $\infty$-categories is \emph{left exact} (resp.\ \emph{right exact})
if it is $\omega$-left exact (resp.\ $\omega$-right exact),
where $\omega$ denotes the lowest transfinite ordinal number.

\item
A functor of $\infty$-category is \emph{exact}
if it is both left exact and right exact.
\end{enumerate}
\end{dfn}

We will repeatedly use the following criterion of exactness in the main text.

\begin{fct}[{\cite[Proposition 5.3.2.9]{Lur1}}]\label{fct:ic:exact}
Let $F: \iB \to \iC$ be a functor of $\infty$-categories and 
$\kappa$ be a regular cardinal.
\begin{enumerate}[nosep]
\item
If $f$ is $\kappa$-left exact, then $F$ preserves all $\kappa$-small colimits
which exists in $\iB$.

\item
If $\iB$ admits $\kappa$-small limits and $F$ preserves $\kappa$-small colimits,
then $F$ is $\kappa$-right exact.
\end{enumerate}
Dual statements hold for right-exactness.
\end{fct}

Recall the notion of adjoint functors (Definition \ref{dfn:ic:adj}).
Now we have the following criterion on existence of adjunctions.

\begin{fct}[{\cite[Corollary 5.5.2.9]{Lur1}}] 
\label{fct:ic:adj}
Let $F: \iB \to \iC$ be a functor of presentable $\infty$-categories.
\begin{enumerate}[nosep]
\item 
$F$ has a right adjoint if and only if it is right exact.
\item
$F$ has a left adjoint if and only if it is left exact.
\end{enumerate}
\end{fct}

\section{$\infty$-topoi}
\label{a:it}

In this section we give some complementary accounts on $\infty$-topoi.

\subsection{Sheaves on $\infty$-sites}

Let $\iC$ be an $\infty$-category.
Recall the Yoneda embedding $\yon: \iC \to \iPSh(\iC)$ 
in Definition \ref{dfn:pr:Yoneda}.
By \cite[Proposition 6.2.2.5]{Lur1}, for each $X \in \iC$ we have a bijection 
\begin{equation}\label{eq:ic:sieve:bij}
 \sigma_X: \Sub(\yon(X)) \longsimto \Sie(X), \quad 
 (i: U \inj \yon(X)) \longmapsto \oc{\iC}{X}(U)
\end{equation}
between the set $\Sub(\yon(X))$ of monomorphisms $U \inj \yon(X)$ 
(Definition \ref{dfn:ic:mono})
and the set $\Sie(X)$ of all sieves on $X$.
Here $\oc{\iC}{X}(U)$ denotes the full sub-$\infty$-category of $\iC$ 
spanned by those objects $f: Y \to X$ of $\oc{\iC}{X}$ such that 
there exists a commutative triangle
\[
 \xymatrix{\yon(Y) \ar[rr]^{\yon(f)} \ar[rd] & & \yon(X) \\ & U \ar[ru]_i}
\]

Now we recall

\begin{dfn}[{\cite[Definition 5.5.4.2]{Lur1}}]\label{dfn:ic:tau-local}
For an $\infty$-category $\iB$ and a set $S$ of morphisms in $\iB$,
we call $Z \in \iC$ to be \emph{$S$-local} if for any $(s: X \to Y) \in S$
the composition with $s$ induces an isomorphism 
$\Map_{\iC}(Y,Z) \simto \Map_{\iC}(X,Z)$ 
in the homotopy category $\Ho \iS$ of spaces.
\end{dfn}

Then we can introduce 

\begin{dfn}\label{dfn:it:sh}
Let $(\iC,\tau)$ be an $\infty$-site.
\begin{enumerate}[nosep]
\item 
Define the set $S_\tau$ by
\[
 S_\tau := \bigcup_{X \in \iC} \sigma_X^{-1}(\Cov_{\tau}(X)),
\]
the set of all monomorphisms $U \to \yon(X)$ 
corresponding to covering sieves ${\iC}^{(0)}_{/X}$ of $\tau$ 
under the bijection $\sigma_X$ \eqref{eq:ic:sieve:bij}.
A presheaf $\shF \in \iPSh(\iC)$ is called a \emph{$\tau$-sheaf} 
if it is $S_\tau$-local.
\item
We denote by 
\[
 \iSh(\iC,\tau) \subset \iPSh(\iC)
\]
the full sub-$\infty$-category spanned by $\tau$-sheaves.
\end{enumerate}
\end{dfn}

As explained in Fact \ref{fct:pr:sh} 
the $\infty$-category $\iSh(\iC,\tau)$ is always an $\infty$-topos.

\subsection{Yoneda embedding of $\infty$-topoi}
\label{ss:it:yon-topos}

Let $(\iC,\tau)$ be an $\infty$-site and 
$L: \iPSh(\iC) \to \iSh(\iC,\tau)$ be a left adjoint of the inclusion
$\iSh(\iC,\tau) \inj \iPSh(\iC)$ of the $\infty$-topos $\iSh(\iC,\tau)$.
Also let $\yon: \iC \inj \iPSh(\iC)$ be the Yoneda embedding 
(Definition \ref{dfn:pr:Yoneda})
We first recall

\begin{fct*}[{\cite[Lemma 6.2.2.16]{Lur1}}]
Let $(\iC,\tau)$, $L$ and $\yon$ as above,
and let $i: U \to \yon(X)$ be a monomorphism in $\iPSh(\iC)$ 
corresponding to a sieve $\oc{\iC}{X}^{(0)}$ on $X \in \iC$ 
under the bijection \eqref{eq:ic:sieve:bij}.
Then $L \circ i$ is an equivalence 
if and only if $\oc{\iC}{X}^{(0)}$ is a covering sieve.
\end{fct*}

This fact implies that given an $\infty$-site $(\iC,\tau)$
one can recover $\tau$ from $\iSh(\iC,\tau) \subset \iPSh(\iC)$.
Applying this fact to the identity $\yon(X) \to \yon(X)$ and 
recalling that $\oc{\iC}{X} \subset \oc{\iC}{X}$ is a covering sieve 
(Definition \ref{dfn:ic:inf-site} (1) (a)), 
we find that the composition 
\begin{equation}\label{eq:pr:Yoneda:iSh} 
 \iC \xrr{\yon} \iPSh(\iC) \xrr{L} \iSh(\iC,\tau) 
\end{equation}
is a fully faithful functor of $\infty$-categories.
Thus, we may say that $\tau$ is sub-canonical for the $\infty$-topos $\iS(\iC,\tau)$.
(using the terminology in the ordinary Grothendieck topology).

\begin{dfn}\label{dfn:pr:Yoneda:iSh}
The composition \eqref{eq:pr:Yoneda:iSh} 
will also be called the \emph{Yoneda embedding},
and will be denoted by the same symbol $\yon: \iC \to \iSh(\iC,\tau)$.
\end{dfn}

\subsection{Hypercomplete $\infty$-topoi}

The $\infty$-topoi of $\tau$-sheaves discussed in the previous part
has a distinguished property among general $\infty$-topoi.

\begin{dfn}[{\cite[Definition 6.5.1.10, \S 6.5.2]{Lur1}}]\label{dfn:ic:hcpl}
Let $\tT$ be an $\infty$-topos.
\begin{enumerate}[nosep]
\item
Let $n \in \bbN \cup \{\infty\}$.
A morphism $f: X \to Y$ in an $\infty$-topos $\tT$ is called \emph{$n$-connective}
if it is an effective epimorphism (Definition \ref{dfn:it:eff-epi})
and $\pi_k(f)$ is trivial for each $k=0,1,\ldots,n$.
\item
An object $X \in \tT$ is called \emph{hypercomplete} 
if it is local (Definition \ref{dfn:ic:tau-local}) 
with respect to the class of $\infty$-connective morphisms.
We denote by $\tT^{\wedge}$ the sub-$\infty$-category of $\tT$ 
spanned by hypercomplete objects.
\item 
$\tT$ is called \emph{hypercomplete} if $\tT^{\wedge}=\tT$.
\end{enumerate} 
\end{dfn}

Next we introduce the notion of hypercoverings following \cite[\S 6.5.3]{Lur1}.
See also \cite[\S 3.2]{TVe1} for a model-theoretic explanation.

Let us recall the category $\CS$ of combinatorial simplices 
(\S \ref{ss:ic:ss}) and
the nerve $\Ner(\catC)$ of an ordinary category $\catC$ 
(Definition \ref{dfn:ic:nerve}).

\begin{dfn*}[{\cite[Definition 6.1.2.2]{Lur1}}]
A \emph{simplicial object in an $\infty$-category $\iC$} 
is a simplicial map $U_{\bullet}: \Ner(\CS)^{\op} \to \iC$.
The $\infty$-category of simplicial objects in $\iC$ is denoted by $\iC_{\Delta}$.
\end{dfn*}

Following \cite[Notation 6.5.3.1]{Lur1}, for each $n \in \bbN$, we denote 
by $\CS^{\le n}$ the full subcategory of $\CS$ spanned by $\{[0],\ldots,[n]\}$.
If $\iC$ is a presentable $\infty$-category, then the restriction functor
$\sk_n: \iC_{\Delta} \to \iFun(\Ner(\CS^{\le n})^{\op},\iC)$ has a right adjoint:
\begin{align}\label{eq:ic:hyper:sk}
 \sk_n: \iC_{\Delta} \adjunc \iFun(\Ner(\CS^{\le n})^{\op}) :r.
\end{align}
In fact, $r$ is given by the right Kan extension \cite[\S4.3.2]{Lur1} 
along the inclusion functor $\Ner(\CS^{\le n})^{\op} \inj \Ner(\CS)^{\op}$.
We set
\[
 \cosk_n:= r \circ \sk_n: \iC_{\Delta} \longto \iC_{\Delta}
\]
and call it the \emph{$n$-coskeleton functor}.

Recalling that an $\infty$-topos is presentable (Fact \ref{fct:pr:Giraud}), 
we have the following definition.

\begin{dfn}[{\cite[Definition 6.5.3.2]{Lur1}}]
\label{dfn:ic:top:hc}
Let $\tT$ be an $\infty$-topos.
A simplicial object $U_{\bullet} \in \tT_{\Delta}$ is called a \emph{hypercovering} 
of $\tT$ if for each $n \in \bbN$ the unit map $U_n \to (\cosk_{n-1} U_{\bullet})_n$
coming from the adjoint \eqref{eq:ic:hyper:sk} is an effective epimorphism
(Definition \ref{dfn:it:eff-epi}).
\end{dfn}

As noted in \cite[Remark 6.5.3.3]{Lur1},
a hypercovering $U_{\bullet}$ of $\tT$ is a simplicial object 
such that the morphisms $U_0 \to \one_{\tT}$, $U_1 \to U_0 \times U_0$, 
$U_2 \to U_0 \times U_0 \times U_0$, $\ldots$ are effective epimorphisms.
Here $\one_{\tT}$ denotes a final object of $\tT$.

Next we give the definition of a geometric realization of a simplicial object.
We denote by $\CS_+$ the category of possibly empty finite linearly ordered sets.
We can regard $\CS \subset \CS_+$ as a full subcategory.

\begin{dfn*}[{\cite[Notation 6.1.2.12]{Lur1}}]
Let $\iC$ be an $\infty$-category and $U_{\bullet} \in \iC_{\Delta}$.
Regarding $U_{\bullet}$ as a diagram in $\iC$ indexed by $\Ner(\CS)^{\op}$,
we denote by 
\[
 \abs{U_{\bullet}}: \Ner(\CS_+)^{\op} \longto \iC
\] 
a colimit for $U_{\bullet}$ if it exists,
and call it a \emph{geometric realization of $U_{\bullet}$}.
\end{dfn*}

\begin{rmk*}
\begin{enumerate}[nosep]
\item
As noted in \cite[Remark 6.1.2.13]{Lur1},
$\abs{U_{\bullet}}$ is determined up to contractible ambiguity.
\item
For a hypercovering $U_{\bullet}$ of an $\infty$-topos $\tT$,
a geometric realization of $U_{\bullet}$ always exists 
since $\tT$ admits arbitrary colimits (Corollary \ref{cor:pr:lcl}),
so that the notation $\abs{U_{\bullet}}$ makes sense.
By the item (1), we call it 
\emph{the} geometric realization of $U_{\bullet}$.
\end{enumerate}
\end{rmk*}

\begin{dfn*}
A hypercovering $U_{\bullet}$ of an $\infty$-topos $\tT$ is called \emph{effective}
if $\abs{U_{\bullet}}$ is a final object of $\tT$.
\end{dfn*}

Let us cite a criterion for an $\infty$-topos to be hypercomplete.

\begin{fct}[{\cite[Theorem 6.3.5.12]{Lur1}}]\label{fct:ic:hc:eff}
For an $\infty$-topos $\tT$, the following two conditions are equivalent.
\begin{enumerate}[nosep]
\item
$\tT$ is hypercomplete.
\item 
For each $X \in \tT$, every hypercovering $U_{\bullet}$ of $\tT_{/X}$ is effective.
\end{enumerate}
\end{fct}

Now we have the following result.
See also \cite[Theorem 3.4.1]{TVe1} for a discussion in a model-theoretical context.

\begin{fct}[{\cite[Corollary 6.5.3.13]{Lur1}}]
\label{fct:ic:top:hc}
Let $\tT$ be an $\infty$-topos.
Define $S$ to be the collection of morphisms $\abs{U_{\bullet}} \to X$ 
where $U_{\bullet}$ is a hypercovering of $\tT_{/X}$ for an object $X \in \tT$.
Then an object of $\tT$ is hypercomplete if and only if 
it is $S$-local (Definition \ref{dfn:ic:tau-local}).
\end{fct}

This fact implies

\begin{cor}\label{cor:ic:iSh:hc}
For an $\infty$-site $(\iC,\tau)$,
the $\infty$-topos $\iSh(\iC,\tau)$ of sheaves is hypercomplete.
\end{cor}

Combining with the Yoneda embedding $\yon: \iC \to \iSh(\iC,\tau)$
(Definition \ref{dfn:pr:Yoneda:iSh}), 
we also have

\begin{cor}\label{cor:ic:LE:colim-rep}
For an $\infty$-site $(\iC,\tau)$, each object of $\iSh(\iC,\tau)$ 
is equivalent to a colimit of objects in the sub-$\infty$-category 
$\yon(\iC) \subset \iSh(\iC,\tau)$.
\end{cor}

\begin{proof}
Given a sheaf $F \in \iSh(\iC,\tau)$,
we can take a hypercovering $U_{\bullet}$ in $\iPSh(\iC)_{/F}$ 
such that $\abs{U_{\bullet}} \to F$ 
is equivalence by Corollary \ref{cor:ic:iSh:hc} and Fact \ref{fct:ic:hc:eff}.
Each $U_n \in \iPSh(\iC)$ is equivalent to a colimit of objects in $\yon(\iC)$
by Fact \ref{fct:ic:iPSh:gen}.
Thus $F$ is equivalent to a colimit of objects in $\yon(\iC)$.
\end{proof}

%

\subsection{Proper base change in an $\infty$-topos}
\label{ss:is:bc}

In this subsection we recall the proper base change theorem 
in an $\infty$-topos \cite[\S 7.3.1]{Lur1}.

%

Recall the $\infty$-category $\iCat$ of small $\infty$-category. 
It admits small limits and small colimits.
Now the following definition makes sense.

\begin{dfn}[{\cite[Definition 7.3.1.1, 7.3.1.2]{Lur1}}]
\label{dfn:is:bc}
A diagram 
\[
 \xymatrix{
  \iB' \ar[r]^{q'_*} \ar[d]_{p'_*} & \iC' \ar[d]^{p_*} \\
  \iB  \ar[r]_{q_*}                & \iC}
\]
of $\infty$-categories is \emph{left adjointable} 
if the corresponding diagram 
\[
 \xymatrix{
  \Ho \iB' \ar[r]^{q'_*} \ar[d]_{p_*} & \Ho \iC' \ar[d]^{q_*} \\
  \Ho \iB  \ar[r]_{q_*}               & \Ho \iC}
\]
of the homotopy categories commutes up to a specified isomorphism
$\eta : p_* q'_* \to q_* p'_*$,
the functors $q_*$, $q'_*$ of categories admit left adjoints $q^*$, ${q'}^*$
and the morphism 
\[
 \alpha : q^* p_* \xrr{u} q^* p_* q'_* {q'}^* \xrr{\eta} 
          q^* q_* p'_* {q'}^* \xrr{c} p'_* {q'}^*
\]
is an isomorphism of functors.
Here $u$ is the unit and $c$ is the counit associated to each adjunction.
We call $\alpha$ the \emph{base change morphism}.
\end{dfn}

\begin{dfn*}[{\cite[Definition 7.3.1.4]{Lur1}}]
A geometric morphism $p: \tU \to \tT$ of $\infty$-topoi 
corresponding to the adjoint pair $p^*: \tU \rlto \tT :p_*$
is \emph{proper} if for any cartesian rectangle
\[
 \xymatrix{
  \tT'' \ar[r] \ar[d] & \tT' \ar[r] \ar[d] & \tT \ar[d]^{p_*} \\
  \tU'' \ar[r]        & \tU' \ar[r]        & \tU}
\]
of $\infty$-topoi, the left square is left adjointable.
\end{dfn*}

Thus a proper geometric morphism is defined to be one 
for which the base change theorem holds.

We collect some formal properties of proper geometric morphisms.

\begin{fct*}[{\cite[Proposition 7.3.1.6]{Lur1}}]
\begin{enumerate}[nosep]
\item 
Any equivalence of $\infty$-topoi is a proper geometric morphism.
\item
The class of proper geometric morphisms is closed under
equivalence, pullback by any morphism and composition.
\end{enumerate}
\end{fct*}

In \cite[Theorem 7.3.16]{Lur1}
it is shown that for a proper map $p: X \to Y$ 
of topological spaces with $X$ completely regular
(i.e., homeomorphic to a subspace of a compact Hausdorff space), 
the associated geometric morphism $p_*: \iSh(X) \to \iSh(Y)$ is proper. 

\section{Stable $\infty$-categories}
\label{a:stb}

\subsection{Definition of stable $\infty$-categories}
\label{ss:stb:stab-icat}

In this subsection we cite from \cite[Chap.\ 1]{Lur2} 
the necessary notion and statements on stable $\infty$-category.

\begin{dfn*}[{\cite[Definition 1.1.1.1]{Lur2}}]
A \emph{zero object} of an $\infty$-category $\iC$ is an object 
which is both initial and final in the sense of Definition \ref{dfn:ic:ini-fin}.
\end{dfn*}

\begin{dfn}[{\cite[Definition 1.1.1.4, 1.1.1.6]{Lur2}}]\label{dfn:stb:fcs}
Let $\iC$ be an $\infty$-category with a zero object $0$.
\begin{enumerate}[nosep]
\item 
A \emph{triangle} in $\iC$ is a square of the form
\[
 \xymatrix{X \ar[r] \ar[d] & Y \ar[d] \\ 0 \ar[r] & Z}
\]
We sometimes denote such a triangle simply by $X \to Y \to Z$.

\item
A triangle is called a \emph{fiber sequence} (resp.\ \emph{cofiber sequence})
if it is a pullback square (resp.\ pushout square) 
in the sense of Definition \ref{dfn:ic:pb-po}.

\item
Let $f: X \to Y$ be a morphism in $\iC$.
A \emph{fiber} of $f$ is a fiber sequence  of the form 
\[
 \xymatrix{W \ar[r] \ar[d] & X \ar[d]^{f} \\ 0 \ar[r] & Y}
\]
A \emph{cofiber} of $f$ is a cofiber sequence of the form 
\[
 \xymatrix{X \ar[r]_{f} \ar[d] & Y \ar[d] \\ 0 \ar[r] & Z}
\]
\end{enumerate}
\end{dfn}

The word triangle in the above definition will be used only in this subsection
(in fact, up to Definition \ref{dfn:stb:stb}).
Note that in the main text we use the word triangle in $\iC$ to mean 
a simplicial map $\Delta^2 \to \iC$.

\begin{dfn}[{\cite[Definition 1.1.1.9]{Lur2}}]\label{dfn:stb:stb}
An $\infty$-category $\iC$ is called \emph{stable} 
if it satisfies the following three conditions.
\begin{itemize}
\item 
$\iC$ has 
a zero object $0 \in \iC$.
\item
Every morphism in $\iC$ has a fiber and a cofiber.
\item
A triangle in $\iC$ is a fiber sequence if and only if it is a cofiber sequence.
\end{itemize}
\end{dfn}

For a stable $\infty$-category $\iC$ 
one can define the suspension functor $\Sigma: \iC \to \iC$ 
and the loop functor $\Omega: \iC \to \iC$ as follows \cite[\S1.1.2]{Lur2}.

Let us assume for a while only that $\iC$ has a zero object.
Consider the full sub-$\infty$-category $\iM^{\Sigma}$
of $\iFun(\Delta^1 \times \Delta^1,\iC)$ 
spanned by the pushout squares of the form
\[
 \xymatrix{X \ar[r] \ar[d] & 0 \ar[d] \\ 0' \ar[r] & Y}
\]
with $0$ and $0'$ zero objects of $\iC$.
As explained in \cite[p.23]{Lur2}, if morphisms in $\iC$ have cofibers, 
then the evaluation at $X$ induces a trivial fibration $i: \iM^{\Sigma} \to \iC$.
Let $s: \iC \to \iM^{\Sigma}$ be a section of $i$.
Let also $f: \iM^{\Sigma} \to \iC$ be the functor given by evaluation at $Y$.

\begin{dfn}\label{dfn:stb:suspension}
Let $\iC$ be 
an $\infty$-category which has a zero object 
and where every morphism has a cofiber.
The \emph{suspension functor} $\Sigma = \Sigma_{\iC}: \iC \to \iC$ of $\iC$ 
is defined to be the composition $\Sigma := f \circ s$
of $f$ and $s$ constructed above.
\end{dfn}

Dually, denote by $\iM^{\Omega}$ the full sub-$\infty$-category
of $\iFun(\Delta^1 \times \Delta^1,\iC)$ 
spanned by the pullback squares of the above form.
If morphisms in $\iC$ have fibers, then the evaluation at the vertex $Y$
induces a trivial fibration $f': \iM^{\Sigma} \to \iC$.
Let $s': \iC \to \iM^{\Omega}$ be a section of $f'$.
Let also $i': \iM^{\Omega} \to \iC$ be the functor given by the evaluation at $X$.

\begin{dfn*}
Let $\iC$ be 
an $\infty$-category which has a zero object 
and where every morphism has a fiber.
The \emph{loop functor} $\Omega = \Omega_{\iC}: \iC \to \iC$ of $\iC$ 
is defined to be the composition $\Omega := i' \circ s'$
of $i'$ and $s'$ constructed above.
\end{dfn*}

\begin{fct*}
If $\iC$ is stable, then $\iM^{\Sigma}=\iM^{\Omega}$,
so that $\Sigma$ and $\Omega$ are mutually inverse equivalences on $\iC$.
\end{fct*}

\begin{dfn}\label{dfn:stb:[n]}
Let $\iC$ be a stable $\infty$-category.
For $n \in \bbN$, we denote by $X \mapsto X[n]$ 
the $n$-th power of the suspension functor $\Sigma$,
and by $X \mapsto X[-n]$ the $n$-th power of the loop functor $\Omega$.
We call them \emph{translations} or \emph{shifts} on $\iC$.
\end{dfn}

We also have the equivalence on the homotopy category $\Ho \iC$ 
induced by $[n]:\iC \to \iC$, which will be denoted by the same symbol $[n]$
and called translations on $\Ho \iC$.

We cite from \cite{Lur2} a construction of new stable $\infty$-category from old one.

\begin{fct*}[{\cite[Propsition 1.1.3.1, Lemma 1.1.3.3]{Lur2}}]
Let $\iC$ be a stable $\infty$-category.
\begin{enumerate}[nosep]
\item
For a simplicial set $K$,
the $\infty$-category $\iFun(K,\iC)$ of functors is stable.

\item
Let $\iC' \subset \iC$ be a full sub-$\infty$-category
which is stable under cofibers and translations. 
Then $\iC'$ is a stable subcategory of $\iC$.
\end{enumerate}
\end{fct*}

We now recall the structure of a triangulated category 
on the homotopy category $\Ho \iC$.
A diagram in $\Ho \iC$ of the form
\[
 \xymatrix{X \ar[r]^f & Y \ar[r]^g & Z \ar[r]^(0.4){h} & X[1]}
\]
is called a \emph{distinguished triangle} if there exists a diagram
$\Delta^1 \times \Delta^2 \to \iC$ of the form
\[
 \xymatrix{X  \ar[r]^{\wt{f}} \ar[d] & Y \ar[r]   \ar[d]^g & 0 \ar[d] \\
           0' \ar[r]                 & Z \ar[r]_{\wt{h}}   & W          }
\]
satisfying the following four conditions:
\begin{itemize}[nosep]
\item $0,0' \in \iC$ are zero objects.
\item Both squares are pushout square in $\iC$.
\item The maps $\wt{f}$ and $\wt{g}$ in $\iC$ 
      represent $f$ and $g$ in $\Ho \iC$ respectively.
\item $h$ is the composition of the homotopy class of $\wt{h}$ with the equivalence
      $W \simeq X[1]$ determined by the outer rectangle.
\end{itemize}

\begin{fct}[{\cite[Theorem 1.1.2.14]{Lur2}}]\label{fct:stb:stb}
Let $\iC$ be a stable $\infty$-category.
Then the translations on $\Ho \iC$ and the distinguished triangles 
give $\Ho \iC$ the structure of a triangulated category.
\end{fct}

By this fact, we know in particular that 
$\Ho \iC$ has the structure of an additive category.
For objects $X,Y$ of a stable $\infty$-category $\iC$ and an integer $n$,
we denote by $\Ext_{\iC}^n(X,Y)$ the abelian group $\Hom_{\Ho \iC}(X[-n],Y)$.

Finally we recall the notion of exact functors of stable $\infty$-categories.
See \cite[\S 1.1.4]{Lur2} for the detail.

\begin{dfn}\label{dfn:stb:exact}
A functor $F: \iC \to \iC'$ of stable $\infty$-categories is called \emph{exact} 
if the following two conditions are satisfied.
\begin{itemize}[nosep]
\item $F$ carries zero objects in $\iC$ to zero objects in $\iC'$.
\item $F$ carries fiber sequences to fiber sequences.
\end{itemize}
\end{dfn}

This notion of exactness is compatible with Definition \ref{dfn:ic:exact}
by the following fact.

\begin{fct*}[{\cite[Proposition 1.1.4.1]{Lur2}}]
For a functor $F: \iC \to \iC'$ of stable $\infty$-categories,
the following three conditions are equivalent.
\begin{enumerate}[nosep,label=(\roman*)]
\item 
$F$ commutes with finite limits.

\item
$F$ commutes with finite colimits.

\item
$F$ is exact in the sense of Definition \ref{dfn:stb:exact}.
\end{enumerate}
\end{fct*}

\subsection{$t$-structure}
\label{ss:stb:t-str}

We collect the basics of $t$-structure on a stable $\infty$-category
citing from \cite[\S 1.2]{Lur2}.

Recall the notion of $t$-structure on a triangulated category
in the sense of \cite{BBD}:

\begin{dfn}\label{dfn:stb:BBD}
A \emph{$t$-structure} on a triangulated category $\catD$ is a pair 
$(\catD_{\le0},\catD_{\ge0})$ of full subcategories 
satisfying the following three conditions.
\begin{enumerate}[nosep,label=(\roman*)]
\item 
For any $X \in \catD_{\ge0}$ and  $Y \in \catD_{\le 0}$, 
we have $\Hom_{\catD}(X,Y [-1]) = 0$.
\item
We have $\catD_{\ge 0} [1] \subset \catD_{\ge 0}$ and 
$\catD_{\le 0} [-1] \subset \catD_{\le 0}$.
\item
For any $X \in \catD$, there exists a fiber sequence $X' \to X \to X'$ 
in the nerve $\Ner(\catD)$ (see Definition \ref{dfn:stb:fcs}), 
where $X' \in \catD_{\ge 0}$ and $X'' \in \catD_{\le 0}[-1]$.
\end{enumerate}
\end{dfn}


\begin{dfn}[{\cite[Definition 1.2.1.4]{Lur2}}]\label{dfn:stb:t-str}
\begin{enumerate}[nosep]
\item
Let $\iC$ be an $\infty$-category.
A $t$-structure on $\iC$ is a $t$-structure on the homotopy category $\Ho \iC$.
\item
Let $n \in \bbZ$.
Given a $t$-structure on $\iC$, we denote by $\iC_{\ge n}$ and $\iC_{\le n}$
the full sub-$\infty$-categories spanned by those objects belonging to 
$(\Ho \iC)_{\ge n}$ and $(\Ho \iC)_{\le n}$, respectively.
\end{enumerate}
\end{dfn}

By \cite[Proposition 1.2.1.5]{Lur2}, 
for a stable $\infty$-category $\iC$ equipped with a $t$-structure, 
the full sub-$\infty$-categories $\iC_{\ge n}$ is a localization of $\iC$.
Thus, following \cite[Notation 1.2.1.7]{Lur2}, it makes sense to define the functor 
\[
 \tau_{\le n}: \iC \longto \iC_{\le n}
\]
to be a left adjoint to the inclusion $\iC_{\le n} \inj \iC$.
We also denote by $\tau_{\ge n}: \iC \to \iC_{\ge n}$
a right adjoint to the inclusion $\iC_{\ge n} \inj \iC$.

\begin{rmk*}[{\cite[Remark 1.2.1.3]{Lur2}}]
A $t$-structure on a stable $\infty$-category $\iC$ 
is determined by either of the corresponding localizations 
$\iC_{\le 0}, \iC_{\ge 0} \subset \iC$.
We call it the \emph{$t$-structure determined by $(\iC_{\le 0},\iC_{\ge 0})$}.
\end{rmk*}

As in the case of non-$\infty$-categorical case, 
we have the bounded sub-$\infty$-categories.

\begin{dfn}\label{dfn:stb:+-b}
Let $\iC$ be an $\infty$-category equipped with a $t$-structure.
We define the sub-$\infty$-categories $\iC^+, \iC^-, \iC^b$ of $\iC$ by
\[
 \iC^+ := \cup_n \iC_{\le n}, \quad
 \iC^- := \cup_n \iC_{\ge-n}, \quad
 \iC^b := \iC^+ \cap \iC^-.
\]
We call $\iC$ to be \emph{left bounded} if $\iC = \iC^+$,
\emph{right bounded} if $\iC = \iC^-$, and \emph{bounded} if $\iC = \iC^b$,
\end{dfn}

Note that these sub-$\infty$-categories $\iC^{*}$ ($* \in \{\pm,b\}$) are stable.

By \cite[Proposition 1.2.1.10]{Lur2} we have an equivalence 
$\tau_{\le m} \circ \tau_{\ge n} \to \tau_{\ge n}\circ\tau_{\le m}$
of functors $\iC \to \iC_{\le m} \cap \iC_{\ge n}$
as in the non-derived case \cite[\S 1]{BBD}.

\begin{dfn}[{\cite[Notation 1.2.1.7]{Lur2}}]
\label{dfn:stb:heart}
\begin{enumerate}[nosep]
\item 
We define the \emph{heart} of $\iC$ to be 
\[
 \iC^{\heartsuit} := \iC_{\le 0} \cap \iC_{\ge 0} \subset \iC.
\]

\item
We define
$\pi_0:=\tau_{\le 0} \circ \tau_{\ge 0}: \iC \to \iC^{\heartsuit}$

\item
For $n \in \bbZ$ we define $\pi_n: \iC \to \iC^{\heartsuit}$ to be 
the composition of $\pi_0$ with the shift functor $X \mapsto X[-n]$.
\end{enumerate}
\end{dfn}

For later use, we introduce

\begin{dfn*}[{\cite[\S 1.2.1, p.44]{Lur2}}]
Let $\iC$ be a stable $\infty$-category equipped with a $t$-structure.
The \emph{left completion} $\wh{\iC}$ of $\iC$ is a 
limit of the tower of $\infty$-categories
\[
 \cdots      \xrr{\tau_{\le 2}}  \iC_{\le 2} \xrr{\tau_{\le  1}}
 \iC_{\le 1} \xrr{\tau_{\le 0}}  \iC_{\le 0} \xrr{\tau_{\le -1}} \cdots.
\]
\end{dfn*}

As explained in loc.\ cit.,
we have the following description of $\wh{\iC}$.
Let $\Ner(\bbZ)$ denote the nerve of (the category associated to) 
the linearly ordered set $\bbZ$.
Then $\wh{\iC}$ is the full sub-$\infty$-category of $\iFun(\Ner(\bbZ),\iC)$
spanned by $F: \Ner(\bbZ) \to \iC$ such that
$F(n) \in \iC_{\le -n}$ for each  $n \in \bbZ$
and the morphism $\tau_{\le -n} F(m) \to F(n)$ 
induced by $F(m) \to F(n)$ is an equivalence 
for each pair $m \le n$.

\begin{fct*}[{\cite[Proposition 1.2.1.17]{Lur2}}]
Let $\iC$ be a stable $\infty$-category equipped with a $t$-structure.
\begin{enumerate}
\item 
The left completion $\wh{\iC}$ is stable.

\item
$\wh{\iC}$ has a $t$-structure 
determined by $(\wh{\iC}_{\le 0},\wh{\iC}_{\ge 0})$,
where $\wh{\iC}_{\le 0}$ and $\wh{\iC}_{\ge 0}$ are 
full sub-$\infty$-categories of $\wh{\iC}$ spanned by functors 
factoring through $\iC_{\le 0}$ and $\iC_{\ge 0}$ respectively.

\item
There is a canonical functor $\iC \to \wh{\iC}$,
which is exact and induces an equivalence $\iC_{\ge 0} \to \wh{\iC}_{\ge 0}$.
\end{enumerate}
\end{fct*}

\begin{dfn}\label{dfn:stb:cmpl}
A stable $\infty$-category $\iC$ equipped with a $t$-structure
is called \emph{left complete} if the canonical functor $\iC \to \wh{\iC}$
is an equivalence.
The \emph{right completion} and the \emph{right completeness} 
are defined dually.
\end{dfn}

We also give some properties of a $t$-structure.

\begin{dfn}\label{dfn:stb:t-prp}
Let $\iC$ be a stable $\infty$-category 
equipped with a $t$-structure determined by $(\iC_{\le 0},\iC_{\ge 0})$.
\begin{enumerate}[nosep]
\item 
The $t$-structure is \emph{accessible}
if $\iC_{\ge0}$ is a presentable $\infty$-category.

\item
The $t$-structure is \emph{compatible with filtered colimits}
if the $\infty$-category $\iC_{\le0}$ is stable under filtered colimits.

\item
Assume that $\iC$ is a symmetric monoidal $\infty$-category.
The $t$-structure is \emph{compatible with the symmetric monoidal structure} 
if $\iC_{\ge 0}$ contains the unit object of $\iC$ 
and is stable under tensor product.
\end{enumerate}
\end{dfn}

\subsection{Derived $\infty$-category}
\label{ss:stb:iDa}

Following \cite[\S 1.3]{Lur2} we recall a construction of 
stable $\infty$-category from an abelian category.
We will use the cohomological notation of complexes,
opposed to the homological notation in loc.\ cit.
Thus a complex $M=(M^*,d)$ is a sequence of morphisms 
\[
 \cdots \xrr{d^{-2}} M^{-1} \xrr{d^{-1}} M^0 \xrr{d^0} M^1 \xrr{d^1} \cdots
\]
such that $d^{n+1} \circ d^n=0$ for any $n \in \bbZ$.

\subsubsection{Construction}

Let $\catD$ be a dg-category over a commutative ring $k$
(Definition \ref{dfn:mod:dg}).
%
For $X,Y \in \catD$, we denote the complex of morphisms from $X$ to $Y$ 
by $\Hom_{\catD}(X,Y) = (\Hom_{\catD}(X,Y)^*,d)$.

\begin{dfn}[{\cite[Construction 1.3.16]{Lur2}}]
\label{dfn:stb:Ndg}
For a dg-category $\catD$, we define a simplicial set $\Ndg(\catD)$ as follows.
For $n \in \bbN$, we define $\Ndg(\catD)_n$ to be the set consisting of
$(\{X_i \mid 0 \le i \le n\}, \{f_I\})$ where
\begin{itemize}[nosep]
\item 
$X_i$ is an object of $\catD$.

\item 
For each subset $I=\{i_-<i_m<\cdots<i_1<i_+\} \subset [n]$ with $m \in \bbN$,
$f_I$ is an element of the $k$-module $\Hom_{\catD}(X_{i_-},X_{i_+})^{-m}$ 
satisfying 
$d f_I = \sum_{1 \le j \le m} (-1)^j (f_{I \setminus \{i_j\}} 
 - f_{\{i_j<\cdots<i_1<i_+\}}\circ f_{\{i_-<i_m<\cdots<i_j\}})$.
\end{itemize}
We omit the description of face and degeneracy maps. 
$\Ndg(\catD)$ is called the \emph{differential graded nerve}, 
or the \emph{dg-nerve}, of $\catD$.
\end{dfn}

As explained in \cite[Example 1.3.1.8]{Lur2}, 
lower simplices of $\Ndg(\catD)$ are given by 
\begin{itemize}
\item 
A $0$-simplex of $\Ndg(\catD)$ is an object of $\catD$.

\item
A $1$-simplex is a morphism $f: X_0 \to X_1$ of $\catD$,
i.e., an element $f \in \Hom_{\catD}(X_0,X_1)^0$ with $d f = 0$.

\item
A $2$-simplex consists of $X_0,X_1,X_2 \in \catD$, 
$f_{i j} \in \Hom_{\catD}(X_i,X_j)^{0}$ for $(i,j)=(0,1),(1,2),(0,2)$
with $d f_{i j}=0$,
and $g \in \Hom_{\catD}(X_0,X_2)^{-1}$ with $d g = f_{0 1} \circ f_{1 2} - f_{0 2}$.
\end{itemize}

\begin{fct*}[{\cite[Proposition 1.3.1.10]{Lur2}}]
For any dg-category $\catD$, 
the simplicial set $\Ndg(\catD)$ is an $\infty$-category.
\end{fct*}

\begin{ntn*}
Let $\catA$ be an additive category.
\begin{enumerate}[nosep]
\item
$\C(\catA)$ is the category of complexes in $\catA$.
It has a natural structure of a dg-category over $\bbZ$, 
and hereafter we consider $\C(\catA)$ as a dg-category.
\item
$\C^{+}(\catA) \subset \C(\catA)$ is the full subcategory 
spanned by complexes bounded below, i.e., 
spanned by those $M$ such that $M^n = 0$ for $n \ll 0$.

\item
$\C^{-}(\catA) \subset \C(\catA)$ is the full subcategory 
spanned by complexes $M$ bounded above, i.e., $M^n = 0$ for $n \gg 0$.
\end{enumerate}
\end{ntn*}

For the dg-category $\C(\catA)$ we can construct the dg-nerve $\Ndg(\C(\catA))$.
We can now introduce the derived $\infty$-category for abelian category
with enough projective or injective objects.

\begin{dfn}[{\cite[Definition 1.3.2.7, Variant 1.3.2.8]{Lur2}}]
\label{dfn:stb:iDm/iDp}
Let $\catA$ be an abelian category.
\begin{enumerate}[nosep]
\item
Assume $\catA$ has enough injective objects, and let $\catA_{\txinj} \subset \catA$ 
be the full subcategory spanned by injective objects.
We define the $\infty$-category $\iDp(\catA)$ to be
\[
 \iDp(\catA) := \Ndg(\C^{+}(\catA_{\txinj}))
\]
and call it the \emph{lower bounded derived $\infty$-category} of $\catA$.

\item 
Assume $\catA$ has enough projective objects, and let $\catA_{\proj} \subset \catA$
be the full subcategory spanned by projective objects.
We define the $\infty$-category $\iDm(\catA)$ to be
\[
 \iDm(\catA) := \Ndg(\C^{-}(\catA_{\proj}))
\]
and call it the \emph{upper bounded derived $\infty$-category} of $\catA$.
\end{enumerate}
\end{dfn}

As noted in \cite[Variant 1.3.2.8]{Lur2},
we have an equivalence $\iDp(\catA)^{\op} \simeq \iDm(\catA^{\op})$.
We will mainly discuss on $\iDp(\catA)$ with $\catA$ enough injectives hereafter.

For any additive category $\catA$, the $\infty$-category $\Ndg(\C(\catA))$ 
is stable by \cite[Proposition 1.3.2.10]{Lur2}.
Then one can deduce

\begin{fct*}[{\cite[Corollary 1.3.2.18]{Lur2}}]
For an abelian category $\catA$ with enough injective objects,
the $\infty$-category $\iDp(\catA)$ is stable.
\end{fct*}

We also have a description of $\iDu{\pm}(\catA)$ by localization 
(Definition \ref{dfn:ic:loc}),
which is similar to the definition of the ordinary derived category.

\begin{fct*}[{\cite[Theorem 1.3.4.4, Proposition 1.3.4.5]{Lur2}}]
Let $\catA$ be an abelian category with enough injective objects.
We denote by $W$ be the collection of those morphisms in $\C^{+}(\catA)$ 
(regarded as an ordinary category) which are quasi-isomorphisms of complexes,
and by $W_{\dg}$ quasi-isomorphisms in $\Ndg(\C^{+}(\catA))$. 
Then we have canonical equivalences 
\[
 \Ner(\C^{+}(\catA))[W^{-1}] \simeq 
 \Ndg(\C^{+}(\catA))[W_{\dg}^{-1}] \simeq \iDp(\catA).
\]
A dual description exists for an abelian category with enough projective objects.
\end{fct*}

\subsubsection{Grothendieck abelian category}
\label{ss:stb:G}

Following the terminology of \cite{Lur1,Lur2},
we say a category $\catC$ is \emph{presentable}
if it satisfies the following two conditions.
\begin{itemize}[nosep]
\item 
$\catC$ admits arbitrary small limits and small colimits.
\item
$\catC$ is generated under small colimits by a set of $\kappa$-compact objects
(Definition \ref{dfn:ic:cpt}) for some regular cardinal number $\kappa$.
\end{itemize}
See Fact \ref{fct:ic:pr=lpr} for the related claim 
on presentable $\infty$-categories.

Now let us recall 

\begin{dfn}[{\cite[Definition 1.3.5.1]{Lur2}}]\label{dfn:stb:G} 
An abelian category $\catA$ is called \emph{Grothendieck} 
if it satisfies the following conditions.
\begin{itemize}[nosep]
\item 
$\catA$ is presentable as an ordinary category.

\item
The collection of monomorphisms in $\catA$ 
is closed under small filtered colimits.
\end{itemize}
\end{dfn}

Let us also recall a model structure of
the dg-category $\C(\catA)$ of complexes in $\catA$.

\begin{fct}[{\cite[Propsoition 1.3.5.3]{Lur2}}]\label{fct:stb:Ch:model}
Let $\catA$ be a Grothendieck abelian category.
Then the category $\C(\catA)$ has the following model structure.
\begin{itemize}[nosep]
\item 
A morphism $f: M \to N$ in $\C(\catA)$ is a cofibration 
if for any $k \in \bbZ$ the map $f: M_k \to N_k$ is a monomorphism in $\catA$.
\item
A morphism $f: M \to N$ in $\C(\catA)$ is a weak equivalence 
if it is a quasi-isomorphism.
\end{itemize}
We call it the \emph{injective model structure} of $\C(\catA)$.
\end{fct}

We have the following characterization of fibrant objects 
in this model category.

\begin{fct*}
Let $\catA$ be a Grothendieck abelian category and $M \in \C(\catA)$.
If $M$ is fibrant, then each $M_n$ is an injective object of $\catA$. 
Conversely, if each $M_n$ is injective and $M_n \simeq 0$ for $n \gg 0$,
then $M$ is fibrant.
\end{fct*}

As a corollary, one can reprove that
a Grothendieck abelian category has enough injective objects
\cite[Corollary 1.3.5.7]{Lur2}.

As in Definition \ref{dfn:ic:um},
we denote by $\catA^{\circ}$ the subcategory of fibrant-cofibrant objects
in $\C(\catA)$ with the injective model structure.

\begin{dfn}[{\cite[Definition 1.3.5.8]{Lur2}}]\label{dfn:stb:iDa}
Let $\catA$ be a Grothendieck abelian category.
We define an $\infty$-category $\iDa(\catA)$ to be 
\[
 \iDa(\catA) := \Ndg(\C(\catA)^{\circ})
\]
and call it the \emph{unbounded derived $\infty$-category} of $\catA$.
\end{dfn}

Here is a list of properties of 
the unbounded derived $\infty$-category $\iDa(\catA)$.

\begin{fct}\label{fct:stb:iDa}
Let $\catA$ be a Grothendieck abelian category.
\begin{enumerate}[nosep]
\item 
$\iDa(\catA)$ is stable \cite[Proposition 1.3.5.9]{Lur2}.

\item
The natural inclusion $\iDa(\catA) \inj \Ndg(\C(\catA))$
has a left adjoint $L$ which is a localization functor
(Definition \ref{dfn:ic:loc-func})
\cite[Proposition 1.3.5.13]{Lur2}.

\item
$\iDa(\catA)$ is equivalent to the underlying $\infty$-category 
$\Nsp(\C(\catA)^{\circ})$ of $\C(\catA)$ 
regarded as a discrete simplicial model category
\cite[Proposition 1.3.5.13]{Lur2}.

\item
$\iDa(\catA)$ is presentable as an $\infty$-category 
\cite[Proposition 1.3.5.21 (1)]{Lur2}.
\end{enumerate}
\end{fct}

\subsubsection{$t$-structure}
\label{ss:stb:D-t}

Let us cite from \cite[\S 1.3]{Lur2}
the natural $t$-structures on the derived $\infty$-categories.

\begin{fct*}[{\cite[Proposition 1.3.2.19]{Lur2}}]
Let $\shA$ be an abelian category with enough injective objects.
We define $\iDp(\catA)_{\ge 0} \subset \iDp(\catA)$ to be
the full sub-$\infty$-category spanned by those objects $\shM$
such that the homology $H_n(\shM) \in \catA$ vanishes for $n<0$.
We define $\iDp(\catA)_{\le 0}$ similarly.
Then the pair 
\[
 \left(\iDp(\catA)_{\le 0},\iDp(\catA)_{\ge 0}\right)
\]
determines a $t$-structure on $\iDp(\catA)$,
and there is a canonical equivalence 
$\iDp(\catA)^{\heartsuit} \simeq \Ner(\catA)$.
\end{fct*}


\begin{fct}[{\cite[Proposition 1.3.5.21 (2), (3)]{Lur2}}]
\label{fct:stb:iDa:t-str}
Let $\catA$ be a Grothendieck abelian category.
We denote by $\iDa(\catA)_{\ge 0} \subset \iDa(\catA)$ 
the full sub-$\infty$-category spanned by 
those objects $M$ such that $H_n(M) \simeq 0$ for $n<0$.
$\iDa(\catA)_{\le 0}$ is defined similarly.
Then the pair
\[
 (\iDa(\catA)_{\le 0}, \iDa(\catA)_{\ge 0})
\]
determines a $t$-structure on $\iDa(\catA)$,
which is accessible, right complete and compatible with filtered colimits
(Definition \ref{dfn:stb:cmpl}, \ref{dfn:stb:t-prp}).
\end{fct}

%

Let us also cite the following useful result.

\begin{fct}[{\cite[Theorem 1.3.3.2]{Lur2}}]\label{fct:stb:ext-fun}
Let $\catA$ be an abelian category 
with enough injective (resp.\ projective) objects, 
$\iC$ be a stable $\infty$-category equipped with a left complete $t$-structure,
and $\iE \subset \iFun(\iDp(\catA),\iC)$ 
(resp.\ $\iE \subset \iFun(\iDm(\catA),\iC)$) be the full sub-$\infty$-category
spanned by those left (resp.\ right) $t$-exact functors which carry 
injective (resp.\ projective) objects of $\catA$ into $\iC^{\heartsuit}$. 
Then the construction 
\[
 F \longmapsto \tau_{\le 0} \circ (\rst{F}{\iDp(\catA)^{\heartsuit}})
\]
(resp.\  $F \mapsto \tau_{\le 0} \circ (\rst{F}{\iDm(\catA)^{\heartsuit}})$)
determines an equivalence from $\iE$ to the nerve of the
category of left (resp.\ right) exact functors $\catA \to \iC^{\heartsuit}$.
\end{fct}

\begin{cor*}
Let $\catA$ be an abelian category with enough injective (resp.\ projective) 
objects, and $\iC$ be a stable $\infty$-category equipped with a right 
(resp.\ left) complete $t$-structure (Definition \ref{dfn:stb:cmpl}).
Then any left (resp.\ right) exact functor $\catA \to \iC^{\heartsuit}$ 
of abelian categories can be extended to a $t$-exact functor 
$\iDp(\catA) \to \iC$ (resp.\ $\iDm(\catA) \to \iC$) 
uniquely up to a canonical equivalence.
\end{cor*}

\begin{dfn}[{\cite[Definition 1.3.3.1]{Lur2}}]\label{dfn:stb:text}
Let $f: \iC \to \iC'$ be a functor of stable $\infty$-categories 
equipped with $t$-structures.
\begin{enumerate}[nosep]
\item
$f$ is \emph{left $t$-exact} if it is exact (Definition \ref{dfn:ic:exact}) 
and carries $\iC_{\le 0}$ into $\iC'_{\le 0}$.

\item 
$f$ is \emph{right $t$-exact} if it is exact 
and carries $\iC_{\ge 0}$ into $\iC'_{\ge 0}$,

\item
$f$ is \emph{$t$-exact} if it is both left and right $t$-exact.
\end{enumerate}
\end{dfn}

Then we have another corollary of Fact \ref{fct:stb:ext-fun}
which is standard in the ordinary derived category.

\begin{cor*}[{\cite[Example 1.3.3.4]{Lur2}}]
Let $f: \catA \to \catB$ be a functor between abelian categories.
\begin{enumerate}[nosep]
\item
If $\catA$ and $\catB$ have enough injective objects and $f$ is right exact,
then $f$ extends to a left $t$-exact functor
\[
 \dR f: \iDp(\catA) \longto \iDp(\catB)
\]
which is unique up to contractible ambiguity.
We call it the \emph{right derived functor} of $f$.

\item
Dually, if $\catA$ and $\catB$ have enough projective objects
and $f$ is right exact, then $f$ extends to a right $t$-exact functor
\[
 \dL f: \iDm(\catA) \longto \iDm(\catB)
\]
which is unique up to contractible ambiguity.
We call it the \emph{left derived functor} of $f$.
\end{enumerate}
\end{cor*}

\section{Spectra and stable modules}
\label{a:sp}

Here we explain spectra, ring spectra and stable modules following \cite{Lur2}.
Let us remark that the contents in \S \ref{ss:sp:ring} will not be used 
in the main text except Fact \ref{fct:sp:spl=E}.

\subsection{Spectra}
\label{ss:sp:sp}

Let us give an important example of a stable $\infty$-category,
the \emph{$\infty$-category $\iSp$ of spectra}.
We need some preliminary definitions.

\begin{dfn*}[{\cite[Definition 7.2.2.1]{Lur1}}]\label{dfn:sp:ptd}
A \emph{pointed object} of an $\infty$-category $\iC$ is a morphism $X_*: 1 \to \iC$ 
where $1$ is a final object of $\iC$ (Definition \ref{dfn:ic:ini-fin}).
We denote by $\iC_*$ the full sub-$\infty$-category of $\iFun(\Delta^1,\iC)$
spanned by pointed objects of $\iC$.
\end{dfn*}

Recall that $\iS$ denotes the $\infty$-category of spaces 
(Definition \ref{dfn:ic:iS}).
Thus $\iS_*$ denotes the $\infty$-category of pointed objects in $\iS$.
Noting that $\iS$ has a final object, we choose and denote it by $* \in \iS$.

\begin{dfn*}[{\cite[Definition 1.4.2.5]{Lur2}}]
We denote by $\iS^{\fin} \subset \iS$ the smallest full sub-$\infty$-category 
which contains a final object $* \in \iS$ 
and is stable under taking finite colimits.
We also denote by $\iS^{\fin}_* := (\iS^{\fin})_*$ 
the $\infty$-category of pointed objects of $\iS^{\fin}$.
\end{dfn*}


We need another definition.

\begin{dfn*}[{\cite[Definition 1.4.2.1]{Lur2}}]
Let $F: \iC \to \iB$ be a functor of $\infty$-categories
\begin{enumerate}
\item 
Assume $\iC$ admits pushouts.
$F$ is called \emph{excisive} 
if $F$ carries pushout squares in $\iC$ to pullback squares in $\iB$.
\item
Assume $\iC$ has a final object $*$.
$F$ is called \emph{reduced} if $F(*)$ is a final object of $\iB$.
\end{enumerate}
\end{dfn*}

Now we can introduce

\begin{dfn}[{\cite[Definition 1.4.2.8, Definition 1.4.3.1]{Lur2}}]
\label{dfn:sp:sp}
\begin{enumerate}
\item
Let $\iC$ be an $\infty$-category admitting finite limits.
A \emph{spectrum object} of $\iC$ is defined to be 
a reduced excisive functor $F:\iS^{\fin}_* \to \iC$.
We denote by $\iSp(\iC)$ the full sub-$\infty$-category 
of $\iFun(\iS^{\fin}_*,\iC)$ spanned by spectrum objects,
and call it the \emph{$\infty$-category of spectra in $\iC$}.

\item 
A \emph{spectrum} is a spectrum object of the $\infty$-category $\iS$ of spaces.
We denote by $\iSp: = \iSp(\iS_*)$ the \emph{$\infty$-category of spectra}.
\end{enumerate}
\end{dfn}

\begin{fct}[{\cite[Corollary 1.4.2.17]{Lur2}}]
\label{fct:sp:sp-stable}
If $\iC$ is an $\infty$-category admitting finite limits,
then the $\infty$-category $\iSp(\iC)$ is stable.
\end{fct}

In particular we have the shift functor $[n]$ with $n \in \bbZ$ 
on  $\iSp(\iC)$ (Definition \ref{dfn:stb:[n]}).
Then we can introduce

\begin{dfn}[{\cite[Notation 1.4.2.20]{Lur2}}]
\label{dfn:sp:Omega^inf}
Let $\iC$ be an $\infty$-category admitting finite limits.
\begin{enumerate}[nosep]
\item
Let $S^0$ be the $0$-sphere regarded as an object of $\iS^{\fin}_*$.
We denote the evaluation functor at $S^0 \in \iS^{\fin}_*$ by
\[
 \Omega^{\infty}: \iSp(\iC) \longto \iC.
\]
\item
For $n \in \bbZ$ we denote by
\[
 \Omega^{\infty-n}: \iSp(\iC) \longto \iC
\]
the composition $\Omega^{\infty} \circ  [n]$,
where $[n]$ denotes the shift functor on $\iSp(\iC)$.
\end{enumerate}
\end{dfn}

For $n \in \bbN$, the functor $\Omega^{\infty-n}$ can be regarded as
the evaluation functor at the pointed $n$-sphere $S^n \in \iS^{\fin}_*$.

The following fact means that a spectrum object can be regarded as 
a series of pointed objects together with loop functors
as the classical homotopy theory claims.

\begin{fct}[{\cite[Proposition 1.4.2.24, Remark 1.4.2.25]{Lur2}}]
\label{fct:sp:sp-stab}
For an $\infty$-category $\iC$ admitting finite limits,
$\iSp(\iC)$ is equivalent to the limit of the tower
$\cdots \to \iC_* \xrightarrow{\Omega} \iC_* \xrightarrow{\Omega} \iC_*$
of $\infty$-categories.
\end{fct}

Using $\Omega^{\infty}$ we can endow a $t$-structure on $\iSp(\iC)$.

\begin{fct}[{\cite[Proposition 1.4.3.4]{Lur2}}]\label{fct:sp:sp-t}
Let $\iC$ be a presentable $\infty$-category, 
and $\iSp(\iC)_{\le -1} \subset \iSp(\iC)$ be the full sub-$\infty$-category
spanned by those objects $X$ 
such that $\Omega^{\infty}(X)$ is a final object of $\iC$.
Then $\iSp(\iC)_{\le -1}$ determines a $t$-structure on $\iSp(\iC)$. 
\end{fct}

Let us also recall the sphere spectrum.
By \cite[Proposition 1.4.4.4]{Lur2}, 
if $\iC$ is a presentable $\infty$-category,
then the functor $\Omega^{\infty}:\iSp(\iC) \to \iC$ admits a left adjoint
\[
 \Sigma^{\infty}: \iC \longto \iSp(\iC).
\]

\begin{dfn}\label{dfn:sp:Sigma^inf}
For $n \in \bbN$, 
we denote the composition with the shift functor $[n]$ on $\iSp(\iC)$ by
\[
 \Sigma^{\infty+n} := [n] \circ \Sigma:  \iC \longto \iSp(\iC).
\]
We also denote the image of a final object $1_{\iC} \in \iC$ 
under the functor $\Sigma^{\infty+n}$ by
\[
 S^n_{\iC} := \Sigma^{\infty+n}(1_{\iC}) \in \iSp(\iC).
\]
\end{dfn}

Obviously $\Sigma^{\infty+n}$ is a left adjoint of $\Omega^{\infty-n}$.

In the case $\iC = \iS$ we have

\begin{fct}\label{fct:stb:iSp:mon}
The stable $\infty$-category $\iSp$ of spectra has the following properties.
\begin{enumerate}
\item
$\iSp$ is freely generated by 
the sphere spectrum $S^0_{\iS}$ under (small) colimits
\cite[Corollary 1.4.4.6]{Lur2}.

\item 
$\iSp$ is equipped with the symmetric monoidal structure 
induced by smash products \cite[\S 4.8.2]{Lur2}.

\item
The heart $\iSp^{\heartsuit}$ of the $t$-structure (Fact \ref{fct:sp:sp-t})
is equivalent to the nerve of the ordinary category of abelian groups
\cite[Proposition 1.4.3.6]{Lur2}.
\end{enumerate}
\end{fct}

\subsection{Ring spectra}
\label{ss:sp:ring}

For $n \in \bbN$,
we denote by $\bbE_n^{\otimes}$ the \emph{$\infty$-operad of little $n$-cubes}
in the sense of \cite[Definition 5.1.0.2]{Lur2}.
We have a natural sequence 
\[
 \bbE_0^{\otimes} \longinj \bbE_1^{\otimes} \longinj \bbE_2^{\otimes} \longinj \cdots
\]
of $\infty$-operads.
By \cite[Corollary 5.1.1.5]{Lur2},
the colimit of this sequence is equivalent to the commutative $\infty$-operad 
\cite[Example 2.1.1.18]{Lur2}, 
which we denote by $\bbE_{\infty}^{\otimes}$.

Let $k$ be an $\bbE_{\infty}$-ring.
We denote by $\iMod_k(\iSp)$ the $\infty$-category of $k$-module spectra
(see \cite[Notation 7.1.1.1]{Lur2}),
which has a symmetric monoidal structure.

\begin{dfn}[{\cite[Notation 0.3]{Lur7}}]\label{dfn:stb:iCAlg_k}
Let $k$ be an $\bbE_{\infty}$-ring.
We denote by 
\[
 \iCAlg_k = \iCAlg_k(\iMod_{k}(\iSp))
\]
the $\infty$-category of commutative ring objects 
in the symmetric monoidal $\infty$-category $\iMod_{k}(\iSp)$,
and call an object of $\iCAlg_k$ a \emph{commutative $k$-algebra spectrum}.
If $k$ is connective, then we denote by $\iCAlg^{\cn}_k \subset \iCAlg_k$
the full sub-$\infty$-category spanned by connective commutative $k$-algebra spectra.
\end{dfn}

An ordinary commutative ring $k$ can be regarded as an $\bbE_{\infty}$-ring.
Let us explain a relationship between the $\infty$-category $\iCAlg^{\cn}_k$
and the $\infty$-category arising from \emph{simplicial commutative $k$-algebras}.

Let us denote by $\sCom_k$ the category of simplicial commutative $k$-algebras.
In other words, an object of $\sCom_k$ is 
a simplicial object in the category of commutative $k$-algebras 
(see \S \ref{ss:ic:ss}).
We have

\begin{fct}\label{fct:stb:sCom:model}
The category $\sCom_k$ has a simplicial model structure 
(Definition \ref{dfn:ic:spl-model}) determined by the following data.
\begin{itemize}[nosep]
\item 
A morphism $A_{\bullet} \to B_{\bullet}$ in $\sCom_k$ is a weak equivalence
if and only if the underlying simplicial map is a weak homotopy equivalence.
\item
A morphism in $\sCom_k$ is a fibration
if and only if the underlying simplicial map is a Kan fibration.
\end{itemize}
\end{fct}

We denote by $\sCom_k^{\circ} \subset \sCom_k$ 
the full subcategory spanned by fibrant-cofibrant objects. 
By Definition \ref{dfn:ic:um}, 
we have the underlying $\infty$-category $\Nsp(\sCom_k^{\circ})$. 

\begin{fct}[{\cite[Proposition 7.1.4.20, Warning 7.1.4.21]{Lur2},
\cite[Proposition 4.1.11]{Lur5}}]
\label{fct:sp:spl=E}
There is a functor 
\[
 \Nsp(\sCom_k^{\circ}) \longto \iCAlg^{\cn}_k
\]
of $\infty$-categories which preserves small limits and colimits
and which admits left and right adjoints.
If moreover the base ring $k$ contains $\bbQ$, then this functor is an equivalence.
\end{fct}

\subsection{Stable modules}
\label{ss:sp:sm}

Following \cite[\S 1.2.11, \S 2.2.1]{TVe2},
we introduce some terminology on stable modules of derived rings.
We will use some terminology and facts on derived $\infty$-categories
(see \S \ref{ss:stb:iDa}).

Let $k$ be a commutative ring.
For a derived $k$-algebra $A$, 
we denote by $\isMod_A$ the $\infty$-category of $A$-modules in 
the $\infty$-category $\isMod_k$ of simplicial $k$-modules 
(\S \ref{sss:dSt:isMod}).
The $\infty$-category $\isMod_A$ admits finite limits.
Thus we have the $\infty$-category $\iSp(\isMod_A)$
of spectrum objects in $\isMod_A$ (Definition \ref{dfn:sp:sp}).
It is a stable $\infty$-category in the sense of Definition \ref{dfn:stb:stb}.
We also have the suspension functor $\Sigma^{\infty}: \isMod_A \to \iSp(\isMod_A)$ 
in the sense of Definition \ref{dfn:sp:Sigma^inf}.


\begin{dfn}\label{dfn:stb:iSp(A)}
For a derived $k$-algebra $A$, we denote 
\[
 \iSp(A) := \iSp(\isMod_A)
\]
and call it the \emph{$\infty$-category of stable $A$-modules}.
For $n \in \bbZ$, the $n$-th shift functor (Definition \ref{dfn:stb:[n]}) 
on $\iSp(A)$ is denoted by $[n]$.
We also denote the suspension functor $\Sigma^{\infty}: \isMod_A \to \iSp(A)$ by 
\[
 \Sigma^{\infty}_A: \isMod_A \longto \iSp(A).
\]
\end{dfn}


Let us explain another description of $\iSp(A)$.
We start with the recollection of the \emph{normalized chain complex}
(see \cite[Chap.\ III, \S 2]{GJ} for the detail).
For a derived $k$-algebra $A \in \sCom_k$, 
we define the chain complex $N(A)$ as follows:
The graded component is defined to be
\[
 N(A)_n := \cap_{i=0}^{n-1} \Ker d_i, 
\]
where $d_i: A_n \to A_{n-1}$ denote the face maps (\S \ref{ss:ic:ss}).
The map defined by
\[
 (-1)^n d_n: N(A)_n \longto N(A)_{n-1},
\]
gives the differential on $N(A)$ 
due to the simplicial identity $d_{n-1}d_n=d_n d_{n-1}$.
The commutative ring structure on $A$ induces 
a structure of a commutative $k$-dg-algebra on the complex $N(A)$.

\begin{dfn*}
The obtained commutative $k$-dg-algebra $N(A)$ is called 
the \emph{normalized chain complex} of $A$.
\end{dfn*}

We denote by $\C(N(A))$ the dg-category of $N(A)$-dg-modules.
Considering it as the model category with the injective model structure 
(Fact \ref{fct:stb:Ch:model}), 
we have the subcategory $\C(N(A))^{\circ}$ of fibrant-cofibrant objects.
Then by taking the dg-nerve (Definition \ref{dfn:stb:Ndg}),
we have an $\infty$-category $\Ndg(\C(N(A))^{\circ})$.
We will mainly use the next description of $\iSp(A)$ in the following presentation.
 
\begin{lem*}
We have an equivalence of $\infty$-categories
\[
 \Ndg(\C(N(A))^{\circ}) \simeq \iSp(A).
\]
\end{lem*}

Let $\Lambda$ be a commutative $k$-algebra.
Then we can identify $\C(N(\Lambda)) = \C(\Lambda)$,
so that we have 

\begin{fct}\label{fct:sp:SP=Dinf}
There is an equivalence 
\[
 \iSp(\Lambda) \simeq \iDa(\Lambda) := \Ndg(\C(\Lambda)^{\circ}).
\]
\end{fct}

We call $\iDa(\Lambda)$ the \emph{derived $\infty$-category of $\Lambda$-modules}.
See Definition \ref{dfn:stb:iDa} for the detail,
where the derived $\infty$-category is denoted by $\iDa(\Mod_\Lambda)$.
Taking the homotopy groups in the above equivalence, we have 

\begin{cor*}
\[
 \Ho \iSp(\Lambda) \simeq D(\Lambda) \simeq \Hot \C(\Lambda),
\]
where $D(\Lambda)$ denotes the derived category of 
unbounded complexes of $\Lambda$-modules (in the ordinary sense)
and $\Hot \C(\Lambda)$ denotes the homotopy category 
of the dg-category $\C(\Lambda)$
equipped with the model structure in Fact \ref{fct:stb:Ch:model}.
\end{cor*}

We have the theory of $t$-structures for stable $\infty$-categories.
See \S \ref{ss:stb:t-str} for an account.
Applying Fact \ref{fct:stb:iDa:t-str} to the present situation,
we obtain the following $t$-structure of $\iSp(\Lambda)$.

\begin{lem}\label{lem:stb:iSp}
Let $\Lambda$ be a commutative $k$-algebra.
Then $\iSp(\Lambda) \simeq \iDa(\Lambda)=\Ndg(\C(\Lambda)^{\circ})$
is a stable $\infty$-category with a $t$-structure determined by
\[
 (\iSp(\Lambda)_{\le 0}, \iSp(\Lambda)_{\ge 0}).
\]
Here we set
\[
 \iSp(\Lambda)_{\ge 0} := \Ndg(\C(\Lambda))_{\ge 0} \cap \iDa(\Lambda), \quad
 \iSp(\Lambda)_{\le 0} := \Ndg(\C(\Lambda))_{\le 0} \cap \iDa(\Lambda), 
\]
where $\Ndg(\C(\Lambda))_{\ge n}$ (resp.\ $\Ndg(\C(\Lambda))_{\le n}$) 
denotes the full sub-$\infty$-category of $\Ndg(\C(\Lambda))$
spanned by $\Lambda$-dg-modules $M$ 
such that $H_k(M) \simeq 0$ for $k<n$ (resp.\ $k>n$).
\end{lem}

%
%


\end{document}